%% file: triangrep.tex
\documentclass[a4paper,11pt,oneside,reqno]{amsart}
\input{packages}

\input{definitions}

\numberwithin{equation}{section}
\AtBeginDocument{\addtocontents{toc}{\protect\thispagestyle{empty}}}

\author{Martina Fruttidoro}
\title{Non-regular trianguline representations}
\date{}
\begin{document}

\maketitle

\begin{abstract}
    Let $\absGal$ be the absolute Galois group of $\Qp$ and let $L$ be a finite extension of $\Qp$. 
    Moreover let $\bar\rho:\absGal\rightarrow GL_n(k_L)$ be a continous representation of $\absGal$, 
    where $k_L$ is the residue field of $L$.
    We find sufficient conditions for which a trianguline representation $\rho:\absGal\rightarrow GL_n(L)$ 
    lifting $\bar\rho$ is a point of the trianguline variety associated to $\bar\rho$.
\end{abstract}

\tableofcontents
\thispagestyle{empty}

\section*{Introduction}
\setcounter{page}{1}
  The $p$-adic local Langlands programme (for $p$ a prime number) conjectures a correspondence between
  $p$-adic Galois representations and certain $p$-adic representations of reductive groups over $p$-adic local fields.
  The theory of $p$-adic automorphic forms allows us to attach a $p$-adic Galois representation
  to a certain $p$-adic automorphic form; thus, this approach provides a plausible candidate for the
  correspondence between $p$-adic Galois representations that arise in this way
  and $p$-adic automorphic representations of $GL_n$.
  One nice aspect of this theory is that $p$-adic automorphic forms come in families parametrised by
  geometric objects called \emph{eigenvarieties}.
  In \cite[§3.2]{breuil2017interpretation}, Breuil, Hellmann and Schraen construct an analogue of an eigenvariety
  called \emph{patched eigenvariety} using the \emph{patching module} constructed by
  Caraiani, Emerton, Gee, Geraghty, Paškūnas and Shin in \cite{caraiani2013patching}.
  In \cite{breuil2017interpretation}, the authors also construct an analogue of an
  eigenvariety in the world of Galois representations, called \emph{trianguline variety}.
  Both the patched eigenvariety and the trianguline variety are defined as
  subspaces of $\defspace\times\T^n$ (see below for definitions of these spaces)
  and conjecturally they coincide (up to power series-rings over deformation rings at auxiliary places away from $p$).
  One of the main results of this thesis is the description of the underlying set of points of the
  trianguline variety in some special cases: we will show that in these cases,
  for a trianguline representation $\rho$, the representation $\rho$ is a point of the trianguline variety.
  Despite what the name might suggest, the above statement is highly non-trivial:
  as we will see, the trianguline variety is by definition the closure of the set of trianguline
  representations with a regularity condition on the parameters and
  it is unclear whether such closure adds all trianguline representations.
  In the following, I will give a more precise summary of the main results.

  Let $L$ be a finite extension of $\Qp$ and let $A$ be an affinoid $L$-algebra.
  Let us denote by $\absGal$ the absolute Galois group of $\Qp$ and let us fix a natural number $n\geq 1$.
  One of the main results in the theory of $p$-adic Galois representations is that there
  exists a fully faithful and exact functor
  \[
    \D^\dag_\rig:\left\{
      \begin{array}{c}
        \text{continous representations }\\
        \rho:\absGal\rightarrow GL_n(A)
      \end{array}\right\}
    \rightarrow\left\{
      \begin{array}{c}
        \text{étale }\phigam\text{-modules}\\
        \text{over }\robba_A\text{ of rank }n
      \end{array}\right\}.
  \]
  A continous representation $\rho:\absGal\rightarrow GL_n(L)$ is
  \emph{trianguline} if $\D^\dag_\rig(\rho)$ is \emph{triangulable},
  i.e. it is a successive extension of $\phigam$-modules of rank 1.
  This kind of representations constitute a very important category of representations
  in the framework of the $p$-adic local Langlands programme: they first appeared in the
  work of Colmez \cite{colmez2008representations}, who showed that 2-dimensional trianguline representations
  are Zariski dense in the deformation space of 2-dimensional $p$-adic Galois representations.
  This result was crucial for proving the $p$-adic local Langlands correspondence
  for $GL_2(\Qp)$.

  In \cite{kedlayaCohomology}, Kedlaya, Pottharst and Xiao prove that it is possible to
  classify $\phigam$-modules of rank 1 in a nice way:
  they show that there exists a bijection
  \begin{align*}
    \{\text{continous characters }\delta:\Qp^\times\rightarrow A^\times\}
    &\leftrightarrow\{\text{free }\phigam\text{-modules of rank }1\text{ over }\robba_A\}\\
    \delta&\mapsto\robba_A(\delta).
  \end{align*}
  The functor $\Hom_{\rr{cont}}(\Qp^\times,-)$ of continous characters of $\Qp^\times$ is representable by a rigid space $\T$,
  which constitutes then a nice moduli space for $\phigam$-modules of rank 1.
  Now let us fix a residual representation $\bar\rho:\absGal\rightarrow GL_n(k_L)$
  (where $k_L$ is the residue field of $L$) and let us denote by $\defspace$ the
  rigid analytic generic fibre of its universal framed deformation space.
  Let us define
  \[
    \regtrivar\coloneqq\left\{
      (\rho,\delta_1,\dots,\delta_n)\in\defspace\times\T^\reg_n\colon
        \begin{array}{c}
          \rho\text{ is trianguline}\\
          \text{with graded pieces}\\
          \text{of a filtration of }\drig(\rho)\\
          \text{ given by }(\delta_1,\dots,\delta_n)
        \end{array}
    \right\}\subset\defspace\times\T^n,
  \]
  where $\T^\reg_n$ is a subspace of $\T^n$ of tuples of characters satisfying a certain
  regularity condition (see (\ref{def of Tnreg}) in the body of the paper).
  The \emph{trianguline variety} is the Zariski closure of $\regtrivar$ inside $\defspace\times\T^n$
  and is denoted by $\trivar$.
  This space was first introduced by Hellmann in \cite{hellmann2012families}
  and its geometry has been studied in \cite{breuil2017smoothness},\cite{breuil2017interpretation}
  and \cite{breuil2019local}, yet there still are some properties to be investigated.
  Consider for example a trianguline representation $\rho:\absGal\rightarrow GL_n(L)$
  and let $(\delta_1,\dots,\delta_n)\in\T^n(L)$ be \emph{parameters} for $\rho$, which means that
  $\drig(\rho)$ has a filtration with graded pieces $\robba_L(\delta_i)$;
  it is well-known that up to conjugation by a matrix in $GL_n(L)$, we can assume $\rho$
  has values in $GL_n(\mathcal{O}_L)$.
  Let $\bar\rho:\absGal\rightarrow GL_n(k_L)$ be the reduction of $\rho$ to the residue field of $L$.
  We can then ask ourselves whether $(\rho,\delta_1,\dots,\delta_n)$ is a point of $\trivar(L)$ or not.

  \begin{conj}
    Let $\rho:\absGal\rightarrow GL_n(L)$ be a trianguline representation with parameters
    $\underline{\delta}\coloneqq(\delta_1,\dots,\delta_n)\in\mathcal{T}^n(L)$ and let
    $\bar\rho:\absGal\rightarrow GL_n(k_L)$ be the reduction of $\rho$ to the residue field of $L$.
    Then $(\rho,\underline{\delta})\in\trivar(L)$.
  \end{conj}

  Of course if $\underline\delta\in\mathcal{T}_n^{\reg}(L)$, then we have that
  $(\rho,\underline{\delta})\in\regtrivar(L)$, hence in particular $(\rho,\underline\delta)\in\trivar(L)$.
  But as remarked before,
  the above conjecture is not obvious at all in case $(\delta_1,\dots,\delta_n)\notin\mathcal{T}_n^{\reg}(L)$.
  We will prove that $(\rho,\underline{\delta})\in\trivar$ if either $n\leq 4$, or the number
  of non-regularities in $\underline{\delta}$ is at most 5, or if $\underline{\delta}$
  doesn't present non-regularities of type $\T^+$ (see \ref{def of T+} for a definition of $\T^+$).
  More precisely, we will prove the following Theorems.

  \begin{restatable}{thmx}{result}
    \label{RESULT1}
    Let $\rho:\absGal\rightarrow GL_n(\mathcal{O}_L)$
    be a trianguline representation of parameters $(\delta_1,\dots,\delta_n)\in\mathcal{T}^n(L)$.
    Moreover let $\bar\rho:\absGal\rightarrow GL_n(k_L)$ be the reduction of $\rho$ to the residue field of $L$.
    If $\delta_i/\delta_j\notin\mathcal{T}^+(L)$ for all $i<j$, then
    $(\rho,\delta_1,\dots,\delta_n)\in\trivar(L)$.
  \end{restatable}

  \begin{restatable}{thmx}{difficultresult}
    \label{result2}
    Let $\rho:\absGal\rightarrow GL_n(\mathcal{O}_L)$
    be a trianguline representation of parameters $(\delta_1,\dots,\delta_n)\in\mathcal{T}^n(L)$.
    Moreover let $\bar\rho:\absGal\rightarrow GL_n(k_L)$ be the reduction of $\rho$ to the residue field of $L$.
    \begin{enumerate}
      \item
        If $n\leq4$, then $(\rho,\delta_1,\dots,\delta_n)\in\trivar(L)$.
      \item
        If $|\{\delta_i/\delta_j\in\T\setminus\T^\reg\colon i<j\}|\leq5$, then $(\rho,\delta_1,\dots,\delta_n)\in\trivar(L)$.
    \end{enumerate}
  \end{restatable}

  For $\phigam$-modules there is a notion of $\phigam$-cohomology, which coincides with Galois cohomology
  in case the $\phigam$-module comes from a family of Galois representations.
  In the statement of Theorem \ref{RESULT1}, the space $\T^+$ is the subspace of $\T$ consisting of
  characters $\delta$ such that the second $\phigam$-cohomology is non-trivial.

  The strategy used to prove the above results is the following: given a trianguline representation
  $\rho:\absGal\rightarrow GL_n(L)$ and $\underline{\delta}\in\T^n(L)$ characters associated
  to a filtration of $\drig(\rho)$, we will show that $(\rho,\underline{\delta})$
  is in the Zariski closure of some ``regular family''.
  More precisely, for each given $\rho$, we will construct a rigid space $X_n$ with a family of triangulable
  $\phigam$-modules $\D_n$ over $X_n$ such that
  \begin{itemize}
    \item
      the family $\D_n$ specializes to $\drig(\rho)$ at some point $x\in X_n(L)$;
    \item
      the triangulation of $\D_n$ gives rise to characters $(\widetilde{\delta}_1,\dots,\widetilde{\delta}_n)\in\T^n(X_n)$
      such that there exists an open dense $U\xhookrightarrow{j}X_n$ for which $j^*\widetilde{\underline{\delta}}\in\T^\reg_n(U)$.
  \end{itemize}
    For a detailed explanation of how the strategy implies the desired results, refer to Theorem \ref{strategy}.

    Following a construction in \cite{emerton2019moduli}, when two
    $\phigam$-modules satisfy some condition on their $\phigam$-cohomology,
    we will be able to construct
    a space with a family of extensions of the two given $\phigam$-modules.
    More precisely,
    given two $\phigam$-modules $D_1$ and $D_2$ over $\robba_A$ such that $H^2_{\phi,\Gamma}(D_1\otimes D_2^\vee)=0$, in
    Section \ref{constructing extensions} we explain how to construct a vector bundle $X$
    over $\Sp(A)$ with a family of $\phigam$-modules $\D$ over $X$ such that for any any rigid space $Y\xrightarrow{g}\mathrm{Sp}(A)$
    over $\mathrm{Sp}(A)$, we have a surjective map
    \begin{align*}
    \Psi_Y:X(Y)&\twoheadrightarrow\mathrm{Ext}^1_{\phi,\Gamma}(g^* D_2,g^* D_1)(Y)\\
    f&\mapsto f^*\D
    \end{align*}
    which is functorial in $Y$ (see Theorem \ref{extensions}).
    It follows from this that if we are in the hypothesis of Theorem \ref{RESULT1}, which means
    that $H^2_{\phi,\Gamma}(\robba_L(\delta_i)\otimes\robba_L(\delta_j)^\vee)=0$
    (i.e. $\delta_i/\delta_j\notin\T^+(L)$) for all $i<j$, then there exists a vector bundle $X_n$
    over an appropriate neighbourhood $\mcal U\subset\T^n$ of $\underline{\delta}$
    with a family of $\phigam$-modules $\D_n$ specializing to $\drig(\rho)$, as stated in Proposition \ref{n dim T-}.
    Since the regular locus is open and dense inside $\T^n$, we have that the pullback of such locus to $X_n$
    is open and dense, as $X_n$ is simply a vector bundle over $\mcal U$
    (in fact the preimage of a dense subspace is dense if the map is open).

    The above is enough to conclude the argument in the hypothesis of Theorem \ref{RESULT1},
    but things complicate a great deal when $\delta_i/\delta_j\in\T^+(L)$ for at least one pair
    $(i,j)$ with $i<j$.
    Indeed in this case, we cannot use Theorem \ref{extensions}:
    in fact the fiber dimension of the second $\phigam$-cohomology is not locally constant,
    hence we cannot find a family such that the base change for the first $\phigam$-cohomology
    is surjective.
    In order to get around this problem, we will make use of an equivalence of categories
    which is an analogue for $\phigam$-modules of the Theorem of Beauville-Laszlo;
    more precisely, the equivalence is between $\phigam$-modules over $\robba_A$ and the data consisting of pairs of
    $\phigam$-modules over $\robba_A[1/t]$ (where $t$ is a particular element of $\robba_A$)
    and $\Gamma$-stable lattices of $A((t))$.
    It should be emphasized that the above statement is a simplification (and false as stated) and the reader
    should refer to Section \ref{beauville-laszlo}
    for the correct statement and a detailed explanation.
    The main point in using this equivalence is that passing from
    $\phigam$-modules over $\robba_A$ to $\phigam$-modules over $\robba_A[1/t]$
    annihilates the second $\phigam$-cohomology; thus it is possible to use Theorem \ref{extensions}
    again in order to construct a rigid space $X_n'$ with a family of $\phigam$-modules $\D_n'$
    specializing to $\drig(\rho)[1/t]$ at some point $x'\in X_n'(L)$.
    As a remark, this is not enough for our purpose, as we want a family of
    $\phigam$-modules specializing to $\drig(\rho)$.
    We will then define the space $X_n$ inside $X_n'\times\rr{Gr}_n$ (where $\rr{Gr}_n$ is the affine Grassmannian)
    consisting of pairs $(D', \Lambda)$, where $D'$ is given
    by the universal $\phigam$-module $\D_n'$ over $X_n'$ and $\Lambda\subset D'[1/t]\otimes_{\robba_A^r[1/t]}A((t))$
    is a $\Gamma$-stable lattice
    (see Theorem \ref{construction of X} and Theorem \ref{construction of X_n}).
    If we let $\D_n$ be the sheaf on $X_n$ corresponding to the pullback to $X_n$ of the pair of $\D_n'$ and
    of the universal lattice over $\rr{Gr}_n$, we have that $\D_n$ specializes to $\drig(\rho)$ at some point $x\in X_n(L)$,
    by virtue of the equivalence written above.
    We still need to show that the family $\D_n$ is densely regular,
    i.e. there exists a Zariski open dense of $X_n$ such that the pullback of $\D_n$
    to such open has regular parameters.
    Now the main issue is that $X_n$ is not a vector bundle over an open $\mcal U$ of $\T^n$,
    hence the pullback of the regular locus $\T_n^\reg$ to $X_n$ might not be dense in $X_n$:
    indeed $X_n$ is not even flat over $\T^n$, as we will see in Remark \ref{not flatness and dim X'}.
    In Section \ref{density of regular locus in dim 2} and Section \ref{density of the regular locus},
    we will study the geometry of the space $X_n$ and give a bound to the dimension of the non-regular locus
    inside $X_n$ when $n\leq 4$ or $|\{\delta_i/\delta_j\in\T\setminus\T^\reg\colon i<j\}|\leq5$
    as in the hypothesis of Theorem \ref{result2}.
    This will allow us to conclude that $X_n$ is irreducible and that
    the pullback of the regular locus $\T_n^\reg$ is dense inside $X_n$, concluding the
    proof of Theorem \ref{result2}.

    \hfill\break

    \textbf{Structure of the paper:}
    In section \ref{section 1} we recollect basics and well-known results concerning $\phigam$-modules.
    section \ref{section 2} is devoted to defining trianguline representations and the trianguline variety.
    The outline of the strategy we use to prove Theorem \ref{RESULT1} and Theorem \ref{result2}
    is explained in section \ref{section 3}.
    In section \ref{section 4} we first explain how to construct a family of extensions of two $\phigam$-modules
    (\S\ref{constructing extensions}) and in \S\ref{section 4.2} we prove Theorem \ref{RESULT1}.
    The rest of the paper is dedicated to the proof of Theorem \ref{result2}:
    section \ref{section 5} provides all the tools that we will need in the course of the proof;
    section \ref{section 6} gives a proof of Theorem \ref{result2} when $n=2$
    and finally section \ref{section 7} finishes the proof of Theorem \ref{result2} generalizing (when possible)
    the results in section \ref{section 6}.

\section*{Acknowledgements}

  My first thanks go to my supervisor Prof. Eugen Hellmann for introducing me to the topics of the
  thesis and for his invaluable guidance, continuous support and patience during my PhD study.

  I would like to offer my special thanks to Dr. Zhixiang Wu for patiently explaining me
  some results from his PhD thesis, which are essential for this dissertation.

  I am very grateful to Luisa Meinke for the mathematical and emotional support
  and for the fun we had together.

  Finally my thanks go to the most important people of all: my friends and family.
  I want to thank Stefano, Andrea, Margherita, Beatrice, Rafaela and my parents
  for encouraging and supporting me during these years,
  but mostly just for being present in my life.

\section{Preliminaries on $(\phi,\Gamma)$-modules}
\label{section 1}

One of the most powerful tools in the study of the theory of $p$-adic Galois representations
is the theory of $\phigam$-modules, providing a tool of investigating such representations
in terms of semi-linear algebra objects.
The category of these modules (with some additional condition) is in fact equivalent
to the category of $p$-adic Galois representations, as will be more precisely stated in this section.
In this section we discuss some basic notions on $\phigam$-modules and their cohomology
as well as their variation in families.
Indeed, in \cite{berger2008familles} Berger and Colmez introduce the notion of families of overconvergent $\phigam$-modules
associated to families of $p$-adic Galois representations.
The connection between these two categories will be better investigated in the second section.

From now on, let $p$ be a prime number.
Let $L$ be a finite extension of $\mathbb{Q}_p$, $\mathcal{O}_L$ its ring of integers and $k_L$ its residue field.

\subsection{The Robba ring}

    For $r=p^{-a},s=p^{-b}\in p^{\mathbb{Q}}$ with $0\leq r\leq s<1$, let
    \[
    \mathbb{B}^{[r,s]}=\{T:r\leq |T|\leq s\}
    =\mathrm{Spa}(\mathbb{Q}_p\langle p^{-b}T,p^aT^{-1}\rangle,\mathbb{Z}_p\langle p^{-b}T,p^aT^{-1}\rangle)
    \]
    be the closed annulus of inner radius $r$ and outer radius $s$ over $\mathbb{Q}_p$.
    Let $\mathcal{R}^{[r,s]}$ denote the ring of analytic functions of $\mathbb{B}^{[r,s]}$
    and let $\mathcal{A}^{[r,s]}$ denote the subring of functions whose valuation has maximal value 1.
    For $0\leq r\leq s'\leq s<1$, we have $\mathcal{R}^{[r,s]}\subset\mathcal{R}^{[r,s']}$ simply defined by restriction.
    We let
    \[
    \mathcal{R}^r=\bigcap_{r\leq s<1}\mathcal{R}^{[r,s]}
    \]
    be the ring of analytic functions converging on
    the half-open annulus
    \[
    \mathbb{B}^r=\{T:r\leq |T|<1\}
    \]
    and $\mathcal{A}^r$ the subring of $\mathcal{R}^r$ consisting of functions whose valuation has maximal value 1.
    As in \cite{kedlaya2006finiteness}, for $r\leq s<1$, we can define the valuation
    \[
    w_s\left (\sum_{n\in \mathbb{Z}}a_nT^n\right ):=\mathrm{min}_{n\in\mathbb{Z}}\{v_p(a_n)+ns\};
    \]
    then $\robba^r$ is complete with respect to the Fr\'echet topology defined by these valuations.

    \begin{dfn}
    The \emph{Robba ring over $\mathbb{Q}_p$} is the ring
    \begin{align*}
        \robba&=\bigcup_{0\leq r<1}\mathcal{R}^r\\
        &=\left\{ f(T) = \sum_{n\in\mathbb{Z}} a_nT^n \colon
        \begin{array}{l}
        a_n \in \mathbb{Q}_p,
        a_nr(f)^n\xrightarrow[n\rightarrow-\infty]{}0\text{ for some } 0\leq r(f)<1\\
        \text{ and }
        a_ns^n\xrightarrow[n\rightarrow+\infty]{}0\text{ for all } r(f)\leq s<1
        \end{array}
        \right\}
    \end{align*}
    of power series converging on some annulus.
    We equip the Robba ring with the direct limit of the Fr\'echet topologies on $\robba^r$.
    Moreover let $\mathcal{A}^\dag\coloneqq\bigcup_{0\leq r<1}\mathcal{A}^r$
    and equip it with the subspace topology.
    \end{dfn}

    We can extend the definition of the Robba ring in order to have coefficients in some affinoid $\mathbb{Q}_p$-algebra $A$, so we define
    \[
    \begin{array}{lr}
        \robba_A\coloneqq\robba\widehat\otimes_{\mathbb Q_p}A,
        &\mathcal{A}^\dag_A\coloneqq\mathcal{A}^\dag\widehat\otimes_{\mathbb Z_p}A,
    \end{array}
    \]
    where the completed tensor product is the completion of the ordinary tensor product
    with respect to the $p$-adic topology on $A$ and the topology on $\robba$ and $\mathcal{A}^\dag$ respectively.
    For more precise definitions, we refer to \cite[§6.1]{hellmann2013arithmetic}.

    We want to equip the Robba ring with a Frobenius and some action of $\Gamma:=\mathbb{Z}_p^\times$.
    This will be useful later to define the semi-linear structure of $\phigam$-modules over the Robba ring.
    We define the $A$-linear endomorphism
    \begin{equation}
    \label{phi over R}
    \phi:\robba_A\rightarrow\robba_A, T\mapsto (1+T)^p-1
    \end{equation}
    and an $A$-linear action
    \[
    \Gamma\times\robba_A\rightarrow\robba_A, (\gamma,T)\mapsto (1+t)^\gamma-1=\sum_{n\geq1}\binom{\gamma}{n}T^n.
    \]
    The two operators commute with each other and they are both continuous.

    Moreover we set
    \[
    t\coloneqq\mathrm{log}(1+T)=\sum_{n=1}^\infty(-1)^{n+1}\frac{T^n}{n}\in\robba.
    \]
    Then $\phi(t)=pt$ and $\gamma\cdot t=\gamma t$ for all $\gamma\in\Gamma$.

    More generally, one can define the notion of \emph{relative Robba ring} over some rigid space $X$ in the following way:
    for $r,s\in p^{\mathbb{Q}}$ with $0\leq r\leq s<1$, let $\mathbb{B}_X^{[r,s]}=X\times_{\mathbb{Q}_p}\mathbb{B}^{[r,s]}$
    and $\mathbb{B}_X^r=X\times_{\mathbb{Q}_p}\mathbb{B}^r$; then define $\robba_X^{[r,s]}=(\mathrm{pr}_X)_*\mathcal{O}_{\mathbb{B}_X^{[r,s]}}$
    and $\robba^r_X=(\mathrm{pr}_X)_*\mathcal{O}_{\mathbb{B}_X^r}$. The relative Robba ring over $X$ is the sheaf
    \[
    \robba_X\coloneqq\lim_{r\rightarrow1}\robba^r_X.
    \]
    The sheaf $\mathcal{R}_X$ is endowed with a continuous $\mathcal{O}_X$-linear endomorphism $\phi:\robba_X\rightarrow\robba_X$
    and a continuous $\mathcal{O}_X$-linear action $\Gamma\times\robba_X\rightarrow\robba_X$.
    In the same way we can define $\mathcal{A}_X^{[r,s]}=(\mathrm{pr}_X)_*\mathcal{O}^+_{\mathbb{B}_X^{[r,s]}}$
    and $\mathcal{A}^r_X=(\mathrm{pr}_X)_*\mathcal{O}^+_{\mathbb{B}_X^r}$. Finally, let
    \[
    \mathcal{A}^\dag_X\coloneqq\lim_{r\rightarrow1}\mathcal{A}^r_X.
    \]

\subsection{$\phigam$-modules}

    Now we introduce the notion of $\phigam$-modules as in \cite{chenevierDensite}.

    \begin{dfn}
    Let $L$ be a finite extension of $\mathbb{Q}_p$.
    \begin{enumerate}[(i)]
        \item
        A \emph{$(\phi,\Gamma)$-module over $\mathcal{R}_L$} is a locally free $\mathcal{R}_L$-module $D$ of finite presentation endowed with
        \begin{itemize}
        \item
            a $\mathcal{R}_L$-semilinear endomorphism $\Phi:D\rightarrow D$ such that $\mathcal{R}_L\cdot\Phi(D)=D$;
        \item
            a continuous $\mathcal{R}_L$-semilinear
            $\Gamma$-action commuting with $\Phi$.
        \end{itemize}
        \item
        A \emph{morphism of $(\phi,\Gamma)$-modules} is a morphism of $\mathcal{R}_L$-modules respecting the actions of $\Gamma$ and $\Phi$.
    \end{enumerate}
    \end{dfn}

    \begin{rem}
    Let us explain in more detail what we mean by \emph{continuous} action of $\Gamma$ on $D$ in the above definition.
    We ask that there exists an $\mathcal{R}_L$-basis $\mathcal{E}=\{e_1,\dots,e_d\}$ of $D$ and $r\in[0,1]$ such that the map
    \begin{align*}
        \Gamma&\rightarrow \mathrm{Mat}_d(\mathcal{R}_L)\\
        \gamma&\mapsto\mathrm{Mat}_{\mathcal{E}}(\gamma)
    \end{align*}
    has values in $\mathrm{Mat}_d(\mathcal{R}_L^r)$ and is continuous.
    \end{rem}

    The motivation behind the study of $\phigam$-modules is the following well-known Theorem
    due to Fontaine, Cherbonnier, Colmez and Berger.

    \begin{thm}
    \label{drig}
    There is a fully faithful functor
    \[
        \drig:\{L\text{-linear continuous representations of }\absGal\}\rightarrow\{\phigam\text{-modules over }\robba_L\}
    \]
    commuting with base change.
    \end{thm}

    For a reference look at \cite[Theorem 2.2.2]{berger2011trianguline}.

    \begin{dfn}
    The $\phigam$-modules in the essential image of the funtor $\drig$ are called \emph{étale $\phigam$-modules}.
    \end{dfn}

    The consequence of Theorem \ref{drig} is that the functor
    \begin{equation}
    \label{equivalence}
    \drig:\left\{\begin{array}{c}
        L\text{-linear continuous}\\
        \text{representations of }\absGal
    \end{array}\right\}
    \rightarrow\{\text{étale }\phigam\text{-modules over }\robba_L\}
    \end{equation}
    is an equivalence of categories.
    Let us remark that the above functor is not an equivalence when we work with families of Galois representations
    and families of $\phigam$-modules because it is not essentially surjective, as we will see in \S \ref{def and basic properties}.

    As mentioned before, the theory of $\phigam$-modules is compatible with its variation in analytic families.

    \begin{dfn}
    Let $X$ be a rigid space.
    A \emph{family of $(\phi,\Gamma)$-modules over $X$} is an $\mathcal{R}_X$-module
    $\D$ which is locally on $X$ finite projective over $\mathcal{R}_X$ with
    \begin{itemize}
        \item
        a $\mathcal{R}_X$-semilinear endomorphism $\Phi:\D\rightarrow\D$
        such that $\Phi:\phi^*\D\rightarrow\D$ is an isomorphism;
        \item
        a continuous $\mathcal{R}_X$-semilinear $\Gamma$-action commuting with $\Phi$.
    \end{itemize}
    A family of $\phigam$-modules $\D$ over $X$ is said to be \emph{étale} if there exists
    a covering $\{U_i\}_{i\in I}$ of $X$ and free $\mathcal{A}_{U_i}^\dag$-submodules
    $D_i$ of $\D_{|U_i}$ such that
    \[
        D_i\otimes_{\mathcal{A}_{U_i}^\dag}\robba_{U_i}=\D_{|U_i}
    \]
    and $\Phi$ induces an isomorphism $\phi^*D_i\rightarrow D_i$.
    \end{dfn}

    In \cite[Corollary 6.11]{hellmann2013arithmetic} it is proved that étaleness is an
    open condition and that it can be checked fiber-wise. More precisely:

    \begin{rem}
    \begin{enumerate}
        \item[(i)]
        Given a family of $\phigam$-modules $\D$ over a rigid space $X$, the subset of $X$
        where the specialization of $\D$ is étale is open.
        \item[(ii)]
        A family of $\phigam$-modules over a rigid space $X$ is étale if and only if
        all fibers are étale.
    \end{enumerate}
    \end{rem}

\subsection{The $\phigam$-cohomology}

    The $\phigam$-cohomology was first introduced by Herr in order to compute the Galois
    cohomology of $p$-adic Galois representations on the side of $\phigam$-modules and it is
    a very useful tool in the study of these objects.

    \begin{dfn}
    \label{phi,Gamma cohomology}
    Let $X$ be a rigid space and let $D$ be a family of $(\phi,\Gamma)$-modules over $X$.
    \begin{itemize}
        \item
        If $p\neq2$, the \emph{Herr complex} of $D$ is the complex of $\mathcal{O}_X$-modules
        \[
            0\rightarrow D\xrightarrow{(\Phi-1,\gamma-1)} D\oplus D\xrightarrow{(\gamma-1)\oplus(1-\Phi)} D\rightarrow0,
        \]
        where $\gamma$ is a topological generator of $\Gamma$.
        \item
        If $p=2$, the \emph{Herr complex} of $D$ is the complex of $\mathcal{O}_X$-modules
        \[
        0\rightarrow D^{\Delta}\xrightarrow{(\Phi-1,\gamma-1)} D^{\Delta}\oplus D^{\Delta}\xrightarrow{(\gamma-1)\oplus(1-\Phi)} D^{\Delta}\rightarrow0,
        \]
        where $\Delta\subset\Gamma$ is the torsion subgroup and $\gamma$ is a topological generator of $\Gamma/\Delta$.
    \end{itemize}
    We will refer to the Herr complex of $D$ as $\mathcal{C}^\bullet(D)$.
    Moreover we denote by $H^\bullet_{\phi,\Gamma}(D)$ the cohomology of the Herr complex,
    called the \emph{$\phigam$-cohomology} of $D$.
    \end{dfn}

    \begin{prop}
    \label{H is fin.gen.}
    Let $X$ be a rigid space and $D$ a family of $\phigam$-modules over $X$.
    The $\mathcal{O}_X$-modules $H^i_{\phi,\Gamma}(D)$ are coherent for all $i=0,1,2$.
    \end{prop}

    The proof of this Proposition can be found in \cite[Theorem 4.4.2]{kedlayaCohomology}.

    One of the most important properties of the $\phigam$-cohomology is that it computes extensions:

    \begin{lem}
    Let $X$ be a rigid space and let $D_1$ and $D_2$ be two families of
    $(\phi,\Gamma)$-modules over $X$. Then for $i=0,1$ there are natural isomorphisms
    \[
        H^i_{\phi,\Gamma}(D_2^\vee\otimes D_1)\cong\mathrm{Ext}^i_{\phi,\Gamma}(D_2,D_1).
    \]
    \end{lem}

    \begin{proof}
    Notice that
    \[
        H^0_{\phi,\Gamma}(D_2^\vee\otimes D_1)\cong\mathrm{Hom}_{\phi,\Gamma}(\mathcal{R}_X,D_2^\vee\otimes D_1)
        \cong\mathrm{Hom}_{\phi,\Gamma}(D_2,D_1)=\mathrm{Ext}^0_{\phi,\Gamma}(D_2,D_1).
    \]
    As for $i=1$, we refer to \cite[Lemme 2.2]{chenevierDensite}.
    \end{proof}

    It is possible to classify all the $\phigam$-modules of rank 1 in a very nice way.
    Let $\mathcal{T}$ be the rigid space representing the functor
    \begin{align*}
    \{\text{rigid analytic spaces over } \mathbb{Q}_p\}&\rightarrow\underline{Sets}\\
    X&\mapsto\mathrm{Hom}_{\mathrm{cont}}(\mathbb{Q}_p^\times,\Gamma(X,\mathcal{O}_X^\times)).
    \end{align*}
    We also denote by $\W$ the rigid space representing the functor
    \begin{align*}
    \{\text{rigid analytic spaces over } \mathbb{Q}_p\}&\rightarrow\underline{Sets}\\
    X&\mapsto\mathrm{Hom}_{\mathrm{cont}}(\mathbb{Z}_p^\times,\Gamma(X,\mathcal{O}_X^\times)).
    \end{align*}
    The representability of the above functors is proved in \cite[Proposition 6.1.1]{kedlayaCohomology}.
    There is a natural projection $\T\rightarrow\W$ given by the restriction of a character to $\Z_p^\times$;
    this map has a section, identifying $\T$ with
    $\W\times\mathbb{G}_m^{\mathrm{an}}$ via $\delta\mapsto(\delta_{|\Z_p^\times},\delta(p))$.

    \begin{dfn}
    Let $\delta\in\mathcal{T}(X)$ for some rigid space $X$.
    \begin{itemize}
        \item
        We define $\robba_X(\delta)$ to be the $\phigam$-module of rank 1 over $\robba_X$ with basis $e$
        such that $\Phi(e)=\delta(p)e$ and $\gamma\cdot e=\delta(\gamma)e$ for all $\gamma\in\Gamma$.
        \item
        If $X$ is an affinoid $L$-rigid space $\mathrm{Sp}(A)$, then we define the \emph{weight of $\delta$}
        following \cite[Notation 5.3.2.]{kedlayaCohomology} as the element
        \[
            \rr{wt}(\delta)\coloneqq\frac{\mathrm{log}(\delta(a))}{\mathrm{log}(a)}\in A
        \]
        for non-torsion $a\in\mathbb{Z}_p^\times$ close enough to 1.
    \end{itemize}
    \end{dfn}

    The following Theorem gives a classification of all $\phigam$-modules of rank 1 (up to twist by a unique line bundle).

    \begin{thm}
    \label{rank 1 phi gamma modules}
    Let $D$ be a $\phigam$-module of rank 1 over $\robba_X$ for some rigid space $X$.
    We have that $D\cong \robba_X(\delta)\otimes_{\mathcal{O}_X}H^0_{\phi,\Gamma}(D(\delta^{-1}))$
    for a uniquely determined $\delta\in\mathcal{T}(X)$.
    Moreover we have that $H^0_{\phi,\Gamma}(D(\delta^{-1}))$ is a line bundle over $X$.
    \end{thm}

    \begin{proof}
    For the proof look \cite[Construction 6.2.4, Theorem 6.2.14]{kedlayaCohomology}.
    \end{proof}

    The $\phigam$-cohomology of $\phigam$-modules of rank 1 does not behave always in the same way.
    In the following we partition the space $\mathcal{T}$ keeping in mind this behaviour.
    Let $x$ be the identity character and $\chi=x|x|$ the cyclotomic character in $\mathcal{T}$.
    Let us define the following subspaces of $\mathcal T$:
    \begin{equation}
    \label{def of T+}
    \begin{split}
        \mathcal{T}^+&\coloneqq\{\chi x^n\colon n\in\mathbb{N}\}\\
        \mathcal{T}^-&\coloneqq\{x^{-n}\colon n\in\mathbb{N}\}\\
        \mathcal{T}^{\mathrm{reg}}&\coloneqq\mathcal{T}\setminus(\mathcal{T}^+\cup\mathcal{T}^-).
    \end{split}
    \end{equation}
    The subsets $\mathcal{T}^+$, $\mathcal{T}^-$ and $\{x^n\colon n\in\mathbb{Z}\}$ are closed in $\mathcal{T}$,
    making $\mathcal{T}^{\mathrm{reg}}$ open.
    Moreover $\mathcal{T}^{\mathrm{reg}}$ is dense in $\mathcal{T}$.

    \begin{dfn}
        We say that a character $\delta\in\mathcal{T}$ is \emph{regular} if $\delta\in\mathcal{T}^{\mathrm{reg}}$.
    \end{dfn}

    \begin{thm}
    Let $\delta\in\mathcal{T}(L)$.
    We have that
    \begin{enumerate}
        \item[(i)]
        $\mathrm{dim}_LH^1(\mathcal{R}_L(\delta))=
        \begin{cases}
        1\text{ if }\delta\in\mathcal{T}^{\mathrm{reg}}(L)\\
        2\text{ if }\delta\in(\mathcal{T}^+\cup\mathcal{T}^-)(L)
        \end{cases}$
        \item[(ii)]
        $\mathrm{dim}_LH^0(\mathcal{R}_L(\delta))=
        \begin{cases}
        0\text{ if }\delta\in(\mathcal{T}\setminus\mathcal{T}^-)(L)\\
        1\text{ if }\delta\in\mathcal{T}^-(L)
        \end{cases}$
        \item[(iii)]
        $\mathrm{dim}_LH^2(\mathcal{R}_L(\delta))=
        \begin{cases}
        0\text{ if }\delta\in(\mathcal{T}\setminus\mathcal{T}^+)(L)\\
        1\text{ if }\delta\in\mathcal{T}^+(L)
        \end{cases}$
    \end{enumerate}
    \end{thm}

    For a proof see \cite[Proposition 2.1, Theorem 2.9]{colmez2008representations} and \cite[Corollary 2.13]{liu2007cohomology}

    More generally, we have:

    \begin{thm}
    \label{cohomology of families}
    Let $A$ be an affinoid $\Qp$-algebra and let $\delta\in\mathcal{T}(A)$.
    \begin{itemize}
        \item[(i)]
        $H^0_{\phi,\Gamma}(\robba_A(\delta))\neq0$ if and only if there exists $i\in\mathbb{N}$ and  $0\neq a\in A$
        such that $a\cdot(\delta-x^i)$ is the constant zero function on $\Qp^\times$.
        \item[(ii)]
        $H^2_{\phi,\Gamma}(\robba_A(\delta))\neq0$ if and only if the closed subspace of $\mathrm{Sp}(A)$
        given by the equation $\delta=\chi x^i$ is non-empty for some $i\in\mathbb{N}$.
        \item[(iii)]
        If $\delta\in\T^\reg(A)$, then $H^0_{\phi,\Gamma}(\robba_A(\delta))=H^2_{\phi,\Gamma}(\robba_A(\delta))=0$
        and $H^1_{\phi,\Gamma}(\robba_A(\delta))$ is free of rank 1 over $A$.
    \end{itemize}
    \end{thm}

    The reader can find the proof of this Theorem in \cite[Corollaire 2.11, Théorème 2.29]{chenevierDensite}.

    Let us moreover define
    \begin{equation}
    \label{def of Tnreg}
    \mathcal{T}_n^{\mathrm{reg}}\coloneqq\{(\delta_1,\dots,\delta_n)\in\mathcal{T}^n\colon
    \delta_i/\delta_j\in\mathcal{T}^{\mathrm{reg}}\text{ for all } i<j\}\subset\mathcal{T}^n,
    \end{equation}
    which is an open dense subspace of $\T^n$ and
    \[
    \W_n^\reg\coloneqq\{(\delta_1,\dots,\delta_n)\in\W^n\colon\delta_i/\delta_j\notin x^\Z\text{ for all }i<j\}\subset\W^n,
    \]
    which is a dense subspace of $\W^n$.

    \begin{lem}
    \label{good neighbourhood}
    For every $\underline{\delta}\in\T^n(K)$ (for some extesion $K$ of $\Qp$), there exists
    an open neighbourhood $\widetilde{\mathcal{U}}_n\subset\T^n$ of $\underline{\delta}$ such
    that the following properties are satisfied.
    Let $(\widetilde{\delta}_1,\dots,\widetilde{\delta}_n)$ be the characters
    corresponding to the open $\widetilde{\mathcal{U}}_n$.
    For all $1\leq i<j\leq n$
    \begin{itemize}
        \item
        for all $u\in\widetilde{\mcal U}_n$, we have
        $\widetilde{\delta}_i/\widetilde{\delta}_j\widehat{\otimes}k(u)\notin\{x^m\colon m\in\Z\}$,
        unless $\delta_i/\delta_j=x^{a_{ij}}$ for some $a_{ij}\in\Z$;
        in this case, we have $\widetilde{\delta}_i/\widetilde{\delta}_j\widehat{\otimes}k(u)\in\{x^m\colon m\in\Z\}$
        only if $\widetilde{\delta}_i/\widetilde{\delta}_j\widehat{\otimes}k(u)=x^{a_{ij}}$.
        \item
        for all $u\in\widetilde{\mcal U}_n$, we have
        $\widetilde{\delta}_i/\widetilde{\delta}_j\widehat{\otimes}k(u)\notin\{\chi x^m\colon m\in\Z\}$,
        unless $\delta_i/\delta_j=\chi x^{b_{ij}}$ for some $b_{ij}\in\Z$;
        in this case, we have $\widetilde{\delta}_i/\widetilde{\delta}_j\widehat{\otimes}k(u)\in\{\chi x^m\colon m\in\Z\}$
        only if $\widetilde{\delta}_i/\widetilde{\delta}_j\widehat{\otimes}k(u)=\chi x^{b_{ij}}$.
    \end{itemize}
    \end{lem}

    \begin{proof}
    For all $1\leq i<j\leq n$, we will define a neighbourhood
    $\mathcal{V}_{ij}\subset\T^n$
    of $\underline{\delta}$ in the following way:
    let
    \[
        \widetilde\delta_i/\widetilde\delta_j:\T^n\rightarrow\T
    \]
    be the map sending $(\widetilde\delta_1,\dots,\widetilde\delta_{n})\in\T^n(A)$
    (for some affinoid algebra $A$) to $\widetilde\delta_i/\widetilde\delta_j$.
    \begin{itemize}
        \item
        If $\delta_i/\delta_j\notin\{x^m\colon m\in\Z\}\cup\{\chi x^m\colon m\in\Z\}$, then we define
        \[
            \mathcal{V}_{ij}\coloneqq(\widetilde\delta_i/\widetilde\delta_j)^{-1}(\T\setminus(\{x^m\colon m\in\Z\}\cup\{\chi x^m\colon m\in\Z\})),
        \]
        which is open, since $\T\setminus(\{x^m\colon m\in\Z\}\cup\{\chi x^m\colon m\in\Z\})$ is.
        \item
        If instead $\delta_i/\delta_j=x^{a_{ij}}$ for some $a_{ij}\in\Z$, then we define
        \[
            \mathcal{V}_{ij}\coloneqq(\widetilde\delta_i/\widetilde\delta_j)^{-1}
            (\T\setminus(\{x^m\colon m\in\Z\setminus\{a_{ij}\}\}\cup\{\chi x^m\colon m\in\Z\})).
        \]
        \item
        Finally, if $\delta_i/\delta_j=\chi x^{b_{ij}}$ for some $b_{ij}\in\Z$, then we define
        \[
            \mathcal{V}_{ij}\coloneqq(\widetilde\delta_i/\widetilde\delta_j)^{-1}
            (\T\setminus(\{x^m\colon m\in\Z\}\cup\{\chi x^m\colon m\in\Z\setminus\{b_{ij}\}\})).
        \]
    \end{itemize}
    Finally, the open $\bigcap_{(1\leq i<j\leq n)}\mathcal{V}_{ij}$
    is a neighbourhood of $\underline{\delta}$ and satisfies the wanted properties.
    \end{proof}

    In what follows we state some well-known properties of $\phigam$-cohomology.

    The following Theorem is proved in \cite[Theorem 0.2]{liu2007cohomology}.

    \begin{thm}
    \label{poincare duality}
    Let $D$ be a $\phigam$-module over $\robba_L$. Then
    \[
        \sum_{i=0}^2(-1)^i\dim_LH^i_{\phi,\Gamma}(D)=\mathrm{rk}(D).
    \]
    \end{thm}

    \begin{thm}
    \label{Herrcomplexprojective}
    Let $A$ be an affinoid $\mathbb{Q}_p$-algebra and let $D$ be a family of $(\phi,\Gamma)$-modules over $\mathrm{Sp}(A)$.
    The Herr complex of $D$ is quasi-isomorphic to a bounded complex of projective $A$-modules
    \[
        M^0\rightarrow M^1\rightarrow M^2
    \]
    that universally computes the cohomology of the Herr complex,
    i.e. for all  affinoid $A$-algebras $B$, the complex $M^\bullet\otimes_A B$
    is quasi-isomorphic to $\mathcal{C}^\bullet(D\widehat{\otimes}_A B)$.
    \end{thm}

    This result follows from \cite[Theorem 4.4.3,4.4.5]{kedlayaCohomology}.

\section{Trianguline representations}
\label{section 2}

Trianguline representations are a certain class of $p$-adic Galois representations,
just like crystalline, semistable and de Rham representations.
They were introduced for the first time by Colmez in order
to study the functoriality of the $p$-adic local Langlands correspondence for $GL_2(\mathbb{Q}_p)$.

\subsection{Definitions and basic properties}
\label{def and basic properties}

  \begin{dfn}
    A Galois representation $\rho:\absGal\rightarrow GL_n(L)$ is a \emph{trianguline representation} if $\drig(\rho)$ is \emph{triangulable},
    which means that there exists a complete filtration
    \[
      0\subset\mathrm{Fil}^1(\drig(\rho))\subset\dots\subset\mathrm{Fil}^n(\drig(\rho))=\drig(\rho)
    \]
    of sub-$\phigam$-modules of $\drig(\rho)$ such that for all $i$
    \[
      \mathrm{Fil}^i(\drig(\rho))/\mathrm{Fil}^{i-1}(\drig(\rho))\cong\mathcal{R}_L(\delta_i)
    \]
    for some $\delta_i\in\mathcal{T}(L)$.
    We say that $(\delta_1,\dots,\delta_n)$ are the \emph{parameters} of the triangulation
    $\mathrm{Fil}^\bullet(\drig(\rho))$ of $\drig(\rho)$ and they are said to be \emph{regular} if
    $(\delta_1,\dots,\delta_n)\in\mathcal{T}_n^{\mathrm{reg}}(L)$.
  \end{dfn}

  Analogously to the theory of \emph{families} of $\phigam$-modules, there is also a
  theory of \emph{families} of $p$-adic Galois representations.

  \begin{dfn}
    Let $X$ be a rigid space.
    A \emph{family of Galois representations over $X$} is a vector bundle $\mathcal{V}$ over $X$
    endowed with a continuous $\absGal$-action.
  \end{dfn}

  \begin{thm}
    \label{fullyfaithfulfunctor}
    There is a fully faithful and exact functor
    \[
      \drig:\left\{
        \begin{array}{l}
          \text{families of Galois}\\
          \text{represetations over }X
        \end{array}
        \right\}
      \rightarrow
      \left\{
        \begin{array}{l}
          \text{families of étale }\\
          \phigam\text{-modules over }X
        \end{array}
        \right\}
    \]
    commuting with base change.
  \end{thm}
  For a reference of the above Theorem, look at \cite[Theorem 2.2.17.]{kedlayaCohomology}.

\subsection{The trianguline variety}

    In the following, let $k_L$ be the residue field of $L$ and
    fix a continuous Galois representation $\bar\rho:\absGal\rightarrow GL_n(k_L)$.

    \begin{dfn}
    \begin{itemize}
        \item
        Let $A$ be a complete local Artinian ring with residue field $k_L$.
        We say that a continuous representation $\rho:\absGal\rightarrow GL_n(A)$ is
        a \emph{deformation} of $\bar\rho$ if the composition of $\rho$ with the induced map
        $GL_n(A)\rightarrow GL_n(k_L)$ is exactly $\bar\rho$.
        \item
        We say that two deformations $\rho_1,\rho_2:\absGal\rightarrow GL_n(A)$ are \emph{equivalent}
        if $\rho_1=M\rho_2 M^{-1}$ for some matrix $M\in\mathrm{ker}(GL_n(A)\rightarrow GL_n(k_L))$.
    \end{itemize}
    \end{dfn}

    The functor that associates to a complete local Artinian ring $A$ the set of
    deformations of $\bar\rho$ is pro-representable by a
    complete local Noetherian $W(k_L)$-algebra $R_{\bar\rho}^\square$, called the
    \emph{universal framed deformation ring of $\bar\rho$}.
    This was proved by Mazur in \cite{mazur1989deforming}.
    We let $\mathfrak X_{\bar\rho}^\square=(\mathrm{Spf}R_{\bar\rho}^\square)^{\mathrm{rig}}$
    be the rigid space associated to the formal scheme $\mathrm{Spf}(R_{\bar\rho}^\square)$.
    We have that $\defspace$ represents the functor which associates
    to an affinoid rigid space $\mathrm{Sp}(A)$ over $W(k_L)[1/p]$ the set of
    deformations $\rho:\absGal\rightarrow GL_n(A^+)$ of $\bar\rho$, where $A^+$ denotes
    the subring of power-bounded elements of $A$ (in particular it is a $W(k_L)$-algebra).

    \begin{dfn}
    The \emph{trianguline variety} $X_\tri^\square(\bar\rho)$ is the rigid analytic space
    over $L$ defined as the Zariski-closure of the rigid space
    \[
        U_{\tri}^\square(\bar\rho):=\{(\rho,\underline\delta)\in\mathfrak X_{\bar\rho}^\square\times\mathcal{T}_n^\reg
        \colon \rho\text{ is trianguline of parameters }\underline\delta\}
    \]
    inside $\mathfrak X_{\bar\rho}^\square\times\mathcal{T}^n$.
    \end{dfn}

    This space was defined and studied in \cite{breuil2017interpretation}.
    In particular we recall the following results.

    \begin{thm}
    \begin{enumerate}
        \item[(i)]
        The space $\trivar$ is equidimensional of dimension $n^2+~\frac{n(n+1)}{2}$;
        \item[(ii)]
        the subset $\regtrivar$ is Zariski-open and Zariski-dense in $\trivar$
        \item[(iii)]
        if $(\rho,\underline{\delta})\in\trivar$, then $\rho$ is trianguline.
    \end{enumerate}
    \end{thm}

    It is then natural to ask ourselves whether all trianguline representations "are
    in" the trianguline variety.
    This is the question that motivates this thesis.

\section{Strategy}
\label{section 3}

    First we show that, given a Galois
    representation on a rigid space, its reduction to the residue field of the ring of
    integers of a point is locally constant.
    Recall that, given an affinoid ring $A$, we denote by $A^+$ the subring of power-bounded
    elements and by $A^{++}\subset A^+$ the ideal of topologically nilpotent elements.

    \begin{prop}
    \label{reduction is locally constant}
    Let $X$ be a rigid space over $L$ and let $\mathcal{V}$ be a family of
    Galois representations over $X$, so that we have $\widetilde{\rho}:\absGal\rightarrow GL(\mathcal{V})$.
    Let $x\in X(L)$ and let $\rho$ be the reduction of $\widetilde{\rho}$ to the point $x$.
    After eventually shrinking $X$, we can choose a basis of $\mathcal{V}$ such that
    $\rho$ has values in $GL_n(\mathcal{O}_L)$.
    Let $\bar\rho$ be the reduction of $\rho$ to the residue field $k_L$ of $L$.
    Then there is an open neighbourhood $U=\mathrm{Sp}(A)\xhookrightarrow{j}X$ of $x$ such that (with the
    basis chosen above) we have $\widetilde\rho_U:=j^*\widetilde\rho:\absGal\rightarrow GL_n(A^+)$.
    Moreover $\widetilde\rho_U$ reduces to $\bar\rho$, which means that there is a commutative diagram
    \[
        \begin{tikzcd}
        \absGal\arrow[r, "\widetilde\rho_U"]\arrow[d, "\overline{\rho}"]
        & GL_n(A^+)\arrow[d, two heads]\\
        GL_n(k_L)\arrow[r, hook]
        & GL_n(A^+/A^{++}),
        \end{tikzcd}
    \]
    where the lower horizontal arrow is induced by the inclusion $k_L\hookrightarrow A^+/A^{++}$, since $A$ is an $L$-algebra.
    \end{prop}

    \begin{rem}
    Let us remark that the representation $\overline{\rho}$ depends on the choice
    of a basis for the vector bundle $\mathcal{V}$; however the isomorphism class of
    the semisimplification of $\overline{\rho}$ is well-defined independently of the chosen basis,
    as showed in \cite[Proposition 5.11]{familiesrepr}.
    \end{rem}

    \begin{proof}
    First of all we show that there is an affinoid neighbourhood
    $V\coloneqq\mathrm{Sp}(B)\xhookrightarrow{i}X$ of $x$ such that
    $\widetilde\rho_V\coloneqq i^*\widetilde\rho$ has values in $GL_n(B^+)$ with respect
    to the basis of $\mathcal{V}$ fixed above.
    By \cite{jannsen1982struktur} we know that $\absGal$ is topologically finitely generated, hence
    let $\{g_1,\dots,g_t\}$ be a set of topological generators of $\absGal$.
    For every $k=1,\dots,t$, let $a_{ij}^k$ be the $(i,j)$-th entry of the matrix
    $\widetilde\rho(g_k)$ and $d^k\coloneqq \mathrm{det}(\widetilde\rho(g_k))$.
    We then have that $a_{ij}^k(x)\in\mathcal{O}_L$, so $|a_{ij}^k(x)|\leq1$
    and $|d^k(x)|=1$ because $d^k(x)\in\mathcal{O}_L^\times$.
    Let us define the affinoid open neighbourhoods $V_{ij}^k\coloneqq \{y\in X:|a_{ij}^k(y)|\leq1\}$
    and $V_d^k\coloneqq \{y\in X:|d^k(y)|=1\}$ of $x$, so that we let
    $$
        V= \mathrm{Sp}(C)\coloneqq \bigcap_{i,j=1,\dots,n\atop k=1,\dots,t}V_{ij}^k\cap V_d^k.
    $$
    In particular for all $k=1,\dots,t$ and $i,j=1\dots,n$, we have
    $|a_{ij}^k|=|a_{ij}^k(y)|\leq1$ and $|d^k|=|d^k(y')|=1$ for some $y,y'\in V$
    thanks to the Maximum Principle, so $a_{ij}^k\in B^+$ and $d^k\in (B^+)^\times$.
    This means that $i^*\widetilde\rho(g_k)\in GL_n(B^+)$ for all $k=1,\dots,t$.
    Now we show that $\widetilde\rho_V\coloneqq i^*\widetilde\rho:\absGal\rightarrow GL_n(B^+)$.
    Let $g\in \absGal$ and observe that for any open neighbourhood $Z$ of $\widetilde\rho_V(g)\in GL_n(B)$,
    there exists $h\in\langle g_1,\dots,g_t\rangle$ such that $\widetilde\rho_V(h)\in Z$.
    So there is a sequence $(\widetilde\rho_V(h_k))_k$ converging to $\widetilde\rho_V(g)$,
    where $(h_k)_k\subset\langle g_1,\dots,g_t\rangle$.
    Let $b_{ij}$ be the $(i,j)$-th entry of $\widetilde\rho_V(g)$ and $b_{ij,k}$ the
    $(i,j)$-th entry of $\widetilde\rho_V(h_k)$.
    For all $i,j=1,\dots,n$ we have a sequence $(b_{ij,k})_k\subset B^+$ converging
    to $b_{ij}$, so $b_{ij}\in B^+$.
    In the same way, we get a sequence $(d_k)_k\subset (B^+)^\times$ converging to
    $\mathrm{det}(\widetilde\rho_V(g))$, which allows us to conclude that $\widetilde\rho_V(g)\in GL_n(B^+)$.
    Finally it is left to show that $\widetilde\rho_V$ reduces to $\overline{\rho}$ modulo
    $B^{++}$, after eventually shrinking $V$.
    Let $f_{ij}^k\coloneqq a_{ij}^k-a_{ij}^k(x)\in B^+$; fix $\epsilon\in(0,1)$ and
    consider $U_{ij}^k\coloneqq \{y\in\mathrm{Sp}(B):|f_{ij}^k(y)|\leq\epsilon\}$
    for all $i,j=1\dots,n$ and $k=1,\dots,t$. We have $x\in U_{ij}^k$ because $f_{ij}^k(x)=0$, so
    $$
        U\coloneqq \mathrm{Sp}(A)\coloneqq \bigcap_{i,j=1,\dots,n\atop k=1,\dots,t}U_{ij}^k
    $$
    is an affinoid open neighbourhood of $x$. We then have $\widetilde\rho_V:\absGal\rightarrow GL_n(A^+)$
    and $|f_{ij}^k|=|f_{ij}^k(y)|\leq \epsilon<1$ for some $y\in\mathrm{Sp}(A)$,
    thus $f_{ij}^k\in A^{++}$ for all $i,j=1,\dots,n$ and $k=1,\dots,t$ by the same
    approximation argument as before.
    In a very similar way to the one used above, it is possible to conclude that
    $\widetilde\rho_U(g)\equiv \rho(g)\text{ mod }A^{++}$ for all $g\in \absGal$.
    \end{proof}

    In what follows we will outline a strategy that we will use in order to show that
    trianguline representations appear as points of the trianguline variety.
    The underlying idea for what follows is that, given a trianguline representation $\rho$,
    we want to construct a rigid space $X$ and a family of trianguline
    Galois representations on $X$ reducing to $\rho$ at some pont $x\in X$ and having regular
    parameters in a dense subset of a neighbourhood of $x$.
    In this way, we can show that $\rho$ is in the closure of a regular family, hence
    in the trianguline variety.

    The following Theorem makes the above idea more precise in terms of $(\phi,\Gamma)$-modules.

    \begin{restatable}{theorem}{strategy}
    \label{strategy}
    Let $\rho:\absGal\rightarrow GL_n(\mathcal{O}_L)$ be a trianguline representation
    with parameters $\underline{\delta}\coloneqq(\delta_1,\dots,\delta_n)\in\mathcal{T}^n(L)$
    and let $\bar{\rho}$ be the reduction of $\rho$ to the residue field of $L$:
    \[
        \bar{\rho}:\absGal\xrightarrow{\rho} GL_n(\mathcal{O}_L)\twoheadrightarrow GL_n(k_L).
    \]
    Assume that there exist
    \begin{itemize}
        \item
        a rigid space $X$ over $L$;
        \item
        a family of $\phigam$-modules $\D$ over $X$;
        \item
        a map $(\widetilde{\delta_1},\dots,\widetilde{\delta_n}):X\rightarrow\mathcal{T}^n$
    \end{itemize}
    such that there are
    \begin{enumerate}
        \item
        a point $x\in X(L)$ such that $\underline{\widetilde{\delta}}\widehat\otimes k(x)=\underline\delta$
        and $\D\widehat{\otimes} k(x)=\drig(\rho)$;
        \item
        a Zariski-open dense subset $U\xhookrightarrow{j}X$ such that
        $j^*\D$ has a filtration with graded pieces $\mathcal{R}_U(j^*\widetilde{\delta_i})$
        and $(j^*\widetilde{\delta_1},\dots,j^*\widetilde{\delta_n})\in\mathcal{T}_n^{\mathrm{reg}}(U)$.
    \end{enumerate}
    Then $(\rho,\underline{\delta})\in X_\mathrm{tri}^\square(\bar{\rho})(L)$.
    \end{restatable}

    \begin{proof}
    By \cite[Theorem 0.2]{kedlaya2011families}, there exists a neighbourhood
    $V=\mathrm{Sp}(B)\xhookrightarrow{i} X$ of $x$ and a family of representations on $V$
    \[
        \widetilde\rho:\absGal\rightarrow GL_n(B^+)
    \]
    such that $\D_\mathrm{rig}^\dag(\widetilde\rho)= i^*\D$.
    By Proposition \ref{reduction is locally constant}, eventually shrinking $V$, we can
    choose a basis of $\mathcal{V}$ such that the following diagram commutes:
    \begin{center}
        \begin{tikzcd}
        \absGal\ar[r, "\widetilde\rho"]\ar[d, "\overline{\rho}"]
        & GL_n(B^+)\ar[d,two heads]\\
        GL_n(k_L)\ar[r,hook]
        & GL_n(B^+/B^{++}).
        \end{tikzcd}
    \end{center}
    Thus $\widetilde\rho$ induces a morphism
    \[
        \psi\coloneqq (\widetilde\rho,i^*\widetilde{\underline{\delta}}):V\rightarrow\mathfrak{X}^\square_{\bar{\rho}}\times\mathcal{T}^n.
    \]
    We now show that $\psi_{|U\cap V}$ factors through $U_\mathrm{tri}^\square(\bar{\rho})^\mathrm{reg}$:
    first of all let $u:U\cap V\hookrightarrow V$ and notice that it is sufficient to prove
    that the family of representations $u^*\widetilde{\rho}$ is trianguline with regular
    parameters $(i\circ u)^*\underline{\widetilde\delta}$.
    By hypothesis and by Theorem \ref{fullyfaithfulfunctor},
    $\drig(u^*\widetilde\rho)=u^*\drig(\widetilde\rho)\cong(i\circ u)^*\D$.
    Hence $\drig(u^*\widetilde\rho)$ has parameters
    $(i\circ u)^*\underline{\widetilde\delta}\in\mathcal{T}^\reg_n(U\cap V)$.
    Thus $\psi(U\cap V)\subset\regtrivar$.
    Notice that $\regtrivar\cap \mathrm{im}(\psi)$ is dense in $\mathrm{im}(\psi)$:
    let $W$ be an open in $\mathrm{im}(\psi)$, so that $\psi^{-1}(W)$ is open in $V$.
    Then $\psi^{-1}(W)\cap U\neq\emptyset$, which implies that $W\cap\regtrivar\neq\emptyset$.
    Thanks to the density of $\regtrivar\cap \mathrm{im}(\psi)$ in $\mathrm{im}(\psi)$,
    we can conclude that
    \[
        \mathrm{im}(\psi)\subset\overline{\mathrm{im}(\psi)}=\overline{\regtrivar\cap \mathrm{im}(\psi)}
        \subset\trivar.
    \]
    In particular, $\psi(x)=(\rho,\underline{\delta})\in\trivar$.
    \end{proof}

    \begin{lem}
    \label{enlarge L}
    Let $\rho:\absGal\rightarrow GL_n(L)$ be a representation and let $\underline{\delta}\in\T^n(L)$.
    Let us assume that there is a finite extension $L\xhookrightarrow{i}L'$ such that $(\rho',\underline{\delta}')\in\trivar(L')$,
    where $\rho'\coloneqq i^*\rho:\absGal\xrightarrow{\rho}GL_n(L)\hookrightarrow GL_n(L')$ and
    $\underline\delta'\coloneqq i^*\underline\delta:\Q_p^\times\xrightarrow{\underline{\delta}}(L^\times)^n\hookrightarrow(L'^\times)^n$.
    Then $(\rho,\underline{\delta})\in\trivar(L)$.
    \end{lem}

    \begin{proof}
    We have that the point $(\rho',\underline{\delta}')\in(\mathfrak X_{\bar\rho}^\square\times\T^n)(L')$
    corresponds by construction to the map
    \[
        (\rho',\underline{\delta}'):\Sp(L')\rightarrow\Sp(L)\xrightarrow{(\rho,\underline\delta)}\mathfrak X_{\bar\rho}^\square\times\T^n.
    \]
    The fact that $(\rho',\underline{\delta}')\in\trivar(L')$ implies that $\im(\rho',\underline{\delta}')\subset\trivar$;
    hence in particular we can deduce that $\im(\rho,\underline{\delta})\subset\trivar$.
    This means $(\rho,\underline\delta)\in\trivar(L)$.
    \end{proof}

\section{Trianguline representations with no $\mathcal{T}^+$-type irregularity}
\label{section 4}

In this section we prove Theorem \ref{RESULT1}, which states that any trianguline representation whose parameters
 don't admit quotients in $\mathcal{T}^+$ are in the trianguline variety.
More precisely, we show the following:

\result*

\subsection{Constructing extensions}
\label{constructing extensions}

    We want to use Theorem \ref{strategy} to prove Theorem \ref{RESULT1}.
    In order to do so, it is essential to construct a rigid space $X$ of extensions
    of $\phigam$-modules.
    In this section we will show how to do so in the case the second $\phigam$-cohomology of
    these $\phigam$-modules is a locally free sheaf.

    \begin{thm}
    \label{bc}
    Let $D$ be a $\phigam$-module over an affinoid ring $A$ and let $B$ be an affinoid $A$-algebra.
    If $H^2_{\phi,\Gamma}(D)$ is a locally free $A$-module,
    then
    \[
        H^1_{\phi,\Gamma}(D)\otimes_AB\cong H^1_{\phi,\Gamma}(D\widehat\otimes_AB).
    \]
    \end{thm}

    \begin{proof}
    By Theorem \ref{Herrcomplexprojective}, there is a complex of projective $A$-modules
    \[
    M^0\xrightarrow{d^0} M^1\xrightarrow{d^1} M^2
    \]
    universally computing the cohomology of the Herr complex of $D$.
    This means that $H^1(D\widehat\otimes_AB)\cong H^1(M^\bullet\otimes_AB)$ for any $A$-algebra $B$.
    Being $H^2(M^\bullet)$ locally free, we have $M^2\cong H^2(M^\bullet)\oplus\mathrm{im}(d^1)$.
    Since both $M^2$ and $H^2(M^\bullet)$ are projective modules, we can conclude $\mathrm{im}(d^1)$ is as well,
    implying that $M^1$ also splits as the direct sum of $\mathrm{im}(d^1)$ and $\mathrm{ker}(d^1)$.
    Again we can deduce that $Z^1:=\mathrm{ker}(d^1)$ is a projective $A$-module.
    Now notice that the complex
    \[
        M^0\xrightarrow{d^0}Z^1
    \]
    universally computes $H^0_{\phi,\Gamma}(D)$ and $H^1_{\phi,\Gamma}(D)$.
    Moreover
    \[
        M^0\xrightarrow{d^0}Z^1\rightarrow H^1(M^\bullet)\rightarrow0
    \]
    is a projective resolution of $H^1(M^\bullet)$.
    Therefore
    \begin{align*}
        H^1_{\phi,\Gamma}(D)\otimes_AB&\cong H^1(M^\bullet)\otimes_AB\cong\mathrm{Tor}_0(H^1(M^\bullet),B)
        \cong (Z^1\otimes_AB)/\mathrm{im}(d^0\otimes\mathrm{id}_B)\\
        &\cong H^1(M^\bullet\otimes_AB)\cong H^1_{\phi,\Gamma}(D\widehat\otimes_AB).
    \end{align*}
    \end{proof}

    \begin{thm}
    \label{extensions}
    Let $D_1, D_2$ be two families of $\phigam$-modules over some rigid analytic affinoid space $\mathrm{Sp}(A)$.
    Assume that $H^2_{\phi,\Gamma}(D_1\otimes_{\robba_A} D_2^\vee)$ is locally free.
    There exists a vector bundle $X$ over $\mathrm{Sp}(A)$ (of dimension equal to the minimal
    number of generators of the module $H^1_{\phi,\Gamma}(D_1\otimes_{\robba_A} D_2^\vee)$ over $A$)
    and a family of $\phigam$-modules
    $\D$ over $X$ such that for any rigid space $Y\xrightarrow{g}\mathrm{Sp}(A)$
    over $\mathrm{Sp}(A)$, we have a surjective map
    \begin{align*}
        \Psi_Y:X(Y)&\twoheadrightarrow\mathrm{Ext}^1_{\phi,\Gamma}(g^* D_2,g^* D_1)(Y)\\
        f&\mapsto f^*\D
    \end{align*}
    which is functorial in $Y$.
    We moreover can choose $X$ such that the above map $\Psi_Y$ is bijective for all rigid spaces
    $Y\xrightarrow{g}\mathrm{Sp}(A)$ if the module $H^1_{\phi,\Gamma}(D_1\otimes_{\robba_A} D_2^\vee)$
    is free over $A$.
    \end{thm}

    \begin{rem}
    In \cite[§5.4]{emerton2019moduli} Emerton and Gee explain how to construct a family
    of extensions of $\phigam$-modules with coefficients in characteristic $p$ in case
    the second $\phigam$-cohomology has constant fiber rank.
    In fact Theorem \ref{extensions} is the analogue in characteristic 0 of that construction.
    \end{rem}

    \begin{proof}
    By Proposition \ref{H is fin.gen.}, we know $H^1_{\phi,\Gamma}(D_1\widehat\otimes_A D_2^\vee)$
    is a coherent $\mathcal{O}_{\mathrm{Sp}(A)}$-module.
    Let $n$ be the minimal number of generators of $H^1_{\phi,\Gamma}(D_1\widehat\otimes_A D_2^\vee)$
    and let $F$ be a finitely generated free $A$-module of rank $n$; in particular, we have a
    surjective map of $A$-modules
    \[
        F\twoheadrightarrow  H^1_{\phi,\Gamma}(D_1\widehat\otimes_A D_2^\vee).
    \]
    We define the rigid space
    \[
        X\coloneqq \mathrm{Sp}(A)\times\mathbb{A}^{n,\mathrm{an}}_{\mathbb{Q}_p}\xrightarrow{h}\mathrm{Sp}(A),
    \]
    which is a vector bundle over $\mathrm{Sp}(A)$.
    Let $Y\to X$ be a map of rigid spaces over $\mathrm{Sp}(A)$
    and let $g$ be the map $Y\rightarrow\Sp(A)$. Let us study the space $X(Y)$:
    \begin{align*}
        X(Y)=\mathrm{Hom}_{\mathrm{Sp}(A)}(Y,X)&\cong
        \mathrm{Hom}_{\mathrm{Sp}(\mathbb{Q}_p)}(Y,\mathbb{A}_{\mathbb{Q}_p}^{n,\mathrm{an}}).
    \end{align*}
    For any rigid $\mathbb{Q}_p$-space $Y$ there is a bijection
    \[
        \mathrm{Hom}_{\mathrm{Sp}(\mathbb{Q}_p)}(Y,\mathbb{A}_{\mathbb{Q}_p}^{n,\mathrm{an}})
        \xrightarrow{\sim}\mathrm{Hom}_{\mathrm{Sp}(\mathbb{Q}_p)}(Y,\mathbb{A}_{\mathbb{Q}_p}^{n})\cong\Gamma(Y,\mathcal{O}_Y)^n
    \]
    between the set of rigid morphisms $ Y\rightarrow\mathbb{A}_{\mathbb{Q}_p}^{n,\mathrm{an}}$
    and the set of morphisms between locally $G$-ringed spaces $ Y\rightarrow\mathbb{A}_{\mathbb{Q}_p}^{n}$.
    We then have the map
    \begin{align*}
        X(Y)&\cong\mathrm{Hom}_{\mathrm{Sp}(\mathbb{Q}_p)}(Y,\mathbb{A}_{\mathbb{Q}_p}^{n,\mathrm{an}})
        \cong \Gamma(Y,\mathcal{O}_Y)^n\\
        &\cong \Gamma(Y,g^*F)\twoheadrightarrow \Gamma(Y,g^*H^1_{\phi,\Gamma}(D_1\otimes_{\mathcal{R}_A} D_2^\vee)).
    \end{align*}
    Since $H^2_{\phi,\Gamma}(D_1\otimes_{\robba_A} D_2^\vee)$ is locally free by hypothesis,
    we get as a consequence that
    \[
        g^*H^1_{\phi,\Gamma}(D_1\otimes_{\robba_A} D_2^\vee)\cong
        H^1_{\phi,\Gamma}(g^*(D_1\otimes_{\robba_A} D_2^\vee))
    \]
    by Theorem \ref{bc}.
    Finally, we have
    \begin{align*}
        \Psi_Y:X(Y)\twoheadrightarrow g^*H^1_{\phi,\Gamma}(D_1\otimes_{\mathcal{R}_A} D_2^\vee)(Y)
        &\cong H^1_{\phi,\Gamma}(g^*(D_1\otimes_{\mathcal{R}_A} D_2^\vee))(Y)\\
        & \cong\mathrm{Ext}^1_{\phi,\Gamma}(g^*D_2,g^*D_1)(Y).
    \end{align*}
    Let us now define the universal family of $(\phi,\Gamma)$-modules over $X$.
    We know we have a surjective map
    \begin{align*}
        \Psi_X:\mathrm{Hom}_{\mathrm{Sp}(A)}(X,X)\twoheadrightarrow \mathrm{Ext}^1_{\phi,\Gamma}(h^*D_2, h^*D_1)(X).
    \end{align*}
    Let $\D\in\mathrm{Ext}^1_{\phi,\Gamma}(h^*D_2, h^*D_1)$ be the image of
    $\mathrm{id}_X$ through $\Psi_X$.
    By the Yoneda Lemma it follows immediately that $\D$ satisfies the universal property specified above.
    In case the module $H^1_{\phi,\Gamma}(D_1\otimes_{\robba_A} D_2^\vee)$ is free over $A$,
    it is sufficient to take the module $F$ exactly equal to $H^1_{\phi,\Gamma}(D_1\otimes_{\robba_A} D_2^\vee)$;
    it is then obvious that all the above maps composing $\Psi_Y$ are isomorphisms.
    \end{proof}

\subsection{Proof of Theorem \ref{RESULT1}}
\label{section 4.2}

    In order to prove Theorem \ref{RESULT1}, we proceed by induction; hence let us first prove the
    2-dimensional case.

    \begin{prop}
    \label{2 dim T-}
    Let $(\delta_1,\delta_2)\in\mathcal{T}^2(L)$ such that $\delta_1\delta_2^{-1}\notin\mathcal{T}^+(L)$
    and let $D\in H^1_{\phi,\Gamma}(\mathcal{R}_L(\delta_1/\delta_2))$.
    Then there exist
    \begin{itemize}
        \item
        a rigid space $X$ over $L$ which is also a vector bundle of dimension
        \[
            \dim_L(H^1_{\phi,\Gamma}(\robba_L(\delta_1/\delta_2)))
        \]
        over an appropriate
        neighbourhood of $(\delta_1,\delta_2)$ in $\mathcal{T}^2$;
        \item
        a family of $\phigam$-modules $\D$ over $X$;
        \item
        a map $(\widetilde\delta_1,\widetilde\delta_2):X\rightarrow\mathcal{T}^2$
    \end{itemize}
    such that there are
    \begin{enumerate}
        \item
        $x\in X$ such that $\D\widehat\otimes k(x)=D$ and
        $(\widetilde\delta_1,\widetilde\delta_2)\widehat\otimes k(x)=(\delta_1,\delta_2)$;
        \item
        a Zariski-open dense subset $U\xhookrightarrow{j}X$ such that $j^*\D$
        has a filtration of sub-$\phigam$-modules with graded pieces $\robba_U(j^*\widetilde\delta_1)$
        and $\robba_U(j^*\widetilde\delta_2)$ and
        $(j^*\widetilde\delta_1,j^*\widetilde\delta_2)\in\Treg_2(U)$.
    \end{enumerate}
    \end{prop}

    \begin{proof}
    Let $\mathcal{U}=\mathrm{Sp}(A)\subset\mathcal{T}^2$ be an open affinoid
    neighbourhood of $(\delta_1,\delta_2)$.
    Let $(\widetilde{\delta_1},\widetilde{\delta_2})\in\mathcal{T}^2(A)$ be the
    characters corresponding to  $\mathcal{U}\hookrightarrow\mathcal{T}^2$.
    Since $\mathcal{T}^+$ is a closed subspace of $\mathcal{T}$, we can choose $\mathcal{U}$
    in such a way so that $\widetilde{\delta_1}/\widetilde{\delta_2}\widehat\otimes k(u)\notin\mathcal{T}^+(k(u))$
    for any $u\in\mathcal{U}$.
    By Theorem \ref{cohomology of families}, we have
    $H^2_{\phi,\Gamma}(\mathcal{R}_A(\widetilde{\delta_1}/\widetilde{\delta_2}))=0$,
    in particular it is a locally free sheaf over $\mathcal{U}$.
    By Theorem \ref{extensions}, there exists a vector bundle
    $Y\rightarrow\mathcal{U}$ of rank $\dim_LH^1(\robba_L(\delta_1/\delta_2))$
    and a family of $(\phi,\Gamma)$-modules $\D'$ on $Y$ such that there exists a surjective map
    \begin{align*}
        Y(k(\delta_1,\delta_2))=Y(L)&\twoheadrightarrow H^1_{\phi,\Gamma}
        (\mathcal{R}_A(\widetilde{\delta_1}/\widetilde{\delta_2})\widehat{\otimes}_AL)
        =\mathrm{Ext}^1_{\phi,\Gamma}(\mathcal{R}_L(\delta_2),\mathcal{R}_L(\delta_1))\\
        y&\mapsto y^*\D'.
    \end{align*}
    As the $\phigam$-module $D$ is an extension of $\robba_L(\delta_2)$
    by $\robba_L(\delta_1)$,
    there is a point $y\in Y(L)$ such that $y^*\D'=D$.
    Let us define the rigid $L$-space
    \begin{center}
        \begin{tikzcd}
        X\coloneqq Y\times\mathrm{Sp}(L)\arrow[r,"p_2"]\arrow[d,"p_1"]
        & \mathrm{Sp}(L)\arrow[d]\\
        Y\arrow[r]
        &\mathrm{Sp}(\mathbb{Q}_p)
        \end{tikzcd}
    \end{center}
    and let $\D\coloneqq p_1^*\D'$.
    Let $x$ be any point of $X$ such that $p_1(x)=y$, so that
    $x^*\D=(p_1\circ x)^*\D'=y^*\D'=D$.
    Now we only have to show that there is a Zariski-open dense subset
    $U\xhookrightarrow{j} Y$ such that $j^*\D$ has regular parameters.
    Let $\mathcal{U}^\reg:=\mathcal{U}\cap\mathcal{T}_2^{\mathrm{reg}}$ which is a
    Zariski-open dense in $\mathcal{U}$.
    The rigid space
    \begin{center}
        \begin{tikzcd}
        U\coloneqq X\times_{\mathcal{U}}\mathcal{U}^\reg\arrow[r,"q_2"]\arrow[d,"q_1"]
        & \mathcal{U}^\reg\arrow[d, "u"]\\
        X\arrow[r,"v"]
        &\mathcal{U}
        \end{tikzcd}
    \end{center}
    is Zariski-open dense in $X$: in fact $Y$ is a vector bundle over $\mathcal{U}$,
    so in particular the map $v$ is open and preimages of dense subsets through an open map are dense.
    Moreover $q_1^*\D$ has parameters
    $u^*(\widetilde\delta_1,\widetilde\delta_2)\in\mathcal{T}^2(\mathcal{U}^\reg)$, hence regular.
    \end{proof}

    \begin{prop}
    \label{n dim T-}
    Let $(\delta_1,\dots,\delta_{n+1})\in\mathcal{T}^{n+1}(L)$ such that
    $\delta_i/\delta_j\notin\mathcal{T}^+(L)$ for $i<j$ and let
    $D_{n+1}\in H^1_{\phi,\Gamma}(D_n(\delta_{n+1}^{-1}))$ where $D_n$ is a triangulable
    $\phigam$-module with parameters $(\delta_1,\dots,\delta_n)\in\T^n(L)$.
    Let us denote by $D_i$ the $i$-th piece of the filtration  with graded pieces
    $\robba_L(\delta_1),\dots,\robba_L(\delta_n)$.
    Then there exist
    \begin{itemize}
        \item
        a rigid space $X_2$ over $L$ which is also a vector bundle over an
        appropriate neighbourhood $\mathcal{U}_1\times\mathcal{U}_{2}$ of $(\delta_1,\delta_2)$ in $\T^2$;
        \item
        rigid spaces $X_i$ over $L$ which are also vector bundles over
        $X_{i-1}\times\mathcal{U}_i$, where $\mathcal{U}_i$ is an
        appropriate neighbourhood of $\delta_i$ in $\T$ for all $3\leq i\leq n+1$;
        \item
        families of $\phigam$-modules $\D_i$ over $X_i$ for all $2\leq i\leq n+1$;
        \item
        maps $(\widetilde\delta_1,\dots,\widetilde\delta_i):X_i\rightarrow\T^i$ for all $2\leq i\leq n+1$
    \end{itemize}
    such that there are
    \begin{enumerate}
        \item
        $x_i\in X_i$ such that $\D_i\widehat\otimes k(x_i)=D_i$ and
        $(\widetilde\delta_1,\dots,\widetilde\delta_i)\widehat\otimes k(x_i)=(\delta_1,\dots,\delta_i)$
        for all $2\leq i\leq n+1$;
        \item
        Zariski open dense subsets $U_i\xhookrightarrow{j_i}X_i$ such that $j_i^*\D_i$
        has a filtration of sub-$\phigam$-modules with graded pieces $\robba_{U_i}(\widetilde{\delta_1}),\dots,
        \robba_{U_i}(\widetilde\delta_i)$ and $j_i^*(\widetilde{\delta_1},\dots,\widetilde{\delta_i})\in\T^\reg_i(U_i)$
        for all $2\leq i\leq n+1$.
    \end{enumerate}
    \end{prop}

    \begin{proof}
    We proceed by induction on $i$.
    Proposition \ref{2 dim T-} shows that there exist
    \begin{itemize}
        \item
        a rigid space $X_2$ over $L$ which is also a vector bundle over an
        appropriate neighbourhood $\mathcal{U}_{1}\times\mathcal{U}_{2}$ of $(\delta_1,\delta_2)$ in $\T^2$;
        \item
        a family of $\phigam$-modules $\D_2$ over $X_2$;
        \item
        a map $(\widetilde\delta_1,\widetilde\delta_2):X_2\rightarrow\T^2$
    \end{itemize}
    such that there are
    \begin{enumerate}
        \item
        $x_2\in X_2$ such that $\D_2\widehat\otimes k(x_2)=D_2$ and
        $(\widetilde\delta_1,\widetilde\delta_2)\widehat\otimes k(x_2)=(\delta_1,\delta_2)$;
        \item
        a Zariski open dense subset $U_2\xhookrightarrow{j_2}X_2$ such that $j_2^*\D_2$
        has a filtration of sub-$\phigam$-modules with graded pieces $\robba_{U_2}(\widetilde{\delta_1}),
        \robba_{U_2}(\widetilde\delta_2)$ and $j_2^*(\widetilde{\delta_2},\widetilde{\delta_2})\in\T^\reg_2(U_2)$.
    \end{enumerate}
    Hence the statement is true for $i=2$.
    Assume now that the statement of the Proposition is true for any $i<n+1$.
    Let $X_n$ be the vector bundle over $X_{n-1}\times\mathcal{U}_n$ as described in the statement
    of the Proposition, which means in particular that
    it is a vector bundle over the neighbourhood
    $\widetilde{\mathcal{U}}_n\coloneqq\mathcal{U}_1\times\dots\times\mathcal{U}_n$ of
    $(\delta_1,\dots,\delta_n)$ inside $\mathcal{T}^n$; let moreover $\D_n$ be the family
    of $\phigam$-modules over $X_n$ satisfying the properties above.
    Let $\mathcal{U}_{n+1}$ be a neighbourhood of $\delta_{n+1}$ in $\mathcal{T}$
    and similarly let $\widetilde\delta_{n+1}$ be the character corresponding to
    $\mathcal{U}_{n+1}\rightarrow\mathcal{T}$.
    Let us define the open neighbourhood $\widetilde{\mathcal{U}}_{n+1}\coloneqq\widetilde{\mathcal{U}}_n\times\mathcal{U}_{n+1}$
    of $(\delta_1,\dots,\delta_{n+1})$ in $\mathcal{T}^{n+1}$.
    We intersect $\widetilde{\mathcal{U}}_{n+1}$ with the open neighbourhood of $(\delta_1,\dots,\delta_{n+1})$
    defined as in Lemma \ref{good neighbourhood}.
    In particular we can assume that
    $\widetilde\delta_i/\widetilde\delta_{n+1}\widehat\otimes k(u)\notin\mathcal{T}^+(k(u))$ for
    all $i<n+1$ and for all $u\in\widetilde{\mathcal{U}}_{n+1}$.
    Let us consider the space $X_n\times\mathcal{U}_{n+1}$, which is a vector bundle over $\widetilde{\mathcal{U}}_{n+1}$.
    For simplicity of notation, we denote by $\D_n$ and $\robba_{\mathcal{U}_{n+1}}(\widetilde\delta_{n+1})$
    the pullback of these two $\phigam$-modules to $X_n\times\mathcal{U}_{n+1}$.
    By the way we have chosen the neighbourhood $\mathcal{U}_{n+1}$, we have
    that the sheaf $H^2_{\phi,\Gamma}(\D_n(\widetilde\delta_{n+1}^{-1}))$
    is locally free over $X_n\times\mathcal{U}_{n+1}$.
    Hence by Theorem \ref{extensions}, there is a vector bundle
    $Y_{n+1}\rightarrow X_n\times\mathcal{U}_{n+1}$
    and a family of $(\phi,\Gamma)$-modules $\D'_{n+1}$ on $Y_{n+1}$
    such that there exists a surjective map
    \begin{align*}
        Y_{n+1}(k(\delta_1,\dots,\delta_{n+1}))=Y_{n+1}(L)&\twoheadrightarrow H^1_{\phi,\Gamma}
        (\D_n(\widetilde{\delta}_{n+1}^{-1})\widehat{\otimes}L)
        =\mathrm{Ext}^1_{\phi,\Gamma}(\mathcal{R}_{L}(\delta_{n+1}),D_n)\\
        y&\mapsto y^*\D'_{n+1}.
    \end{align*}
    This implies that there is $y_{n+1}\in Y_{n+1}(L)$ such that $y_{n+1}^*\D'_{n+1}=D_{n+1}$.
    It is now sufficient to condider the space
    \begin{center}
        \begin{tikzcd}
        X_{n+1}\coloneqq Y_{n+1}\times\mathrm{Sp}(L)\arrow[r,"p_2"]\arrow[d,"p_1"]
        & \mathrm{Sp}(L)\arrow[d]\\
        Y_{n+1}\arrow[r]
        &\mathrm{Sp}(\mathbb{Q}_p)
        \end{tikzcd}
    \end{center}
    and let $\D_{n+1}\coloneqq p_1^*\D'_{n+1}$.
    Any point $x_{n+1}\in X_{n+1}(L)$ such that $p_1(x_{n+1})=y_{n+1}$ has the property that
    $x_{n+1}^*\D_{n+1}=(p_1\circ x_{n+1})^*\D_{n+1}'=y_{n+1}^*\D_{n+1}'=D_{n+1}$.
    Finally, let $\widetilde{\mathcal{U}}_{n+1}^\reg:=\widetilde{\mathcal{U}}_{n+1}\cap\mathcal{T}_{n+1}^{\mathrm{reg}}$ which is a
    Zariski-open dense in $\widetilde{\mathcal{U}}_{n+1}$.
    The rigid space
    \begin{center}
        \begin{tikzcd}
        U_{n+1}\coloneqq X_{n+1}\times_{\widetilde{\mathcal{U}}_{n+1}}\widetilde{\mathcal{U}}_{n+1}^\reg\arrow[r,"q_2"]\arrow[d,"q_1"]
        & \widetilde{\mathcal{U}}_{n+1}^\reg\arrow[d, "u"]\\
        X_{n+1}\arrow[r,"v"]
        &\widetilde{\mathcal{U}}_{n+1}
        \end{tikzcd}
    \end{center}
    is Zariski-open dense in $X_{n+1}$ (in fact $Y_{n+1}$ is a vector bundle over $\widetilde{\mathcal{U}}_{n+1}$,
    making the map $v$ open).
    Moreover $q_1^*\D_{n+1}$ has parameters
    $u^*(\widetilde\delta_1,\dots,\widetilde\delta_{n+1})\in\mathcal{T}^{n+1}(\widetilde{\mathcal{U}}_{n+1}^\reg)$,
    therefore they are regular.
    \end{proof}

    \begin{proof}[Proof of Theorem \ref{RESULT1}]
    Using Theorem \ref{strategy}, the Propositions \ref{2 dim T-} and \ref{n dim T-} above
    prove that for any trianguline representation $\rho:\absGal\rightarrow GL_n(\mathcal{O}_L)$
    satisfying the hypothesis of Theorem \ref{RESULT1},
    we have $(\rho,(\delta_1,\dots,\delta_n))\in\trivar(L)$.
    \end{proof}

\section{Some preparation}
\label{section 5}

  Before approaching the proof of Theorem \ref{result2}, we need some more preparation.

\subsection{The Beauville-Laszlo Theorem}
  \label{beauville-laszlo}

  The results presented in this section are mainly taken from Section 5.2 in \cite{matthiasthesis}.
  Fix $r\in p^{\mathbb{Q}}\cap[0,1)$.
  Then there exists $m_0=m_0(r)\in\mathbb{N}$ (depending on $r$)
  such that for every $m\geq m_0$, the point $1-\zeta_{p^m}$
  lies in the half open annulus $\mathbb{B}^r$, where $\zeta_{p^m}$ is a $p^m$-th root of unity.
  This means that the minimal polynomial of $1-\zeta_{p^m}$ over $\mathbb{Q}_p$
  defines a point $\alpha_m$ in $\mathbb{B}^r$ with residue field $K_m\coloneqq\mathbb{Q}_p(\zeta_{p^m})$.
  Therefore for all $m\geq m_0$, there exists a surjective morphism
  \[
    \mathcal{R}_{\Qp}^r\twoheadrightarrow K_m.
  \]
  The element $t\coloneqq\mathrm{log}(1+T)$ is a uniformizer of the completed local ring
  at each $\alpha_m$; hence we have that $\widehat{(\mathcal{R}^r_{\Qp})}_{\alpha_m}$ is isomorphic
  to a power series ring in one variable and with coefficients in $K_m$.
  Therefore this gives a map
  \[
    \mathcal{R}_{\Qp}^r\rightarrow K_m\llbracket t\rrbracket.
  \]
  Let us then fix $m\geq m_0$ and let us assume that $L$ contains $K_m$.
  For any affinoid $L$-algebra $A$, we have a morphism
  \[
  \iota_{r,m}:\mathcal{R}^r_A=\mathcal{R}^r\widehat{\otimes}_{\mathbb{Q}_p}A\rightarrow
  K_m\llbracket t\rrbracket\otimes_{\mathbb{Q}_p}A\cong\prod_{K_m\hookrightarrow L}A\llbracket t\rrbracket.
  \]

  Let $r\leq s<1$ such that $s\in p^\Q\cap[0,1)$.

  \begin{notation}
    From now on for a ring $R$, we denote by $\Mod(R)$ the category of modules over $R$.
  \end{notation}

  \begin{dfn}
    We define
    \[
      \mathrm{Mod}(\mathcal{R}_A^{[r,s]}[1/t])\times_{\mathrm{Mod}(\prod A((t)))}
      \mathrm{Mod}(\prod_{K_m\hookrightarrow L}A\llbracket t\rrbracket)
    \]
    to be the category of glueing data whose objects are tuples $(M^{[r,s]},\Lambda^{[r,s]},f^{[r,s]})$ where
    \begin{itemize}
      \item
        $M^{[r,s]}\in\mathrm{Mod}(\mathcal{R}_A^{[r,s]}[1/t])$;
      \item
        $\displaystyle\Lambda^{[r,s]}\in\mathrm{Mod} (\prod_{K_m\hookrightarrow L}A\llbracket t\rrbracket)$;
      \item
        $$
          f^{[r,s]}:M^{[r,s]}\otimes_{\mathcal{R}_A^{[r,s]}[1/t]}\prod_{K_m\hookrightarrow L}A((t))\xrightarrow{\sim}\Lambda^{[r,s]}[1/t]
        $$
        is an isomorphism of $\displaystyle\prod_{K_m\hookrightarrow L}A((t))$-modules.
    \end{itemize}
  \end{dfn}

  The following Theorem can be found in \cite[(5.4)]{matthiasthesis} and it is an analogue for
  $\phigam$-modules of the Theorem of Beauville-Laszlo (cf. \cite{beauville1995lemme}).

  \begin{thm}
    \label{equivalence of categories}
    Let us assume that $\alpha_m\in\mathbb{B}^{[r,s]}$ and no other point $\alpha_n$
    obtained from the minimal polynomial of $1-\zeta_{p^n}$ over $\mathbb{Q}_p$
    is contained in $\mathbb{B}^{[r,s]}$ (for example $s= r^{1/p}$).
    There is an equivalence of categories
    \begin{align*}
      \mathrm{Mod}(\mathcal{R}_A^{[r,s]})\xrightarrow{\sim}
      &\mathrm{Mod}(\mathcal{R}_A^{[r,s]}[1/t])\times_{\mathrm{Mod}(\prod A((t)))}
      \mathrm{Mod}(\prod_{K_m\hookrightarrow L}A\llbracket t\rrbracket)\\
      D^{[r,s]}\mapsto&
      (D^{[r,s]}[1/t],D^{[r,s]}\otimes_{\mathcal{R}_A^{[r,s]}}\prod_{K_m\hookrightarrow L}A\llbracket t\rrbracket, f^{[r,s]}),
    \end{align*}
    where the isomorphism $f^{[r,s]}$ of $\prod_{K_m\hookrightarrow L} A((t))$-modules is given by
    \begin{equation}
      \label{definition of map f}
      (D^{[r,s]}\otimes_{\mathcal{R}_A^{[r,s]}}\prod_{K_m\hookrightarrow L}A\llbracket t\rrbracket)[1/t]
      \cong D^{[r,s]}[1/t]\otimes_{\mathcal{R}_A^{[r,s]}[1/t]}\prod_{K_m\hookrightarrow L}A((t))\\
    \end{equation}
    via $t^i(m\otimes_{\mathcal{R}_A^{[r,s]}} 1)\mapsto (m\otimes_{\mathcal{R}_A^{[r,s]}}t^i)
    \otimes_{\mathcal{R}_A^{[r,s]}[1/t]}1$ for $m\in D^{[r,s]}$ and $i\in\Z$.
    The inverse of this equivalence is given by
    \[
      \ker(\overline{f^{[r,s]}})\mapsfrom(M^{[r,s]},\Lambda^{[r,s]}, f^{[r,s]}),
    \]
    where
    \[
      \overline{f^{[r,s]}}:M^{[r,s]}\rightarrow M^{[r,s]}\otimes_{\mathcal{R}_A^{[r,s]}[1/t]}\prod_{K_m\hookrightarrow L}A((t))
      \xrightarrow{f^{[r,s]}}\Lambda^{[r,s]}[1/t]\twoheadrightarrow\Lambda^{[r,s]}[1/t]/\Lambda^{[r,s]}.
    \]
  \end{thm}

  \subsection{$\phi$-modules over $\robba_A^r$}

  We'll see that any $\phigam$-module over the Robba ring can also be represented by a $\phigam$
  -module over the subring $\robba^{r}$ for some appropriate $r$ depending on the module.
  In this section we define the notion of $\phigam$-modules over $\robba^r$ and see some properties.
  Let $0\leq r\leq s<1$ such that $r,s\in p^\Q\cap[0,1)$ and let $A$ be an affinoid
  $L$-algebra.
  The rings $\robba_A^r$ and $\robba_A^{[r,s]}$ are both equipped with a Frobenius morphism
  \begin{align*}
    \phi:&\robba_A^{[r,s]}\rightarrow\robba_A^{[r^{1/p},s^{1/p}]},\\
    \phi:&\robba_A^r\rightarrow\robba_A^{r^{1/p}}
  \end{align*}
  which is simply the restriction of the Frobenius on the Robba ring $\robba_A$.
  Let us define the restriction map
  \[
    \mathrm{res}^r_{r^{1/p}}:\mathbb{B}_A^{r^{1/p}}\hookrightarrow\mathbb{B}_A^r.
  \]

  \begin{dfn}
    \begin{itemize}
      \item
        A $\robba^r_A$-module $D$ is called a \emph{$\phi$-module} if $D$ is projective,
        of finite presentation and there exists an isomorphism
        \[
          \phi^*D\cong (\mathrm{res}^r_{r^{1/p}})^*D.
        \]
      \item
        A \emph{$\phigam$-module over $\robba^r_A$} is a $\phi$-module $D$ over $\robba^r_A$
        endowed with a continuous action of $\Gamma$ commuting with $\phi$.
    \end{itemize}
  \end{dfn}

  \begin{dfn}
    Let $A$ be an affinoid $L$-algebra and let $r_0\in p^\Q\cap[0,1)$.
    \begin{itemize}
      \item[(i)]
        A  \emph{coherent sheaf} on $\robba_A^{r_0}$ consists of a family of finite
        $\robba_A^{[r,s]}$-modules $D^{[r,s]}$ for all $r_0\leq r\leq s<1$
        together with isomorphisms
        \[
          D^{[r',s']}\cong D^{[r,s]}\otimes_{\robba_A^{[r,s]}}\robba_A^{[r',s']}
        \]
        for all $r_0\leq r\leq r'\leq s'\leq s<1$ satisfying the cocycle conditions.
      \item[(ii)]
        A \emph{vector bundle} on $\robba_A^{r_0}$ is a coherent sheaf $(D^{[r,s]})_{r_0\leq r\leq s<1}$
        on $\robba_A^{r_0}$ such that each $D^{[r,s]}$ is a projective module over $\robba_A^{[r,s]}$.
    \end{itemize}
  \end{dfn}

  The following Theorem is due to \cite[Proposition 2.2.7.]{kedlayaCohomology}.

  \begin{thm}
    \label{phi-modules over Rr}
    For every affinoid $L$-algebra $A$ and for any $r_0\in p^\Q\cap[0,1)$,
    there exists an equivalence of categories
    \begin{align*}
      \{\phi\text{-modules over }\mathcal{R}_A^{r_0}\}
      \xrightarrow{\sim}
      &\left\{\begin{array}{l}\text{vector bundles }
      (D^{[r,s]})_{r_0\leq r\leq s<1}\text{ on }\mathcal{R}_A^{r_0}\text{ together with }\\
      \phi^*D^{[r,s]}\cong D^{[r^{1/p},s^{1/p}]}\text{ compatible with restrictions}\end{array}\right\}\\
      D^{r_0}\mapsto
      &(D^{r_0}\otimes_{\mathcal{R}_A^{r_0}}\mathcal{R}_A^{[r,s]})_{r_0\leq r\leq s<1}.
      \end{align*}
  \end{thm}

  As already mentioned in the beginning of this section, any $\phigam$-module $D$ over the Robba ring can be seen as
  a $\phigam$-module over $\robba^r$ for some $r$ depending on $D$.
  More precisely, we have the following Theorem proved in \cite[Lemma 2.2.8]{bellaiche2009families}.

  \begin{thm}
    \label{sub-phigam-module}
    Let $D$ be any $\phigam$-module over $\mathcal{R}_L$. There exists a unique
    $r(D)\in p^{\mathbb{Q}}\cap[0,1)$ such that for all $r\geq r(D)$, there is a unique $\phi$-module $D^r$
    over $\mathcal{R}_L^r$ such that $D=D^r\otimes_{\mathcal{R}_L^r}\mathcal{R}_L$.
    Moreover $D^r$ is stable under the action of $\Gamma$ induced by $D$ on $D^r\otimes_{\mathcal{R}_L^r}\mathcal{R}_L$.
  \end{thm}

  More generally, let $\D$ be a family of free $\phigam$-modules over a rigid affinoid space
  $\Sp(A)$ and let $\{e_1,\dots,e_d\}$ be a basis of $\D$.
  Let us consider the matrix $M$ of $\Phi$ in this basis.
  Then there exists $r(D)\in p^{\mathbb{Q}}\cap[0,1)$ such that $M\in GL_d(\robba_A^r)$
  for all $r\geq r(D)$.
  We can then take $\D^r\coloneqq\oplus_{i=1}^d\robba_A^re_i$ and
  we have the following Remark.

  \begin{rem}
    Let $\D$ be a family of locally free $\phigam$-modules over a rigid space $X$.
    Then there exists $r\in p^\Q\cap [0,1)$ and a $\phi$-module $\D^r$ over $\robba_X^r$
    such that $\D^r\otimes_{\robba_X^r}\robba_X=\D$.
  \end{rem}

  \subsection{From $\phigam$-modules over $\robba$ to $\phigam$-modules over $\robba[1/t]$}

  We will show how to view a $\phigam$-module over $\robba_L$ as a $\phigam$-module over
  $\robba_L[1/t]$ and a lattice without losing any information.

  Let $A$ be an affinoid $L$-algebra.
  First of all notice that the morphism
  \[
    \phi:\robba_A\rightarrow\robba_A, T\mapsto(1+T)^p-1
  \]
  defined in (\ref{phi over R}) can be extended to
  \[
    \phi:\robba_A[1/t]\rightarrow\robba_A[1/t]
  \]
  simply  by $\frac{1}{t}\mapsto \frac{1}{pt}$.

  \begin{dfn}
    A \emph{$\phigam$-module over $\robba_A[1/t]$} is a locally free $\robba_A[1/t]$-module $M$
    of finite presentation endowed with
    \begin{itemize}
      \item
        a $\robba_A[1/t]$-semilinear endomorphism $\Phi: M\rightarrow M$ such that
        $\Phi:\phi^*M\rightarrow M$ is an isomorphism;
      \item
        a continuous $\robba_A[1/t]$-semilinear $\Gamma$-action commuting with $\Phi$.
    \end{itemize}
  \end{dfn}

  Let $(\phi,\Gamma)-\mathrm{Mod}(\mathcal{R}_A^r)$ denote the category of $\phigam$-modules
  over $\mathcal{R}_A^r$ and $(\phi,\Gamma)-\mathrm{Mod}(\mathcal{R}_A^r[1/t])$
  the category of $\phigam$-modules over $\mathcal{R}_A^r[1/t]$.

  \begin{dfn}
    A \emph{lattice $\Lambda$ in $A((t))^n$} is a
    finitely generated projective $A\llbracket t\rrbracket$-module such that
    \[
      \Lambda \otimes_{A\llbracket t\rrbracket}A((t))\cong A((t))^n.
    \]
  \end{dfn}

  \begin{dfn}
    For $r\in p^\Q\cap[0,1)$ and $m\in\N$ satisfying the conditions presented in the beginning of \S\ref{beauville-laszlo},
    let us denote by
    $$\Category$$
    the category whose objects are tuples $(M,\Lambda,f)$ where
    \begin{itemize}
      \item
        $M\in(\phi,\Gamma)-\mathrm{Mod}(\mathcal{R}^r_A[1/t])$;\\
      \item
        $\displaystyle\Lambda\in\mathrm{Mod}(\prod_{K_m\hookrightarrow L}A\llbracket t\rrbracket)$;
      \item
        $\displaystyle f:M\otimes_{\mathcal{R}_A^{r}[1/t]}\prod_{K_m\hookrightarrow L}A((t))
        \xrightarrow{\sim}\Lambda[1/t]$ is an isomorphism of $ \prod A((t))$-modules
    \end{itemize}
    and $\Lambda$ is stable under the action of $\Gamma$ induced on $\Lambda[1/t]$ under $f$
    from the diagonal action of $\Gamma$ on $M\otimes_{\mathcal{R}_A^{r}[1/t]}\prod_{K_m\hookrightarrow L}A((t))$.
  \end{dfn}

  \begin{thm}
    \label{equivalenceBL}
    Let $r_0\in p^\Q\cap[0,1)$ such that $\mathbb{B}^{[r_0,r_0]}$ does not contain any zero for $t$
    (i.e. $r_0\neq|1-\zeta_{p^m}|$ for all $m\in\N$).
    There is an equivalence of categories
    \begin{align}
      \label{equivalenceBeauvilleLaszlo}
      (\phi,\Gamma)-\mathrm{Mod}(\mathcal{R}_A^{r_0})\xrightarrow{\sim}
      &\Categoryr\\
      D^{r_0}\mapsto
      &(D^{r_0}[1/t],D^{r_0}\otimes_{\mathcal{R}_A^{r_0}}\prod_{K_m\hookrightarrow L}A\llbracket t\rrbracket,f),
      \nonumber
    \end{align}
    where
    \[
      f:D^{r_0}[1/t]\otimes_{\mathcal{R}_A^{r_0}[1/t]}\prod_{K_m\hookrightarrow L}
      A((t))\xrightarrow{\sim}
      (D^{r_0}\otimes_{\mathcal{R}_A^{r_0}}\prod_{K_m\hookrightarrow L}A\llbracket t\rrbracket)[1/t]
    \]
    via $e_i'\otimes 1\mapsto e_i\otimes 1$ with $e_1,\dots,e_n$ basis of $D^{r_0}$ over $\mathcal{R}_A^{r_0}$
    and $e_1'=e_1\otimes_{\mathcal{R}_A^{r_0}}1,\dots,e_n'=e_n\otimes_{\mathcal{R}_A^{r_0}}1$ basis
    of $D^{r_0}[1/t]$ over $\mathcal{R}_A^{r_0}[1/t]$.
  \end{thm}

  \begin{proof}
    Let $\Lambda\coloneqq D^{r_0}\otimes_{\mathcal{R}_A^{r_0}}\prod_{K_m\hookrightarrow L}A\llbracket t\rrbracket$.
    First of all, notice that the $\Gamma$-action of $D^{r_0}[1/t]$ extends diagonally to
    \[
      f:D^{r_0}[1/t]\otimes_{\mathcal{R}_A^{r_0}[1/t]}\prod_{K_m\hookrightarrow L}
      A((t))\xrightarrow{\sim}\Lambda[1/t]
    \]
    and stabilizes $\Lambda$ by definition.
    In order to show that the above is an equivalence, we describe the inverse functor:
    let $(M, \Lambda, f)$ be an object of the right hand side of (\ref{equivalenceBeauvilleLaszlo}).
    We denote by $M_0$ and $\Lambda_0$ the restriction of $M$ and $\Lambda$ respectively to
    th annulus $\mathbb{B}^{[r_0,r_1]}$, where $r_1\coloneqq r_0^{1/p}$; moreover
    let us denote by $\Lambda_i\coloneqq\phi^*\Lambda_{i-1}$ and $M_i\coloneqq \phi^*M_{i-1}$ for $i\geq1$.
    If we let $r_i\coloneqq r_0^{1/p^i}$ for $i\geq1$, this yields to a family of objects
    $$
      ((M_i,\Lambda_i,f_i)\in
      \mathrm{Mod}(\mathcal{R}_A^{[r_i,r_{i+1}]}[1/t])\times_{\mathrm{Mod}(\prod A((t)))}
      \mathrm{Mod}(\prod_{K_m\hookrightarrow L}A\llbracket t\rrbracket))_i
    $$
    for $i\geq0$ such that $M_i$ is a $\phi$-module
    over $\mathcal{R}_A^{[r_i,r_{i+1}]}[1/t]$ endowed with an action of $\Gamma$.
    Thanks to Theorem \ref{equivalence of categories},
    this corresponds to a family of vector bundles $(D^{[r_i,r_{i+1}]})_{i\geq0}$ with a
    $\Gamma$-action and such that $\phi^*D^{[r_i,r_{i+1}]}\cong D^{[r_{i+1},r_{i+2}]}$.
    We will now show that we can glue $D^{[r_{i-1},r_{i}]}$ and $D^{[r_i,r_{i+1}]}$
    along the intersection $\mathbb{B}^{[r_i, r_i]}$ for all $i\geq1$.
    By assumption, we have that
    \[
      D^{[r_i,r_{i+1}]}_{|\mathbb{B}^{[r_i, r_{i+1}]}\setminus V(t)} = M_{i|\mathbb{B}^{[r_i, r_{i+1}]}\setminus V(t)}
    \]
    for all $i\geq0$.
    By hypothesis, we have that $V(t)$ does not intersect $\mathbb{B}^{[r_i, r_{i}]}$ for all $i\geq0$, thus
    \[
      D^{[r_{i-1},r_{i}]}_{|\mathbb{B}^{[r_i, r_{i}]}}\cong M_{i-1|\mathbb{B}^{[r_i, r_{i}]}}\cong
      \phi^* M_{i-1|\mathbb{B}^{[r_i, r_{i}]}} \cong M_{i|\mathbb{B}^{[r_i, r_{i}]}} \cong D^{[r_i,r_{i+1}]}_{|\mathbb{B}^{[r_i, r_{i}]}}.
    \]
    Therefore we have that the family $(D^{[r_i,r_{i+1}]})_{i\geq0}$ is compatible with restrictions and
    by Theorem \ref{phi-modules over Rr}, this gives a $\phigam$-module $D^{r_0}$ over $\mathcal{R}_A^{r_0}$.
  \end{proof}

  \begin{rem}
    \label{stable lattices}
    Let $K_\infty:=\bigcup_m K_m$ and let $\Gamma_m:=\mathrm{Gal}(K_\infty/K_m)\subset\Gamma=\mathrm{Gal}(K_\infty/\mathbb{Q}_p)$.
    Let $(M,\Lambda,f)$ be an object of $\Category$.
    Notice that the action of $\Gamma_m$ on
    \[
      M\otimes_{\mathcal{R}_A^{r}[1/t]}\prod_{K_m\hookrightarrow L}A((t))\cong\prod_{K_m\hookrightarrow L}A((t))^n
    \]
    preserves each factor of the product.
    Moreover if we fix an embedding $\tau:K_m\hookrightarrow L$ and a $\Gamma_m$-stable
    lattice $\Lambda_\tau$, then it is possible to recover $\Lambda$ from it:
    let $g_m$ be a generator of $\mathrm{Gal}(K_m/\Q_p)$, then we have
    \[
      g_m^n\cdot\Lambda_\tau=\Lambda_{\tau_n}
    \]
    where $\tau_n:K_m\xrightarrow{g_m^n}K_m\xhookrightarrow{\tau}L$.
    Thus
    \[
      \Lambda=\prod_{n}g_m^n\cdot\Lambda_\tau,
    \]
    where the product ranges over $n\in\{0,\dots,|\mathrm{Gal}(K_m/\Q_p)|-1\}$ and
    $\Lambda_\tau$ is a $\Gamma_m$-stable lattice in $A((t))^n$.
    Hence a $\Gamma$-stable lattice in $\prod A((t))^n$ is equivalent to a $\Gamma_m$-stable lattice in
    one of the factors $A((t))^n$, as the $\Gamma$-action translates the lattice to the other factors.
  \end{rem}

  Notice furthermore that we have the following bijection.

  \begin{prop}
    \label{bijection of lattices and modules}
    Let $M\in\phigam-\Mod(\robba_A[1/t])$ for some affinoid algebra $A$.
    Let $r\in [0,1)\cap\Q$ such that $\mathbb{B}^{[r,r]}$ does not contain any zero for $t$
    and such that there exists a $\phigam$-module $M^r$ over $\robba_A^r[1/t]$
    such that $M^r\otimes_{\robba_A^r[1/t]}\robba_A[1/t]=M$ (this is always possible by Theorem \ref{sub-phigam-module})
    and let us take $m\in\N$
    big enough so that there is a surjective morphism $\robba_{\Qp}^r\twoheadrightarrow K_m$
    as explained at the beginning of Section \ref{beauville-laszlo}.
    Let us assume that $L$ is big enough so that
    $K_m\subset L$ (we can always do that thanks to Lemma \ref{enlarge L})
    and let us fix an embedding $\tau:K_m\hookrightarrow L$.
    Let $\mathrm{Gr}_M^\Gamma$  be the set
    \begin{equation}
      \label{Gamma-stable lattices in M}
      \mathrm{Gr}_M^\Gamma\coloneqq \left\{\Lambda\colon
      \Lambda\text{ is a }\Gamma_m\text{-stable lattice of }M^r\otimes_{\robba_A^r[1/t],\tau}A((t))
      \right\}
    \end{equation}
    and let $M^{\phigam}$ be the set
    \begin{equation}
      \label{submodules of M}
      M^{\phigam}\coloneqq \{D\in\phigam-\Mod(\robba_A)\colon D\subset M\text{ and }D[1/t]=M\}.
    \end{equation}
    Then the functor described in Theorem \ref{equivalenceBL} induces the
    bijection
    \begin{align*}
      M^{\phigam}&\rightarrow \mathrm{Gr}_M^\Gamma\\
      D&\mapsto (D^r\otimes_{\robba_A^r,\tau}A\llbracket t\rrbracket),
    \end{align*}
    where $D^r$ is a $\phigam$-module over $\robba_A^r$ such that
    $D=D^r\otimes_{\robba_A^r}\robba_A$.
  \end{prop}

  \begin{rem}
    Notice that for a $\phigam$-module $D\in M^{\phigam}$,
    we have a morphism of $\phigam$-modules over $\robba_A$
    \[
      \iota:D\hookrightarrow M\xrightarrow{\rr id}M^r\otimes_{\robba_A^r[1/t]}\robba_A[1/t],
    \]
    whose image is obviously $D$ and of the form $D^r\otimes_{\robba_A^r}\robba_A$
    for some sub-$\robba_A^r$-module $D^r$ of $M^r$.
    Let $\gamma$ be a topological generator of $\Gamma$.
    The fact that $\iota$ is a morphism of $\phigam$-modules implies that the image of $D^r$ through
    $\gamma$ is contained in $D^r \otimes_{\robba_A^r}\robba_A$, therefore (after choosing a basis for $D^r$)
    the matrix of $\gamma$ has coefficients in $\robba_A$.
    At the same time, the module $M^r$ is stable under $\gamma$ by hypothesis, thus the matrix
    of $\gamma$ in the same basis has coefficients in $\robba_A^r[1/t]$.
    This implies that the coefficients of the matrix of the action of $\gamma$ on $D^r$
    has coefficients in $\robba^r_A$ and therefore $D^r$ is $\Gamma$-stable.
    The fact that $M^r$ is a $\phigam$-module over $\robba^r_A[1/t]$ implies that
    $\phi^*D^r\cong(\rr{res}_r^{r^{1/p}})^*D^r$.
    Thus we have that such a $D^r$ is a $\phigam$-module over $\robba_A^r$
    and $D=D^r\otimes_{\robba_A^r}\robba_A$.
  \end{rem}

  \begin{proof}
    Let
    \[
      \mathcal{A}\coloneqq\left\{
        (\Lambda,f)\colon\begin{array}{c}
          \Lambda\text{ is a }\Gamma_m\text{-stable projective }A\llbracket t\rrbracket\text{-module and}\\
          f:\Lambda[1/t]\xrightarrow{\sim}M^r\otimes_{\robba_A^r[1/t],\tau}A((t))\text{ isomorphism of }A((t))\text{-modules}
        \end{array}
      \right\},
    \]
    where the action of $\Gamma_m$ on $\Lambda[1/t]$ is induced by the isomorphism $f$.
    Let moreover
    \begin{equation}
      \label{set B}
      \mathcal{B}\coloneqq\left\{
      (D^r,g)\colon \begin{array}{c}
      D^r\text{ is a }\phigam\text{-module over }\robba_A^r\text{ and}\\
      g:D^r[1/t]\xrightarrow{\sim}M^r\text{ isomorphism of }\phigam\text{-modules over }\robba_A^r[1/t]
      \end{array}
      \right\}.
    \end{equation}
    By Remark \ref{stable lattices} and since the functor of Theorem \ref{equivalenceBL} is an equivalence,
    we obviously have that the functor induces a bijection between the sets of isomorphism classes of the above
    two sets.
    If we denote by $\widetilde{\mathcal{A}}$ and $\widetilde{\mathcal{B}}$ the sets of isomorphism classes
    of $\mathcal{A}$ and $\mathcal{B}$ respectively, we just have to show that
    we have bijections
    \[
      \begin{array}{lcr}
        \begin{aligned}
          \mathrm{Gr}_M^\Gamma&\leftrightarrow\widetilde{\mathcal{A}}\\
          \Lambda&\mapsto(\Lambda,\mathrm{id}_{\Lambda[1/t]})
        \end{aligned}
        &\text{ and }
        &\begin{aligned}
          M^{\phigam}&\leftrightarrow\widetilde{\mathcal{B}}\\
          D&\mapsto(D^r,\mathrm{id}_{D^r[1/t]}).
        \end{aligned}
      \end{array}
    \]
    \begin{itemize}
      \item
      Let's start showing the injectivity of the first map: assume that $\Lambda,\Lambda'\in\mathrm{Gr}_M^\Gamma$
      such that $(\Lambda,\mathrm{id}_{\Lambda[1/t]})\cong(\Lambda',\mathrm{id}_{\Lambda'[1/t]})$.
      This means that there is an isomorphism $\xi:\Lambda\xrightarrow{\sim}\Lambda'$ of $A\llbracket t\rrbracket$-modules
      such that the diagram
      \[
        \begin{tikzcd}
          \Lambda\arrow[rr,"\xi"]\arrow[hookrightarrow]{d}
          &&\Lambda'\arrow[hookrightarrow]{d}\\
          \Lambda[1/t]\arrow[r,phantom,"="]
          & M^r\otimes_{\robba_A^r[1/t],\tau}A((t))
          &\Lambda'[1/t]\arrow[l,phantom,"="]
        \end{tikzcd}
      \]
      commutes.
      It is then obvious that $\xi$ must be the identity and so $\Lambda=\Lambda'$.
      As for surjectivity, let $(\Lambda,f)\in\widetilde{\mathcal{A}}$; notice then that
      $f(\Lambda)\subset M^r\otimes_{\robba_A^r[1/t],\tau}A((t))$ is a lattice and it is
      stable under the action of $\Gamma_m$ induced by $M^r$ by definition of the $\Gamma$-action on $\Lambda$
      and by the fact that $\Lambda$ is stable under $\Gamma_m$.
      Finally, we have that $(\Lambda,f)\xrightarrow{f}(f(\Lambda),\mathrm{id}_{f(\Lambda)[1/t]})$
      is an isomorphism in the category $\mathcal{A}$.
      \item
      The bijectivity of the second map is proved in an analogous way.
      More precisely, let $D,D'\in M^{\phigam}$ such that
      $(D^r,\mathrm{id}_{D^r[1/t]})\cong(D'^r,\mathrm{id}_{D'^r[1/t]})$.
      In other words, we have an isomorphism $\xi:D^r\xrightarrow{\sim}D'^r$ of $\phigam$-modules
      over $\mathcal{R}_A^r$ and a commutative diagram
      \[
        \begin{tikzcd}
          D^r\arrow[rr,"\xi"]\arrow[hookrightarrow]{d}
          &&D'^r\arrow[hookrightarrow]{d}\\
          D^r[1/t]\arrow[r,phantom,"="]
          & M^r
          &D'^r[1/t]\arrow[l,phantom,"="].
        \end{tikzcd}
      \]
      Thus $D^r=D'^r$ and $D=D^r\otimes_{\robba_A^r}\robba_A= D'^r\otimes_{\robba_A^r}\robba_A=D'$,
      hence the map is injective.
      On the other hand, let $(D^r,f)\in\widetilde{\mathcal{B}}$; notice that the $\robba_A^r$-module
      $f(D^r)\subset M^r$ is stable under the action of $\Gamma$ inherited by $M^r$:
      in fact by definition of morphism of $\phigam$-modules, we have that the action of
      $\Gamma$ commutes with $f$.
      Moreover, the fact that $\phi^*M^r\cong (\rr{res}_r^{r^{1/p}})^*M^r$ implies that
      $\phi^*f(D^r)\cong (\rr{res}_r^{r^{1/p}})^*f(D^r)$, therefore $f(D^r)$ is a $\phigam$-module over $\robba_A^r$.
      Obviously $f(D^r)[1/t]=M^r$, hence $f(D^r)\otimes_{\robba_A^r}\robba_A\in M^{\phigam}$.
      Finally
      $(D^r,f)\xrightarrow{f}(f(D^r),\mathrm{id}_{f(D^r)[1/t]})$ is an isomorphism in the category $\mathcal{B}$,
      thus the map $M^{\phigam}\rightarrow\widetilde{\mathcal{B}}$ is surjective.
    \end{itemize}
  \end{proof}

  \subsection{The $\phigam$-cohomology of $\phigam$-modules over $\robba[1/t]$}

  In this section we explore some nice results of how to compute the $\phigam$-cohomology
  of $\phigam$-modules over $\robba_L[1/t]$, which is defined using the Herr complex as in
  Definition \ref{phi,Gamma cohomology}.

  \begin{rem}
    \label{phigam cohomology with 1/t}
    Let $D$ be a $\phigam$-module over $\mathcal{R}_A$ for some affinoid $L$-algebra $A$.
    Since direct limits are exact, we have
    \[
      H^i_{\phi,\Gamma}\left(D\left[1/t\right]\right)=
      \lim_{\stackrel{\rightarrow}{n}}H^i_{\phi,\Gamma}(t^{-n}D)
      =\lim_{\stackrel{\rightarrow}{n}}H^i_{\phi,\Gamma}(D(x^{-n})).
    \]
    for $i=0,1,2$.
  \end{rem}

  \begin{lem}
    \label{twist}
    Let $\delta\in\T(L)$
    and let $k\in\mathbb{N}$ such that $\mathrm{wt}(\delta)\notin\{1,\dots,k\}$.
    Then
    \begin{itemize}
      \item[(i)]
      the morphism
      \[
        H^1_{\phi,\Gamma}(\mathcal{R}_L(\delta))\xrightarrow{\sim}H^1_{\phi,\Gamma}(\mathcal{R}_L(\delta x^{-k}))
      \]
      is an isomorphism;
      \item[(ii)]
      $H^0_{\phi,\Gamma}(t^{-k}\robba_L(\delta)/\robba_L(\delta))=0$.
    \end{itemize}
  \end{lem}

  The reader can find Lemma \ref{twist} and its proof in
  \cite[Théorème 2.22 (i), Proposition 2.18 (i)]{colmez2008representations}.

  \begin{lem}
    \label{lem:iso of H0}
    Let $\delta\in\T(L)$ such that $\mathrm{wt}(\delta)\notin\{1,\dots,k\}$.
    Then the map
    \[
      H^0_{\phi,\Gamma}(\robba_L(\delta))\rightarrow H^0_{\phi,\Gamma}(\robba_L(\delta\cdot x^{-k}))
    \]
    is an isomorphism.
  \end{lem}

  \begin{proof}
    Consider the short exact sequence
    \[
      0\rightarrow \robba_L(\delta)\rightarrow t^{-k}\robba_L(\delta)
      \rightarrow t^{-k}\robba_L(\delta)/\robba_L(\delta)\rightarrow0,
    \]
    which induces the long exact sequence of cohomologies
    \[
      0\rightarrow H^0_{\phi,\Gamma}(\robba_L(\delta))\rightarrow H^0_{\phi,\Gamma}(t^{-k}\robba_L(\delta))
      \rightarrow H^0_{\phi,\Gamma}(t^{-k}\robba_L(\delta)/\robba_L(\delta))\rightarrow\dots.
    \]
    By Lemma \ref{twist}, we have that $H^0_{\phi,\Gamma}(t^{-k}\robba_L(\delta)/\robba_L(\delta))=0$,
    which implies the wanted isomorphism.
  \end{proof}

  \begin{prop}
    \label{prop:high dim twist}
    Let $D_n$ be a triangulable $\phigam$-module over $\robba_L$ with parameters
    $(\delta_1,\dots\delta_n)\in\T^n(L)$ and let $\delta\in\T(L)$ such that
    $\mathrm{wt}(\delta_i/\delta)\notin\N_{>0}$ for $i=1,\dots,n$.
    We have
    \[
      H^1_{\phi,\Gamma}(D_n(\delta^{-1}))\xrightarrow{\sim} H^1_{\phi,\Gamma}(D_n(\delta^{-1}x^{-k})).
    \]
    is an isomorphism for all $k\in\N$.
  \end{prop}

  \begin{proof}
    We proceed by induction on $n$.
    \begin{enumerate}
      \item [$n=1:$]
        This case is proved by Lemma \ref{twist}.
      \item [$n\geq2:$]
        We want to show that the map
        \[
          H^1_{\phi,\Gamma}(D_n(\delta^{-1}))\rightarrow H^1_{\phi,\Gamma}(D_n(\delta^{-1}\cdot x^{-k}))
        \]
        is an isomorphisms for all $k \in \N$.
        Let $D_{n-1}$ be the $\phigam$-module over $\robba_L$ such that
        $D_n\in H^1_{\phi,\Gamma}(D_{n-1}(\delta_n^{-1}))$.
        This means we have a short exact sequence
        \[
          0\rightarrow D_{n-1}\rightarrow D_n\rightarrow \robba_L(\delta_n)\rightarrow0,
        \]
        which induces the long exact horizontal rows of the following diagram
        (where all cohomologies are $\phigam$-cohomologies):
        \[
          \begin{tikzcd}[column sep=tiny]
            H^0(\robba_L(\delta_n/\delta))\ar[r]\isoarrow{d}
            & H^1(D_{n-1}(\delta^{-1}))\ar[r]\isoarrow{d}
            & H^1(D_n(\delta^{-1}))\ar[r]\arrow{d}
            & H^1(\robba_L(\delta_n/\delta))\ar[r]\isoarrow{d}
            & 0\\
            H^0(\robba_L(\delta_n\delta^{-1}x^{-k}))\ar[r]
            & H^1(D_{n-1}(\delta^{-1}x^{-k}))\ar[r]
            & H^1(D_n(\delta^{-1}x^{-k}))\ar[r]
            & H^1(\robba_L(\delta_n\delta^{-1}x^{-k}))\ar[r]
            & 0.
          \end{tikzcd}
        \]
        Notice first of all that, since $\mathrm{wt}(\delta_i/\delta)\notin\N_{>0}$, we have that
        $H^2_{\phi,\Gamma}(D_{n-1}(\delta^{-1}))=H^2_{\phi,\Gamma}(D_{n-1}(\delta^{-1}x^{-k}))=0$.
        This explains why the two long exact sequences above are right-exact.
        Furthermore the first vertical arrow is an isomorphism by Lemma \ref{lem:iso of H0},
        the second vertical arrow is an isomorphism by inductive hypothesis and the last
        verical arrow is an isomorphism by Lemma \ref{twist}.
        We can therefore conclude by the five lemma that the middle vertical arrow is an isomorphism.
    \end{enumerate}
  \end{proof}

  \subsection{The affine Grassmannian}

  There exists a "space" with a universal family of lattices: the affine Grassmannian.
  In this section we recall some basic notions about it.

  Recall that a \emph{lattice} in $A((t))^n$ is a finitely generated projective
  $A\llbracket t\rrbracket$-module $\Lambda\subset A((t))^n$ such that $\Lambda[1/t]\cong A((t))^n$.

  Let $n\in\mathbb{N}_{>0}$.
  The \emph{affine Grassmannian} is the functor
  \begin{align*}
    \mathrm{Gr}_n: \underline{Rings}&\rightarrow\underline{Sets}
  \end{align*}
  sending a ring $A$ to the set of
  lattices in $A((t))^n$.

  \begin{thm}
    The affine Grassmannian is representable by an ind-projective scheme.
    More precisely, let $\mathrm{Gr}_{n,i}$ be the subfunctor of the affine Grassmannian
    sending a ring $A$ to the set
    consisting of those lattices $\Lambda$ such that
    \[
      t^iA\llbracket t\rrbracket^n\subset \Lambda\subset t^{-i}A\llbracket t\rrbracket^n
    \]
    for $i\in\mathbb{N}$.
    We have that $\mathrm{Gr}_{n,i}$ is representable by a projective scheme and
    the affine Grassmannian is the direct limit of the subfunctors $\mathrm{Gr}_{n,i}$.
  \end{thm}

  In particular, we have that each $\mathrm{Gr}_{n,i}$ is locally of finite type, hence
  we can consider their rigidification through the GAGA functor and obtain rigid spaces
  $\mathrm{Gr}_{n,i}^{\rig}$.
  Hence we define
  \[
    \mathrm{Gr}_n^\rig\coloneqq\lim_{i\in\mathbb{N}}\mathrm{Gr}_{n,i}^\rig.
  \]
  This is the functor on rigid spaces
  \[
    \mathrm{Gr}_n^\rig:\underline{Rig}\rightarrow\underline{Sets}
  \]
  sending an affinoid rigid space $\mathrm{Sp}(A)$ to the set of lattices in $A((t))^n$.
  Moreover, we denote by $\mathbb L$ the universal object of $\mathrm{Gr}_n^\rig$.

  Another description of the affine Grassmannian is through loop groups: given an algebraic group $G$,
  the \emph{loop group of $G$}
  is the presheaf $LG$ defined by $LG(R)\coloneqq G(R((t)))$ and the \emph{positive loop group of $G$}
  is the presheaf $L^+G$ defined by $L^+G(R)\coloneqq G(R\llbracket t\rrbracket)$.
  Then the affine Grassmannian is the fpqc quotient $LGL_n/L^+GL_n$.

  In this paragraph we study a bit more the geometry of the affine Grassmannian:
  let $B_n\subset GL_n$ be the Borel subgroup of upper triangular matrices of $GL_n$,
  $T_n\subset B_n$ the maximal torus of diagonal matrices and $U_n\subset B_n$ the unipotent
  radical of $B_n$.
  Let $X_*(T_n)$ be the coweight group of $T_n$, so the group of all algebraic homomorphisms $\mathbb{G}_m\rightarrow T_n$,
  which can also be identified with tuples $\mu=(\mu_1,\dots,\mu_n)\in\mathbb{Z}^n$.
  Moreover, let $X_*(T_n)^+$ be the group of dominant coweights of $T_n$; recall that we say that a coweight $(\mu_1,\dots,\mu_n)\in\Z^n$
  is \emph{dominant} if $\mu_1\leq\mu_2\leq\dots\leq\mu_n$.
  For dominant coweights $\lambda,\mu\in X_*(T_n)^+$, we say $\lambda\leq\mu$ if
  $\lambda-\mu$ is a sum of simple co-roots and this defines a partial order on $X_*(T_n)^+$.
  For a ring $R$ we have
  \[
    \mathrm{Gr}_n(R)=\coprod_{\mu\in X_*(T_n)^+} L^+GL_n(R)\cdot t^\mu\cdot L^+GL_n(R)/L^+GL_n(R),
  \]
  where $t^\mu$ is the image of $t$ through the coweight $\mu:R((t))^\times\rightarrow T_n(R((t)))$.
  For a coweight $\mu\in X_*(T_n)$, we define the \emph{Schubert cell}
  \[
    \mathrm{Gr}_n^\mu\coloneqq LGL_n^+\cdot t^\mu\cdot LGL_n^+/LGL_n^+;
  \]
  as a consequence, we have the following \emph{Cartan decomposition}:
  \[
    \mathrm{Gr}_n=\coprod_{\mu\in X_*(T_n)^+}\mathrm{Gr}_n^\mu.
  \]
  Moreover the closure relations of the Schubert cells are given by the partial order on dominant coweights:
  \[
    \mathrm{Gr}_n^{\leq\mu}\coloneqq\overline{\mathrm{Gr}_n^\mu}=\bigcup_{\lambda\leq\mu}\mathrm{Gr}_n^\lambda.
  \]
  The space $\mathrm{Gr}_n^{\leq\mu}$ is called \emph{Schubert variety} and the Schubert cell $\mathrm{Gr}_n^\mu$
  is open in $\mathrm{Gr}_n^{\leq\mu}$.
  For a coweight $\mu\in X_*(T)^+$, let $P_\mu$ be the parabolic subgroup of $GL_n$
  corresponding to $\mu$, i.e. the parabolic subgroup generated by $B_n$ and by the
  root subgroups $U_\alpha$ for those roots $\alpha$ satisfying $\langle\alpha,\mu\rangle\leq0$.
  For a ring $R$, let us fix the standard lattice $\Lambda_0\coloneqq R\llbracket t\rrbracket^n\in\mathrm{Gr}_n(R)$
  and a basis $e_1,\dots,e_n$ of $\Lambda_0$ over $R\llbracket t\rrbracket$.
  Denote by $\mathrm{Fil}^\bullet(V)$ the standard filtration of $V\coloneqq \Lambda_0[1/t]$
  given by $0\subset R((t))e_1\subset R((t))e_1\oplus R((t))e_2\subset\dots\subset V$.
  There is a natural projection
  \begin{align}
    \label{pi_mu}
    \pi_\mu:\mathrm{Gr}_n^\mu=L^+GL_n\cdot t^\mu\cdot L^+GL_n/L^+GL_n&\twoheadrightarrow GL_n/P_\mu\\
    (g\cdot t^\mu\text{ mod }L^+GL_n)&\mapsto (\bar{g}\text{ mod }P_\mu)\nonumber,
  \end{align}
  where $\bar g\coloneqq g\text{ mod }t$.
  Recall that $GL_n/P_\mu$ is a flag variety and
  for a lattice $\Lambda\in \mathrm{Gr}_n^\mu(R)$, we have that the flag $\pi_\mu(\Lambda)\in GL_n/P_\mu(R)$
  is given by the filtration $\Fil^\bullet(R^n)$ where
  \begin{equation}
    \label{def pi_mu}
    \Fil^i(R^n)\coloneqq(t^{-\mu_i}\Lambda\cap\Lambda_0)/(t^{-\mu_i}\Lambda\cap t\Lambda_0).
  \end{equation}
  We moreover have a section
  \begin{align*}
    \sigma_\mu:GL_n/P_\mu\rightarrow L^+GL_n\cdot t^\mu\cdot L^+GL_n/L^+GL_n=\mathrm{Gr}_n^\mu,
  \end{align*}
  which sends a decreasing filtration $\Fil^\bullet(R^n)$ of $R^n$ to the lattice
  \begin{equation}
    \label{def sigma_mu}
    \sigma_\mu(\Fil^\bullet(R^n))\coloneqq \sum_{i=1}^n\Fil^i(R^n)\otimes_Rt^{\mu_i}R\llbracket t\rrbracket.
  \end{equation}

  We define the maps
  \begin{equation}
    \label{def pi and sigma}
    \rr{Gr}_n=\coprod_{\mu\in X_*(T_n)^+}\rr{Gr}_n^\mu
    \mathrel{\mathop{\rightleftarrows}^{\pi}_{\sigma}}
    \coprod_{\mu\in X_*(T_n)^+}GL_n/P_\mu
  \end{equation}
  taking $\Lambda\in\rr{Gr}_n$ to $\pi_\mu(\Lambda)$ if $\Lambda \in \rr{Gr}_n^\mu$
  and taking $gP_\mu\in GL_n/P_\mu$ to $\sigma_\mu(gP_\mu)$.

  Finally, given $\mu\in X_*(T_n)$, we define the locally closed ind-scheme
  \[
    Y_\mu\coloneqq LU_n\cdot t^\mu\cdot L^+GL_n/L^+GL_n:\underline{Rings}\rightarrow\underline{Sets}
  \]
  of $\mathrm{Gr}_n$.

  \begin{lem}
    \label{Y_mu gives filtrations}
    Let $\mu=(\mu_1,\dots,\mu_n)\in X_*(T_n)$ and let $\Lambda\in Y_\mu(A)$ for some affinoid algebra $A$.
    Then $\Fil^i(\Lambda)\coloneqq\Lambda\cap\mathrm{Fil}^i(V)$ is a locally free module over $A\llbracket t\rrbracket$
    for all $i=1,\dots, n$.
  \end{lem}

  \begin{proof}
    Let $\mathbb{M}$ be the universal lattice on $Y_\mu$, hence locally on the
    affinoid $\Sp(A)$ the lattice $\mathbb{M}$
    has a basis of the form
    \[
      \{t^{\mu_1}e_1,t^{\mu_2}e_2+x_{(1,2)}e_1,\dots,t^{\mu_n}e_n+x_{(n-1,n)}e_{n-1}+\dots+x_{(1,n)}e_1\}
    \]
    with $x_{(i,j)}\in A((t))$ for all $i<j$.
    Then we have that $\Fil^i(\mathbb{M})$ is the free $A\llbracket t\rrbracket$-module generated
    by the first $i$ vectors of the above basis.
  \end{proof}

  Hence for a lattice $\Lambda\in Y_\mu(A)$, let us define a filtration of $\Lambda$ as
  $\mathrm{Fil}^i(\Lambda)\coloneqq\Lambda\cap\mathrm{Fil}^i(V)$ (in fact we know this is a filtration
  by Lemma \ref{Y_mu gives filtrations}).
  Then we have that
  \[
    Y_\mu(A)=\{\Lambda\in\mathrm{Gr}_n(A): \mathrm{gr}^i(\Lambda)= t^{\mu_i}\mathrm{gr}^i(\Lambda_0)\text{ for }i=1,\dots,n\},
  \]
  where the equality $\mathrm{gr}^i(\Lambda)= t^{\mu_i}\mathrm{gr}^i(\Lambda_0)$ is
  intended as lattices of $\mathrm{gr}^iA((t))^n$.

  \begin{lem}
    \label{Gr=union of Y_mu}
    We have $\rr{Gr}=\coprod_{\mu\in X_*(T_n)}Y_\mu$ as sets.
  \end{lem}

  \begin{proof}
    First of all notice that $\Lambda\in Y_\mu\cap Y_\lambda$ if and only if there exist
    $V,V'\in LU$ such that $Vt^\mu L^+G=\Lambda=V't^\lambda L^+G$; hence $\Lambda\in Y_\mu\cap Y_\lambda$
    if and only if there exists $V\in LU$ such that $t^{-\lambda}Vt^\mu\in L^+G$:
    it is obvious that this is possible only if $\lambda=\mu$.
    Now let $\Lambda\in\rr{Gr}_n(K)$ for some field $K$: we show that $\Lambda\in Y_\mu$ for some $\mu\in X_*(T_n)$.
    Let $\Lambda^i\coloneqq\Lambda\cap\Fil^i(V)$ for all $1\leq i\leq n$
    (as $K\llbracket t\rrbracket$ is a PID, $\Lambda^i$ is a finite projective $K\llbracket t\rrbracket$-module)
    and let us consider the commutative diagram
    \[
      \begin{tikzcd}
        0\ar{r}&\Lambda^i\ar{r}\ar[hookrightarrow]{d}
        &\Lambda^{i+1}\ar{r}\ar[hookrightarrow]{d}&\Lambda^{i+1}/\Lambda^i\ar{r}\ar[hookrightarrow]{d}&0\\
        0\ar{r}&\Fil^i(V)\ar{r}&\Fil^{i+1}(V)\ar{r}&\rr{gr}^{i+1}(V)\ar{r}&0.
      \end{tikzcd}
    \]
    Then $\Lambda^{i+1}/\Lambda^i$ is a $K\llbracket t\rrbracket$-lattice of $\rr{gr}^{i+1}(V)=K((t))e_{i+1}$,
    which means that there exists $\mu_{i+1}\in\Z$ such that
    $\Lambda^{i+1}/\Lambda^i= t^{\mu_{i+1}}e_{i+1}\cdot K\llbracket t\rrbracket=t^{\mu_{i+1}}\rr{gr}^{i+1}(\Lambda_0)$.
    It is then obvious that $\Lambda\in Y_\mu$, where $\mu=(\mu_1,\dots,\mu_n)$ constructed as described.
  \end{proof}

  \begin{lem}
    \label{preimage in Y_mu of filtrations}
    Let $\lambda\in X_*(T_n)^+$, $\mu\in X_*(T_n)$ and $w_0\in W_n$ the longest element
    of the Weyl group.
    If $\mu=ww_0\lambda$ for some $w\in W_n$,
    we have that the diagram
    \begin{equation}
      \begin{tikzcd}
        B_nwP_\lambda/P_\lambda\arrow[r]\arrow[d,"\sigma_\lambda"] & GL_n/P_\lambda\arrow[d, "\sigma_\lambda"]\\
        Y_\mu\arrow[r] & \rr{Gr}_n
      \end{tikzcd}
    \end{equation}
    is a fiber product.
  \end{lem}

  \begin{proof}
    We will show that
    \[
      \sigma_{\lambda}(B_nwP_\lambda/P_\lambda) = Y_\mu\cap\sigma_{\lambda}(GL_n/P_\lambda).
    \]
    Notice first of all that $\sigma_{\lambda}(GL_n/P_\lambda) = Gt^\lambda L^+G/L^+G
    =Gt^{w_0\lambda}L^+G/L^+G$, where $w_0\in W_n$ is the longest element;
    moreover $\sigma_{\lambda}(B_nwP_\lambda/P_\lambda)=B_nwt^{\lambda}L^+G/L^+G$ by definition of $\sigma_\lambda$.
    As $B_nwt^{\lambda}L^+G/L^+G=U_nwt^\lambda L^+G/L^+G$,
    we prove that
    $Y_\mu\cap\sigma_{\lambda}(GL_n/P_\lambda)= U_nwt^\lambda L^+G/L^+G$.
    As both schemes are reduced and locally closed subschemes of $LG/L^+G$, it's enough
    to work on $L$-points for some field $L$.
    Observe that $B_nt^{w_0\lambda} L^+G/L^+G = t^{w_0\lambda} L^+G/L^+G$, in fact:
    for any $(b_{ij})_{1\leq i,j\leq n}\in B_n$, we have
    \[
      (b_{ij})_{1\leq i,j\leq n}\cdot t^{w_0\lambda} = t^{w_0\lambda} \cdot (b_{ij}t^{\lambda_{n-j+1}-\lambda_{n-i+1}})_{1\leq i,j\leq n};
    \]
    since $(b_{ij})_{1\leq i,j\leq n}$ is upper-triangular, we have that $b_{ij}=0$ for $i>j$.
    As a consequence of this fact and of the fact that $\lambda$ is dominant, we have that the matrix
    $(b_{ij}t^{\lambda_{n-j+1}-\lambda_{n-i+1}})_{1\leq i,j\leq n}\in L^+G$.
    Moreover by the Bruhat decomposition, we have
    \[
      GL_n = \bigsqcup_{w\in W}U_nwB_n,
    \]
    hence
    \begin{align*}
      Gt^\lambda L^+G/L^+G&=Gt^{w_0\lambda}L^+G/L^+G = \bigsqcup_{w\in W_n} U_nwB_nt^{w_0\lambda} L^+G/L^+G\\
      &= \bigsqcup_{w\in W_n} U_nwt^{w_0\lambda} L^+G/L^+G
      =\bigsqcup_{w\in W_n}U_nt^{ww_0\lambda} L^+G/L^+G.
    \end{align*}
    Obviously we have $U_nt^{ww_0\lambda} L^+G/L^+G\subset LU_nt^{ww_0\lambda}L^+G/L^+G$
    and
    \[
      U_nwt^{w_0\lambda} L^+G/L^+G\cap LU_nt^\mu L^+G/L^+G=\emptyset
    \]
    for all $\mu\neq ww_0\lambda$.
    This shows that
    \[
      Y_\mu\cap\sigma_{\lambda}(GL_n/P_\lambda) = U_nt^{ww_0\lambda} L^+G/L^+G = \sigma_{\lambda}(B_nwP_\lambda/P_\lambda),
    \]
    where $ww_0\lambda = \mu$.
  \end{proof}

  \begin{lem}
    \label{decomposition of Y_mu}
    Let us fix the standard lattice $\Lambda_0$ and a standard basis $\{b_1,\dots,b_n\}$ of $\Lambda_0$.
    Let us decompose $\Lambda_0$ as $\Lambda_0=\bigoplus_{i=1}^l\Lambda_{0,i}$,
    where $\Lambda_{0,i}$ is a lattice of rank $n_i$; let us moreover assume that
    $\Lambda_{0,i}$ has a basis consisting of a
    subset of the standard basis for $\Lambda_0$ and let us fix the standard basis of $\Lambda_{0,i}$
    to be such basis for all $i$.
    Notice we can order the basis elements in such a way so that they respect the order
    in the standard basis of $\Lambda_0$.
    Let $\mu\in X_*(T_n)$ and $\lambda_i\in X_*(T_{n_i})^+$ for all $1\leq i\leq l$.
    We have
    \[
      Y_\mu\times_{\rr{Gr}_n}\prod_{i=1}^lGL_{n_i}/P_{\lambda_i} =
      \prod_{i=1}^lY_{\mu_i}\times_{\rr{Gr}_{n_i}}GL_{n_i}/P_{\lambda_i}
    \]
    for some $\mu_i\in X_*(T_{n_i})$, where the map $\prod_{i=1}^lGL_{n_i}/P_{\lambda_i}\rightarrow\rr{Gr}_n=\prod_{i=1}^l\rr{Gr}_{n_i}$
    is given by
    \[
      (g_iP_{\lambda_i})_{i=1,\dots,l}\mapsto(\sigma_{\lambda_i}(g_iP_{\lambda_i}))_{i=1,\dots,l}.
    \]
  \end{lem}

  \begin{proof}
    We show that
    \[
      Y_\mu\cap\prod_{i=1}^l\sigma_{\lambda_i}(GL_{n_i}/P_{\lambda_i})
      =\prod_{i=1}^l(Y_{\mu_i}\cap\sigma_{\lambda_i}(GL_{n_i}/P_{\lambda_i}))
    \]
    for some $\mu_i\in X_*(T_{n_i})$.
    Observe that, as all schemes are reduced locally closed subspaces of $LG/L^+G$,
    we can work on $L$-points for $L$ a field.
    For all $i$, we denote by $\{e_{1,i},\dots,e_{n_i,i}\}$ the standard basis of $\Lambda_{0,i}$
    consisting of elements of the standard basis $\{b_1,\dots,b_n\}$ of $\Lambda_0$
    such that if $e_{j,i}=b_{j'}$, $e_{k,i}=b_{k'}$ and $j\leq k$, then $j'\leq k'$.
    So let $\Lambda\in (Y_\mu\cap\prod_{i=1}^l\sigma_{\lambda_i}(GL_{n_i}/P_{\lambda_i}))(L)$.
    This means that $\Lambda = (\sigma_{\lambda_i}(g_iP_{\lambda_i}))_{i=1,\dots,l}$
    for some $g_iP_{\lambda_i}\in GL_{n_i}/P_{\lambda_i}$.
    We will show that for all $1\leq j\leq n$ we have
    \[
      \prod_{i=1}^l\sigma_{\lambda_i}(g_iP_{\lambda_i})\cap \prod_{k=1}^jL((t))b_k
      = \prod_{i=1}^l(\sigma_{\lambda_i}(g_iP_{\lambda_i})\cap(L((t))e_{1,i}\oplus\dots\oplus L((t))e_{k_{(j,i)},i})),
    \]
    where $\{b_1,\dots,b_j\}=\bigcup_{i=1}^l\{e_{1,i},\dots, e_{k_{(j,i)},i}\}$.
    Let
    \[
      x\in\prod_{i=1}^l\sigma_{\lambda_i}(g_iP_{\lambda_i})\cap \prod_{k=1}^jL((t))b_k,
    \]
    in particular $x=\sum_{i=1}^lx_i$ with $x_i\in \sigma_{\lambda_i}(g_iP_{\lambda_i})$.
    As $\sigma_{\lambda_i}(g_iP_{\lambda_i})\subset \prod_{k=1}^{n_i}L((t))e_{k,i}$,
    we have that $x_i=\sum_{h=1}^{n_i}y_he_{h,i}$ with $y_h\in L((t))$.
    As a consequence, we have
    \[
      x =\sum_{i=1}^l\sum_{h=1}^{n_i}y_he_{h,i}\in \prod_{k=1}^jL((t))b_k.
    \]
    This proves the first inclusion. The inverse inclusion is obvious.
    As a consequence, we have
    \begin{align*}
      \Fil^j(\Lambda)&=\prod_{i=1}^l\sigma_{\lambda_i}(g_iP_{\lambda_i})\cap \prod_{k=1}^jL((t))b_k\\
      &=\prod_{i=1}^l(\sigma_{\lambda_i}(g_iP_{\lambda_i})\cap(L((t))e_{1,i}\oplus\dots\oplus L((t))e_{k_{(j,i)},i}))\\
      &=\prod_{i=1}^l\Fil^{k_{(j,i)}}(\sigma_{\lambda_i}(g_iP_{\lambda_i})).
    \end{align*}
    In particular observe that for all $1\leq j\leq n$
    \begin{align*}
      \rr{gr}^j(\Lambda)&=\Fil^j(\Lambda)/\Fil^{j-1}(\Lambda)\\
      &=\left(\prod_{i=1}^l\Fil^{k_{(j,i)}}(\sigma_{\lambda_i}(g_iP_{\lambda_i})) \right)/
      \left(\prod_{i=1}^l\Fil^{k_{(j-1,i)}}(\sigma_{\lambda_i}(g_iP_{\lambda_i}))\right)\\
      &= \Fil^{k_{(j,i')}}(\sigma_{\lambda_{i'}}(g_{i'}P_{\lambda_{i'}}))/\Fil^{k_{(j-1,i')}}(\sigma_{\lambda_{i'}}(g_{i'}P_{\lambda_{i'}}))\\
      &= \rr{gr}^{k_{(j,i')}}(\sigma_{\lambda_{i'}}(g_{i'}P_{\lambda_{i'}})),
    \end{align*}
    where $i'$ is such that $b_j=e_{k_{j,i'}}$.
    This proves that
    \[
      \Lambda\in Y_\mu\cap\prod_{i=1}^l\sigma_{\lambda_i}(GL_{n_i}/P_{\lambda_i})
      \Leftrightarrow \Lambda\in
      \prod_{i=1}^l(Y_{\mu_i}\cap\sigma_{\lambda_i}(GL_{n_i}/P_{\lambda_i})),
    \]
    where $(\mu_i\colon 1\leq i\leq l)$ is a permutation of $\mu$.
  \end{proof}

  \begin{dfn}
    Let us fix $n\geq 2$ and
    let $\mathfrak g_n$ be the Lie algebra of $GL_n$.
    \begin{itemize}
      \item[(i)]
      We define the closed subspace
      \begin{equation}
        \label{def of Fil^n}
        \Fil^{\st}\coloneqq\{(N,gP_\mu)\in\mathfrak g_n\times\coprod_{\mu\in X_*(T_n)^+}GL_n/P_\mu
        \colon NgP_\mu\subseteq gP_\mu
        \}
      \end{equation}
      of $\mathfrak g_n\times\coprod_{\mu\in X_*(T_n)^+}GL_n/P_\mu$.
      \item[(ii)]
      Let us fix $\lambda\in X_*(T_n)^+$, we define the fiber product
      \[
        \begin{tikzcd}
          \Fil^{\st}_\lambda\arrow[r]\arrow[d]&\mathfrak g_n\times GL_n/P_\lambda\arrow[hookrightarrow]{d}\\
          \Fil^{\st}\arrow[hookrightarrow]{r}&\mathfrak g_n\times\coprod_{\mu\in X_*(T_n)^+}GL_n/P_\mu,
        \end{tikzcd}
      \]
      which is a closed subspace of $\mathfrak g_n\times GL_n/P_\lambda$.
    \end{itemize}
  \end{dfn}

  \begin{prop}
    \label{dim of Fil^N_lambda}
    The codimension of $\Fil^{\st}_\lambda$ inside $\mathfrak g_n\times GL_n/P_\lambda$
    is at most $\dim GL_n/P_\lambda$.
  \end{prop}

  \begin{proof}
    Let $(N,gP_\lambda)\in\mathfrak{g}_n\times GL_n/P_\lambda$.
    We have that $(N,gP_\lambda)\in\Fil^{\st}_\lambda$ if and only if $NgP_\lambda\subseteq gP_\lambda$
    by construction of $\Fil^{\st}_\lambda$;
    this means that
    \[
      (N,gP_\lambda)\in\Fil^{\st}_\lambda\Leftrightarrow g^{-1}Ng\in P_\lambda.
    \]
    Then the space $\Fil^{\st}_\lambda$ is cut out by the condition $g^{-1}Ng\in P_\lambda$,
    which corresponds to the equations annihilating certain entries of the matrix $g^{-1}Ng$;
    the number of these equations is exactly $\dim \overline{U}_{P_\lambda}$,
    where $\overline{U}_{P_\lambda}$ is the
    unipotent radical of the opposite parabolic of $P_\lambda$.
    Hence we have
    \[
      \operatorname{codim} \Fil^{\st}_\lambda\leq \dim \overline{U}_{P_\lambda}= \dim GL_n/P_\lambda.
    \]
  \end{proof}

  \begin{dfn}
    For $w\in W_n$ and $\lambda\in X_*(T_n)^+$, let us define $V_w^\lambda$ to be the fiber product
    \begin{equation}
      \label{def of V_w}
      \begin{tikzcd}
        V_w^\lambda\arrow[r]\arrow[d] & \mathfrak{b}_n\times B_nwP_\lambda/P_\lambda\arrow[d]\\
        \Fil^{\st}_\lambda\arrow[hookrightarrow]{r}& \mathfrak{g}_n\times GL_n/P_\lambda,
      \end{tikzcd}
    \end{equation}
    where $\mathfrak{b}_n$ is the Lie algebra of the Borel $B_n$ of upper triangular matrices.
  \end{dfn}

  \begin{rem}
    \label{right number of equations}
    The space $V_w^\lambda$ is smooth of dimension $\dim\mathfrak{b}_n$ by \cite{breuil2019local}[Proposition 2.2.1].
    As the space $\mathfrak{b}_n\times B_nwP_\lambda/P_\lambda$ is smooth of dimension $\dim\mathfrak{b}_n+\dim B_nwP_\lambda/P_\lambda$,
    we have that $V_w^\lambda$ inside $\mathfrak{b}_n\times B_nwP_\lambda/P_\lambda$ is cut out
    by $\dim B_nwP_\lambda/P_\lambda$ equations by \cite[\href{https://stacks.math.columbia.edu/tag/0H1G}{Tag 0H1G}]{stacks-project}.
  \end{rem}

  \subsection{Almost de Rham representations}
  \label{pdR}

  In what follows, $\BdR$ and $\BdR^+$ are topological rings for the so called "natural topology"
  (\cite[§3.2]{fontaine2004arithmetique}).
  Let us choose a compatible system of primitive $p^n$-th roots of unity $(\epsilon^{(n)})_{n\in\N}$;
  we have $[\epsilon]\in\BdR^+$, where $[\cdot]$ denotes the Teichmüller map.
  Moreover $\log[\epsilon]\coloneqq\sum_{n=1}^\infty\frac{([\epsilon]-1)^n}{n}$ converges in $\BdR^+$
  and is a uniformizer; we will denote it by $t$ and we have that $\BdR\coloneqq\BdR^+[1/t]$.
  Let $C\coloneqq\BdR^+/t\BdR^+$ be the residue field of $\BdR^+$.
  Recall that both the rings $\BdR$ and $\BdR^+$ are endowed with a continuous action of $\absGal$.
  We now recall the definition of \emph{almost de Rham representations} as defined in
  \cite[§3.7]{fontaine2004arithmetique}.

  \begin{dfn}
    \begin{itemize}
      \item[(i)]
        A \emph{$\BdR$-representation of $\absGal$} is a finite dimensional free $\BdR$-module
        with a continuous semi-linear action of $\absGal$.
        We denote by $\Rep_{\BdR}(\absGal)$ the category of $\BdR$-represetations of $\absGal$.
      \item[(ii)]
        Let $W\in\Rep_{\BdR}(\absGal)$, we say that $W$ is \emph{almost de Rham} if it contains
        a $\BdR^+$-lattice $W^+$ that is stable under $\absGal$ and
        such that the Sen weights of the $C$-represetation $W^+/tW^+$
        are all integers.
        We denote by $\Rep_{\mathrm{pdR}}(\absGal)$ the full subcategory of almost de Rham representations.
    \end{itemize}
  \end{dfn}

  \begin{rem}
    \begin{itemize}
      \item[(i)]
        By compactness of $\absGal$, we have that any $W\in\Rep_{\BdR}(\absGal)$ contains
        a $\absGal$-stable $\BdR^+$-lattice.
      \item[(ii)]
        If $W$ is a Hodge-Tate representation of $\absGal$ over $C$, then the Sen weights
        coincide with the Hodge-Tate weights.
    \end{itemize}
  \end{rem}

  We have a functor
  \begin{equation}
    \label{def of DpdR}
    \begin{split}
    D_{\mathrm{pdR}}:\Rep_{\BdR}(\absGal)&\rightarrow\Rep_{\Q_p}(\mathbb{G}_a)\\
    W&\mapsto(\BdR[\log t]\otimes_{\BdR}W)^{\absGal}
    \end{split}
  \end{equation}
  defined in \cite[§4.3]{fontaine2004arithmetique} with the property that
  $\dim_{\Q_p}D_{\mathrm{pdR}}(W)\leq\dim_{\BdR}W$.

  The following Theorem is proved in \cite[Théorème 4.1.]{fontaine2004arithmetique}
  and \cite[Proposition 3.1.1]{breuil2019local}.

  \begin{thm}
    A represetation $W\in\Rep_{\BdR}(\absGal)$ is almost de Rham if and only if
    $\dim_{\Q_p}D_{\mathrm{pdR}}(W)=\dim_{\BdR}W$.
    Moreover the functor $D_{\mathrm{pdR}}$ induces an equivalence of categories between
    $\Rep_{\mathrm{pdR}}(\absGal)$ and $\Rep_{\Q_p}(\mathbb{G}_a)$.
  \end{thm}

  \begin{rem}
    \label{rem:representation of Ga and nilpotent}
    If $E$ is a field of characteristic 0, we have that the category $\Rep_E(\mathbb{G}_a)$
    is equivalent to the category whose objects are pairs $(V,N_V)$ of finite-dimensional $E$-vector
    spaces $V$ endowed with an $E$-linear nilpotent endomorphism $N_V$ and whose morphisms
    are linear maps of $E$-vector spaces commuting with the nilpotent endomorphism $N_V$.
  \end{rem}

  The following is \cite[Proposition 3.2.1.]{breuil2019local}.

  \begin{prop}
    \label{prop:lattices and filtrations}
    Let $W\in\Rep_{\mathrm{pdR}}(\absGal)$. We have a bijection
    \begin{align*}
      \{\absGal\text{-stable }\BdR^+\text{-lattices of }W\}&\leftrightarrow
      \{\text{filtrations of }D_{\mathrm{pdR}}(W)\}\\
      W^+&\mapsto \Fil^\bullet_{W^+}(D_{\pdR}(W)),
    \end{align*}
    where $\Fil^i_{W^+}(D_\pdR(W))\coloneqq (t^i\BdR^+[\log t]\otimes_{\BdR^+}W^+)^{\absGal}$
    for $i\in\Z$.
    Here a filtration of a representation $V\in\Rep_{\Q_p}(\mathbb{G}_a)$ is a
    decreasing, exhaustive and separated filtration by subobjects in $\Rep_{\Q_p}(\mathbb{G}_a)$.
  \end{prop}

  Thanks to Remark \ref{rem:representation of Ga and nilpotent} and Proposition \ref{prop:lattices and filtrations},
  given $W\in\Rep_{\mathrm{pdR}}(\absGal)$ we have a bijection
  \begin{equation}
    \begin{split}
      \{\absGal\text{-stable }\BdR^+\text{-lattices of }W\}&\leftrightarrow
      \{N\text{-stable filtrations of }D_{\mathrm{pdR}}(W)\}\\
      W^+&\mapsto \Fil^\bullet_{W^+}(D_{\pdR}(W)),
    \end{split}
  \end{equation}
  where the filtration $\Fil^\bullet_{W^+}(D_{\pdR}(W))$ is defined as in Proposition \ref{prop:lattices and filtrations}.

  In \cite{berger2008construction}, Berger defines a topological ring $B_e$ endowed with
  a continuous $\absGal$-action and introduces the notion of \emph{$B$-pairs},
  proving that there exists an equivalence of categories of such objects with $\phigam$-modules.

  \begin{dfn}
    A \emph{$B$-pair} is a pair $(W_e,W_\dR^+)$ where $W_e$ is a free $B_e$-module of
    finite rank endowed with a semi-linear and continuous action of $\absGal$ and
    $W_\dR^+$ is a $\BdR^+$-lattice of $W_\dR\coloneqq \BdR\otimes_{B_e}W_e$ stable under
    the action of $\absGal$.
  \end{dfn}

  The following result is \cite[Théorème 2.2.7]{berger2008construction}.

  \begin{thm}
    \label{phi,gamma modules and B-pairs}
    There is an equivalence of categories
    \begin{align*}
      W:\phigam-\Mod(\robba)&\rightarrow\{B\text{-pairs}\}\\
      D&\mapsto(W_e(D),W_\dR^+(D)).
    \end{align*}
    Moreover the definition of the map $W_e(D)$ actually only depends on $D[1/t]$.
  \end{thm}

  The functoriality of these constructions make it possible to consider representations
  with different coefficients.
  From now on until the end of the section, we let $A$ be a finite dimensional algebra over $\Q_p$;
  we introduce the following notations for the
  rings $\BdR$ and $\BdR^+$ with coefficients in $A$:
  \[
  \begin{array}{cc}
    \BdRA\coloneqq\BdR\otimes_{\Qp}A, & \BdRA^+\coloneqq\BdR^+\otimes_{\Qp}A.
  \end{array}
  \]

  The following notions are defined in \cite[Section 3.1 and 3.2]{breuil2019local}.

  \begin{dfn}
    \begin{enumerate}
      \item
      A \emph{$\BdRA$-representation of $\absGal$} is a $\BdR$-representation
      $W$ of $\absGal$ together with a morphism of $\Qp$-algebras
      $A\rightarrow \mathrm{End}_{\Rep_{\BdR}(\absGal)}(W)$ which makes $W$ a finite free
      $\BdRA$-module.
      We denote by $\Rep_{\BdRA}(\absGal)$ the category of $\BdRA$-representations.
      \item
      A \emph{$\BdRA^+$-representation of $\absGal$} is a $\BdR^+$-representation $W^+$ of $\absGal$
      together with a morphism of $\Qp$-algebras $A\rightarrow\rr{End}_{\Rep_{\BdR^+}(\absGal)}(W^+)$
      which makes $W^+$ a finite free $\BdRA^+$-module.
      \item
      We say that a $\BdRA$-representation is \emph{almost de Rham} if the underlying
      $\BdR$-representation is.
      We denote by $\Rep_{\pdR,A}(\absGal)$ the category of almost de Rham $\BdRA$-representations.
      \item
      An \emph{$(A,B)$-pair of $\absGal$} is a pair $(W_e,\WdR^+)$ such that
      \begin{itemize}
        \item[(i)]
        $W_e$ is a finite $B_e\otimes_{\Qp}A$-module with a continuous semi-linear $\absGal$-action
        which is free as $B_e$-module;
        \item[(ii)]
        $\WdR^+\subset\WdR\coloneqq\BdR\otimes_{B_e}W_e$ is a $\absGal$-stable $\BdRA^+$-
        lattice of $\WdR$.
      \end{itemize}
    \end{enumerate}
  \end{dfn}

  We denote by $\Rep_A(\mathbb{G}_a)$ the category of pairs $(V_A,N_A)$ where $N_A$ is
  a nilpotent endomorphism of a finite free $A$-module $V_A$.

  The following is \cite[Lemma 3.1.4]{breuil2019local}.

  \begin{thm}
    The functor $D_\pdR$ induces an equivalence of categories between
    $\Rep_{\pdR,A}(\absGal)$ and $\Rep_A(\mathbb{G}_a)$.
  \end{thm}

  By \cite[Lemma 3.2.2]{breuil2019local}, we have that given $W\in\Rep_{\pdR,A}(\absGal)$,
  there is a bijection
  \begin{equation}
    \label{bijection breuil}
    \begin{split}
      \left\{ \absGal\text{-stable }\BdRA^+\text{-lattices of }W\right\}&
      \leftrightarrow\left\{ \begin{array}{c}
        N_A\text{-stable exhaustive}\\
        \text{filtrations of }D_{\pdR}(W)
      \end{array}\right\}\\
      W^+&\mapsto \Fil^\bullet_{W^+}(D_{\pdR}(W)),
    \end{split}
  \end{equation}
  where the filtration $\Fil^\bullet_{W^+}(D_{\pdR}(W))$ is defined as in Proposition \ref{prop:lattices and filtrations}.

  We moreover have the following result.

  \begin{thm}
    \label{phi,gamma modules and B-pairs with coefficients}
    Let $A$ be a finite dimensional $\Qp$-algebra.
    There is an equivalence of categories
    \begin{align*}
      W:\phigam-\Mod(\robba_A)&\rightarrow(A,B)\text{-pairs}\\
      D&\mapsto(W_e(D),\WdR^+(D)),
    \end{align*}
    where the functor $W_e(D)$ only depends on $D[1/t]$ and
    $\WdR^+(D)\coloneqq\BdR^+\otimes_{\robba^r}D^r$ is independent of $r$, where
    $r\in p^\Q\cap[0,1)$ such that there exists
    a $\phigam$-module $D^r$ over $\robba_A^r$ such that $D^r\otimes_{\robba_A^r}\robba_A=D$.
  \end{thm}

  Nakamura proved the above Theorem when $A$ is a finite extension of $\Qp$ in \cite{nakamura2009classification},
  but since the same proof holds also for finite dimensional $\Qp$-algebras,
  we report it here almost verbatim.

  \begin{proof}
    As already stated in Theorem \ref{phi,gamma modules and B-pairs},
    for the $\Qp$-coefficient case it was proved by Berger in \cite[Theorem 2.2.7]{berger2008construction}.
    Let $D$ be a $\phigam$-module over $\robba_A$.
    First of all observe that, as $A$ is a finite dimensional $\Qp$-algebra,
    we have that $\robba_A=\robba\otimes_{\Qp}A$ and $\robba_A^r=\robba^r\otimes_{\Qp}A$;
    hence we have in particular that $D$ is also a $\phigam$-module over $\robba$ and
    $D=D\otimes_{\Q_p}A$.
    It is then obvious that $W_e(D)$ is a finite $B_e\otimes_{\Qp}A$-module
    with a continuous semi-linear $\absGal$-action which is also free as $B_e$-module.
    Now let $W(D)$ be the $B$-pair obtained by Theorem \ref{phi,gamma modules and B-pairs}
    and let $a\in A$; multiplication by $a$ gives an endomorphism of $D$ as $\phigam$-module
    over $\robba$.
    From the functoriality, this gives an endomorphism of $W(D)$.
    Thus we have shown that there is an $A$-action on $W(D)$, in particular there is
    a morphism of $\Qp$-algebras
    \[
      A\rightarrow\rr{End}_{\rr{Rep}_{\BdR^+}(\absGal)}(\WdR^+(D)).
    \]
    Hence $\WdR^+(D)$ is a $\BdRA^+$-representation of $\absGal$
    and obviously $\WdR^+(D)$ is a $\absGal$-stable $\BdRA^+$-lattice of $W_e(D)\otimes_{B_e}\BdR$.
    This proves that $(W_e(D),\WdR^+(D))$ is an $(A,B)$-pair.
    On the other hand, let $W=(W_e,\WdR^+)$ be an $(A,B)$-pair.
    In particular notice that $\WdR^+[1/t]=\BdR\otimes_{Be}W_e$.
    Then we have a morphism of $\Qp$-algebras
    \[
      A\rightarrow\rr{End}_{\rr{Rep}_{\BdR}(\absGal)}(\WdR^+(D)).
    \]
    Since the action of $A$ is $\BdR$-linear, we have that the above morphism translates
    into actions of $A$ on the $(A,B)$-pair $W$.
    Let us denote by $D(W)$ the $\phigam$-module over $\robba$ obtained through the
    equivalence proved by Berger.
    By functoriality, this implies that there is a $\Qp$-linear action of $A$ on the $\phigam$-module $D(W)$
    over $\robba$.
    Moreover, since $D(W)$ is a free module over $\robba$, we have that
    $D(W)=D(W) \otimes_{\Qp}A$ is a free module over $\robba_A=\robba\otimes_{\Qp}A$.
    Then the $\phigam$-module $D(W)$ is a $\phigam$-module over $\robba\otimes_{\Qp}A=\robba_A$.
  \end{proof}

  Let us define the functor $\WdR$ as in \cite{breuil2019local}:
  \begin{align*}
    \WdR:\phigam-\Mod(\robba_A[1/t])&\rightarrow\{\BdRA\text{-representations of }\absGal\}\\
    M&\mapsto\WdR^+(D)\otimes_{\BdRA^+}\BdRA=\WdR^+(D)[1/t],
  \end{align*}
  where $D$ is a $\phigam$-module over $\robba_A$ such that $D[1/t]=M$.
  It is possible to show that $\WdR(M)$ doesn't depend on $D$.

  \begin{thm}
    \label{bijection of modules and lattices}
    Let $M$ be a $\phigam$-module over $\robba_A[1/t]$.
    We have a bijection
    \begin{align*}
      M^{\phigam}&\rightarrow
      \{\absGal\text{-stable }\BdRA^+\text{-lattices of }\WdR(M)\}\\
      D&\mapsto \WdR^+(D),
    \end{align*}
    where $M^{\phigam}$ is defined as in (\ref{submodules of M}).
  \end{thm}

  \begin{proof}
    Let $\mathcal{B}$ be the set (\ref{set B}) and let
    \[
      \mathcal{C}\coloneqq\left\{(W^+,f)\colon
      \begin{array}{c}
        W^+\text{ is a }\absGal\text{-stable projective }\BdRA^+\text{-module and}\\
        f:W^+[1/t]\xrightarrow{\sim}\WdR(M)\text{ isomorphism of }\BdRA\text{-modules}
      \end{array}
      \right\};
    \]
    here the $\absGal$-action on the module $W^+$ is induced by the action on $\WdR(M)$ through
    the isomorphism $f$.
    Observe that the equivalence of Theorem \ref{phi,gamma modules and B-pairs with coefficients}
    implies that the functor $\WdR^+$ induces a bijection between the set $\widetilde{\mathcal{B}}$ of isomorphism
    classes of $\mathcal{B}$ and the set of isomorphism classes $\widetilde{\mathcal{C}}$ of
    $\mathcal{C}$.
    In order to prove the wanted bijection, we will proceed in a very similar way to the proof
    of Theorem \ref{bijection of lattices and modules}: we already proved
    that $M^{\phigam}$ is in bijection with $\widetilde{\mathcal{B}}$, so what is left to prove
    is that we have the following bijection
    \begin{align*}
      \{\absGal\text{-stable }\BdRA^+\text{-lattices of }\WdR(M)\}&\rightarrow\widetilde{\mathcal{C}}\\
      W^+&\mapsto(W^+,\mathrm{id}_{W^+[1/t]}).
    \end{align*}
    For injectivity, let us take
    \[
      W^+_1,W^{+}_2\in
    \{\absGal\text{-stable }\BdRA^+\text{-lattices of }\WdR(M)\}
    \]
    such that $(W^+_1,\mathrm{id}_{W^+_1[1/t]})\cong(W^{+}_2,\mathrm{id}_{W^{+}_2[1/t]})$.
    This means that there is an isomorphism $\xi:W^+_1\xrightarrow{\sim}W^+_2$ of $\BdRA^+$-modules
    such that the diagram
    \[
      \begin{tikzcd}
        W^+_1\arrow[rr,"\xi"]\arrow[hookrightarrow]{d}
        &&W^+_2\arrow[hookrightarrow]{d}\\
        W^+_1[1/t]\arrow[r,phantom,"="]
        & \WdR(M)
        &W^+_2[1/t]\arrow[l,phantom,"="]
      \end{tikzcd}
    \]
    commutes.
    It is then easy to see that $W^+_1=W^+_2$.
    Now let $(W^+,f)\in\widetilde{\mathcal{C}}$ and consider $f(W^+)$, which is a $\absGal$-stable
    lattice of $\WdR(M)$.
    Moreover $(W^+,f)\xrightarrow{f}(f(W^+),\mathrm{id}_{f(W^+)[1/t]})$ is an isomorphism in
    the category $\mathcal{C}$, which proves surjectivity.
  \end{proof}

  \begin{cor}
    \label{bijection of lattices and filtrations}
    Let $M$ be a $\phigam$-module over $\robba_A[1/t]$ such that $\WdR(M)$ is almost deRham.
    We have a bijection
    \begin{equation}
      \label{bijection stable lattices and stable filtrations}
      \begin{split}
        \mathrm{Gr}_M^\Gamma &\rightarrow \{N\text{-stable filtrations of }\DpdR(\WdR(M))\}
      \end{split}
    \end{equation}
    where
    $\mathrm{Gr}_M^\Gamma$ is defined as in (\ref{Gamma-stable lattices in M})
    and $N$ is the nilpotent endomorphism associated to $\DpdR(\WdR(M))$.
  \end{cor}

  \begin{proof}
    The result follows from Proposition \ref{prop:lattices and filtrations},
    Proposition \ref{bijection of lattices and modules}
    and Theorem \ref{bijection of modules and lattices}.
  \end{proof}

  \subsection{$\BdR$-representations of non-integral weights}

    In this section, $A$ will be an Artin local $L$-algebra, where $L$ is a finite extension of $\Qp$,
    hence in particular it is a finite-dimensional $\Qp$-algebra.
    Moreover throughout this section, we fix a triangulable $\phigam$-module $M$ over $\robba_A[1/t]$,
    $r\in p^{\mathbb{Q}}\cap[0,1)$ such that there exists a $\phigam$-module $M^r$ over $\robba_A^r[1/t]$
    such that $M=M^r\otimes_{\robba_A^r[1/t]}\robba_A$.
    We fix moreover $m_0\in\mathbb{N}$ such that for every $m\geq m_0$, the point $1-\zeta_{p^m}$
    lies in the half open annulus $\mathbb{B}^r$ and an embedding $\tau:K_{m}\hookrightarrow L$
    for $m\geq m_0$.

    In the previous section we have proved that if $\WdR(M)$ is almost deRham,
    then we have a bijection between the $\Gamma_m$-stable
    lattices of $M^r\otimes_{\robba_A^r[1/t],\tau}A((t))$ and the $N$-stable filtrations
    of $D_{\pdR}(\WdR(M))$, as shown in Corollary \ref{bijection of lattices and filtrations}.
    In this section we want to show that a similar result is true even in the case $\WdR(M)$ is not
    almost deRham.

    First we need some preparation:
    let $W_0^+$ be a $\BdRA^+$-representation of $\absGal$.
    Let us denote by $\mathcal{A}$ the set of the Sen weights of the $A\otimes_{\Qp}C$-representation
    $W_0^+/tW_0^+$ of $\absGal$.
    Let us assume that $\mcal A\subset A$ and let us denote by $\bar{\mcal A}$ the set of
    equivalence classes $(\mcal A+\Z)/\Z\subset A/\Z$.
    Let us denote by $W$ the $\BdRA$-representation $W_0^+[1/t]$ of $\absGal$.
    Assume that for every $a\in\mcal A$ there exists a continuous character
    $\delta:\Qp\times\rightarrow A^\times$ such that $\rr{wt}(\delta)=a$.

    \begin{rem}
      Notice that the above assumption is always satisfied in case $W_0^+=\WdR^+(D)$
      for some triangulable $\phigam$-module $D$ over $\robba_A$; in fact in this case
      the Sen weights of the $A\otimes_{\Qp}C$-representation $W_0^+/tW_0^+$ are exactly
      the weights of the parameters of $D$.
    \end{rem}

    In the following, for a continuous character $\delta:\Qp^\times\rightarrow A^\times$, we
    denote by $\WdR^+(\delta)\coloneqq\WdR^+(\robba_A(\delta))$ and by
    $\WdR(\delta)\coloneqq\WdR^+(\delta)[1/t]$.

    \begin{thm}
      \label{non integral weights}
      With the same notation as above, let $W_{\bar a}$ be the unique maximal subspace of $W$
      such that $W_{\bar a}(\delta_{\bar a}^{-1})\coloneqq W_{\bar a}\otimes_{\BdR}\WdR(\delta_{\bar a}^{-1})$
      is almost de Rham for some continuous character  $\delta_{\bar a}:\Qp^\times\rightarrow A^\times$
      with $\rr{wt}(\delta_{\bar a})+\Z=\bar a$.
      We have that
      \[
        W=\bigoplus_{\bar a\in\bar{\mcal A}}W_{\bar a}
      \]
      as $\BdRA$-representations.
      Moreover if $W^+$ is a $\absGal$-stable $\BdRA^+$-lattice of $W$ such that
      the set of mod $\Z$ classes of the Sen weights of $W^+$ is equal to $\bar{\mcal A}$
      counting multiplicities, then we have
      \[
        W^+=\bigoplus_{\bar a \in\bar{\mcal A}}W_{\bar a}^+
      \]
      where $W_{\bar a}^+\coloneqq W^+\cap W_{\bar a}$ is  a $\absGal$-stable $\BdRA^+$-lattice of $W_{\bar a}$.
      Moreover we have that for every $\bar a\in\bar{\mcal A}$, the
      $A\otimes_{\Qp}C$-representation $W_{\bar a}^+/tW_{\bar a}^+$ has Sen weights in $\bar a$.
    \end{thm}

    The above Theorem is proved in \cite[§ 5.2,5.3]{wu2022trianguline}.

    Notice then that if $W_{\bar a}^+$ is a $\BdRA^+$-representation of $\absGal$ with Sen weights
    in $\bar a$ and $\delta_{\bar a}:\Qp^\times\rightarrow A^\times$ is a continuous
    character with $\rr{wt}(\delta_{\bar a})+\Z=\bar a$,
    then $W_{\bar a}^+(\delta_{\bar a}^{-1})\coloneqq W_{\bar a}^+\otimes_{\BdRA^+}\WdR^+(\delta_{\bar a}^{-1})$
    has integral Sen weights, hence it is almost deRham.
    Let $W_{\bar a}$ be a $\BdRA$-representation of $\absGal$
    and let $W_{\bar a}^+$ be a $\absGal$-stable $\BdRA^+$-lattice of $W_{\bar a}$
    as in Theorem \ref{non integral weights};
    we then define
    \[
      D_{\pdR, \bar a}(W_{\bar a})\coloneqq D_{\pdR}(W_{\bar a}^+(\delta_{\bar a}^{-1})[1/t]),
    \]
    which is a finite free $A$-module endowed with an $A$-linear nilpotent endomorphism $N_{\bar a}$.
    In \cite[§ 5.3]{wu2022trianguline} it is proved that the module $D_{\pdR, \bar a}(W_{\bar a})$
    does not depend on the choice of the character $\delta_{\bar a}$.

    For simplicity from now on, when we say that a $\absGal$-stable $\BdRA^+$-lattice $W^+$ \emph{has Sen weights in $\bar{\mcal A}$},
    we mean that the set of mod $\Z$ classes of the Sen weights of $W^+$ is equal to $\bar{\mcal A}$
    counting multiplicities.

    \begin{lem}
      \label{bijection of G-stable lattices and N-stable filtrations for non-integral weights}
      Let $W$ be a $\BdRA$-representation of $\absGal$ such that there exists a
      $\absGal$-stable $\BdRA^+$-lattice $W_0^+$ with Sen weights in $\mathcal{A}\subset A$.
      For every $\bar a\in\bar{\mcal A}$, let us fix a continuous character $\delta_{\bar a}:\Qp^\times\rightarrow A^\times$
      such that $\rr{wt}(\delta_{\bar a})+\Z= \bar a$ and let $W_{\bar a}$, $W_{0,\bar a}^+$ as
      in Theorem \ref{non integral weights}.
      There is a bijection
      \begin{align*}
        \left\{
          \begin{array}{c}
            \absGal\text{-stable }\\
            \BdRA^+\text{-lattices of }W\\
            \text{with Sen weights in }\bar{\mcal A}
          \end{array}
        \right\}&
        \leftrightarrow
        \left\{(\Fil^\bullet(D_{\pdR,\bar a}(W_{\bar a})))_{\bar a\in\bar{\mcal A}}\colon
          \begin{array}{c}
            \Fil^\bullet(D_{\pdR,\bar a}(W_{\bar a}))\text{ is an }\\N_{\bar a}\text{-stable filtration of}\\
            D_{\pdR,\bar a}(W_{\bar a})\text{ for all }\bar a\in\bar{\mcal A}
          \end{array}
        \right\}\\
        W^+=\bigoplus_{\bar a\in\bar{\mcal A}}W_{\bar a}^+
        &\mapsto (\Fil^\bullet_{W_{\bar a}^+(\delta_{\bar a}^{-1})}(D_{\pdR,\bar a}(W_{\bar a})))_{\bar a\in\bar{\mcal A}},
      \end{align*}
      where the filtrations $\Fil^\bullet_{W_{\bar a}^+(\delta_{\bar a}^{-1})}(D_{\pdR,\bar a}(W_{\bar a}))$
      are defined as in Proposition \ref{prop:lattices and filtrations}.
    \end{lem}

    \begin{proof}
      Let us prove injectivity: assume $W_1^+,W_2^+$ are $\absGal$-stable
      $\BdRA^+$-lattices of $W$ with Sen weights in $\mcal A$ such that for all $\bar a\in \bar{\mcal A}$
      \[
        \Fil^\bullet_{W_{1,\bar a}^+(\delta_{\bar a}^{-1})}(D_{\pdR,\bar a}(W_{\bar a}))=
        \Fil^\bullet_{W_{2,\bar a}^+(\delta_{\bar a}^{-1})}(D_{\pdR,\bar a}(W_{\bar a})).
      \]
      By the bijection (\ref{bijection breuil}),
      we have that $W_{1,\bar a}^+(\delta_{\bar a}^{-1})=W_{2,\bar a}^+(\delta_{\bar a}^{-1})$
      for all $\bar a\in\bar{\mathcal{A}}$.
      Then $W_{1,\bar a}^+=W_{1,\bar a}^+(\delta_{\bar a}^{-1})\otimes_{\BdRA^+}\WdR^+(\delta_{\bar a})
      =W_{2,\bar a}^+(\delta_{\bar a}^{-1})\otimes_{\BdRA^+}\WdR^+(\delta_{\bar a})=W_{2,\bar a}^+$.
      Moreover by Theorem \ref{non integral weights}, we have
      $W_1^+=\bigoplus_{\bar a\in\bar{\mcal A}}W_{1,\bar a}^+
      =\bigoplus_{\bar a\in\bar{\mcal A}}W_{2,\bar a}^+=W_2^+$ and this proves injectivity.
      Now let $(\Fil^\bullet(D_{\pdR,\bar a}(W_{\bar a})))_{\bar a\in\bar{\mcal A}}$
      be a tuple of filtrations of $D_{\pdR,\bar a}(W_{\bar a})$ such that each filtration is stable under
      $N_{\bar a}$ for every $\bar a\in\bar{\mcal A}$.
      By the bijection (\ref{bijection breuil}), we have that $\Fil^\bullet(D_{\pdR,\bar a}(W_{\bar a}))=
      \Fil^\bullet_{W_{\bar a}^+(\delta_{\bar a}^{-1})}(D_{\pdR,\bar a}(W_{\bar a}))$
      for some $\absGal$-stable
      $\BdRA^+$-lattice $W_{\bar a}^+(\delta_{\bar a}^{-1})$ of the $\BdRA$-representation
      $W_{0,\bar a}^+(\delta_{\bar a}^{-1})[1/t]$.
      It is then easy to see that the tuple $(\Fil^\bullet(D_{\pdR,\bar a}(W_{\bar a})))_{\bar a\in\bar{\mcal A}}$
      is the image of $\bigoplus_{\bar a \in \bar{\mcal A}}W_{\bar a}^+$ through the map described in
      the Lemma, where $W_{\bar a}^+\coloneqq W_{\bar a}^+(\delta_{\bar a}^{-1})\otimes_{\BdRA^+}\WdR^+(\delta_{\bar a})$.
      Finally, we have that $\bigoplus_{\bar a \in \bar{\mcal A}}W_{\bar a}^+$ is a $\absGal$-stable
      $\BdRA^+$-lattice of $W$, since $W_{\bar a}^+$ is $\absGal$-stable for all $\bar a\in \bar{\mcal A}$
      and $W_{\bar a}^+[1/t]=W_{\bar a}$ by construction.
      This concludes the proof.
    \end{proof}

    \begin{lem}
      \label{sen weights are the same mod Z}
      Let $D,D'$ be two triangulable $\phigam$-modules over $\robba_A$
      of parameters $(\delta_1,\dots,\delta_n)$ and $(\delta_1',\dots,\delta_n')$.
      Let $0\subset D_1\subset\dots\subset D_n$ and $0\subset D_1'\subset\dots\subset D_n'$
      be the filtrations of $D$ and $D'$ respectively giving rise to the parameters
      $\underline{\delta}$ and $\underline{\delta}'$.
      If for all $1\leq i\leq n$ we have that $D_i[1/t]=D_i'[1/t]$, then
      $\rr{wt}(\delta_i)\equiv \rr{wt}(\delta_i')\mod\Z$.
    \end{lem}

    \begin{proof}
      Observe that for all $1\leq i\leq n$ we have that
      \[
        \robba_A(\delta_i)[1/t]\cong D_i[1/t]/D_{i-1}[1/t]=D_i'[1/t]/D_{i-1}'[1/t]\cong \robba_A(\delta_i')[1/t].
      \]
      As a consequence we have that $\robba_A(\delta_i')=t^{m_i}\robba_A(\delta_i)=\robba_A(x^{m_i}\delta_i)$
      for some $m_i\in\Z$.
      This implies that $\delta_i'=x^{m_i}\delta_i$ by Theorem \ref{rank 1 phi gamma modules}
      and thus $\rr{wt}(\delta_i')=\rr{wt}(\delta_i)+m_i
      \equiv\rr{wt}(\delta_i)\mod \Z$.
    \end{proof}

    The following result is \cite[Lemma A.2.1]{wu2022trianguline}.

    \begin{lem}
      \label{lemma zhixiang1}
      Let $V$ be a continuous semi-linear $K_m((t))$-representation of $\Gamma$ and let $\gamma$
      be a topological generator of $\Gamma$.
      Let $V^{(\gamma-1)\text{-nil}}$ be the space of elements $x$ in $V$ such that $\gamma-1$ acts
      nilpotently on $x$.
      There is an isomorphism of $A$-modules
      \begin{align*}
        (K_m[\log(t)]\otimes_{K_m}V)^\Gamma&\rightarrow V^{(\gamma-1)\text{-nil}}\\
        \sum_{i=1}^la_i\log(t)^i&\mapsto a_0\\
        \sum_{i=1}^l\frac{(-1)^i}{i!}\nabla^i(a)\log(t)^i&\mapsfrom a,
      \end{align*}
      where $\nabla^{l+1}(a)=0$.
    \end{lem}

    \begin{lem}
      \label{replacement of Lemma zhixiang}
      Let $\Lambda$ be a free $\prod_{\tau:K_m\hookrightarrow L}A\llbracket t\rrbracket$-module endowed with
      a $\Gamma$-action with integral Sen weights.
      Then
      \[
        D_\pdR(\Lambda\otimes_{K_m\llbracket t\rrbracket}\BdR)=(\Lambda[1/t]\otimes_{K_m}K_m[\log(t)])^\Gamma.
      \]
    \end{lem}

    The proof of this Lemma is the same as the proof of \cite[Lemma A.2.5]{wu2022trianguline},
    even though the statement is slightly different.

    \begin{rem}
      \label{identify Lambda0/t with DpdR}
      Let $M$ be the $\phigam$-module over $\robba_A[1/t]$ fixed at the beginning of this Section
      and let us fix a triangulable $\phigam$-module $D$ over $\robba_A$
      such that $D[1/t]=M$.
      Moreover let $\mcal A$ be the set of Sen weights of $\WdR^+(D)$ and let
      $\WdR(M)=\bigoplus_{\bar a\in\bar{\mcal A}}W_{\bar a}$
      and $\WdR^+(D)=\bigoplus_{\bar a\in\bar{\mcal A}}W_{\bar a}^+$ as in Theorem \ref{non integral weights}.
      Let $\Lambda\coloneqq D^r\otimes_{\robba_A^r,\tau}A\llbracket t\rrbracket$, then
      through an argument similar to \cite[Remark 5.2.5]{wu2022trianguline}, we can get a decomposition of $\Lambda$:
      let $\nabla\coloneqq\frac{\log(\gamma^{p^s})}{\log(\chi(\gamma^{p^s}))}$ for $s$ big enough
      be the Sen operator as in \cite[Proposition III.1.1]{berger2004equations}
      acting on $\Lambda$
      such that $\nabla(t^n)=nt^n$ for $n\in\Z$, where $\gamma$ is a topological generator of $\Gamma_m$.
      For all $n\in\N$, we have that $\nabla$ acts on $\Lambda/t^n\Lambda$,
      hence we have a decomposition
      \[
        \Lambda/t^n\Lambda=\bigoplus_{\bar a\in\bar{\mcal A}}(\Lambda/t^n\Lambda)_{\bar a}
      \]
      as $A$-modules consisting of (direct sums of) generalized eigenspaces for the $A$-linear operator $\nabla$
      (the generalized eigenspaces are the eigenspaces for $\nabla$ seen as an endomorphism of an $L$-vector space,
      but are easily seen to be stable under the $A$-action).
      Since the action of $\Gamma_m$ and $\nabla$ commute, we have that each
      $(\Lambda/t^n\Lambda)_{\bar a}$ is stable under the $\Gamma_m$-action,
      thus $\Lambda_{\bar a}\coloneqq\varprojlim_{n}(\Lambda/t^n\Lambda)_{\bar a}$
      is an $A\llbracket t\rrbracket$-representation of $\Gamma_m$ and we obviously have a decomposition
      \[
        \Lambda=\bigoplus_{\bar a\in\bar{\mcal A}}\Lambda_{\bar a}.
      \]
      Notice that we can choose a basis of generalized eigenvectors for $\nabla$ that respects the
      triangulation of $D$.
      Observe moreover that
      \[
        \bigoplus_{\bar a\in\bar{\mcal A}}W_{\bar a}^+=\WdR^+(D)=\Lambda\otimes_{K_m\llbracket t\rrbracket}\BdR^+
        =\bigoplus_{\bar a\in \bar{\mcal A}}(\Lambda_{\bar a}\otimes_{K_m\llbracket t\rrbracket}\BdR^+).
      \]
      It is then obvious that $\Lambda_{\bar a}[1/t]\otimes_{K_m((t))}\BdR =
      W_{\bar a}[1/t]$ by the uniqueness in Theorem \ref{non integral weights},
      since they have the same Sen weights;
      hence
      \[
        W_{\bar a}^+=\Lambda_{\bar a}\otimes_{K_m\llbracket t\rrbracket}\BdR^+.
      \]
      For simplicity of notations, let
      \[
        \Lambda_{\bar a}(\delta_{\bar a}^{-1})\coloneqq
        \Lambda_{\bar a}\otimes_{A\llbracket t\rrbracket}(\robba_A^r(\delta_{\bar a}^{-1})\otimes_{\robba_A^r}A\llbracket t\rrbracket)
      \]
      for some character $\delta_{\bar a}$ such that $\rr{wt}(\delta_{\bar a})+\Z=\bar a$,
      so that $\Lambda_{\bar a}(\delta_{\bar a}^{-1})$ has integral Sen weights.
      By Lemma \ref{replacement of Lemma zhixiang}, we have
      \begin{align*}
        D_{\pdR,\bar a}(W_{\bar a})&=D_{\pdR}(W_{\bar a}\otimes_{\BdR}\WdR(\delta_{\bar a}^{-1}))
        =D_{\pdR}(\Lambda_{\bar a}(\delta_{\bar a}^{-1})\otimes_{K_m\llbracket t\rrbracket}\BdR)\\
        &=(\Lambda_{\bar a}(\delta_{\bar a}^{-1})[1/t]\otimes_{K_m}K_m[\log(t)])^\Gamma.
      \end{align*}
      Therefore by Lemma \ref{lemma zhixiang1}, we have an isomorphism of $A$-modules
      \begin{align*}
        D_{\pdR,\bar a}(W_{\bar a})&\rightarrow (\Lambda_{\bar a}(\delta_{\bar a}^{-1})[1/t])^{(\gamma-1)\text{-nil}}\\
        \sum_{i=1}^la_i\log(t)^i&\mapsto a_0\\
        \sum_{i=1}^l\frac{(-1)^i}{i!}\nabla^i(a)\log(t)^i&\mapsfrom a,
      \end{align*}
      where $\nabla^{l+1}(a)=0$.
      Let $\{x_{\bar a,1}',\dots,x'_{\bar a,n_{\bar a}}\}$ be an $A$-basis of $D_{\pdR,\bar a}(W_{\bar a})$ and
      let $\{y_{\bar a,1}',\dots,y_{\bar a,n_{\bar a}}'\}$ be the $A$-basis of
      $(\Lambda_{\bar a}(\delta_{\bar a}^{-1})[1/t])^{(\gamma-1)\text{-nil}}$
      obtained as image of the basis $\{x_{\bar a,1}',\dots,x'_{\bar a,n_{\bar a}}\}$.
      Let $\Lambda(\delta_{\bar a}^{-1})_{0,\bar a}$ be the $A\llbracket t\rrbracket$-module spanned by
      $\{y_{\bar a,1}',\dots,y_{\bar a,n_{\bar a}}'\}$.
      By construction we have that $\Lambda(\delta_{\bar a}^{-1})_{0,\bar a}[1/t]=\Lambda_{\bar a}(\delta_{\bar a}^{-1})[1/t]$,
      hence $\Lambda(\delta_{\bar a}^{-1})_{0,\bar a}$ is a sub-$A\llbracket t\rrbracket$-module of
      $M^r(\delta_{\bar a}^{-1})\otimes_{\robba_A^r,\tau}A\llbracket t\rrbracket$ and it is stable under the action of
      $\Gamma_m$ inherited by $M^r(\delta_{\bar a}^{-1})$, as $(\gamma-1)$ acts nilpotently
      on the basis $\{y_{\bar a,1}',\dots,y_{\bar a,n_{\bar a}}'\}$.
      Consider the isomorphism of $A((t))$-modules
      \[
        M^r(\delta_{\bar a}^{-1})^r\otimes_{\robba_A^r[1/t],\tau}A((t))\rightarrow M^r\otimes_{\robba_A^r[1/t],\tau}A((t))
      \]
      and let $\{y_{\bar a,1},\dots,y_{\bar a,n_{\bar a}}\}$ be the elements obtained as the image of
      the $A\llbracket t\rrbracket$-basis $\{y'_{\bar a,1},\dots,y'_{\bar a,n_{\bar a}}\}$
      of $\Lambda(\delta_{\bar a}^{-1})_{0,\bar a}$ through the above map.
      We denote by $\Lambda_{0,\bar a}$ the $A\llbracket t\rrbracket$-module spanned by the
      elements $\{y_{\bar a,1},\dots,y_{\bar a,n_{\bar a}}\}$ in $M^r\otimes_{\robba_A^r[1/t],\tau}A((t))$;
      we then have that $\Lambda_{0,\bar a}$ is $\Gamma_m$-stable and $\Lambda_{0,\bar a}[1/t]=\Lambda_{\bar a}[1/t]$.
      Let us define
      \[
        \Lambda_0\coloneqq\bigoplus_{\bar a\in\bar{\mcal A}}\Lambda_{0,\bar a},
      \]
      which is a $\Gamma_m$-stable lattice of $M^r\otimes_{\robba_A^r[1/t],\tau}A((t))$.
      We have that
      \[
        \{y_{\bar a,1},\dots,y_{\bar a,n_{\bar a}}\colon\bar a\in\bar{\mcal A}\}
      \]
      is an $A\llbracket t\rrbracket$-basis of $\Lambda_0$ and hence we can choose it as
      an $A$-basis of $\Lambda_0/t\Lambda_0$.
      Observe that we can then identify the two $A$-modules $\Lambda_0/t\Lambda_0$ and
      $\bigoplus_{\bar a\in\bar{\mcal A}}D_{\pdR,\bar a}(W_{\bar a})$ through the isomorphism
      \begin{equation}
        \label{identification}
        \begin{split}
          \mcal F:\Lambda_0/t\Lambda_0&\xrightarrow{\sim}\bigoplus_{\bar a\in\bar{\mcal A}}D_{\pdR,\bar a}(W_{\bar a})\\
          y_{\bar a,i}&\mapsto \sum_{i=1}^l\frac{(-1)^i}{i!}\nabla^i(y'_{\bar a,i})\log(t)^i,
        \end{split}
      \end{equation}
      where $\nabla^{l+1}(y'_{\bar a,i})=0$.
      Obviously for all $\bar a\in\bar{\mcal A}$, we have that $\mcal F$ restricts to an isomorphism
      \[
        \mcal F:\Lambda_{0,\bar a}/t\Lambda_{0,\bar a}\xrightarrow{\sim} D_{\pdR,\bar a}(W_{\bar a}).
      \]
    \end{rem}

    \begin{lem}
      \label{iso on tLambda0'}
      Let $\Lambda(\delta_{\bar a}^{-1})_{0,\bar a}$ be the lattice constructed in Remark \ref{identify Lambda0/t with DpdR}
      for some $\bar a\in\bar{\mcal A}$.
      Then $\nabla:t\Lambda(\delta_{\bar a}^{-1})_{0,\bar a}\rightarrow t\Lambda(\delta_{\bar a}^{-1})_{0,\bar a}$ is an $A$-linear isomorphism.
    \end{lem}

    \begin{proof}
      Observe that $\nabla$ acts nilpotently on the basis $\{y'_{\bar a,1},\dots,y'_{\bar a,n_{\bar a}}\}$
      of $\Lambda(\delta_{\bar a}^{-1})_{0,\bar a}$ by construction.
      Hence there is an $A$-basis of $\Lambda(\delta_{\bar a}^{-1})_{0,\bar a}/t\Lambda(\delta_{\bar a}^{-1})_{0,\bar a}$ such that
      $\nabla$ has upper triangular matrix with nilpotent elements on the diagonal.
      Then the claim of the Lemma is a direct consequence of \cite[Lemma 3.1]{matthiasthesis}.
    \end{proof}

    \begin{lem}
      \label{iso on t^iLambda cap tLambda0}
      Let us choose a trivialization of $\Lambda(\delta_{\bar a}^{-1})_{0,\bar a}$
      and let $\Lambda_{\bar a}'\in\rr{Gr}_{n_{\bar a}}(A)$ stable under the action of $\Gamma_m$.
      Then we have
      \[
        \nabla:t^i\Lambda_{\bar a}'\cap t\Lambda(\delta_{\bar a}^{-1})_{0,\bar a}
        \rightarrow t^i\Lambda_{\bar a}'\cap t\Lambda(\delta_{\bar a}^{-1})_{0,\bar a}
      \]
      is an isomorphism for all $i\in\Z$.
    \end{lem}

    \begin{proof}
      If $i<<0$, then we have $t^i\Lambda_{\bar a}'\cap t\Lambda(\delta_{\bar a}^{-1})_{0,\bar a}=t\Lambda(\delta_{\bar a}^{-1})_{0,\bar a}$
      and the claim is true by Lemma \ref{iso on tLambda0'}.
      Now it's enough to observe that
      \[
        \nabla(t^i\Lambda_{\bar a}'\cap t\Lambda(\delta_{\bar a}^{-1})_{0,\bar a})
      \subseteq t^i\Lambda_{\bar a}'\cap t\Lambda(\delta_{\bar a}^{-1})_{0,\bar a}
      \]
      for all $i\in\Z$ because both $\Lambda_{\bar a}'$ and $\Lambda(\delta_{\bar a}^{-1})_{0,\bar a}$ are $\Gamma_m$-stable,
      as showed in Remark \ref{identify Lambda0/t with DpdR}.
      Therefore $\nabla_{|t^i\Lambda_{\bar a}'\cap t\Lambda(\delta_{\bar a}^{-1})_{0,\bar a}}$
      is an automorphism.
    \end{proof}

    \begin{lem}
      \label{Gamma-stable lattices and Gamma-stable filtrations}
      Let $\Lambda_0$ be the $\Gamma_m$-stable lattice of $M^r\otimes_{\robba_A^r[1/t],\tau}A((t))$
      defined in Remark \ref{identify Lambda0/t with DpdR}.
      We have that $\Lambda_0$ decomposes as $\bigoplus_{\bar a\in\bar{\mcal A}}\Lambda_{0,\bar a}$, where
      each $\Lambda_{0,\bar a}$ is a free $A\llbracket t\rrbracket$-module stable under the action of $\Gamma_m$
      Then the action of $\Gamma_m$ on $\Lambda_{0,\bar a}$ induces a $\Gamma_m$-action
      on the $A$-module $\Lambda_{0,\bar a}/t\Lambda_{0,\bar a}$ for all $\bar a\in\bar{\mcal A}$.
      There exists a bijection
      \begin{align*}
        \left\{
          \begin{array}{c}
            \Gamma_m\text{-stable }A\llbracket t\rrbracket\text{-lattices}\\
            \text{of }M^r\otimes_{\robba_A^r[1/t],\tau}A((t))\\
            \text{ with Sen weights in }\bar{\mcal A}
          \end{array}\right\}&
          \mathrel{\mathop{\rightleftarrows}^{f}_{g}}
          \left\{(\Fil^\bullet(\Lambda_{0,\bar a}/t\Lambda_{0,\bar a}))_{\bar a\in\bar{\mcal A}}\colon
          \begin{array}{c}
            \Fil^\bullet(\Lambda_{0,\bar a}/t\Lambda_{0,\bar a})\text{ is a}\\
            \Gamma_m\text{-stable filtration}\\
            \text{of }\Lambda_{0,\bar a}/t\Lambda_{0,\bar a}\\
            \text{for all }\bar a\in\bar{\mcal A}
          \end{array}\right\}\\
        \Lambda=\bigoplus_{\bar a\in\bar{\mcal A}}\Lambda_{\bar a}&\mapsto(\pi(\Lambda_{\bar a}))_{\bar a\in\bar{\mcal A}}\\
        \bigoplus_{\bar a\in\bar{\mcal A}}\sigma(\Fil^\bullet(\Lambda_{0,\bar a}/t\Lambda_{0,\bar a}))&
        \mapsfrom(\Fil^\bullet(\Lambda_{0,\bar a}/t\Lambda_{0,\bar a}))_{\bar a\in\bar{\mcal A}},
      \end{align*}
      where the maps $\pi$ and $\sigma$ are defined as in (\ref{def pi and sigma}), while the decomposition of $\Lambda$
      as direct sum of $\Gamma_m$-stable sub-$A\llbracket t\rrbracket$-modules $\Lambda_{\bar a}$ is given
      by the same argument in Remark \ref{identify Lambda0/t with DpdR}.
    \end{lem}

    \begin{proof}
      First of all notice that if $\Lambda_{\bar a}\in\rr{Gr}_{n_{\bar a}}^{\mu_{\bar a},\rig}(A)$
      is a $\Gamma_m$-stable lattice of $\Lambda_{0,\bar a}[1/t]$,
      then $\pi_{\mu_{\bar a}}(\Lambda_{\bar a})$ is $\Gamma_m$-stable:
      recall the definition of $\pi_{\mu_{\bar a}}(\Lambda_{\bar a})$ in (\ref{def pi_mu})
      and let us fix a topological generator $\gamma$ of $\Gamma_m$,
      then we have $\gamma(\Lambda_{0,\bar a})\subseteq\Lambda_{0,\bar a}$ and $\gamma(t\Lambda_{0,\bar a})\subseteq t\Lambda_{0,\bar a}$
      because $\Lambda_{0,\bar a}$ and $t\Lambda_{0,\bar a}$ are $\Gamma_m$-stable by Remark \ref{identify Lambda0/t with DpdR}.
      Moreover $\gamma(t^{-\mu_{\bar a,i}}\Lambda_{\bar a})=
      \chi(\gamma)^{-\mu_{\bar a,i}}t^{-\mu_{\bar a,i}}\gamma(\Lambda_{\bar a})\subseteq t^{-\mu_{\bar a,i}}\Lambda_{\bar a}$
      by the fact that $\Lambda_{\bar a}$ is $\Gamma_m$-stable; therefore
      $\gamma(t^{-\mu_{\bar a,i}}\Lambda_{\bar a}\cap\Lambda_{0,\bar a})\subseteq t^{-\mu_{\bar a,i}}\Lambda_{\bar a}\cap\Lambda_{0,\bar a}$
      and $\gamma(t^{-\mu_{\bar a,i}}\Lambda_{\bar a}\cap t\Lambda_{0,\bar a})\subseteq t^{-\mu_{,\bar a,i}}\Lambda_{\bar a}\cap t\Lambda_{0,\bar a}$,
      hence
      \[
        \gamma((t^{-\mu_{\bar a,i}}\Lambda_{\bar a}\cap\Lambda_{0,\bar a})/(t^{-\mu_{\bar a,i}}\Lambda_{\bar a}\cap t\Lambda_{0,\bar a}))\subseteq
      (t^{-\mu_{\bar a,i}}\Lambda_{\bar a}\cap\Lambda_{0,\bar a})/(t^{-\mu_{\bar a,i}}\Lambda_{\bar a}\cap t\Lambda_{0,\bar a}).
      \]
      We already know that $\pi_{\mu_{\bar a}}\circ\sigma_{\mu_{\bar a}}=\rr{id}$
      because $\sigma_{\mu_{\bar a}}$ is a section of $\pi_{\mu_{\bar a}}$,
      which implies that $f\circ g=\rr{id}$.
      Let now $\Lambda$ be a $\Gamma_m$-stable lattice with Sen weights in $\bar{\mcal A}$ and
      let us decompose $\Lambda$ as $\bigoplus_{\bar a\in\bar{\mcal A}}\Lambda_{\bar a}$
      as in Remark \ref{identify Lambda0/t with DpdR}.
      By the assumption of $\Lambda$ having the same mod $\Z$ classes of Sen weights as $\Lambda_0$
      counted with multiplicity, we have that $\Lambda_{\bar a}$ is free of rank $n_{\bar a}$
      over $A\llbracket t\rrbracket$ for all $\bar a\in\bar{\mcal A}$.
      Let $\mu_{\bar a}\in X_*(T)^+$ such that $\Lambda_{\bar a}\in\rr{Gr}_{n_{\bar a}}^{\mu_{\bar a},\rig}(A)$
      and let $\Fil^\bullet(\Lambda_{0,\bar a}/t\Lambda_{0,\bar a})=\pi_{\mu_{\bar a}}(\Lambda_{\bar a})$.
      We consider $\Fil^i(\Lambda_{0,\bar a}/t\Lambda_{0,\bar a})$ as a submodule of $\Lambda_{0,\bar a}$ and we show that
      $\Fil^i(\Lambda_{0,\bar a}/t\Lambda_{0,\bar a})\subset t^{-\mu_{\bar a,i}}\Lambda_{\bar a}$.
      This will end the proof, since then we have
      \[
        \sigma_{\mu_{\bar a}}(\pi_{\mu_{\bar a}}(\Lambda_{\bar a}))
        =\sum_{i}\Fil^i(\Lambda_{0,\bar a}/t\Lambda_{0,\bar a})\otimes_A t^{\mu_{\bar a,i}}A\llbracket t\rrbracket
        \subseteq \Lambda_{\bar a};
      \]
      moreover $\Lambda_{\bar a}$ and $\sigma_{\mu_{\bar a}}(\pi_{\mu_{\bar a}}(\Lambda_{\bar a}))$
      have the same elementary divisors, hence they are exactly the same.
      So let $x\in \Fil^i(\Lambda_{0,\bar a}/t\Lambda_{0,\bar a})
      =(t^{-\mu_i}\Lambda_{\bar a}\cap\Lambda_{0,\bar a})/(t^{-\mu_i}\Lambda_{\bar a}\cap t\Lambda_{0,\bar a})$.
      Let us consider the isomorphism of $A((t))$-modules
      \[
        M^r\otimes_{\robba_A^r[1/t],\tau}A((t))\xrightarrow{\sim}M(\delta_{\bar a}^{-1})^r\otimes_{\robba_A^r[1/t],\tau}A((t))
      \]
      and let $x'$ be the image of $x$ through the above isomorphism.
      As in Remark \ref{identify Lambda0/t with DpdR} we have that $\Lambda(\delta_{\bar a}^{-1})_{0,\bar a}$
      is the image of $\Lambda_{0,\bar a}$ and let
      $\Lambda_{\bar a}'$ be the image of $\Lambda_{\bar a}$ through the above isomorphism.
      Then $x'\in (t^{-\mu_i}\Lambda'_{\bar a}\cap\Lambda(\delta_{\bar a}^{-1})_{0,\bar a})/(t^{-\mu_i}\Lambda'_{\bar a}\cap t\Lambda(\delta_{\bar a}^{-1})_{0,\bar a})$,
      so we can write $x'=x_1'+tx_2'$ where $x_1'\in t^{-\mu_i}\Lambda'_{\bar a}$ and $x_2'\in\Lambda(\delta_{\bar a}^{-1})_{0,\bar a}$.
      Moreover $(\gamma -1)$ is nilpotent on $x'$, since
      $(\gamma-1)$ is nilpotent on the $A$-module $\Lambda(\delta_{\bar a}^{-1})_{0,\bar a}/t\Lambda(\delta_{\bar a}^{-1})_{0,\bar a}$
      by the construction in Remark \ref{identify Lambda0/t with DpdR}
      and this implies that $\nabla$ is nilpotent too.
      Therefore we have
      \[
        0=\nabla^k(x')=\nabla^k(x_1')+\nabla^k(tx_2')
      \]
      for some $k\in\N$.
      As both $\Lambda(\delta_{\bar a}^{-1})_{0,\bar a}$ and $\Lambda_{\bar a}'$ are $\Gamma_m$-stable, we have that they are stable
      also for the action of $\nabla$, hence
      $-\nabla^k(tx_2')=\nabla^k(x_1')\in t^{-\mu_i}\Lambda_{\bar a}'\cap t\Lambda(\delta_{\bar a}^{-1})_{0,\bar a}$.
      By Lemma \ref{iso on t^iLambda cap tLambda0}, we have $tx_2'\in t^{-\mu_i}\Lambda_{\bar a}'$
      and thus $x'=x_1'+tx_2'\in t^{-\mu_i}\Lambda_{\bar a}'$.
      By definition of $x'$, this implies that $x\in t^{-\mu_i}\Lambda_{\bar a}$.
    \end{proof}

    \begin{thm}
      \label{final bijection}
      Let $D$ be the triangulable $\phigam$-module over $\robba_A$ such that $D[1/t]=M$ as
      fixed at the beginning of Remark \ref{identify Lambda0/t with DpdR},
      moreover let $(\delta_{1},\dots,\delta_{n})\in\T^n(A)$ be its parameters
      and let $\mcal A\coloneqq\{\rr{wt}(\delta_{i})\colon 1\leq i\leq n\}$.
      Let us define the space
      \begin{equation}
        \rr{Gr}^\Gamma_{M,\bar{\mcal A}}\subset\rr{Gr}^\Gamma_{M}
      \end{equation}
      of $\Gamma_m$-stable lattices in $M^r\otimes_{\robba_A^r,\tau}A((t))$ with Sen weights
      in $\bar{\mcal A}$.
      For every $\bar a\in\bar{\mcal A}\coloneqq(\mcal A+\Z)/\Z$, let $W_{\bar a}$ be the $\BdRA$-representations of $\absGal$
      such that $\bigoplus_{\bar a \in\bar{\mcal A}}W_{\bar a}=\WdR(M)$ as in Theorem \ref{non integral weights}.
      There exists a bijection
      \begin{align*}
        \rr{Gr}_{M,\bar{\mcal A}}^\Gamma&\rightarrow
        \left\{(\Fil^\bullet(D_{\pdR,\bar a}(W_{\bar a})))_{\bar a\in\bar{\mcal A}}\colon
        \begin{array}{c}
          \Fil^\bullet(D_{\pdR,\bar a}(W_{\bar a}))\text{ is an }N_{\bar a}\text{-stable}\\
          \text{filtration of } D_{\pdR,\bar a}(W_{\bar a})\text{ for all }\bar a\in\bar{\mcal A}
        \end{array}
      \right\},
      \end{align*}
      which sends a $\Gamma_m$-stable lattice $\Lambda=\bigoplus_{\bar a\in\bar{\mcal A}}\Lambda_{\bar a}$
      to $(\pi(\Lambda_{\bar a}))_{\bar a\in\bar{\mcal A}}$,
      where $\pi(\Lambda_{\bar a})$ is defined as in (\ref{def pi and sigma}).
      The claim makes sense, since we can identify $\Lambda_{0,\bar a}/t\Lambda_{0,\bar a}$ and
      $D_{\pdR,\bar a}(W_{\bar a})$ through the isomorphism (\ref{identification}).
    \end{thm}

    \begin{proof}
      First of all we show that the set $ \rr{Gr}_{M,\bar{\mcal A}}^\Gamma$ is in bijection with the
      set of $\absGal$-stable $\BdRA^+$-lattices of $\WdR(M)$ with Sen weights in $\bar{\mcal A}$.
      We already know that there is a bijection
      \[
        \rr{Gr}_M^\Gamma\leftrightarrow\{\absGal\text{-stable }\BdRA^+\text{-lattices of }\WdR(M)\}
      \]
      thanks to Proposition \ref{bijection of lattices and modules} and Theorem \ref{bijection of modules and lattices}.
      Then it is sufficient to show that if $\Lambda\in \rr{Gr}_{M,\bar{\mcal A}}^\Gamma$, then
      the corresponding $\BdRA^+$-lattice has modulo $\Z$ equivalence classes of its Sen weights in $\bar{\mcal A}$.
      Let $E$ be the $\phigam$-module over $\robba_A$ associated to $\Lambda$ through the bijection of Proposition
      \ref{bijection of lattices and modules}; then the associated $\BdRA^+$-lattice
      is $\WdR^+(E)=\Lambda\otimes_{K_m\llbracket t\rrbracket}\BdR^+$, which has the same Sen weights as $\Lambda$.
      By Lemma \ref{bijection of G-stable lattices and N-stable filtrations for non-integral weights},
      this proves that the stated bijection exists and it is given by
      \[
        \Lambda=\bigoplus_{\bar a\in\bar{\mcal A}}\Lambda_{\bar a}\mapsto (\Fil^\bullet_{W_{\bar a}^+}(D_{\pdR,\bar a}(W_{\bar a})))_{\bar a\in\bar{\mcal A}},
      \]
      where $W_{\bar a}\coloneqq W_{\bar a}^+[1/t]$ and
      $W_{\bar a}^+\coloneqq \Lambda_{\bar a}\otimes_{K_m\llbracket t\rrbracket}\BdR^+$.
      Now we show that for all $\bar a\in\bar{\mcal A}$, the flag given by the filtration
      $\Fil^\bullet_{W_{\bar a}^+}(D_{\pdR,\bar a}(W_{\bar a}))$
      corresponds to $\pi(\Lambda_{\bar a})$ through the isomorphism $\mcal F$ constructed in Remark \ref{identify Lambda0/t with DpdR}.
      Notice that for all $\bar a\in\bar{\mcal A}$, $\Lambda_{0,\bar a}$ and $\Lambda_{\bar a}$
      have the same rank, since by hypothesis
      the Sen weights of $\Lambda_0$ and $\Lambda$ are the same and have the same multiplicity
      modulo $\Z$.
      Let then $\mu_{\bar a}\in X_*(T)^+$ such that $\Lambda_{\bar a}\in\rr{Gr}_{n_{\bar a}}^{\mu_{\bar a},\rig}$;
      in other words, there is a basis $\{b_{\bar a,1},\dots,b_{\bar a,n_{\bar a}}\}$ of
      $\Lambda_{0,\bar a}$ such that $\Lambda_{\bar a}$ is spanned by
      $\{t^{\mu_{\bar a,1}}b_{\bar a,1},\dots,t^{\mu_{\bar a,n_{\bar a}}}b_{\bar a,n_{\bar a}}\}$.
      Moreover let $\{y_{\bar a,1},\dots,y_{\bar a ,n_{\bar a}}\}$ be the basis of $\Lambda_{0,\bar a}$
      constructed in Remark \ref{identify Lambda0/t with DpdR}.
      Since $\Lambda_{\bar a}$ is $\Gamma_m$-stable, we can choose the basis $\{b_{\bar a,1},\dots,b_{\bar a,n_{\bar a}}\}$
      such that $b_{\bar a,i}=\sum_{j=1}^{n_{\bar a}}a_jy_{\bar a,j}$ with $a_j\in A$
      by Lemma \ref{Gamma-stable lattices and Gamma-stable filtrations}.
      Furthermore by definition of $\pi_\mu$ (\ref{def pi_mu}), we have
      \begin{equation}
        \label{first filtration}
        \pi_{\mu_{\bar a}}(\Lambda_{\bar a})=(t^{-\mu_{\bar a,\bullet}}\Lambda_{\bar a}\cap\Lambda_{0,\bar a})
        /(t^{-\mu_{\bar a,\bullet}}\Lambda_{\bar a}\cap t\Lambda_{0,\bar a})
        =\bigoplus_{\mu_{\bar a,j}\leq\mu_{\bar a,\bullet}}A\cdot b_{\bar a,j}.
      \end{equation}
      Observe that
      \begin{align*}
        \Fil^i_{W_{\bar a}^+}(D_{\pdR,\bar a})(W_{\bar a})
        &=\left(t^{i}\BdR^+[\log(t)]\otimes_{\BdR^+}
        (\Lambda_{\bar a}\otimes_{K_m\llbracket t\rrbracket}\BdR^+)
        \otimes_{\BdR^+}\WdR^+(\delta_{\bar a}^{-1})\right)^{\absGal}\\
        &=\left(\BdR^+[\log(t)]\otimes_{\BdR^+}
        (t^i\Lambda_{\bar a}\otimes_{K_m\llbracket t\rrbracket}\BdR^+)
        \otimes_{\BdR^+}\WdR^+(\delta_{\bar a}^{-1})\right)^{\absGal}.
      \end{align*}
      Notice that $(\Lambda_{\bar a}\otimes_{K_m\llbracket t\rrbracket}\BdR^+)
      \otimes_{\BdR^+}\WdR^+(\delta_{\bar a}^{-1})$ is a $\absGal$-stable $\BdRA^+$-lattice with integral Sen weights,
      thanks to the twist by $\WdR^+(\delta_{\bar a})$.
      For $j$ such that $\mu_{\bar a,j}\leq -i$, we have that
      $b_{\bar a,j}=\sum_{k=1}^{n_{\bar a}}a_ky_{\bar a,k}\in t^i\Lambda_{\bar a}\otimes_{K_m\llbracket t\rrbracket}\BdR^+$.
      We will show that for all $k\geq0$ and for all $j$ such that $\mu_{\bar a,j}\leq -i$ we have
      \[
        \nabla^k(b_{\bar a,j})\in t^i\Lambda_{\bar a}\otimes_{K_m\llbracket t\rrbracket}\BdR^+:
      \]
      it is sufficient to prove that $\gamma^k(b_{\bar a,j})\in t^i\Lambda_{\bar a}\otimes_{K_m\llbracket t\rrbracket}\BdR^+$
      for all $k\geq0$;
      this is true, since the filtration $\pi_{\mu_{\bar a}}(\Lambda_{\bar a})$ is
      $\Gamma_m$-stable by Lemma \ref{Gamma-stable lattices and Gamma-stable filtrations}, hence
      $\gamma^k(b_{\bar a,j})\in \bigoplus_{\mu_{\bar a,l}\leq-i}A\llbracket t\rrbracket\cdot b_{\bar a,l}$.
      As a consequence, if we let $b_{\bar a,j}'$ denote the element corresponding to $b_{\bar a,j}$
      through the isomorphism of $\BdRA^+$-modules
      \[
        (t^i\Lambda_{\bar a}\otimes_{K_m\llbracket t\rrbracket}\BdR^+)
        \otimes_{\BdR^+}\WdR^+(\delta_{\bar a}^{-1})\rightarrow t^i\Lambda_{\bar a}\otimes_{K_m\llbracket t\rrbracket}\BdR^+,
      \]
      we have that $\nabla^k(b'_{\bar a,j})\in  (t^i\Lambda_{\bar a}\otimes_{K_m\llbracket t\rrbracket}\BdR^+)
      \otimes_{\BdR^+}\WdR^+(\delta_{\bar a}^{-1})$ for all $k\geq0$.
      Observe that $\nabla^k(b'_{\bar a,j})=\sum_{i=1}^{n_{\bar a}}a_i\nabla^k(y'_{\bar a,i})$, so
      by construction of $\{y_{\bar a,1}',\dots,y'_{\bar a,n_{\bar a}}\}$ in Remark \ref{identify Lambda0/t with DpdR},
      we have $\nabla^{K+1}(b'_{\bar a,j})=0$ for some $K\geq0$.
      Thus
      \begin{align*}
        \mcal F(b_{\bar a,j})&=\sum_{k=0}^K\frac{(-1)^{k}}{k!}\nabla^k(b'_{\bar a,j})(\log(t))^k
        =\sum_{i=1}^{n_{\bar a}}a_i\left(\sum_{k=0}^K\frac{(-1)^{k}}{k!}\nabla^k(y'_{\bar a,j})(\log(t))^k\right)\\
        &\in\BdR^+[\log(t)]\otimes_{\BdR^+}(t^i\Lambda_{\bar a}\otimes_{K_m\llbracket t\rrbracket}\BdR^+)
        \otimes_{\BdR^+}\WdR^+(\delta_{\bar a}^{-1}).
      \end{align*}
      Now notice that $\sum_{k=0}^K\frac{(-1)^{k}}{k!}\nabla^k(y'_{\bar a,j})(\log(t))^k\in D_{\pdR,\bar a}(W_{\bar a})$
      by the construction in Remark \ref{identify Lambda0/t with DpdR}, thus
      \[
        \mcal F(b_{\bar a,j})=\sum_{k=0}^K\frac{(-1)^{k}}{k!}\nabla^k(b'_{\bar a,j})(\log(t))^k\in
        \Fil^i_{W_{\bar a}^+}(D_{\pdR,\bar a}(W_{\bar a}))
      \]
      and they constitute an $A$-basis of $\Fil^i_{W_{\bar a}^+}(D_{\pdR,\bar a}(W_{\bar a}))$.
      Finally, observe that
      \begin{align*}
        \pi_{\mu_{\bar a}}(\Lambda_{\bar a})=\bigoplus_{\mu_{\bar a,j}\leq\mu_{\bar a,\bullet}}A\cdot b_{\bar a,j}
        \xrightarrow{\mcal F}
        &\bigoplus_{\mu_{\bar a,j}\leq\mu_{\bar a,\bullet}}
        A\cdot\mcal F(b_{\bar a,j})
        &=\Fil_{W_{\bar a}^+}^{-\mu_{\bar a,\bullet}}(D_{\pdR,\bar a}(W_{\bar a})).
      \end{align*}
    \end{proof}

\section{Trianguline representations of dimension 2}
\label{section 6}

    In this section we prove that any trianguline representation of dimension 2 is in the
    trianguline variety.
    More precisely, we show the following Theorem, which is a special case of Theorem \ref{result2}.

    \begin{thm}
    \label{DIM2}
    Let $\rho:\absGal\rightarrow GL_2(\mathcal{O}_L)$ be a trianguline representation
    and $\overline{\rho}:\absGal\rightarrow GL_2(k_L)$ a reduction of $\rho$
    to the residue field of $L$.
    Let $(\delta_1,\delta_2)\in\mathcal{T}^2(L)$ be the parameters of $\rho$.
    We have $(\rho,(\delta_1,\delta_2))\in\trivar(L)$.
    \end{thm}

    Theorem \ref{DIM2} is obviously true if $\delta_1/\delta_2\in\Treg(L)$, because then
    $(\delta_1,\delta_2)\in\Treg_2(L)$.
    In section 4 we proved that $(\rho,(\delta_1,\delta_2))\in\trivar(L)$ if $\delta_1/\delta_2\in\mathcal{T}^-(L)$.
    Therefore the only case left to treat is $\delta_1/\delta_2\in\mathcal{T}^+(L)$, which
    is the most complicated: since $H^2_{\phi,\Gamma}(\robba_L(\delta_1/\delta_2))\neq0$,
    it is not possible to choose a neighbourhood $\mathcal{U}$ of $(\delta_1,\delta_2)$
    in $\mathcal{T}^2$ such that $H^2_{\phi,\Gamma}(\robba_{\mathcal{U}}(\widetilde\delta_1/\widetilde\delta_2))$
    is locally free, where $(\widetilde\delta_1,\widetilde\delta_2)$ are the characters
    corresponding to $\mathcal{U}\hookrightarrow\mathcal{T}^2$.
    This implies that we cannot use Theorem \ref{extensions} to construct a rigid space with a
    family of $\phigam$-modules specializing to $\drig(\rho)$, as we did
    in the proof of Proposition \ref{2 dim T-}.
    Of course this makes things more difficult, but we'll see that it's still possible
    to somehow reduce ourselves to the case in which $\delta_1/\delta_2\notin\mathcal{T}^+(L)$
    with the machinery developed in section 5 and be able then to use the results of section 4 to construct such a
    rigid space $X$ with a family of $\phigam$-modules $\D$ over $X$.
    The downside is now that the space $X$ is no more a vector bundle over a neighbourhood
    $\mathcal{U}$ of $(\delta_1,\delta_2)$.
    Hence it will be harder to show that the family of $\phigam$-modules $\D$
    is densely regular.

    Let us assume then that $\drig(\rho)$ is an extension of $\robba_L(\delta_2)$ by $\robba_L(\delta_1)$
    such that $\delta_1/\delta_2\in\T^+(L)$, which means that $\delta_1/\delta_2$
    is of the form $\chi x^n$ for some $n\in\mathbb{N}$.

\subsection{Construction of space of extensions}

  \begin{lem}
    \label{iso of H^1}
    Let $\delta_1,\delta_2\in\T(L)$ such that $\delta_1/\delta_2=\chi x^n$ for some $n\in\mathbb{N}$.
    We have the isomorphism
    \begin{equation}
      \label{from 1/t to twist}
      H^1_{\phi,\Gamma}(\robba_L(\delta_1/\delta_2)[1/t])\cong H^1_{\phi,\Gamma}(\robba_L(\chi x^{-1}))
    \end{equation}
    of $L$-modules.
  \end{lem}

  \begin{proof}
    By Remark \ref{phigam cohomology with 1/t}, we have an isomorphism
    \[
      H^1_{\phi,\Gamma}(\robba_L(\delta_1/\delta_2)[1/t])\cong
      \lim_{\stackrel{\rightarrow}{i}} H^1_{\phi,\Gamma}(\robba_L(\delta_1/\delta_2 x^{-i}))=
      \lim_{\stackrel{\rightarrow}{i}} H^1_{\phi,\Gamma}(\robba_L(\chi x^{n-i})).
    \]
    Let us consider the character $\chi x^{-1}\in\T^{\mathrm{wr}}(L)\cap\T^\reg(L)$,
    which has weight $0$.
    By Lemma \ref{twist}, we have that
    \[
      H^1_{\phi,\Gamma}(\robba_L(\chi x^{-1}))\rightarrow H^1_{\phi,\Gamma}(\robba_L(\chi x^{-i}))
    \]
    is an isomorphism for all $i\geq 1$.
    This means that all the transition maps
    \[
      H^1_{\phi,\Gamma}(\robba_L(\chi x^{n-i}))\rightarrow H^1_{\phi,\Gamma}(\robba_L(\chi x^{n-i-1}))
    \]
    of the direct limit
    $\lim_i H^1_{\phi,\Gamma}(\robba_L(\chi x^{n-i}))$
    are isomorphisms for $i\geq n+1$.
    As a consequence, we have that
    \[
      H^1_{\phi,\Gamma}(\robba_L(\delta_1/\delta_2)[1/t])\cong H^1_{\phi,\Gamma}(\robba_L(\chi x^{-1})).
    \]
  \end{proof}

  Now let $\rho:G_{\Qp}\rightarrow GL_2(L)$ be a trianguline representation whose $\phigam$-module
  $\drig(\rho)$ is an extension of $\robba_L(\delta_2)$ by $\robba_L(\delta_1)$.
  Let us denote by $D'\in H^1_{\phi,\Gamma}(\robba_L(\chi x^{-1}))$ the extension corresponding to
  $\drig(\rho)[1/t]$ through the isomorphism (\ref{from 1/t to twist}).
  Notice that in particular we have that $\drig(\rho)[1/t]=D'[1/t]$.
  Let $a_1\coloneqq\rr{wt}(\delta_1)$ and $a_2\coloneqq\rr{wt}(\delta_2)$;
  since $\delta_1/\delta_2=\chi x^{n}$, we have $a_1-a_2=n+1$.
  If $a_1,a_2\in\Z$, let $\mu_1=a_1,\mu_2=a_2$; otherwise, let $\mu_1=0$ and $\mu_2=-n-1$.

  \begin{rem}
    Observe that the characters $\delta_1x^{-\mu_1},\delta_2x^{-\mu_2}$ either both have weights equal to 0
    or they have the same non-integer weight.
  \end{rem}

  The $\phigam$-module $D'$ is an extension of $\robba_L(\delta_2x^{-\mu_2})$ by $\robba_L(\delta_1 x^{-\mu_1})$.
  In fact we have that $-\mu_1+\mu_2=-\rr{wt}(\delta_1)+\rr{wt}(\delta_2)=-\rr{wt}(\delta_1/\delta_2)=-n-1$,
  thus $\delta_1x^{-\mu_1}/\delta_2x^{-\mu_2}=\delta_1/\delta_2\cdot x^{-n-1}=\chi x^{-1}$.
  Since $\delta_1/\delta_2\cdot x^{-n-1}=\chi x^{-1}\notin\T^+(L)$ by construction,
  we can apply Proposition \ref{2 dim T-} and get the following result.

  \begin{prop}
    \label{construction of X'}
    There exist
    \begin{itemize}
      \item
        a rigid space $X'$ over $L$ which is also a vector bundle of dimension 1 over an
        appropriate neighbourhood $\mathcal{U}'$ of $(\delta_1x^{-\mu_1},\delta_2 x^{-\mu_2})$ in $\T^\reg _2$;
      \item
        a family of $\phigam$-modules $\D'$ over $X'$;
      \item
        a map $(\widetilde\delta_1',\widetilde\delta_2'):X'\rightarrow\T^2$
    \end{itemize}
    such that there are
    \begin{enumerate}
      \item
        $x'\in X'$ such that $\D'\widehat\otimes k(x')=D'$ and
        $(\widetilde\delta_1',\widetilde\delta_2')\widehat\otimes k(x')=(\delta_1x^{-\mu_1},\delta_2 x^{-\mu_2})$;
      \item
        a Zariski open dense subset $U'\xhookrightarrow{j}X'$ such that $j^*\D'$
        has a filtration of sub-$\phigam$-modules with graded pieces $\robba_{U'}(\widetilde\delta_1')$
        and $\robba_{U'}(\widetilde\delta_2')$ and $j^*(\widetilde\delta_1',\widetilde\delta_2')\in\T^\reg_2(U').$
    \end{enumerate}
  \end{prop}

  \begin{rem}
    \label{perfect parametrization}
    After eventually shrinking $X'$ to a neighbourhood of $x'$, we may assume that $X'=\mathrm{Sp}(R')$
    is affinoid.
    Notice moreover that the rigid space $X'$ and the family of extensions $\D'$
    parametrizes all the extensions $H^1_{\phi,\Gamma}(\robba_{X'}(\widetilde\delta_1'/\widetilde\delta_2'))$.\
    In other words, the map
    \begin{align*}
      \Psi_Y:X'(Y)&\twoheadrightarrow H^1_{\phi,\Gamma}(\robba_Y(g^*\widetilde\delta_1'/g^*\widetilde\delta_2'))\\
      f&\mapsto f^*\D'
    \end{align*}
    from Theorem \ref{extensions} is actually a bijection for all $g:Y\rightarrow\mathcal{U}'$
    map of rigid spaces.
    This follows from Theorem \ref{extensions}: in fact $(\delta_1 x^{-\mu_1},\delta_2 x^{-\mu_2})\in \T^\reg_2(L)$
    and $\dim_L H^1_{\phi,\Gamma}(\robba_L(\delta_1/\delta_2\cdot x^{-n-1}))=1$.
    Therefore we can choose the neighbourhood $\mathcal{U}'$ inside $\T^\reg_2$ in such a way that
    $H^1_{\phi,\Gamma}(\robba_{\mathcal{U}'}(\widetilde\delta_1'/\widetilde\delta_2'))$ is free of rank 1
    over $\mathcal{O}_{\mathcal{U}'}(\mathcal{U}')$.
  \end{rem}

  By Theorem \ref{sub-phigam-module}, there exist $r_0\in p^\Q\cap[0,1)$ and
  a unique $\phigam$-module $D^{r_0}$ over $\robba_L^{r_0}$ such that $\drig(\rho)=D^{r_0}\otimes_{\robba^{r_0}_L}\robba_L$;
  moreover there is $r_1\in p^\Q\cap[0,1)$ such that
  the family of $\phigam$-modules $\D'$ can be already defined over $\robba_{X'}^{r_1}$.
  Let us denote by $r$ the maximum of $r_0$ and $r_1$, in such a way that both $\drig(\rho)$
  and $\D'$ are defined over $\robba_L^r$ and $\robba_{X'}^r$ respectively.
  In other words, there exist $D^r$ a $\phigam$-module over $\robba^r_L$ such that
  $\drig(\rho)=D^r\otimes_{\robba^r_L}\robba_L$ and a $\phigam$-module $\D'^{r}$
  over $\robba_{X'}^r$
  such that $\D'^{r}\otimes_{\robba_{X'}^r}\robba_{X'}=\D'$.
  Let $m(r)$ be the smallest positive integer such that $p^{m-1}(p-1)\geq r$ for all $m\geq m(r)$,
  so that all the points of the form $1-\zeta_{p^m}$ (where $\zeta_{p^m}$ is a $p^m$-th root of unity)
  lie in the open annulus $\mathbb{B}^r$.
  Let us fix $m\geq m(r)$ and let us assume that $L$ contains $K_m$
  (by Lemma \ref{enlarge L} we know we can always make this assumption).
  As a consequence of Theorem \ref{equivalenceBL}, the $\phigam$-module $D^r$ corresponds to the
  tuple $(D^r[1/t],D^r\otimes_{\robba_L^r}\prod_{K_m\hookrightarrow L}L\llbracket t\rrbracket,f)$ through
  the equivalence \ref{equivalenceBeauvilleLaszlo}, where
  \[
    f:D^r[1/t]\otimes_{\robba^r_L[1/t]}\prod_{K_m\hookrightarrow L}L((t))\xrightarrow{\sim}
  (D^r\otimes_{\robba_L^r}\prod_{K_m\hookrightarrow L}L\llbracket t\rrbracket)[1/t]
  \]
  is defined as in Theorem \ref{equivalenceBL}.
  Let us denote by $\Lambda$ the $\Gamma$-stable finite projective module
  $D^r\otimes_{\robba_L^r}\prod_{K_m\hookrightarrow L}L\llbracket t\rrbracket$.
  By Remark \ref{stable lattices}, we have
  \[
    \Lambda=\prod_{K_m\xhookrightarrow{\tau} L}\Lambda_\tau
  \]
  where $\Lambda_\tau$ is a $\Gamma_m$-stable lattice over $L\llbracket t\rrbracket$.
  On the other hand, given such a lattice $\Lambda_\tau$, it is possible to recover $\Lambda$.
  From now on, since the two notions are equivalent,
  we fix an embedding $\tau:K_m\hookrightarrow L$ and we will denote
  by $\Lambda$ the $\Gamma_m$-stable lattice $\Lambda_\tau$ over $L\llbracket t\rrbracket$.
  In the same way, let $\Lambda'$ be the $\Gamma_m$-stable lattice over
  $L\llbracket t\rrbracket$ corresponding to the $\tau$-factor of the lattice associated to $D'$.
  Let us denote by $\mathbb{L}'$ the $\Gamma_m$-stable lattice corresponding to the
  factor $\tau:K_m\hookrightarrow L$ of the lattice $\D'^r\otimes_{\robba^r_{X'}}
  \prod_{K_m\hookrightarrow L}R'\llbracket t\rrbracket$ associated to $\D'$.
  After choosing a basis of $\D'$,
  we choose a trivialization of $\mathbb{L}'$ and fix a basis of $\mathbb{L}'$
  respecting the triangulation of $\D'$.

  We then have the following result.

  \begin{lem}
    \label{relative positions of lattices}
    $\Lambda\in Y_{(\mu_1,\mu_2)}(L)$.
  \end{lem}

  \begin{proof}
    We have to show that $\mathrm{gr}^1(\Lambda)=t^{\mu_1}\mathrm{gr}^1(\Lambda')$ and
    $\mathrm{gr}^2(\Lambda)=t^{\mu_2}\cdot\mathrm{gr}^2(\Lambda')$.
    Let us denote by $\delta_1'\coloneqq \delta_1 x^{-\mu_1}$ and $\delta_2'\coloneqq\delta_2x^{-\mu_2}$.
    Since we chose a basis of $\Lambda'$ respecting the triangulation of $D'$, we have that
    $\mathrm{Fil}^\bullet(L((t))^2)$
    is induced by the filtration of $D'[1/t]$.
    This implies that
    \[
      \mathrm{gr}^1(\Lambda')=\robba_L^r(\delta_1')\otimes_{\robba_L^r,\tau}L\llbracket t\rrbracket
    \]
    and
    \[
      \mathrm{gr}^2(\Lambda')=\robba_L^r(\delta_2')\otimes_{\robba_L^r,\tau}L\llbracket t\rrbracket.
    \]
    Let us denote by $V$ the module $L((t))^2$ for simplicity.
    We have
    \begin{align*}
      \mathrm{gr}^1(\Lambda)&=\Lambda\cap\mathrm{Fil}^1(V)
      =\Lambda\cap\left (\robba_L^r(\delta_1')[1/t]\otimes_{\robba^r_L[1/t],\tau}L((t))\right )\\
      &=\robba_L^r(\delta_1)\otimes_{\robba_L^r,\tau}L\llbracket t\rrbracket
      =t^{\mu_1}\robba_L^r(\delta_1 x^{-\mu_1})\otimes_{\robba_L^r,\tau}L\llbracket t\rrbracket
      =t^{\mu_1}\cdot\mathrm{gr}^1(\Lambda')
    \end{align*}
    and
    \begin{align*}
      \mathrm{gr}^2(\Lambda)&=\Lambda/\mathrm{gr}^1(\Lambda)
      =\left (D^r\otimes_{\robba_L^r,\tau}L\llbracket t\rrbracket\right )
      /\left (\robba_L^r(\delta_1)\otimes_{\robba_L^r,\tau}L\llbracket t\rrbracket \right )\\
      &=\robba_L^r(\delta_2)\otimes_{\robba_L^r,\tau}L\llbracket t\rrbracket
      =\robba_L^r(\delta_2'\cdot x^{\mu_2})\otimes_{\robba_L^r,\tau}L\llbracket t\rrbracket\\
      &=t^{\mu_2}\cdot\robba^r_L(\delta_2')\otimes_{\robba_L^r,\tau}L\llbracket t\rrbracket
      =t^{\mu_2}\cdot\mathrm{gr}^2(\Lambda').
    \end{align*}
  \end{proof}

  Consider the space $X'\times\mathrm{Gr}_2^{\rig}$, where $X'$ is the rigid
  space defined in Proposition \ref{construction of X'};
  we will define a subspace consisting of $\Gamma_m$-stable lattices.
  First of all, let us fix a topological generator $\gamma$ of the group $\Gamma_m$.

  \begin{dfn}
    Let us define the map $\gamma\cdot:X'\times\mathrm{Gr}_2^{\rig}\rightarrow\mathrm{Gr}_2^{\rig}$ in the following way:
    for any affinoid $L$-algebra $A$ and for any $(y',\Lambda_A)\in(X'\times\mathrm{Gr}_2^{\rig})(A)$, let
    \[
      \gamma\cdot(y',\Lambda_A)\coloneqq \gamma\cdot\Lambda_A,
    \]
    where the action of $\gamma$ on $\Lambda_A$ is induced from the
    action of $\Gamma_m$ on $y'^*\D'$ through the isomorphism
    \[
      (y'^*\D'^r)[1/t]\otimes_{\robba_A^r[1/t],\tau}A((t))
      \cong \Lambda_A[1/t].
    \]
    Notice that $\gamma\cdot\Lambda_A\in\Gr_2(A)$, because $\gamma$ induces an automorphism on
  $(y'^*\D'^r)[1/t]$, hence the map is well-defined.

  \end{dfn}

  \begin{thm}
    \label{construction of X}
    There exist
    \begin{itemize}
      \item
      an affinoid rigid space $X=\mathrm{Sp}(R)$ over $L$;
      \item
      a family of triangulable $\phigam$-modules $\D$ over $X$;
      \item
      a map $(\widetilde\delta_1,\widetilde\delta_2):X\rightarrow\T^2$ which consists of parameters for $\D$
    \end{itemize}
    such that:
    \begin{enumerate}
      \item
      there is $x\in X$ such that $\D\widehat\otimes k(x)=\drig(\rho)$ and
      $(\widetilde\delta_1,\widetilde\delta_2)\widehat\otimes k(x)=(\delta_1,\delta_2)$;
      \item
      $(\widetilde\delta_1,\widetilde\delta_2)=(\widetilde\delta_1'x^{\mu_1},\widetilde\delta_2'x^{\mu_2})$.
    \end{enumerate}
  \end{thm}

  \begin{proof}
    Let us define the ind-rigid space $Y$ to be the fiber product
    \begin{equation}
      \label{def of Y}
      \begin{tikzcd}
        Y\coloneqq X'\times\Gr_2\times_{\Gr_2\times\Gr_2}\Gr_2\ar[r]\ar[d, "p_2"] & \mathrm{Gr}_2^{\rig}\ar[d, "\delta"]\\
        X'\times \Gr_2\ar[r,"(\gamma\cdot{,}\mathrm{pr})"] & \Gr_2\times\Gr_2
      \end{tikzcd}
    \end{equation}
    where the rightmost map $\delta$ is the diagonal map.
    Notice that for any point $((y',\Lambda_A),\Lambda'_A)\in Y(A)$, we have
    $(\gamma\cdot(y',\Lambda_A),\Lambda_A)=(\Lambda_A',\Lambda_A')$.
    Since $\Lambda$ is stable under the action of $\Gamma_m$ induced by the point $x'$,
    we have that $x\coloneqq((x',\Lambda),\Lambda)\in Y(L)$.
    We define $X$ to be the fiber product
    \[
      X\coloneqq Y\times_{\Gr_2}Y_{(\mu_1,\mu_2)}^\rig;
    \]
    we have then that $x\in X(L)$ because of Lemma \ref{relative positions of lattices} and,
    eventually restricting $X$, we can assume that $X$ is an affinoid neighbourhood $\mathrm{Sp}(R)$ of $x$.
    Consider the pullback $p_2^*(\D',\mathbb{L})$ to $X$ (where $\mathbb{L}$ is the universal lattice on $\Gr_2$)
    and the isomorphism of $\prod_{K_m\hookrightarrow L}R(( t))$-modules
    \[
      \widetilde{f}:p_2^*\D'^r[1/t]\otimes_{\robba_R^r[1/t]}\prod_{K_m\hookrightarrow L}R((t))
      \xrightarrow{\sim}\left (\prod_n g_m^n\cdot p_2^*\mathbb{L}\right )[1/t]
    \]
    defined as in Theorem \ref{equivalenceBL}, where $g_m$ is a generator of $\mathrm{Gal}(K_m/\Q_p)$
    and $n$ ranges over $\{0,\dots,|\mathrm{Gal}(K_m/\Q_p)|-1\}$.
    Let us denote by $\D$ the sheaf
    \[
      \left (p_2^*\D', p_2^*\mathbb{L},\widetilde{f}\right )
    \]
    on $X$, which can be seen as a sheaf of $\phigam$-modules thanks to the equivalence \ref{equivalenceBeauvilleLaszlo};
    then we have that $\D\widehat\otimes k(x)=(D^r[1/t],\Lambda,f)$.
    We have therefore constructed a space $X$ with a sheaf $\D$ specializing at some point $x\in X$ to
    $(D^r[1/t],\Lambda,f)$, which corresponds to the $\phigam$-module $\drig(\rho)$ through
    the equivalence \ref{equivalenceBeauvilleLaszlo}.
    Moreover the family of $\phigam$-modules $\D$ is obviously already defined over $\robba_X^r$,
    as $\D'$ is;
    hence there is a family of $\phigam$-modules $\D^r$ over $\robba_X^r$ such that
    $\D=\D^r\otimes_{\robba_X^r}\robba_X$.
    Notice that the sheaf of $\phigam$-modules $\D$ comes with a filtration of sub-$\phigam$-modules,
    which is the one induced by the filtration of sub-$\phigam$-modules of $\D'$.
    More precisely, for any affinoid $L$-algebra $A$ and for $y=((y',\Lambda_A),\Lambda_A)\in X(A)$, we have
    \[
      \mathrm{Fil}^i(y^*\D)=y^*\D\cap\mathrm{Fil}^i(y'^*\D'[1/t]).
    \]
    Recall that $\D'[1/t]=\D[1/t]$ by construction, so the intersection above makes sense.
    Moreover notice that $\mathrm{Fil}^i(y^*\D)$ is stable under $\phi$ and $\Gamma$,
    since both $y^*\D$ and $\mathrm{Fil}^i(y'^*\D'[1/t])$ are.
    Finally we have that $\Fil^i(y^*\D)$ is locally free by Lemma \ref{Y_mu gives filtrations}.
    Therefore $\mathrm{Fil}^\bullet(y^*\D)$ provides a triangulation of $\D$
    and hence it gives parameters $(\widetilde\delta_1,\widetilde\delta_2)\in\T^2(X)$ for the family $\D$.
    The last thing left to show is that $(\widetilde\delta_1,\widetilde\delta_2)=(\widetilde\delta_1'x^{\mu_1},\widetilde\delta_2'x^{\mu_2})$.
    We claim that
    \[
      \mathrm{gr}^1(p_2^*\mathbb{L})=\robba_R^r(\widetilde\delta_1)\otimes_{\robba_R^r,\tau}R\llbracket t\rrbracket
    \]
    and
    \[
      \mathrm{gr}^2(p_2^*\mathbb{L})=\robba_R^r(\widetilde\delta_2)\otimes_{\robba_R^r,\tau}R\llbracket t\rrbracket.
    \]
    Let $q:X\rightarrow X'$ be the projection map; since we fixed a basis
    of the standard lattice $\mathbb{L}'$ over $R'$ respecting
    the triangulation of $\D'$, we have that
    \begin{align*}
      \mathrm{Fil}^1(R((t))^2)& =q^*\mathrm{Fil}^1(R'((t))^2)=
      q^*(\robba_{R'}^r(\widetilde\delta_1')[1/t]\otimes_{\robba_{R'}^r[1/t],\tau}R'((t)))\\
      &=\robba_R^r(\widetilde\delta_1')[1/t]\otimes_{\robba_R^r[1/t],\tau}R((t)).
    \end{align*}
    Thus
    \begin{align*}
      \mathrm{gr}^1(p_2^*\mathbb{L})&=\mathrm{Fil}^1(p_2^*\mathbb{L})=
      p_2^*\mathbb{L}\cap \mathrm{Fil}^1(R((t))^2)\\
      &=\left (\D^r\otimes_{\robba_R^r,\tau}R\llbracket t\rrbracket\right )\cap
      \left (\robba_R^r(\widetilde\delta_1')[1/t]\otimes_{\robba_R^r[1/t],\tau}R((t))\right )\\
      &=\robba_R^r(\widetilde\delta_1)\otimes_{\robba_R^r,\tau}R\llbracket t\rrbracket
    \end{align*}
    and
    \begin{align*}
      \mathrm{gr}^2(p_2^*\mathbb{L})&= p_2^*\mathbb{L}/\mathrm{Fil}^1(p_2^*\mathbb{L})
      =\left (\D^r\otimes_{\robba_R^r,\tau}R\llbracket t\rrbracket\right )
      /\left (\robba_R^r(\widetilde\delta_1)\otimes_{\robba_R^r,\tau}R\llbracket t\rrbracket\right )\\
      &\cong\robba_R^r(\widetilde\delta_2)\otimes_{\robba_R^r,\tau}R\llbracket t\rrbracket.
    \end{align*}
    Restricting ourselves to lattices in $Y_{(\mu_1,\mu_2)}$ as in the definition of the space $X$, we have
    \[
      \robba_R^r(\widetilde\delta_1)\otimes_{\robba_R^r,\tau}R\llbracket t\rrbracket=
      t^{\mu_1}\robba_R^r(\widetilde\delta_1')\otimes_{\robba_R^r,\tau}R\llbracket t\rrbracket
      =\robba_R^r(x^{\mu_1}\widetilde\delta_1')\otimes_{\robba_R^r,\tau}R\llbracket t\rrbracket
    \]
    and
    \begin{align*}
      \robba_R^r(\widetilde\delta_2)\otimes_{\robba_R^r,\tau}R\llbracket t\rrbracket&=
      t^{\mu_2}\cdot \robba_R^r(\widetilde\delta_2')\otimes_{\robba_R^r,\tau}R\llbracket t\rrbracket
      =\robba_R^r(x^{\mu_2}\widetilde\delta_2')\otimes_{\robba_R^r,\tau}R\llbracket t\rrbracket,
    \end{align*}
    which concludes the proof.
  \end{proof}

  We have then shown that there is a rigid space $X$ and a family of $\phigam$-modules
  $\D$ such that $\D\widehat\otimes k(x)\cong\drig(\rho)$ at some point $x\in X(L)$.
  The only thing left to show is that the sheaf of $\phigam$-modules
  $\D$ is regular in a dense subspace of a neighbourhood of the point $x\in X(L)$.
  In order to do so, we must argue in a different way from section 4, since the rigid space $X$ is
  not a vector bundle over an appropriate neighbourhood $\mathcal{U}$ of $(\delta_1,\delta_2)$ in $\T^2$.
  As a consequence, we cannot conclude that $(\widetilde\delta_1,\widetilde\delta_2)^{-1}(\mathcal{U}\cap\Treg_2)$
  is dense in $X$, as the map $X\rightarrow\T^2$ is not necessarily open.

\subsection{Density of the regular locus}
  \label{density of regular locus in dim 2}

  To shorten the notation, let us denote by $G$ the map
  $(\widetilde\delta_1,\widetilde\delta_2):X\rightarrow\T^2$ sending a point $y$ of $X$ to
  the parameters associated to the filtration of $y^*\D$.
  Moreover, let
  \[
    H:\T^2\rightarrow\W^2,
    (\widetilde\eta_1,\widetilde\eta_2)\mapsto(\widetilde\eta_{1|\mathbb{Z}_p^\times},\widetilde\eta_{2|\mathbb{Z}_p^\times})
  \]
  be the map sending continuous characters of $\Q_p^\times$ to their restriction to $\mathbb{Z}_p^\times$.
  Finally, Let
  \begin{equation}
    \label{map F}
    F:X\xrightarrow G\T^2\xrightarrow{H}\W^2
  \end{equation}
  be the composition of the maps $G$ and $H$.

  \begin{lem}
    \label{fiber dimension of G}
    If $(\widetilde\eta_1,\widetilde\eta_2)\in\T^2(K)\cap\im (G)$ for some extension $K$ of $L$, then
    \[
      \dim(G^{-1}(\widetilde\eta_1,\widetilde\eta_2))=\dim_K H^1_{\phi,\Gamma}(\robba_K(\widetilde\eta_1/\widetilde\eta_2)).
    \]
  \end{lem}

  \begin{proof}
    We have a map
    \begin{equation}
      \label{fiber of G}
      G^{-1}(\widetilde\eta_1,\widetilde\eta_2)
      =\{y\in X(K)\colon \D\widehat\otimes k(y)\in H^1_{\phi,\Gamma}(\robba_K(\widetilde\eta_1/\widetilde\eta_2))\}
      \rightarrow
      H^1_{\phi,\Gamma}(\robba_K(\widetilde\eta_1/\widetilde\eta_2))
    \end{equation}
    via $y\mapsto\D\widehat\otimes k(y)$.
    In order to show the identity of dimensions, we must prove that the map is bijective.
    We start with surjectivity: let $E\in H^1_{\phi,\Gamma}(\robba_K(\widetilde\eta_1/\widetilde\eta_2))$,
    then we have that (by construction of $X'$ and of the family $\D'$) there is $y'\in X'(K)$ such that
    $(\D'\widehat\otimes k(y'))[1/t]=E[1/t]\in H^1_{\phi,\Gamma}(\robba_K(\widetilde\eta_1/\widetilde\eta_2\cdot x^{-n-1}))$.
    We still need to show that the lattice $\Lambda_y$ corresponding to the factor $\tau:K_m\hookrightarrow L$
    of the $\prod_{K_m\hookrightarrow L}K\llbracket t\rrbracket$-module
    associated to $E$ has the right graded pieces; in other words we want
    \[
      \Lambda_y\coloneqq E^r\otimes_{\robba_K^r,\tau}K\llbracket t\rrbracket \in Y_{(\mu_1,\mu_2)}(K),
    \]
    where $E^r$ is the $\phigam$-module over $\robba_K^r$ such that $E=E^r\otimes_{\robba_K^r}\robba_K$.
    This $\phigam$-module exists, since the whole family $\D'$ is defined over $\robba_{X'}^r$.
    We have that
    \begin{align*}
      \mathrm{Fil}^1(K((t))^2)&=\mathrm{Fil}^1(\mathbb{L}'[1/t])\otimes k(y')
      =(\robba_{R'}^r(\widetilde\delta_1')[1/t]\otimes_{\robba_{R'}^r[1/t],\tau}R'((t)))\otimes k(y')\\
      &=\robba_K^r(\widetilde\eta_1x^{-\mu_1})[1/t]\otimes_{\robba^r_K[1/t],\tau}K ((t)),
    \end{align*}
    so
    \begin{align*}
      \mathrm{gr}^1(\Lambda_y)&=\Lambda_y\cap \mathrm{Fil}^1(K((t))^2)\\
      &=\left (E^r\otimes_{\robba_K^r,\tau}K\llbracket t\rrbracket\right )
      \cap \left (\robba_K^r(\widetilde\eta_1 x^{-\mu_1})[1/t]\otimes_{\robba^r_K[1/t],\tau}K ((t))\right )\\
      &=\robba_K^r(\widetilde\eta_1)\otimes_{\robba^r_K,\tau}K \llbracket t\rrbracket
      =t^{\mu_1}\mathrm{gr}^1(\mathbb{L}'\otimes k(y'))
    \end{align*}
    and
    \begin{align*}
      \mathrm{gr}^2(\Lambda_y)&=\left (E^r\otimes_{\robba_K^r,\tau}K\llbracket t\rrbracket\right )
      /\left (\robba_K^r(\widetilde\eta_1)\otimes_{\robba^r_K,\tau}K \llbracket t\rrbracket \right )\\
      &=\robba_K^r(\widetilde\eta_2)\otimes_{\robba^r_K,\tau}K \llbracket t\rrbracket
      =t^{\mu_2}\robba_K^r(\widetilde\eta_2\cdot x^{-\mu_2})\otimes_{\robba^r_K,\tau}K \llbracket t\rrbracket\\
      &=t^{\mu_2}\mathrm{gr}^2(\mathbb{L}'\otimes k(y')).
    \end{align*}
    Hence we have $y\coloneqq(y',\Lambda_y)\in X(K)$ and $\D\widehat\otimes k(y)=E$.
    This proves the surjectivity of the above map.
    Now suppose by contradiction that $\D\widehat\otimes k(y_1)
    =\D\widehat\otimes k(y_2)=E\in H^1_{\phi,\Gamma}(\robba_K(\widetilde\eta_1/\widetilde\eta_2))$
    for some $y_1,y_2\in X(K)$.
    In particular we have $y_1=(y_1',\Lambda_y)$ and $y_2=(y_2',\Lambda_y)$ for some $y_1',y_2'\in X'(K)$
    such that $(\D'\widehat\otimes k(y_1'))[1/t]=(\D'\widehat\otimes k(y_2'))[1/t]=E[1/t]$
    and $\Lambda_y\coloneqq E^r\otimes_{\robba_K^r,\tau}K\llbracket t\rrbracket$.
    But since the map
    \begin{align*}
      \Psi_K:X'(K)&\rightarrow H^1_{\phi,\Gamma}(\robba_K(\widetilde\eta_1/\widetilde\eta_2\cdot x^{-n-1}))\\
      f&\mapsto f^*\D'
    \end{align*}
    is bijective (as explained in Remark \ref{perfect parametrization}), we have $y_1'=y_2'$.
    This shows injectivity and concludes the proof.
  \end{proof}

  \begin{thm}
    \label{F equidimensional 3}
    The map
    \[
      F:X\xrightarrow{G}\T^2\xrightarrow{H}\W^2
    \]
    has equidimensional fibers of dimension 3.
  \end{thm}

  \begin{proof}
    Let $(\eta_1,\eta_2)\in\W^2(K)\cap\imF$ and let us study the fiber $F^{-1}(\eta_1,\eta_2)$.
    We have that
    \[
      \dim H^{-1}(\eta_1,\eta_2)=2,
    \]
    which is easy to see thanks to the identification $\T\cong\W\times\mathbb{G}_m^\an$.
    Now we want to consider two different cases:
    \begin{enumerate}
      \item
        If $\eta_1/\eta_2\notin x^{\N_{>0}}$, then notice that $H^{-1}(\eta_1,\eta_2)\cap\mathrm{im}(G)\subset\T^\reg_2(K)$:
        if $\eta_1/\eta_2\notin x^\Z$, then it is obvious that $H^{-1}(\eta_1,\eta_2)\subset\T^\reg_2(K)$;
        on the other hand, if $\eta_1/\eta_2=x^{-n}$ for some $n\in\N$, then
        notice that
        \[
          \{(\widetilde\eta_1,\widetilde\eta_2)\in H^{-1}(\eta_1,\eta_2)\colon \widetilde\eta_1/\widetilde\eta_2= x^{-n}\}
          \cap \mathrm{im}(G)=\emptyset:
        \]
        in fact we chose $\mathcal{U}'\subset \T^\reg_2$ and we have
        \[
          \mathrm{im}(G)=\mathcal{U}=
          \{(\widetilde\delta_1',\widetilde\delta_2'\cdot x^{-n-1})\colon (\widetilde\delta_1',\widetilde\delta_2')\in\mathcal{U}'\}.
        \]
        It is then easy to see that the above implies that
        \[
          \mathrm{im}(G)\cap\{(\widetilde\eta_1,\widetilde\eta_2)\colon
          \widetilde\eta_1/\widetilde\eta_2\in x^{-\mathbb{N}}\}=\emptyset.
        \]
        Let $(\widetilde\eta_1,\widetilde\eta_2)\in H^{-1}(\eta_1,\eta_2)\cap\mathrm{im}(G)$;
        by Lemma \ref{fiber dimension of G}, we have
        \[
          \dim G^{-1}(\widetilde\eta_1,\widetilde\eta_2)=\dim_K H^1_{\phi,\Gamma}(\robba_K(\widetilde\eta_1/\widetilde\eta_2))=1.
        \]
        The last identity follows from the fact that $(\widetilde\eta_1,\widetilde\eta_2)\in\T^\reg_2(K)$.
        Therefore,
        \[
          \dim F^{-1}(\eta_1,\eta_2)=3.
        \]
      \item
        In the case in which $\eta_1/\eta_2=x^n$ for some $n\in\N_{>0}$,
        it is convenient to partition the space $H^{-1}(\eta_1,\eta_2)$ in the following way:
        \[
          H^{-1}(\eta_1,\eta_2)=(H^{-1}(\eta_1,\eta_2)\cap\T^\reg_2)\dot\cup
          \{(\widetilde\eta_1,\widetilde\eta_2)\in H^{-1}(\eta_1,\eta_2)\colon \widetilde\eta_1/\widetilde\eta_2=\chi x^{n-1}\}.
        \]
        Let us denote by $\mathcal{A}$ and $\mathcal{B}$ respectively the first and second subspace.
        \begin{itemize}
          \item
            We have that, as $\T^\reg_2$ is open in $\T^2$
            \[
              \dim\mathcal{A}=\dim H^{-1}(\eta_1,\eta_2)=2.
            \]
            Let $(\widetilde\eta_1,\widetilde\eta_2)\in \mathcal{A}\cap\mathrm{im}(G)$;
            \[
              \dim G^{-1}(\widetilde\eta_1,\widetilde\eta_2)=\dim_K H^1_{\phi,\Gamma}(\robba_K(\widetilde\eta_1/\widetilde\eta_2))=1.
            \]
          \item
            For a point $(\widetilde\eta_1,\widetilde\eta_2)\in\mathcal{B}\cap\mathrm{im}(G)$, we have
            \[
              \dim G^{-1}(\widetilde\eta_1,\widetilde\eta_2)=\dim_K H^1_{\phi,\Gamma}(\robba_K(\widetilde\eta_1/\widetilde\eta_2))=2.
            \]
            Finally, (through the identification $\T\cong\W\times\mathbb{G}_m^\an$) we have
            \begin{align*}
              \dim\mathcal{B}&=\dim(\{(\widetilde\eta_1,\widetilde\eta_2)\in H^{-1}(\eta_1,\eta_2)
              \colon \widetilde\eta_1/\widetilde\eta_2=\chi x^{n-1}\})\\
              &=\dim(\{(\eta_1,\eta_2,\widetilde\eta_1(p),\widetilde\eta_2(p))\in\W^2\times\Gm^{\an,2}
              \colon \widetilde\eta_1(p)/\widetilde\eta_2(p)=p^{n-1}\})\\
              &=1.
            \end{align*}
        \end{itemize}
        Therefore we have
        \[
          \dim F^{-1}(\eta_1,\eta_2)=\dim\mathcal{A}+1=\dim\mathcal{B}+2=3.
        \]
    \end{enumerate}
  \end{proof}

  \begin{rem}
    \label{not flatness and dim X'}
    \begin{itemize}
      \item[(i)]
      From the proof of the previous Theorem, we can already notce that the fibers of the
      map $G$ are not equidimensional, hence in particular $G$ is not flat.
      \item[(ii)]
      Moreover we have that
      \[
        \dim (F^{-1}(\W^\reg_2))=3+\dim\W^\reg_2=5=\dim\mathcal{U}'+1=\dim X',
      \]
      since $X'$ is a vector bundle of dimension $\dim_LH^1_{\phi,\Gamma}(\robba_L(\delta_1/\delta_2'))=1$
      over $\mathcal{U}'$, as proved in Proposition \ref{construction of X'}.
    \end{itemize}
  \end{rem}

  From now on we will study the geometry of the space $X$, in order to conclude from the
  bound of the fiber dimension that $F^{-1}(\W_2^\reg)$ is dense in $X$.

  \begin{lem}
    \label{same Sen weights}
    Let $K$ be an extension of $L$ and let $A$ be an Artin local algebra with residue field $K$.
    Let $\Lambda$ be a free $K\llbracket t\rrbracket$-module endowed with an action of $\Gamma_m$
    and let $\Lambda_A$ be a free $A\llbracket t\rrbracket$-module endowed with an action of $\Gamma_m$
    such that $\Lambda_A\otimes_AK=\Lambda$.
    Then we have that the Sen weights of $\Lambda$ (viewed as a representation with $L$-coefficients)
    and $\Lambda_A$ are the same.
  \end{lem}

  \begin{proof}
    For simplicity let us denote by $\overline{\Lambda}_A$ the $A$-module $\Lambda_A/t\Lambda_A$
    and we show that it is a successive extension of $\overline{\Lambda}\coloneqq \Lambda/t\Lambda$.
    In fact, if we let $\mathfrak{m}_A$ denote the maximal ideal of $A$, then the module
    $\overline{\Lambda}_A$ has a filtration
    \[
      \overline{\Lambda}_A\supset\mathfrak{m}_A\overline{\Lambda}_A\supset\dots\supset\mathfrak{m}_A^{l}\overline{\Lambda}_A\supset 0,
    \]
    where $l\in\N$ such that $\mathfrak{m}_A^{l+1}=0$, which exists as $A$ is Artinian.
    For $1\leq i\leq l$, let $k\coloneqq\dim_K\mathfrak{m}_A^i/\mathfrak{m}_A^{i+1}$ and let us fix a basis
    $\{m_1,\dots,m_k\}$ of $\mathfrak{m}_A^i/\mathfrak{m}_A^{i+1}$ over $K$.
    Then we have an isomorphism of $K$-vector spaces
    \begin{equation}
      \label{graded pieces}
        \mathfrak{m}_A^i\overline{\Lambda}_A/\mathfrak{m}_A^{i+1}\overline{\Lambda}_A
        \cong \mathfrak{m}_A^i/\mathfrak{m}_A^{i+1}\otimes_A\overline{\Lambda}_A\xrightarrow{\sim}\bigoplus_{j=1}^k\overline{\Lambda},
    \end{equation}
    where for every $m=\sum_{i=1}^ka_im_i\in \mathfrak{m}_A^i/\mathfrak{m}_A^{i+1}$
    and every $l\in\overline{\Lambda}_A$, we have
    \[
      m\otimes l\mapsto (a_i(l+\mathfrak{m}_A)+t\Lambda)_{i=1,\dots,k}.
    \]
    Observe moreover that $\overline{\Lambda}$ and $\overline{\Lambda}_A$ inherit
    the $\Gamma_m$-action of $\Lambda$ and $\Lambda_A$ and that the isomorphism (\ref{graded pieces})
    is $\Gamma_m$-equivariant.
    Notice furthermore that the $\Gamma_m$-action on $\overline{\Lambda}_A$ preserves the filtration,
    thus we have that, once we have chosen an appropriate basis of $\overline{\Lambda}_A$,
    the matrix of the Sen operator is an upper triangular block matrix with diagonal blocks the
    matrix of the Sen operator on $\overline{\Lambda}$.
    This concludes the proof, as the Sen weights are the eigenvalues of the Sen operator.
  \end{proof}

  \begin{rem}
    \label{same Sen weights for W^+}
    The same proof as Lemma \ref{same Sen weights} shows that if $W^+$ is a $B_{\dR,K}$-representation
    of $\absGal$ and $W_A^+$ is a $\BdRA^+$-representation of $\absGal$ such that
    $W_A^+\otimes_AK=W^+$, then $W_A^+$ and $W^+$ have the same Sen weights.
  \end{rem}

  The following results will be useful also in the $n$-dimensional case for $n>2$, so we write
  the proof for the $n$-dimensional case directly here.
  Note that in the following, in case $n = 2$ we have $X_n$ is simply $X$,
  $X_n'$ is $X'$, $\D_n$ is $\D$ and $\D_n'$ is $\D'$.

  \begin{rem}
    \label{standard lattices are the same}
    Let us consider the stalk
    $\Sp(\widehat{\mcal O}_{X_n,x_n})$
    and for simplicity let us denote by $\D_n$ and $\D_n'$ the pullback of $\D_n$ and $\D_n'$ to the stalk.
    Let $\widehat \Lambda\coloneqq \D_n^r\otimes_{\robba_R^r,\tau}R\llbracket t\rrbracket$
    and $\widehat\Lambda'\coloneqq \D_n'^r\otimes_{\robba_R^r,\tau}R\llbracket t\rrbracket$, where
    $R\coloneqq \widehat{\mcal O}_{X_n,x_n}$.
    Observe that $\widehat{\mcal O}_{X_n,x_n} \varprojlim_m\mcal O_{X_n,x_n}/\mathfrak{m}^m$,
    where $\mathfrak{m}$ is the maximal ideal associated to the point $x_n$.
    Let us consider the maps
    \[
      \Spec(O_{X_n,x_n}/\mathfrak{m}^m)\xrightarrow{q_m}\Spf(\widehat{O}_{X_n,x_n}).
    \]
    The ring $O_{X_n,x_n}/\mathfrak{m}^m$ is a finite-dimensional Artin local $\Qp$-algebra for all $m\geq1$
    hence by Lemma \ref{same Sen weights}, $q_m^*\widehat{\Lambda}'$
    has the same Sen weights as $\Lambda'$.
    Notice that by construction, if $a_1$ and $a_2$ are two Sen weights of $q_m^*\widehat\Lambda'$
    and $a_1\equiv a_2\pmod\Z$, then $a_1=a_2$; moreover observe that if $a_1\in\Z$, then $a_1=0$.
    Therefore if we let $\mcal A$ be the set of Sen weights of $\Lambda'$,
    we have that the lattice $q_m^*\widehat{\Lambda}'$ decomposes as explained in Remark \ref{identify Lambda0/t with DpdR}
    as
    \[
      q_m^*\widehat{\Lambda}'=\bigoplus_{a\in\mcal A}(q_m^*\widehat{\Lambda}')_a,
    \]
    where $(q_m^*\widehat{\Lambda}')_a$ is a $\Gamma_m$-stable sublattice for all $a\in\mcal A$.
    In this way we obtain a decomposition of $\widehat{\Lambda}'$:
    \[
      \widehat{\Lambda}'=\bigoplus_{a\in\mcal A}\widehat{\Lambda}'_a,
    \]
    where $\widehat{\Lambda}'_a$ is a $\Gamma_m$-stable sublattice.
    Moreover $\widehat\Lambda$ and $\widehat\Lambda'$ have the same
    Sen weights modulo $\Z$ by Lemma \ref{sen weights are the same mod Z}.
    Thus in exactly the same manner we obtain a decomposition
    \[
      \widehat\Lambda=\bigoplus_{a\in\mcal A}\widehat\Lambda_{a},
    \]
    where each $\widehat\Lambda_a$ is a $\Gamma_m$-stable $R\llbracket t\rrbracket$-lattice.
    For each $a\in\mcal A$, let $\delta_a$ be a character of weight $a$ if $a\neq0$
    and let $\delta_a$ be the trivial character when $a=0$.
    Then after choosing an appropriate basis of $\widehat\Lambda_a'(\delta_a^{-1})$,
    the matrix of $\gamma-1$ is upper triangular with zero on the diagonal, where $\gamma$ is a
    topological generator of $\Gamma_m$.
    As a consequence, we have that
    $(\widehat\Lambda_a'[1/t](\delta_a^{-1}))^{(\gamma-1)\text{-nil}}\otimes_RR\llbracket t\rrbracket=\widehat\Lambda_a'(\delta_a^{-1})$.
    In other words, $\widehat\Lambda'$ is the lattice $\Lambda_0$ constructed in Remark \ref{identify Lambda0/t with DpdR}
    and the basis $\{y_{\bar a,1},\dots,y_{\bar a,n_{\bar a}\colon \bar a\in\bar{\mcal A}}\}$
    defined in the Remark respects the triangulation of $\D'$ (up to reordering), hence we can
    choose it as the standard basis of $\widehat\Lambda'$.
    Moreover $\widehat\Lambda_{ a}[1/t]=\widehat\Lambda_{a}'[1/t]$ for all $\bar a\in\mcal A$.
  \end{rem}

  From now on, let us denote by $Y_\mu$ the space $Y_{(\mu_1,\dots,\mu_n)}$ for simplicity of notation.
  Let us decompose $\Lambda$ and $\Lambda'$ as in Remark \ref{identify Lambda0/t with DpdR}:
  let $\mcal{A}$ be the set of Sen weights of the $\Gamma_m$-stable lattice $\Lambda'$
  (which are the same as the mod $\Z$ classes of the Sen weights of $\Lambda$
  by Lemma \ref{sen weights are the same mod Z}); then we have
  \[
    \begin{array}{l l l}
      \Lambda=\bigoplus_{ a\in\mcal A}\Lambda_{ a}
      & \text{ and }
      & \Lambda'=\bigoplus_{ a\in{\mcal A}}\Lambda'_{ a},
    \end{array}
  \]
  where $\Lambda_{ a}$ and $\Lambda_{ a}'$ are $\Gamma_m$-stable sub-$L\llbracket t\rrbracket$-modules
  of the same rank $n_{a}$. After fixing $\Lambda_a'$ to be the standard lattice of $\Gr_{n_a}(L)$, let
  $\mu_{a}\in X_*(T_{n_a})^+$ such that $\Lambda_{ a}\in\rr{Gr}_{n_{ a}}^{\mu_{ a},\rig}$.
  Moreover let $\lambda_a\in X_*(T_{n_a})$ such that
  \[
    Y_\mu\times_{\rr{Gr}_n}\prod_{a\in\mcal A}GL_{n_a}/P_{\mu_a}=
    \prod_{a\in\mcal A}Y_{\lambda_a}\times_{\rr{Gr}_{n_a}}GL_{n_a}/P_{\mu_a}
  \]
  (which is possible by Lemma \ref{decomposition of Y_mu})
  and let $w_a$ in the Weyl group of $GL_{n_a}$ such that $\lambda_a=w_aw_0\mu_a$ for all $a\in\mcal A$.

  \begin{lem}
    \label{def of maps}
    There exists a commutative diagram
    \begin{equation}
      \label{wanted fiber product}
      \begin{tikzcd}
        \Spf(\widehat{\mathcal{O}}_{X_n,x_n})\arrow[r, hookrightarrow]\arrow[d]
        &\Spf(\widehat{R})\arrow[d]\\
        \Spf(\widehat{S})\arrow[r, hookrightarrow]
        &\Spf(\widehat{T}),
      \end{tikzcd}
    \end{equation}
    where
    \begin{align*}
      R&\coloneqq \mathcal{O}_{X_n'\times Y_{\mu}^\rig \times_{\Gr_n}\prod_{ a\in{\mcal A}}(GL_{n_{ a}}/P_{\mu_{a}})^\rig,
      (x_n,(\pi_{\mu_{ a}}(\Lambda_{ a}))_{ a\in\mcal A})},\\
      S&\coloneqq \mathcal{O}_{\prod_{ a\in{\mcal A}}V_{w_{a}}^{\mu_a},
      (N_{ a},\pi_{\mu_{ a}}(\Lambda_{ a}))_{a\in{\mcal A}}},\\
      T&\coloneqq \mathcal{O}_{\prod_{ a\in{\mcal A}}\mathfrak{b}_{n_{ a}}^\rig
      \times (B_nw_{ a}P_{\mu_{ a}}/P_{\mu_{ a}})^\rig,
      (N_{a},\pi_{\mu_{a}}(\Lambda_{a}))_{ a\in{\mcal A}}}.
    \end{align*}
    In the above $\mathfrak b_{n_{ a}}$ is the Lie algebra of the Borel of upper
    triangular matrices $B_{n_a}$, while the space $V_{w_a}^{\mu_a}$ is defined in (\ref{def of V_w}).
    Finally $N_{ a}$ is the matrix of the nilpotent endomorphism acting on
    \[
      D_{\pdR,a}(\WdR(x_n^*\D_n[1/t])_{ a})
    \]
    for all $a\in{\mcal A}$, after having chosen a basis of $D_{\pdR,a}(\WdR(x_n^*\D_n[1/t])_{ a})$.
    Observe that if we choose a basis of $D_{\pdR,a}(\WdR(x_n^*\D_n[1/t])_{ a})$
    respecting the triangulation of $x_n^*\D_n$,
    then we have that the matrix $N_a$ is upper triangular; hence we have $N_a\in\mathfrak{b}_{n_a}$
    for all $a\in\mcal A$.
  \end{lem}

  \begin{proof}
    Recall that for a rigid space $X$ and a point $x\in X(K)$, we have that $\Spf(\widehat{\mathcal{O}}_{X,x})$
    prorepresents the functor
    \begin{align*}
      \{\text{Artin local }K\text{-algebras with residue field }K\}&\rightarrow \underline{Sets}
    \end{align*}
    sending an Artin local $K$-algebra $A$ with maximal ideal $\mathfrak{m}$
    to the set of morphisms $\Sp(A)\xrightarrow{f} X$ such that
    the diagram
    \[
      \begin{tikzcd}
        \Sp(A)\arrow[r, "f"] &X\\
        \Sp(A/\mathfrak{m})=\Sp(K)\arrow[u]\arrow[ur,"x"]
      \end{tikzcd}
    \]
    commutes (in fact if $K$ is a complete non-Archimedean valued field, then Artin
    local $K$-algebras are affinoid $K$-algebras).
    Let $\widehat\Lambda$ be the lattice associated to $\hat x_n^*\D_n$, where
    $\hat x_n:\Sp(\widehat{\mathcal{O}}_{X_n,x_n})\rightarrow X_n$.
    As explained in Remark \ref{standard lattices are the same}, we get a decomposition
    \[
      \widehat\Lambda=\bigoplus_{a\in\mcal A}\widehat\Lambda_a,
    \]
    where $\mcal A$ is the set of Sen weights of $\Lambda'$ and $\widehat{\Lambda}_a$ is a
    $\Gamma_m$-stable lattice of rank $n_a$.
    First of all notice that the map
    \[
      (\pi)_{a\in{\mcal A}}:\Sp(\widehat{\mathcal{O}}_{X_n,x_n})\rightarrow
      \prod_{ a\in{\mcal A}}\coprod_{\lambda_{a}\in X_*(T)^+}(GL_{n_{a}}/P_{\lambda_{a}})^\rig
    \]
    given by $(\pi(\widehat\Lambda_{a}))_{ a\in{\mcal A}}$
    factors through $\prod_{ a\in{\mcal A}}(GL_{n_{a}}/P_{\mu_{a}})^\rig$
    (where $(GL_{n_{ a}}/P_{\lambda_{ a}})^\rig$
    is the rigid analytification of the $L$-scheme of locally finite type $GL_{n_{a}}/P_{\lambda_{a}}$):
    in fact we have that
    $(\pi)_{a\in{\mcal A}}(x_n)\in \prod_{ a\in{\mcal A}}(GL_{n_{ a}}/P_{\mu_{ a}})^\rig$
    and $\Sp(\widehat{\mathcal{O}}_{X_n,x_n})$ is connected.
    The upper horizontal map of the diagram (\ref{wanted fiber product}) is defined then by the commutative diagram
    \begin{equation}
      \label{diagram4}
      \begin{tikzcd}
        \Sp(\widehat{\mathcal{O}}_{X_n,x_n})\ar{r}{\hat x_n}\ar{d}{(\pi)_{ a\in{\mcal A}}}
        &X_n'\times Y_{\mu}^\rig\ar{d}{j}\\
        \prod_{ a\in{\mcal A}}(GL_{n_{a}}/P_{\mu_{a}})^\rig\ar{r}{\oplus_{a\in{\mcal A}}\sigma_{\mu_{a}}}
        &\Gr_n,
      \end{tikzcd}
    \end{equation}
    which gives a map
    \begin{equation}
      \label{def of (id,pi_mu)}
      (\hat{x}_n,(\pi)_{a\in{\mcal A}}):\Sp(\widehat{\mathcal{O}}_{X_n,x_n})\rightarrow
      X_n'\times Y_{\mu}^\rig\times_{\Gr_n}\prod_{ a\in{\mcal A}}(GL_{n_{ a}}/P_{\mu_{a}})^\rig.
    \end{equation}
    Here by $\oplus_{ a\in{\mcal A}}\sigma_{\mu_{ a}}$ we mean the map
    \[
      (\Fil^\bullet(\Lambda'_a/t\Lambda'_a))_{a\in{\mcal A}}\mapsto
      \bigoplus_{ a\in{\mcal A}}\sigma_{\mu_{a}}(\Fil^\bullet(\Lambda'_a/t\Lambda'_a)).
    \]
    In fact the diagram (\ref{diagram4}) is commutative:
    since $\widehat \Lambda_{a}$ is $\Gamma_m$-stable,
    by Lemma \ref{Gamma-stable lattices and Gamma-stable filtrations} we have that
    \[
      \widehat\Lambda=(\oplus_{ a\in{\mcal A}}\sigma_{\mu_{a}})
      \circ ((\pi)_{a\in{\mcal A}})(\widehat \Lambda),
    \]
    thus the diagram commutes.
    It is obvious that
    \[
      (\hat x_n,(\pi)_{ a\in{\mcal A}})\circ x_n=(x_n,(\pi_{\mu_{a}}(\Lambda_{a}))_{ a\in{\mcal A}}).
    \]
    The lower horizontal map of (\ref{wanted fiber product}) is defined by the composition
    \[
      i:\Sp(\widehat{S})\rightarrow
      \prod_{ a\in{\mcal A}}V_{w_a}^{\mu_a}\hookrightarrow
      \prod_{ a\in{\mcal A}}\mathfrak{b}_{n_{a}}^\rig\times(GL_{n_{a}}/P_{\mu_{a}})^\rig,
    \]
    where $\mathfrak{b}_{n_{a}}^\rig$ is the rigid analytification of $\mathfrak{b}_{n_{a}}$;
    it is again obvious that
    \[
      i\circ (N_{ a},\pi_{\mu_{a}}(\Lambda_a))_{a\in{\mcal A}}
      =(N_{a},\pi_{\mu_{a}}(\Lambda_{ a}))_{ a\in{\mcal A}}.
    \]
    As for the left vertical arrow,
    let us consider again the maps
    \begin{align*}
      \Sp(\mcal{O}_{X_n,x_n}/\mathfrak{m}^m)\xrightarrow{q_m}\Sp(\widehat{\mathcal{O}}_{X_n,x_n})\xrightarrow{\hat x_n}X_n
    \end{align*}
    and the pullbacks $q_m^*\hat x_n^*\D_n$.
    We have that $\mcal{O}_{X_n,x_n}/\mathfrak{m}^m$ is a finite dimensional $\Qp$-algebra and by
    Remark \ref{same Sen weights for W^+} we have that $\WdR^+(q_m^*\hat x_n^*\D_n)$ has the same Sen weights as $\WdR^+(x_n^*\D_n)$.
    So there are maps
    \[
      (\widetilde{N}_{m, a},\pi_{\mu_a})_{ a\in{\mcal A}}:
      \Sp(\mathcal{O}_{X_n,x_n}/\mathfrak{m}^m)\rightarrow \prod_{ a\in{\mcal A}}\Fil^{\st,\rig}_{\mu_{\bar a}},
    \]
    which are determined by
    the $\widetilde N_{m, a}$-stable filtrations associated to the lattice
    $q_m^*\widehat \Lambda$
    through the bijection of Theorem \ref{final bijection}.
    As we have already remarked, after choosing an appropriate basis of
    $D_{\pdR,a}(\WdR(q_m^*x_n^*\D_n[1/t])_{ a})$, we have that $\widetilde{N}_{m,a}\in\mathfrak{b}_{n_a}$.
    Moreover we have that (as $(q_m^*\widehat{\Lambda})_a$ is $\Gamma_m$-stable)
    \[
      q_m^*\widehat{\Lambda}=\bigoplus_{a \in\mcal A}(q_m^*\widehat{\Lambda})_a
    =\bigoplus_{a \in\mcal A}\sigma_{\mu_a}(\pi_{\mu_a}(q_m^*\widehat{\Lambda})_a),
    \]
    thus
    \[
      q_m^*\widehat{\Lambda}\in Y_\mu^\rig\times_{\Gr_n}\prod_{a\in\mcal A}GL_{n_a}/P_{\mu_a}
      =\prod_{a\in\mcal A}Y_{\lambda_a}\times_{\Gr_{n_a}}GL_{n_a}/P_{\mu_a}
    \]
    by Lemma \ref{decomposition of Y_mu}.
    In particular, this implies that $(q_m^*\widehat{\Lambda})_a\in Y_{\lambda_a}$.
    By Lemma \ref{preimage in Y_mu of filtrations}, we have
    \[
      \pi_{\mu_a}((q_m^*\widehat{\Lambda})_a)\in B_{n_a}w_aP_{\mu_a}/P_{\mu_a}
    \]
    for $w_a\in W$ such that $w_aw_0\mu_a=\lambda_a$.
    This implies that the maps $(\widetilde{N}_{m, a},\pi_{\mu_a})_{ a\in{\mcal A}}$
    factor through $\prod_{a\in\mcal A}V_w^{\mu_a}$.
    Then these maps give a map
    \[
      (\widetilde{N}_{ a},\pi_{\mu_{ a}})_{a\in{\mcal A}}:\Sp(\widehat{\mathcal{O}}_{X_n,x_n})
      \rightarrow \prod_{ a\in{\mcal A}}V_w^{\mu_{a}},
    \]
    which is the desired left vertical arrow of the diagram (\ref{wanted fiber product}).
    By the definition of $N_{ a}$ and by Theorem \ref{final bijection}, we have in fact that
    $(\widetilde{N}_{a},\pi_{\mu_{ a}})_{a\in{\mcal A}}\circ x_n
    =(N_{a},\pi_{\mu_{a}}(\Lambda_a))_{a\in{\mcal A}}$.
    Observe that
    \[
      \widehat{R}=\varprojlim_m\left( R/\mathfrak{n}^m\right),
    \]
    where $\mathfrak{n}$ is the maximal ideal corresponding to the point
    $(x_n,(\pi_{\mu_{a}}(\Lambda_a))_{a\in{\mcal A}})$.
    Observe that by Lemma \ref{decomposition of Y_mu} and Lemma \ref{preimage in Y_mu of filtrations}, we have
    \begin{align*}
      X_n'\times Y_{\mu}^\rig \times_{\Gr_n}\prod_{ a\in{\mcal A}}(GL_{n_{a}}/P_{\mu_{a}})
      &=X_n'\times\prod_{ a\in{\mcal A}}Y_{\lambda_a}^\rig\times_{\Gr_{n_a}}(GL_{n_{a}}/P_{\mu_{a}})^\rig\\
      &=X_n'\times\prod_{ a\in{\mcal A}}(B_{n_a}w_aP_{\mu_a}/P_{\mu_a})^\rig.
    \end{align*}
    Let us consider the maps
    \begin{align*}
      \Sp(R/(\mathfrak{n})^m)
      \xrightarrow{p_m}\Sp(\widehat{R})\xrightarrow{q}&
      X_n'\times\prod_{ a\in{\mcal A}}(B_{n_a}w_aP_{\mu_a}/P_{\mu_a})^\rig\xrightarrow{p}X_n'
    \end{align*}
    and the pullback $p_m^*q^*p^*\D_n'$.
    The ring $R/(\mathfrak{n})^m$
    is a finite dimensional $\Qp$-algebra for all $m\in\N$, and by Remark \ref{same Sen weights for W^+}
    and Lemma \ref{sen weights are the same mod Z},
    we have that $\WdR^+(p_m^*q^*p^*\D_n')$ has the same mod $\Z$ Sen weights as $\WdR^+(x_n^*\D_n)$.
    After choosing a basis of $D_{\pdR, a}(\WdR(p_m^*q^*p^*\D_n'[1/t])_{ a})$
    respecting the triangulation of $p_m^*q^*p^*\D_n'$ for all $ a\in{\mcal A}$,
    we have maps
    \[
      \Sp(R/(\mathfrak{n})^m)
      \rightarrow \prod_{a\in{\mcal A}}\mathfrak{b}_{n_{ a}}^\rig\times(Bw_aP_{\mu_a}/P_{\mu_{a}})^\rig
    \]
    given by the matrices
    of the nilpotent endomorphisms $(N_{m,a})_{a\in{\mcal A}}$ acting on
    $D_{\pdR,a}(\WdR(p_m^*q^*p^*\D_n'[1/t])_a)$,
    which is then an element of $\prod_{ a\in{\mcal A}}\mathfrak{b}_{n_{a}}^\rig$;
    these maps give us the desired map
    \[
      ((\widetilde{N}_{a})_{ a\in{\mcal A}},\rr{id}):\Sp(\widehat{R})
      \rightarrow \prod_{a\in{\mcal A}}\mathfrak{b}_{n_{a}}^\rig\times(Bw_aP_{\mu_a}/P_{\mu_{a}})^\rig:
    \]
    in fact we have
    $((\widetilde{N}_{ a})_{a\in{\mcal A}},\rr{id})\circ(x_n,(\pi_{\mu_{a}}(\Lambda_a))_{a\in{\mcal A}})
    =(N_{a},\pi_{\mu_{a}}(\Lambda_a))_{a\in{\mcal A}}$.
    Now that we have defined the maps, it is still left to show that the diagram (\ref{wanted fiber product}) commutes.
    But this is very easy to see, in fact once we have chosen a basis of
    \[
      \varprojlim_mD_{\pdR, a}(\WdR(q_m^*\hat x_n^*\D_n[1/t])_{a})
    \]
    respecting the traingulation of $q_m^*\hat x_n^*\D_n$ for all $a \in{\mcal A}$,
    we have that both compositions are given by
    \[
      \left (\varprojlim_m\widetilde N_{m,a},\pi_{\mu_{a}}(\widehat \Lambda_{a})\right )_{ a\in{\mcal A}}.
    \]
  \end{proof}

  \begin{prop}
    \label{stalk is closed immersion}
    Let us assume the same conditions as Lemma \ref{def of maps}.
    We have that
    \[
      \Spf(\widehat{\mathcal{O}}_{X_n,x_n})\hookrightarrow
      \Spf(\widehat{\mathcal{O}}_{X_n',x_n'})\times\Spf(\widehat{\mathcal{O}}_{\prod_{ a\in{\mcal A}}B_{n_a}w_aP_{\mu_a}/P_{\mu_{ a}},
      (\Lambda, (\pi_{\mu_{a}}(\Lambda_{a}))_{a\in{\mcal A}})})
    \]
    is a closed immersion and $\Spf(\widehat{\mathcal{O}}_{X_n,x_n})$ is cut out by
    $\sum_{ a\in{\mcal A}}\dim B_{n_a}w_aP_{\mu_a}/P_{\mu_{a}}$
    equations.
  \end{prop}

  \begin{proof}
    First of all we prove that the diagram
    (\ref{wanted fiber product})
    is a fiber product.
    The functor of points of a $K$-formal scheme is determined by its values on
    Artin local $K$-algebras with residue field $K$,
    so let us assume that we have an Artin local $K$-algebra $A$ with residue field $K$
    and maps
    \begin{align*}
      \Spf(A)&\xrightarrow{f_1}\Spf(\widehat{R})\\
      \Spf(A)&\xrightarrow{f_2} \Spf(\widehat{S})
    \end{align*}
    such that
    \[
      \begin{tikzcd}
        \Spf(A) \arrow[r,"f_1"]\arrow[d,"f_2"]
        &\Spf(\widehat{R})\arrow{d}{\left((\widetilde{N}_{\bar a})_{\bar a\in\bar{\mcal A}},\rr{id}\right)}\\
        \Spf(\widehat{S})\arrow[hookrightarrow]{r}{i}
        &\Spf(\widehat{T})
      \end{tikzcd}
    \]
    commutes.
    By the fact that the formal scheme of the completion of the stalk prorepresents the functor
    of lifts, we have that there exists a morphism
    \[
      \Sp(A)\xrightarrow{f_2} \prod_{a\in{\mcal A}}V_{w_a}^{\mu_{a},\rig}
    \]
    inducing $\Sp(A)\xrightarrow{f_2}\Sp(\widehat{S})$
    such that the diagram
    \begin{equation}
      \label{diagram}
      \begin{tikzcd}
        \Sp(L)\ar[dr,"a"]\ar{drr}{\left(x_n,(\pi_{\mu_{ a}}(\Lambda_ a))_{a\in{\mcal A}}\right)}
        \ar{ddr}[swap]{\left(N_{a},\pi_{\mu_{ a}}(\Lambda_ a)\right)_{a\in{\mcal A}}}
        & &\\
        &\Sp(A)\ar[r,"f_1"]\ar[d,"f_2"]
        & \Sp(\widehat{R})\ar{d}{\left((\widetilde{N}_{ a})_{a\in{\mcal A}},\rr{id}\right)}\\
        & \prod_{ a\in{\mcal A}}V_{w_a}^{\mu_{ a},\rig}\ar[hookrightarrow]{r}{i}
        & \prod_{ a\in{\mcal A}}\mathfrak{b}_{n_{a}}^\rig\times(B_{n_a}w_aP_{\mu_a}/P_{\mu_{ a}})^\rig
      \end{tikzcd}
    \end{equation}
    commutes.
    The map
    \[
      X_n'\times Y_{\mu}^\rig\times_{\Gr_n}\prod_{ a\in{\mcal A}}(GL_{n_{a}}/P_{\mu_{ a}})^\rig
      \xrightarrow{(\rr{id},\oplus_{a\in{\mcal A}}\sigma_{\mu_{ a}})}
      X_n'\times Y_{\mu}^\rig\times_{\Gr_n}\rr{Gr}^{\rig}_n
    \]
    induces a map
    \begin{align*}
      u:\Sp(A)&\xrightarrow{f_1}
      \Sp\left (\widehat{R}\right )
      \xrightarrow{(\mathrm{id},\oplus_{a\in{\mcal A}}\sigma_{\mu_{ a}})}
      \Sp\left (\widehat{\mcal O}_{X_n'\times Y_{\mu}^\rig\times_{\Gr_n}\rr{Gr}^{\rig}_n,(x_n,\Lambda)}\right ),
    \end{align*}
    in fact since $\Lambda_a$ is $\Gamma_m$-stable,
    we have that
    $\Lambda=\bigoplus_{a\in{\mcal A}}\sigma_{\mu_{ a}}(\pi_{\mu_{ a}}(\Lambda_{a}))$
    by Lemma \ref{Gamma-stable lattices and Gamma-stable filtrations}.
    We show that
    \[
      \im(u)\subseteq \Sp\left(\widehat{\mcal O}_{X_n,x_n}\right):
    \]
    since $X_n$ is the closed locus
    of $X_n'\times Y_{\mu}^\rig$ of $\Gamma_m$-stable lattices, we just need to show that
    $u^*\mathbb{L}$ is $\Gamma_m$-stable, where $\mathbb{L}$ is the universal lattice above
    $Y_{\mu}^\rig$.
    Observe that $u^*\mathbb{L}=\bigoplus_{a\in{\mcal A}}u^*\mathbb L_{ a}$
    for some smaller $A\llbracket t\rrbracket$-modules $u^*\mathbb L_{ a}$,
    as $u$ factors through $\Sp(\widehat{R})$ by definition:
    in fact we have that by definition of $R$, $(\mathrm{id},\oplus_{ a\in{\mcal A}}\sigma_{\mu_{ a}})^*\mathbb L$
    decomposes into a direct sum of sub-$\widehat{R}\llbracket t\rrbracket$-modules.
    Thanks to Theorem \ref{final bijection}, showing that $u^*\mathbb{L}$ is $\Gamma_m$-stable is equivalent
    to showing that for all $a\in{\mcal A}$ we have that
    $\pi_{\mu_{a}}(u^*\mathbb{L}_{ a})$ is a stable filtration under the nilpotent
    endomorphism $N_{ a}$ given by the $\mathbb{G}_a$-representation associated to the $A$-module
    $D_{\pdR, a}(\WdR(u^*\D_n'[1/t])_{ a})$.
    In other words, we want to show that the image of the map
    \begin{align*}
      \Sp(A)&\xrightarrow{u}
      \Sp\left (\widehat{\mcal O}_{X_n'\times Y_{\mu}^\rig\times_{\Gr_n}\rr{Gr}^{\rig}_n,(x_n,\Lambda)}\right )
      \xrightarrow{(\rr{id},(\pi_{\mu_{ a}})_{a\in{\mcal A}})}
      \Sp\left (\widehat{R}\right )\\
      &\xrightarrow{((\widetilde{N_{a}})_{ a\in{\mcal A}},\rr{id})}
      \prod_{a\in{\mcal A}}\mathfrak{b}_{n_{a}}^\rig\times(B_{n_a}w_aP_{\mu_a}/P_{\mu_{a}})^\rig
    \end{align*}
    is contained in $\prod_{a\in{\mcal A}}V_{w_a}^{\mu_{a},\rig}$.
    Now notice that
    \begin{align*}
      ((\widetilde{N_{ a}})_{a\in{\mcal A}},\rr{id})\circ(\rr{id},(\pi_{\mu_{a}})_{a\in{\mcal A}})\circ u&
      = ((\widetilde{N_{a}})_{a\in{\mcal A}},\rr{id})\circ(\rr{id},(\pi_{\mu_{a}})_{a\in{\mcal A}})
      \circ (\mathrm{id},\oplus_{a\in{\mcal A}}\sigma_{\mu_{a}})\circ f_1\\
      &= ((\widetilde{N}_{a})_{a\in{\mcal A}},\mathrm{id})\circ  f_1
      = i\circ f_2,
    \end{align*}
    where the second to last equality is due to the fact that
    $(\pi_{\mu_{a}})_{a\in{\mcal A}}\circ(\oplus_{a\in{\mcal A}}\sigma_{\mu_{a}})=\rr{id}$
    and the last equality is due to the commutativity of the diagram (\ref{diagram}).
    Hence we have proved that the map $u$
    factors through $\Sp(\widehat{\mcal O}_{X_n,x_n})$.
    Now notice that the diagram
    \begin{equation}
      \label{diagram2}
      \begin{tikzcd}[row sep=large, column sep=large]
        \Sp(A)\ar[dr,"u"]\ar{drr}{f_1}\ar{ddr}[swap]{f_2}
        & &\\
        &\Sp\left(\widehat{\mcal O}_{X_n,x_n}\right)\ar{r}[swap]{\left(\rr{id},(\pi_{\mu_{ a}})_{ a\in{\mcal A}}\right)}
        \ar{d}{\left(\widetilde{N}_{ a},\pi_{\mu_{ a}}\right)_{ a\in{\mcal A}}}
        & \Sp(\widehat{R} )
        \ar{d}{\left((\widetilde{N}_{a})_{a\in{\mcal A}},\rr{id}\right)}\\
        & \prod_{ a\in{\mcal A}}V_{w_a}^{\mu_{ a},\rig}\ar[hookrightarrow]{r}{i}
        & \prod_{ a\in{\mcal A}}\mathfrak{b}_{n_{ a}}^\rig\times(B_{n_a}w_aP_{\mu_a}/P_{\mu_{ a}})^\rig
      \end{tikzcd}
    \end{equation}
    is commutative:
    \[
      (\rr{id},(\pi_{\mu_{a}})_{ a\in{\mcal A}})\circ u
      =(\rr{id},(\pi_{\mu_{ a}})_{ a\in{\mcal A}})\circ(\mathrm{id},\oplus_{ a\in{\mcal A}}\sigma_{\mu_{ a}})\circ f_1=f_1
    \]
    and
    \begin{align*}
      i\circ(\widetilde{N}_{ a},\pi_{\mu_{ a}})_{ a\in{\mcal A}}\circ u
      &=((\widetilde{N}_{ a})_{ a\in{\mcal A}},\rr{id})\circ(\rr{id},(\pi_{\mu_{ a}})_{ a\in{\mcal A}})
      \circ(\mathrm{id},\oplus_{ a\in{\mcal A}}\sigma_{\mu_{ a}})\circ f_1\\
      &=((\widetilde{N}_{ a})_{ a\in{\mcal A}},\rr{id})\circ f_1=i\circ f_2.
    \end{align*}
    Since $i$ is a closed immersion, we can conclude that
    $(\widetilde{N}_{ a},\pi_{\mu_{ a}})_{ a\in{\mcal A}}\circ u=f_2$.
    Finally, assume there exists $v:\Sp(A)\rightarrow \Sp(\widehat{\mcal O}_{X_n,x_n})$ such that
    the diagram (\ref{diagram2}) is commutative when we substitute $v$ to $u$.
    We will denote by $(u',\Lambda_u),(v',\Lambda_v)\in X_n(A)$ the maps
    corresponding to $u$ and $v$ respectively.
    Since $\Lambda_u$ is $\Gamma_m$-stable and we have $\Lambda_u\otimes_AL=\Lambda$,
    we have a decomposition
    $\Lambda_u=\bigoplus_{a\in{\mcal A}}\Lambda_{u,a}$ and
    similarly for $\Lambda_v$.
    Then we have
    \begin{align*}
      \left (u',\Lambda_u,(\pi_{\mu_{a}}(\Lambda_{u,a}))_{ a\in{\mcal A}}\right )
      &=(\rr{id},(\pi_{\mu_{ a}})_{ a\in{\mcal A}})\circ u=f_1
      =(\rr{id},(\pi_{\mu_{a}})_{ a\in{\mcal A}})\circ v\\
      &=\left(v',\Lambda_v,(\pi_{\mu_{a}}(\Lambda_{v,a}))_{ a\in{\mcal A}}\right),
    \end{align*}
    which implies that $u=(u',\Lambda_u)=(v',\Lambda_v)=v$.
    In particular we have that the diagram
    \begin{equation}
      \label{diagram3}
      \begin{tikzcd}
        \Sp\left(\widehat{\mcal O}_{X_n,x_n}\right)\ar{r}{\left(\rr{id},(\pi_{\mu_{a}})_{a\in{\mcal A}}\right)}
        \ar{d}{\left(\widetilde{N}_{ a},\pi_{\mu_{ a}}\right)_{ a\in\mcal A}}
        & \Sp(\widehat{R} )
        \ar{d}{\left((\widetilde{N}_{a})_{a\in{\mcal A}},\rr{id}\right)}\\
        \prod_{ a\in{\mcal A}}V_{w_a}^{\mu_{ a},\rig}\ar[hookrightarrow]{r}{i}
        & \prod_{ a\in{\mcal A}}\mathfrak{b}_{n_{a}}^\rig\times(Bw_aP_{\mu_a}/P_{\mu_{a}})^\rig
      \end{tikzcd}
    \end{equation}
    is a fiber product and hence
    the morphism
    \[
      (\rr{id},(\pi_{\mu_{a}})_{a\in{\mcal A}}):\Sp(\widehat{\mcal O}_{X_n,x_n})\rightarrow
      \Sp(\widehat{R})
    \]
    is a closed immersion.
    Finally, notice that
    \[
      \begin{tikzcd}
        \Sp(A)\ar{r}{u} &\Sp\left(\widehat{\mcal O}_{X_n,x_n}\right)\\
        \Sp(L)\ar{u}{a}\ar{ur}[swap]{x_n}
      \end{tikzcd}
    \]
    commutes: observe that
    \[
      (\rr{id}, (\pi_{\mu_{ a}})_{ a\in{\mcal A}})\circ u\circ a=f_1\circ a
      =(x_n,\pi_{\mu_{ a}}(\Lambda_a)_{ a\in{\mcal A}})=(\rr{id}, (\pi_{\mu_{a}})_{ a\in{\mcal A}})\circ x_n,
    \]
    where the first equality is given by the commutativity of diagram (\ref{diagram2}),
    the second equality is by commutativity of diagram (\ref{diagram}) and the last equality
    is by definition of the map $(\rr{id},(\pi_{\mu_{ a}})_{a\in{\mcal A}})$ as in (\ref{def of (id,pi_mu)}).
    By the fact that $(\rr{id},(\pi_{\mu_{ a}})_{ a\in{\mcal A}})$ is a closed immersion and hence it has a left inverse,
    we obtain that $u\circ a=x_n$.
    As a consequence, we have that there exists a unique map
    \[
      u:\Spf(A)\rightarrow \Spf(\widehat{\mcal O}_{X_n,x_n})
    \]
    such that the following diagram is commutative:
    \[
      \begin{tikzcd}
        \Spf(A) \arrow[drr,"f_1"]\arrow{ddr}[swap]{f_2}\ar{dr}{u} & &\\
        &\Spf(\widehat{\mcal O}_{X_n,x_n})\ar{r}\ar{d}
        &\Spf(\widehat{R})\arrow[d]\\
        &\Spf(\widehat{S})\arrow[r, hookrightarrow]
        &\Spf(\widehat{T}).
      \end{tikzcd}
    \]
    This means that the diagram (\ref{wanted fiber product}) is a fiber product.
    Notice that, since $B_{n_a}w_aP_{\mu_a}/P_{\mu_{a}}$ is an affine space, we have
    \[
      \Sp(L\langle T_1,\dots, T_M\rangle)=\prod_{ a\in{\mcal A}}(B_{n_a}w_aP_{\mu_a}/P_{\mu_{a}})^\rig,
    \]
    where $M\coloneqq\sum_{a\in{\mcal A}}\dim B_{n_a}w_aP_{\mu_a}/P_{\mu_{a}}$.
    Recall that $X_n'=\Sp(L\langle x_1,\dots, x_d\rangle)$, since it is a
    vector bundle of dimension
    \[
      d \coloneqq \sum_{i=2}^n\dim_L H^1_{\phi,\Gamma}(D_{i-1}'(\delta_i^{-1}x^{\mu_i}))
    \]
    over an affinoid neighbourhood $\widetilde{\mathcal{U}}'$
    of $(\delta_1x^{-\mu_1},\dots,\delta_nx^{-\mu_n})$ in $\T^n$ ($d=1$ in case $n=2$).
    Thus
    \begin{align*}
      \widehat{R} &
      = L\langle x_1,\dots,x_d, T_1,\dots, T_M\rangle\sphat_{(x_n,(\pi_{\mu_{a}}(\Lambda_a))_{ a\in{\mcal A}})}\\
      &=L\llbracket x_1,\dots,x_d, T_1,\dots, T_M\rrbracket\\
      &=L\llbracket x_1,\dots, x_d\rrbracket\otimes_L L\llbracket T_1,\dots,T_M\rrbracket\\
      &= L\langle x_1,\dots,x_d\rangle\sphat_{x_n'}\widehat{\otimes}_L
      L\langle T_1,\dots, T_M\rangle\sphat_{(\pi_{\mu_{a}}(\Lambda_a))_{a\in{\mcal A}}}\\
      &= \widehat{\mathcal{O}}_{X_n',x_n'}\widehat{\otimes}_L
      \widehat{\mathcal{O}}_{\prod_{a\in{\mcal A}}Bw_aP_{\mu_a}/P_{\mu_{a}},
      (\pi_{\mu_{a}}(\Lambda_a))_{ a\in{\mcal A}}},
    \end{align*}
    where the completions are with respect to the subscript maximal ideals.
    Thus we have
    \begin{align*}
      &\Spf(\widehat{R})=
      \Spf(\widehat{\mathcal{O}}_{X_n',x_n'})\times\Spf(\widehat{\mathcal{O}}_{\prod_{ a\in{\mcal A}}B_{n_a}w_aP_{\mu_a}/P_{\mu_{a}},
      (\pi_{\mu_{a}}(\Lambda_a))_{\bar a\in\bar{\mcal A}}}).
    \end{align*}
    Recall by Remark \ref{right number of equations} that
    $V_{w_a}^{\mu_{a}}\hookrightarrow\mathfrak{b}_{n_a}\times B_{n_a}w_aP_{\mu_a}/P_{\mu_{a}}$ is cut out by
    exactly $\dim B_{n_a}w_aP_{\mu_a}/P_{\mu_{a}}$ equations.
    Finally by the fact that the square (\ref{wanted fiber product}) is a fiber product,
    it is then obvious that
    \[
      \Spf(\widehat{\mathcal{O}}_{X_n,x_n})\hookrightarrow
      \Spf(\widehat{\mathcal{O}}_{X_n',x_n'})\times\Spf(\widehat{\mathcal{O}}_{\prod_{a\in{\mcal A}}B_{n_a}w_aP_{\mu_a}/P_{\mu_{a}},
      (\pi_{\mu_{a}}(\Lambda_a))_{a\in{\mcal A}}})
    \]
    is a closed immersion and $\Spf(\widehat{\mathcal{O}}_{X_n,x_n})$ is cut out by
    $\sum_{a\in{\mcal A}}\dim B_{n_a}w_aP_{\mu_a}/P_{\mu_{a}}$ equations.
  \end{proof}

  We will show that the dimension of $\Spec(\widehat{\mathcal{O}}_{X_n,x_n})$ is
  at least the dimension of $\Spec(\widehat{\mathcal{O}}_{X_n',x_n'})$.

  \begin{cor}
    \label{X cut out by M equations}
    We have
    \begin{equation*}
      \dim \mathrm{Spec}(\widehat{\mathcal{O}}_{X_n,x_n})\geq\dim\mathrm{Spec}(\widehat{\mathcal{O}}_{X_n',x_n'}).
    \end{equation*}
  \end{cor}

  \begin{proof}
    Let $M\coloneqq\sum_{a\in{\mcal A}}\dim B_{n_a}w_{ a}P_{\mu_{ a}}/P_{\mu_{ a}}$.
    By Proposition \ref{stalk is closed immersion}, we have that there are
    $f_1,\dots,f_M\in\widehat{\mathcal{O}}_{X_n',x_n'}\llbracket T_1,\dots T_M\rrbracket$ such that
    \[
      \widehat{\mathcal{O}}_{X_n,x_n}=\widehat{\mathcal{O}}_{X_n',x_n'}\llbracket T_1,\dots,T_M\rrbracket/(f_1,\dots,f_M).
    \]
    This immediately concludes the proof.
  \end{proof}

  We make use of the following result from \cite[Exercise 3.20]{hartshorne2013algebraic}.

  \begin{thm}
    \label{thmHartshorne}
    Let $X$ be an integral scheme of finite type over a field $k$.
    For any closed point $P\in X$, $\dim X=\dim \mathcal{O}_{X,P}$.
  \end{thm}

  \begin{rem}
    \label{Hartshorne_for_irreducible}
    Notice first of all that since the dimension of a scheme $X$ doesn't depend on the reducible structure
    of $\mathcal{O}_X$, we have the above result for irreducible schemes more in general.
    Moreover observe that the notion of "dimension" for a scheme (resp. a rigid space) is
    equal to the dimension of the ring associated to an open affine (resp. affinoid) subspace
    of the scheme (resp. rigid space).
    As a consequence, Theorem \ref{thmHartshorne} is valid also for irreducible rigid spaces.
  \end{rem}

  Now let us go back to considering only the 2-dimensional case.
  We show that any irreducible component of $X$ has dimension at least the dimension
  of the regular locus.

    \begin{thm}
      \label{dimZ}
      Let $Z$ be an irreducible component of $X_n$ containing the point $x_n$.
      Then
      \[
        \dim Z\geq\dim (F^{-1}(\W_n^\reg)).
      \]
    \end{thm}

    \begin{proof}
      Since $Z$ is irreducible, we get
      \[
        \dim Z=\dim\mathrm{Spec}(\mathcal{O}_{Z,z})
      \]
      for any point $z\in Z(K)$ by Remark \ref{Hartshorne_for_irreducible}.
      For a Noetherian local ring $A$, we have $\dim\mathrm{Spec}(A)=\dim\mathrm{Spec}(\widehat{A})$,
      thus
      \[
        \dim\mathrm{Spec}(\mathcal{O}_{X_n,x_n})=\dim\mathrm{Spec}(\widehat{\mathcal{O}}_{X_n,x_n}).
      \]
      Furthermore we have that
      \[
        \widehat{\mathcal{O}}_{X_n,x_n}=\widehat{\mathcal{O}}_{X_n',x_n'}
        \llbracket T_1,\dots,T_M\rrbracket/(f_1,\dots,f_M)
      \]
      by Proposition \ref{stalk is closed immersion}.
      Therefore
      \[
        \dim \Spec(\widehat{\mathcal{O}}_{Z,x_n})\geq\dim\Spec(\widehat{\mathcal{O}}_{X',x_n'}).
      \]
      Finally putting the above together, we get
      \begin{align*}
        \dim Z&=\dim\mathrm{Spec}(\mathcal{O}_{Z,x_n})
        =\dim \mathrm{Spec}(\widehat{\mathcal{O}}_{Z,x_n})
        \geq\dim\mathrm{Spec}(\widehat{\mathcal{O}}_{X_n',x_n'})\\
        &=\dim \mathrm{Spec}(\mathcal{O}_{X_n',x_n'})= \dim(F^{-1}(\W_n^\reg)),
      \end{align*}
      where the last equality is given by  Remark \ref{Hartshorne_for_irreducible} and Remark \ref{Wreg and X' same dim}
      (Remark \ref{not flatness and dim X'} for dimension $n=2$).
    \end{proof}

    We are now back to considering the case of dimension 2.

    \begin{lem}
      \label{proof homeomorphism}
      Let $W$ be the intersection of $\W_2^\reg$ and $\im(F)$ in $\W^2$, let
      \[
        W'\coloneqq\{(\eta_1,\eta_2\cdot x^{n+1})\colon(\eta_1,\eta_2)\in W\}\subset H(\mathcal{U}')
      \]
      and let $T'$ be the intersection of $\mathcal{U}'$ and $H^{-1}(W')$ in $\T^2$.
      Recall the definition of $X'\coloneqq\mathcal{U}'\times\mathbb{A}_L^{1,\an}$ for
      some open neighbourhood $\mcal U'$ of $(\delta_1,\delta_2x^{\mu_2})$ in $\T^2$
      as in the proof of Proposition \ref{extensions}.
      The map
      \begin{equation}
        \label{homeomorphism}
        \pi:F^{-1}(W)\rightarrow T'\times\mathbb{A}_L^{1,\an}\subset\mathcal{U}'\times\mathbb{A}_L^{1,\an}=X'
      \end{equation}
      given by the restriction of the map $X\rightarrow X'$ to $F^{-1}(W)$ is a homeomorphism.
    \end{lem}

    \begin{proof}
      Let $z=(z',\Lambda_z)\in F^{-1}(W)$ with
      $z'=(\widetilde\eta_1',\widetilde\eta_2',a)\in\mathcal{U}'\times \mathbb{A}_L^{1,\an}= X'$.
      Then by Theorem \ref{construction of X}, we have
      that $z^*\D$ has parameters $(\widetilde\eta_1',\widetilde\eta_2'\cdot x^{-n-1})$.
      On the other hand $z\in F^{-1}(W)$, hence $H(\widetilde\eta_1',\widetilde\eta_2'\cdot x^{-n-1})\in W$,
      which proves that $H(\widetilde\eta_1',\widetilde\eta_2')\in W'$ and as a consequence
      $(\widetilde\eta_1',\widetilde\eta_2')\in T'$.
      This shows that the above map is well defined and has image in $T'\times\mathbb{A}_L^{1,\an}$.
      We will now show that the map $\pi$ is bijective.
      Notice that we have the following commutative diagram
      \[
        \begin{tikzcd}
          F^{-1}(W)\arrow[r, "G"]\arrow[d, "\pi"] & T\coloneqq H^{-1}(W)\cap\mathcal{U}\isoarrow{d}\\
          T'\times\mathbb{A}_L^{1,\an}\arrow[r] & T',
        \end{tikzcd}
      \]
      where the map $G:F^{-1}(W)\rightarrow T$ is the restriction
      of the map $G$ in (\ref{map F}) to $F^{-1}(W)$ and the right vertical map is just
      the twisting of characters.
      We will then show that
      \[
        \pi:F^{-1}(W)=\bigcup_{(\widetilde\eta_1,\widetilde\eta_2)\in T}G^{-1}(\widetilde\eta_1,\widetilde\eta_2)
        \rightarrow \bigcup_{(\widetilde\eta_1,\widetilde\eta_2 x^{n+1})\in T'}(\widetilde\eta_1,\widetilde\eta_2 x^{n+1})\times\mathbb{A}_L^{1,\an}
        =T'\times\mathbb{A}_L^{1,\an}
      \]
      restricts to a bijection on the fibers
      \[
        G^{-1}(\widetilde\eta_1,\widetilde\eta_2)\xrightarrow{\pi}(\widetilde\eta_1,\widetilde\eta_2 x^{n+1})\times\mathbb{A}_L^{1,\an}
      \]
      for all $(\widetilde\eta_1,\widetilde\eta_2)\in T(K)$ for some $K$ extension of $L$.
      It is enough to notice that we have a commutative diagram
      \[
        \begin{tikzcd}
          G^{-1}(\widetilde\eta_1,\widetilde\eta_2)\arrow[r,"\pi"]\isoarrow{d, "(\ref{fiber of G})"}
          &(\widetilde\eta_1,\widetilde\eta_2 x^{n+1})\times\mathbb{A}_L^{1,\an}\isoarrow{d}\\
          H^1_{\phi,\Gamma}(\robba_K(\widetilde\eta_1/\widetilde\eta_2))\arrow[r, "\sim"]
          & H^1_{\phi,\Gamma}(\robba_K(\widetilde\eta_1/\widetilde\eta_2\cdot x^{-n-1})),
        \end{tikzcd}
      \]
      where the map (\ref{fiber of G}) is a bijection as proved in Lemma \ref{fiber dimension of G}
      and the notation $H^1_{\phi,\Gamma}(\robba_K(\widetilde\eta_1/\widetilde\eta_2))$ (resp.
      $H^1_{\phi,\Gamma}(\robba_K(\widetilde\eta_1/\widetilde\eta_2\cdot x^{-n-1}))$) indicates
      the rigid space associated to the $k(\widetilde{\underline{\eta}})$-vector space
      $H^1_{\phi,\Gamma}(\robba_K(\widetilde\eta_1/\widetilde\eta_2))$
      (resp. $H^1_{\phi,\Gamma}(\robba_K(\widetilde\eta_1/\widetilde\eta_2\cdot x^{-n-1}))$);
      the horizontal bijection is given by Lemma \ref{twist}
      (in fact $\mathrm{wt}(\widetilde\eta_1/\widetilde\eta_2)\notin\Z$, otherwise $H(\widetilde{\eta_1},\widetilde{\eta_2})\notin\W_2^\reg$);
      finally the rightmost vertical bijection is given by Remark \ref{perfect parametrization}.
      We have thus proved that $\pi$ is a bijection.
      We  moreover have that
      \begin{align*}
        F^{-1}(W)&=G^{-1}(H^{-1}(W)\cap\mathcal{U})=G^{-1}(T)\\
        &=X\times_{\mathcal{U}}T=X'\times\Gr_2\times_{\Gr_2\times\Gr_2}Y_{(0,-n-1)}\times_{\mathcal{U}}T\\
        & = \mathcal{U}'\times\mathbb{A}^{1,\an}_L\times\Gr_2\times_{\Gr_2\times\Gr_2}Y_{(0,-n-1)}\times_{\mathcal{U}}T\\
        & = \mathcal{U}'\times_{\mathcal{U}}T\times\mathbb{A}^{1,\an}_L\times\Gr_2\times_{\Gr_2\times\Gr_2}Y_{(0,-n-1)}\\
        & = T'\times\mathbb{A}^{1,\an}_L\times\Gr_2\times_{\Gr_2\times\Gr_2}Y_{(0,-n-1)}.
      \end{align*}
      Recall that for locally ringed spaces $X$, $Y$ over a field $k$, the natural projection $X\times_kY\rightarrow X$
      is open.
      In our case, this translates to the map
      \[
        \pi:T'\times\mathbb{A}_L^{1,\an}\times\Gr_2\times_{\Gr_2\times\Gr_2}Y_{(0,-n-1)}
      \rightarrow T'\times\mathbb{A}_L^{1,\an}
      \]
      being open.
      We have than shown that $\pi$ is continuous, open and a bijection, hence a homeomorphism.
    \end{proof}

    In the following we will make use of this result from \cite[Exercise 3.22]{hartshorne2013algebraic}.

    \begin{thm}
      \label{Hartshorne dimension fiber}
      Let $f:X\rightarrow Y$ be a dominant morphism of integral schemes of finite type over a field $k$.
      Let $e=\dim X-\dim Y$ be the relative dimension of $X$ over $Y$.
      For any point $y\in f(X)$ we have that every irreducible component of the fiber
      $X_y$ has dimension $\geq e$.
    \end{thm}

    \begin{rem}
      \label{Hartshorne dimension fiber for rigid spaces}
      For the same reasons as written in Remark \ref{Hartshorne_for_irreducible}, we have that
      Theorem \ref{Hartshorne dimension fiber} holds for irreducible rigid spaces.
    \end{rem}

    \begin{thm}
      \label{X is irreducible}
      Let $Z$ be an irreducible component of $X$ containing the point $x$.
      Then the space $Z$ is equal to the closure of $F^{-1}(\W_2^\reg)$;
      in particular, there is only one irreducible component containing the point $x$.
    \end{thm}

    \begin{proof}
      Let $W\coloneqq\W_2^\reg\cap\im(F)\subset\W^2$
      and let $Z'$ be the closure of $F^{-1}(W)$ in $X$.
      By Lemma \ref{proof homeomorphism}, we have that $F^{-1}(W)$ is irreducible,
      as $X'$ is obviously irreducible; as a consequence $Z'$ is irreducible, since
      closure of irreducible.
      Now we show that $Z'$ is an irreducible component of $X$: assume $\widetilde Z\supset Z'$
      is an irreducible component of $X$ containing $Z'$.
      Then $F^{-1}(W)$ is open in $Z'$ and $\widetilde{Z}$, so
      \[
        \dim\widetilde Z=\dim F^{-1}(W)=\dim Z'.
      \]
      This shows that $Z'=\widetilde Z$, since they are both irreducible of the same dimension.
      Now we prove that $Z'$ is equal to $Z$:
      assume  by contradiction that $Z\neq Z'$ and let
      \[
        V\coloneqq X\setminus\bigcup_{Z''\neq Z}Z''\subset Z
      \]
      where the union ranges over all the irreducible components $Z''$ of $X$ that are not $Z$.
      We have that $V$ is open and irreducible and so
      $F(V)$ is irreducible.
      Moreover we have that $F(V)\cap\W_2^\reg=\emptyset$:
      if $\underline{\eta}\in F(V)\cap\W_2^\reg$, then $Z'\cap V\neq\emptyset$,
      which is a contradiction by definition of $V$.
      Therefore we have that $F(V)$ is closed in $\W^2$ and
      \[
        \dim F(V)<\dim \W^2=2.
      \]
      Using Remark \ref{Hartshorne dimension fiber for rigid spaces}, notice that for $\underline{\eta}\in F(V)$,
      we have that every irreducible component of
      $F^{-1}(\underline{\eta})$ has dimension that is at least $\dim V-\dim F(V)$.
      Finally, recall that by Theorem \ref{dimZ},
      $\dim Z\geq \dim X'$; in particular $\dim V=\dim Z\geq \dim X'=\dim F^{-1}(\W^\reg_2)$
      by Remark \ref{not flatness and dim X'}.
      We then have
      \[
        \dim F^{-1}(\underline{\eta})\geq \dim V-\dim F(V)>\dim V-2\geq \dim F^{-1}(\W_2^\reg)-2=3.
      \]
      We have then a contradiction, since $\dim F^{-1}(\underline{\eta})=3$ for all
      $\underline{\eta}\in\im(F)$ by Theorem \ref{F equidimensional 3}.
    \end{proof}

    \begin{cor}
      The subspace $F^{-1}(\W^\reg_2)$ is dense in the irreducible component of $X$ containing the point $x$.
    \end{cor}

    \begin{proof}
      This is a direct consequence of Theorem \ref{X is irreducible}.
    \end{proof}

\section{Construction of family of extensions for trianguline \texorpdfstring{\\}{} representations of higher dimension}
\label{section 7}

We would like to generalize the strategy from section 5 in order for it to work
for trianguline representations of dimension $n>2$.
The goal is to prove the following Theorem.

\difficultresult*

We again want to use the following strategy:

\strategy*

Thus in the first section we will see how to construct such a space $X$.
More precisely, given a trianguline representation $\rho:\absGal\rightarrow GL_n(\mathcal{O}_L)$
with parameters $\underline{\delta}\in\T^n(L)$, we will show that there exist
\begin{itemize}
\item
a rigid space $X$ over $L$;
\item
a family of $\phigam$-modules $\D$ over $X$;
\item
a map $(\widetilde{\delta_1},\dots,\widetilde{\delta_n}):X\rightarrow\mathcal{T}^n$
\end{itemize}
such that there is a point $x\in X(L)$ such that
$\underline{\widetilde{\delta}}\widehat\otimes k(x)=\underline\delta$
and $\D\widehat{\otimes} k(x)=\drig(\rho)$.
In the second section we will show that such a family has regular parameters in a dense subspace
locally at $x$.

\subsection{Construction of space of extensions}

  Let $n\geq2$ and fix $\rho:\absGal\rightarrow GL_n(\mathcal{O}_L)$ a trianguline representation of
  parameters $(\delta_1,\dots,\delta_n)\in\T^n(L)$ such that $\delta_i/\delta_j\in\T^+(L)$ for some $i<j$.
  For simplicity, let us denote by $D_i$ the $i$-th piece of the filtration $\Fil^i(\drig(\rho))$
  of $\drig(\rho)$ giving rise to the parameters $(\delta_1,\dots,\delta_n)$.
  Let $\bar\rho:\absGal\rightarrow GL_n(k_L)$ be the reduction of $\rho$ to the residue field of $L$.
  In this section we want to construct a space $X$ with a family of $\phigam$-modules
  specializing to $\drig(\rho)$.

  The following is a generalization of Lemma \ref{iso of H^1}.

  \begin{lem}
    \label{iso of H^1 in high dim}
    Let $(\delta_1,\dots,\delta_n)\in\T^n(L)$ such that $\delta_i/\delta_j\in\T^+(L)$
    for some $i<j$ and let $D_{n}$ be a triangulable $\phigam$-module over $\robba_L$
    of parameters $(\delta_1,\dots,\delta_n)$.
    Let us denote by $D_i$ the $i$-th piece of the filtration of $D_n$ giving rise to parameters
    $(\delta_1,\dots,\delta_n)$, so that $D_i$ is an extension of
    $\robba_L(\delta_1),\dots,\robba_L(\delta_{i})$ for $1\leq i\leq n$.
    There exist
    \begin{enumerate}
      \item
      characters $(\delta_1',\delta_2',\dots,\delta_n')\in\T^n(L)$ where $\delta_i'\in\delta_i x^\Z$
      and such that if $\rr{wt}(\delta_i')\equiv\rr{wt}(\delta_j')\pmod\Z$ for some $i,j$,
      then $\rr{wt}(\delta_i')=\rr{wt}(\delta_j')$;
      moreover if $\rr{wt}(\delta_i')\in\Z$, then $\rr{wt}(\delta_i')=0$;
      \item
      a triangulable $\phigam$-module $D_{n}'$ over $\robba_L$ of parameters $(\delta_1',\dots,\delta_n')$
    \end{enumerate}
    such that (if we let $D_i'$ denote the $i$-th piece of the filtration of $D_n'$ giving rise
    to the parameters $(\delta_1',\dots,\delta_n')$)
    \begin{itemize}
      \item
      $\delta_i'/\delta_j'\notin\T^+(L)$ for all $1\leq i<j\leq n$;
      \item
      if $\delta_i/\delta_j\in\T^-(L)$ then $\delta_i'/\delta_j'\in\T^-(L)$ for all $1\leq i<j\leq n$;
      \item
      $H_{\phi,\Gamma}(D_{i-1}(\delta_i^{-1})[1/t])\cong H^1_{\phigam}(D_{i-1}'(\delta_i'^{-1}))$
      for all $2\leq i\leq n$.
    \end{itemize}
  \end{lem}

  \begin{proof}
    We will prove the statement by induction.
    \begin{itemize}
      \item[$n=2$:]
      This case is proved by Lemma \ref{iso of H^1}.
      \item[$n>2:$]
      Assume by inductive hypothesis that there are $(\delta_1',\delta_2',\dots,\delta_{n-1}')\in\T^n(L)$
      such that $\delta_i'/\delta_j'\notin\T^+(L)$ for all $1\leq i<j\leq n-1$;
      there is moreover an extension $D_{n-2}'$ of $\robba_L(\delta_1'),\robba_L(\delta_2'),\dots,\robba_L(\delta_{n-2}')$
      such that
      \begin{equation}
        \label{proof:induction step}
        H_{\phi,\Gamma}(D_{n-2}(\delta_{n-1}^{-1})[1/t])\cong H^1_{\phigam}(D_{n-2}'(\delta_{n-1}'^{-1})).
      \end{equation}
      We consider different cases:
      \begin{itemize}
        \item[-]
          if $\rr{wt}(\delta_n)\in\Z$, then let $\mu_n\coloneqq\rr{wt}(\delta_n)$;
        \item[-]
          if $\rr{wt}(\delta_n)\notin\Z$ and $\rr{wt}(\delta_i/\delta_n)\notin\Z$
          for all $i<n$, then let $\mu_n\coloneqq 0$;
        \item[-]
          finally if $\rr{wt}(\delta_n)\notin\Z$ and $\rr{wt}(\delta_i/\delta_n)\in\Z$
          for some $i<n$, then let $j\coloneqq\min\{i<n\colon \rr{wt}(\delta_i/\delta_n)\in\Z\}$
          and let $\mu_n\coloneqq -\rr{wt}(\delta_j/\delta_n)$.
      \end{itemize}
      Let now $\delta_n'\coloneqq\delta_n x^{-\mu_n}$.
      Observe that either $\rr{wt}(\delta_i'/\delta_n')\notin\Z$ or $\rr{wt}(\delta_i'/\delta_n')=0$
      for all $i<n$;
      in particular $\delta_i'/\delta_n'\notin\T^+(L)$ and $\mathrm{wt}(\delta_i'/\delta_n')\notin\N_{>0}$
      for all $1\leq i\leq n-1$.
      Let us denote by $D_{n-1}'\in H^1_{\phigam}(D_{n-2}'(\delta_{n-1}'^{-1}))$ the extension
      corresponding to $D_{n-1}[1/t]\in H_{\phi,\Gamma}(D_{n-2}(\delta_{n-1}^{-1})[1/t])$
      through the isomorphism (\ref{proof:induction step}).
      In particular we have that $D_{n-1}[1/t]=D_{n-1}'[1/t]$.
      Then we have that
      \begin{align*}
        H^1_{\phi,\Gamma}(D_{n-1}(\delta_n^{-1})[1/t])&=H^1_{\phi,\Gamma}(D_{n-1}'(\delta_n^{-1})[1/t])
        =\lim_{\stackrel{\rightarrow}{k}}H^1_{\phi,\Gamma}(D_{n-1}'(\delta_n^{-1}x^{-k})).
      \end{align*}
      Notice that for all $k\geq -\mu_n$, we have
      \[
        H^1_{\phi,\Gamma}(D_{n-1}'(\delta_n'^{-1}))\xrightarrow{\sim}H^1_{\phi,\Gamma}(D_{n-1}'(\delta_n^{-1}x^{-k}))
      \]
      is an isomorphism thanks to Proposition \ref{prop:high dim twist}.
      We then have
      \[
        H^1_{\phi,\Gamma}(D_{n-1}(\delta_n^{-1})[1/t])=\lim_{\stackrel{\rightarrow}{k}}H^1_{\phi,\Gamma}(D_{n-1}'(\delta_n^{-1}x^{-k}))
        \cong  H^1_{\phi,\Gamma}(D_{n-1}'(\delta_n'^{-1})).
      \]
      Finally what is left to prove is that if $\delta_i/\delta_j\in\T^-(L)$,
      then $x^{-\mu_i}\delta_i/x^{-\mu_j}\delta_j\in\T^-(L)$ for all $1\leq i<j\leq n$.
      By inductive hypothesis, this is true for all $1\leq i<j\leq n-1$.
      Now it is enough to notice that if $\delta_i/\delta_n\in\T^-$, then
      $\rr{wt}(\delta_i)\equiv\rr{wt}(\delta_n)\pmod \Z$ and by the way we chose $\mu_i,\mu_n$
      we have that $\rr{wt}(\delta_i')-\rr{wt}(\delta_n')=0$; as a consequence,
      $\delta_i'/\delta_n'=x^0\in\T^-$.
    \end{itemize}
  \end{proof}

  By Lemma \ref{iso of H^1 in high dim}, there exists a tuple
  $\mu\coloneqq(\mu_1,\mu_2,\dots,\mu_n)$ of integers such that
  $x^{-\mu_i}\delta_i/x^{-\mu_j}\delta_j\notin\T^+(L)$,
  $\mathrm{wt}(x^{-\mu_i}\delta_i/x^{-\mu_j\delta_j})\notin\Z\setminus\{0\}$
  and if $\delta_i/\delta_j\in\T^-(L)$ then $x^{-\mu_i}\delta_i/x^{-\mu_j}\delta_j=x^0\in\T^-(L)$ for all $1\leq i<j\leq n$.
  Moreover there exists an extenion $D_{n}'$ of $\robba_L(\delta_1x^{-\mu_1}),\robba_L(\delta_2x^{-\mu_2}),\dots,\robba_L(\delta_{n}x^{-\mu_{n}})$
  such that there are isomorphisms
  \[
    H^1_{\phi,\Gamma}(D_{i-1}(\delta_{i}^{-1})[1/t])\xrightarrow{\sim} H^1_{\phigam}(D_{i-1}'(\delta_{i}^{-1}x^{\mu_{i}}))
  \]
  for all $2\leq i\leq n-1$, where $D_i'$ is the $i$-th piece of the filtration of $D_n'$.
  In particular we have $D_n'[1/t]=\drig(\rho)[1/t]$ and $D_i'[1/t]=D_i[1/t]$
  for all $1\leq i\leq n-1$.
  We are going to generalize Proposition \ref{construction of X'} in higher dimensions.

  \begin{prop}
    \label{construction of Xn'}
    There exist
      \begin{itemize}
        \item
          a rigid space $X_2'$ over $L$ which is also a vector bundle over an
          appropriate neighbourhood $\mathcal{U}_1'\times\mathcal{U}_2'$ of $(\delta_1x^{-\mu_1},\delta_2x^{-\mu_2})$ in $\T^2$;
        \item
          rigid spaces $X'_i$ over $L$ which are also vector bundles over
          $X_{i-1}'\times\mathcal{U}_i'$, where $\mathcal{U}_i'$ is an
          appropriate neighbourhood of $\delta_ix^{-\mu_i}$ in $\T$ for all $3\leq i\leq n$;
        \item
          families of $\phigam$-modules $\D_i'$ over $X'_i$ for all $2\leq i\leq n$;
        \item
          maps $(\widetilde\delta_1',\dots,\widetilde\delta_i'):X_i'\rightarrow\T^i$ for all $2\leq i\leq n$
      \end{itemize}
      such that there are
      \begin{enumerate}
        \item
          $x'_i\in X_i'$ such that $\D_i'\widehat\otimes k(x'_i)=D'_i$ and
          \[
            (\widetilde\delta_1',\dots,\widetilde\delta_i')\widehat\otimes k(x'_i)=(\delta_1x^{-\mu_1},\delta_2x^{-\mu_2},\dots,\delta_ix^{-\mu_i})
          \]
          for all $2\leq i\leq n$;
        \item
          Zariski-open dense subsets $U'_i\xhookrightarrow{j_i}X'_i$ such that $j_i^*\D_i'$
          has a filtration of sub-$\phigam$-modules with graded pieces $\robba_{U'_i}(\widetilde{\delta_1}'),\dots,
          \robba_{U'_i}(\widetilde\delta_i')$ and $j_i^*(\widetilde{\delta_1}',\dots,\widetilde{\delta_i}')\in\T^\reg_i(U'_i).$
      \end{enumerate}
  \end{prop}

  \begin{proof}
    The statement descends directly from Proposition \ref{n dim T-}.
  \end{proof}

  \begin{rem}
    \label{good U_n'}
    Notice that by Lemma \ref{good neighbourhood}, we can choose the neighbourhoods $\mathcal{U}_i'$
    of $\delta_ix^{-\mu_i}$ in such a way so that for any
    $u\in\widetilde{\mathcal{U}}_n'(K)\coloneqq\mathcal{U}'_1(K)\times\mathcal{U}_2'(K)\times\dots\times\mathcal{U}_n'(K)$,
    we have
    \begin{itemize}
      \item
        $(\widetilde{\delta_i}'/\widetilde{\delta}_j')\widehat{\otimes}k(u)=x^{a_{ij}}$ for some $a_{ij}\in\Z$
        only if $x^{-\mu_i}\delta_i/x^{-\mu_j}\delta_j=x^{a_{ij}}$;
      \item
        $(\widetilde{\delta_i}'/\widetilde{\delta}_j')\widehat{\otimes}k(u)=\chi x^{a_{ij}}$ for some $a_{ij}\in\Z$
        only if $x^{-\mu_i}\delta_i/x^{-\mu_j}\delta_j=\chi x^{a_{ij}}$.
    \end{itemize}
  \end{rem}

  As in the two-dimensional case, by Theorem \ref{sub-phigam-module}, there exist $r_0\in p^\Q\cap[0,1)$ and
  a unique $\phigam$-module $D^{r_0}_n$ over $\robba_L^{r_0}$ such that $\drig(\rho)=D^{r_0}_n\otimes_{\robba^{r_0}_L}\robba_L$;
  moreover there is $r_1\in p^\Q\cap[0,1)$ such that
  the family of $\phigam$-modules $\D'_n$ can be already defined over $\robba_{X'_n}^{r_1}$.
  Let us denote by $r$ the maximum of $r_0$ and $r_1$, in such a way that both $\drig(\rho)$
  and $\D'_n$ are defined over $\robba_L^r$ and $\robba_{X'_n}^r$ respectively.
  In other words, there exist $D^r_n$ a $\phigam$-module over $\robba^r_L$ such that
  $\drig(\rho)=D^r_n\otimes_{\robba^r_L}\robba_L$ and a $\phigam$-module $\D'^{r}_n$
  over $\robba_{X'_n}^r$ such that
  $\D'^{r}_n\otimes_{\robba_{X_n'}^r}\robba_{X_n'}=\D'_n$.
  We also denote by $D_i^r$ and $D'^r_i$ the $\phigam$-modules over $\robba_L^r$ such that
  $D_i^r\otimes_{\robba_L^r}\robba_L=D_i$ and $D'^r_i\otimes_{\robba_L^r}\robba_L=D'_i$
  for all $1\leq i\leq n-1$.
  Finally for any $2\leq i\leq n-1$,
  we let $\D_i'^r$ be the $\phigam$-module over $\robba_{X_i'}$ such that
  $\D'^{r}_i\otimes_{\robba_{X_i'}^r}\robba_{X_i'}=\D'_i$.
  Let $m(r)$ be the smallest positive integer such that $p^{m-1}(p-1)\geq r$ for all $m\geq m(r)$,
  so that all the points of the form $1-\zeta_{p^m}$ (where $\zeta_{p^m}$ is a $p^m$-th root of unity)
  lie in the open annulus $\mathbb{B}^r$.
  Let us fix $m\geq m(r)$ and let us assume that $L$ contains $K_m$
  (by Lemma \ref{enlarge L} we know we can always make this assumption).
  As a consequence of Theorem \ref{equivalenceBL}, the $\phigam$-module $D^r_i$ corresponds to the
  tuple $(D^r_i[1/t],D^r_i\otimes_{\robba_L^r}\prod_{K_m\hookrightarrow L}L\llbracket t\rrbracket,f_i)$ through
  the equivalence \ref{equivalenceBeauvilleLaszlo} for all $1\leq i\leq n$, where
  $f_i:D^r_i[1/t]\otimes_{\robba^r_L[1/t]}\prod_{K_m\hookrightarrow L}L((t))\xrightarrow{\sim}
  (D^r_i\otimes_{\robba_L^r}\prod_{K_m\hookrightarrow L}L\llbracket t\rrbracket)[1/t]$
  is defined as in Theorem \ref{equivalenceBL}.
  Let us denote by $\Lambda_i$ the $\Gamma$-stable finite projective module
  $D^r_i\otimes_{\robba_L^r}\prod_{K_m\hookrightarrow L}L\llbracket t\rrbracket$
  for all $1\leq i\leq n$.
  By Remark \ref{stable lattices}, we have
  \[
    \Lambda_i=\prod_{K_m\xhookrightarrow{\tau} L}\Lambda_{i,\tau}
  \]
  where $\Lambda_{i,\tau}$ is a $\Gamma_m$-stable lattice over $L\llbracket t\rrbracket$.
  On the other hand, given such a lattice $\Lambda_{i,\tau}$, it is possible to recover $\Lambda_i$.
  From now on, since the two notions are equivalent,
  we fix an embedding $\tau:K_m\hookrightarrow L$ and we will denote
  by $\Lambda_i$ the $\Gamma_m$-stable lattice $\Lambda_{i,\tau}$ over $L\llbracket t\rrbracket$
  for all $1\leq i\leq n$.
  In the same way, let $\Lambda'_i$ be the $\Gamma_m$-stable lattice over
  $L\llbracket t\rrbracket$ corresponding to the $\tau$-factor of the lattice associated to $D'^r_i$.
  Let us denote by $\mathbb{L}'_i$ the $\Gamma_m$-stable lattice corresponding to the
  factor $\tau:K_m\hookrightarrow L$ of the lattice $\D'^r_i\otimes_{\robba^r_{X_i'}}
  \prod_{K_m\hookrightarrow L}R'_i\llbracket t\rrbracket$ associated to $\D'_i$ for all $2\leq i\leq n$.
  After choosing a basis of $\D'_i$,
  we choose a trivialization of $\mathbb{L}'_i$ and fix a basis of $\mathbb{L}'_i$
  respecting the triangulation of $\D'_i$ for all $1\leq i\leq n$.

  We then have the following result.

  \begin{lem}
    \label{relative positions of lattices in dim n}
    $\Lambda_i\in Y_{(\mu_1,\mu_2,\dots,\mu_i)}(L)$ for all $2\leq i\leq n$.
  \end{lem}

  \begin{proof}
    We have to show that
    $\mathrm{gr}^j(\Lambda_i)=t^{\mu_j}\cdot\mathrm{gr}^j(\Lambda_i')$ for all $1\leq j\leq i\leq n$.
    Let us denote by $\delta_i'\coloneqq\delta_ix^{-\mu_i}$.
    Since we chose a basis of $\Lambda'_i$ respecting the triangulation of $D'_i$, we have that
    $\mathrm{Fil}^\bullet(L((t))^i)$
    is induced by the fltration of $D'_i[1/t]$.
    This implies that
    \[
      \mathrm{gr}^j(\Lambda_i')=\robba_L^r(\delta_j')\otimes_{\robba_L^r,\tau}L\llbracket t\rrbracket
    \]
    for all $1\leq j\leq i\leq n$.
    Let us denote by $V_i$ the module $L((t))^i$ for simplicity.
    We have
    \begin{align*}
      \mathrm{gr}^1(\Lambda_i)&=\Lambda_i\cap\mathrm{Fil}^1(V_i)
      =\Lambda_i\cap\left (\robba_L^r(\delta_1)[1/t]\otimes_{\robba^r_L[1/t],\tau}L((t))\right )\\
      &=\robba_L^r(\delta_1)\otimes_{\robba_L^r,\tau}L\llbracket t\rrbracket
      =t^{\mu_1}\robba_L^r(x^{-\mu_1}\delta_1)\otimes_{\robba_L^r,\tau}L\llbracket t\rrbracket
      =t^{\mu_1}\mathrm{gr}^1(\Lambda_i')
    \end{align*}
    and
    \begin{align*}
      \mathrm{gr}^j(\Lambda_i)&=\Fil^j(\Lambda_i)/\Fil^{j-1}(\Lambda_i)
      =(\Lambda_i\cap\Fil^j(V_i))/(\Lambda_i\cap\Fil^{j-1}(V_i))
    \end{align*}
    for all $2\leq j\leq i\leq n$.
    Notice that
    \begin{align*}
      \Lambda_i\cap\Fil^j(V_i)&=\Lambda_i\cap(D_j'^r[1/t]\otimes_{\robba_L^r[1/t],\tau}L((t)))
      =\Lambda_i\cap(D_j^r[1/t]\otimes_{\robba_L^r[1/t],\tau}L((t)))\\
      &=D_j^r\otimes_{\robba_L^r,\tau}L\llbracket t\rrbracket
    \end{align*}
    for all $2\leq j\leq i\leq n$.
    Therefore
    \begin{align*}
      \mathrm{gr}^j(\Lambda_i)&=(\Lambda_i\cap\Fil^j(V_i))/(\Lambda_i\cap\Fil^{j-1}(V_i))
      =(D_j^r\otimes_{\robba_L^r,\tau}L\llbracket t\rrbracket)/(D_{j-1}^r\otimes_{\robba_L^r,\tau}L\llbracket t\rrbracket)\\
      &=\robba_L^r(\delta_j)\otimes_{\robba_L^r,\tau}L\llbracket t\rrbracket
      =\robba_L^r(\delta_j'x^{\mu_j})\otimes_{\robba_L^r,\tau}L\llbracket t\rrbracket\\
      &=t^{\mu_j}\robba_L^r(\delta_j')\otimes_{\robba_L^r,\tau}L\llbracket t\rrbracket
      =t^{\mu_j}\mathrm{gr}^j(\Lambda_i')
    \end{align*}
    for all $2\leq j\leq i\leq n$.
  \end{proof}

  Just like in the two-dimensional case we will construct a subspace of $X_n'\times\Gr_n$
  (where $X_n'$ is the space defined in Proposition \ref{construction of Xn'})
  consisting of $\Gamma_m$-stable lattices.
  Again, let us denote by $\gamma$ a choice of a topological generator of $\Gamma_m$.

  \begin{dfn}
    For all $2\leq i\leq n$, let us define the map
    $\gamma\cdot:X'_i\times\mathrm{Gr}_i^{\rig}\rightarrow\mathrm{Gr}_i^{\rig}$ in the following way:
    for any affinoid $L$-algebra $A$ and for any $(y',\Lambda_A)\in(X_i'\times\mathrm{Gr}_i^{\rig})(A)$, let
    \[
      \gamma\cdot(y',\Lambda_A)\coloneqq \gamma\cdot\Lambda_A,
    \]
    where the action of $\gamma$ on $\Lambda_A$ is induced from the
    action of $\Gamma_m$ on $y'^*\D_i'$ through the isomorphism
    \[
      (y'^*\D_i'^r)[1/t]\otimes_{\robba_A^r[1/t],\tau}A((t))
      \cong \Lambda_A[1/t].
    \]
    Notice that $\gamma\cdot\Lambda_A\in\Gr_i(A)$, because $\gamma$ induces an automorphism on
    $(y'^*\D_i'^r)[1/t]$, hence the map is well-defined.
  \end{dfn}

  \begin{thm}
    \label{construction of X_n}
    For all $2\leq i\leq n$ there exist
    \begin{itemize}
      \item
      an affinoid rigid space $X_i=\mathrm{Sp}(R_i)$ over $L$;
      \item
      a family of $\phigam$-modules $\D_i$ over $X_i$;
      \item
      a map $(\widetilde\delta_1,\dots,\widetilde\delta_i):X_i\rightarrow\T^i$ which consists of parameters for $\D_i$
    \end{itemize}
    such that:
    \begin{enumerate}
      \item
      there is $x_i\in X_i$ such that $\D_i\widehat\otimes k(x_i)=D_i$ and
      $(\widetilde\delta_1,\dots,\widetilde\delta_i)\widehat\otimes k(x_i)=(\delta_1,\dots,\delta_i)$;
      \item
      $(\widetilde\delta_1,\dots,\widetilde\delta_i)=(x^{\mu_1}\widetilde\delta_1',\dots,x^{\mu_i}\widetilde\delta_i')$.
    \end{enumerate}
  \end{thm}

  \begin{proof}
    The proof of this result is completely analogous to the proof of Theorem \ref{construction of X},
    thus we basically repeat the same proof, changing just a few details.

    Let us define the ind-rigid space $Y_i$ to be the fiber product
    \[
      \begin{tikzcd}
        Y_i\coloneqq X'_i\times\Gr_i\times_{\Gr_i\times\Gr_i}\Gr_i\ar[r]\ar[d, "p_2"] & \mathrm{Gr}_i^{\rig}\ar[d, "\delta"]\\
        X'_i\times \Gr_i\ar[r,"(\gamma\cdot{,}\mathrm{pr})"] & \Gr_i\times\Gr_i
      \end{tikzcd}
    \]
    where the rightmost map $\delta$ is the diagonal map.
    Notice that for any $((y',\Lambda_A),\Lambda'_A)\in Y_i(A)$, we have
    $(\gamma\cdot(y',\Lambda_A),\Lambda_A)=(\Lambda_A',\Lambda_A')$.
    Since $\Lambda_i$ is stable under the action of $\Gamma_m$ induced by the point $x_i'$,
    we have that $x_i\coloneqq((x_i',\Lambda_i),\Lambda_i)\in Y_i(L)$.
    We define $X_i$ to be the fiber product
    \[
      X_i\coloneqq Y_i\times_{\Gr_i}Y_{(\mu_1,\dots,\mu_i)};
    \]
    we have then that $x_i\in X_i(L)$ because of Lemma \ref{relative positions of lattices in dim n} and,
    eventually restricting $X_i$, we can assume that $X_i$ is an affinoid neighbourhood $\mathrm{Sp}(R_i)$ of $x_i$.
    Consider the pullback $p_2^*(\D_i',\mathbb{L}_i)$ to $X_i$ (where $\mathbb{L}_i$ is the universal lattice on $\Gr_i$)
    and the isomorphism of $\prod_{K_m\hookrightarrow L}R_i(( t))$-modules
    \[
      \widetilde{f}_i:p_2^*\D_i'^r[1/t]\otimes_{\robba_{R_i}^r[1/t]}\prod_{K_m\hookrightarrow L}R_i((t))
      \xrightarrow{\sim}\left (\prod_k g_m^k\cdot p_2^*\mathbb{L}_i\right )[1/t]
    \]
    defined as in Theorem \ref{equivalenceBL}, where $g_m$ is a generator of $\mathrm{Gal}(K_m/\Q_p)$
    and $k$ ranges over $\{0,\dots,|\mathrm{Gal}(K_m/\Q_p)|-1\}$.
    Let us denote by $\D_i$ the sheaf
    \[
      \left (p_2^*\D'_i, p_2^*\mathbb{L}_i,\widetilde{f}_i\right )
    \]
    on $X_i$, which can be seen as a sheaf of $\phigam$-modules thanks to the equivalence \ref{equivalenceBeauvilleLaszlo};
    then we have that $\D_i\widehat\otimes k(x_i)=(D^r_i[1/t],\Lambda_i,f_i)$.
    We have therefore constructed a space $X_i$ with a sheaf $\D_i$ specializing at some point $x_i\in X_i$ to
    $(D_i^r[1/t],\Lambda_i,f_i)$, which corresponds to the $\phigam$-module $D_i$ through
    the equivalence \ref{equivalenceBeauvilleLaszlo}.
    Moreover the family of $\phigam$-modules $\D_i$ is obviously already defined over $\robba_{X_i}^r$,
    as $\D_i'$ is;
    hence there is a family of $\phigam$-modules $\D_i^r$ over $\robba_{R_i}^r$ such that
    $\D_i=\D_i^r\otimes_{\robba_{R_i}^r}\robba_{R_i}$.
    Notice that the sheaf of $\phigam$-modules $\D_i$ comes with a filtration of sub-$\phigam$-modules,
    which is the one induced by the filtration of sub-$\phigam$-modules of $\D_i'$.
    More precisely, for any affinoid $L$-algebra $A$ and for $y=((y',\Lambda_A),\Lambda_A)\in X_i(A)$, we have
    \[
      \mathrm{Fil}^j(y^*\D_i)=y^*\D_i\cap\mathrm{Fil}^j(y'^*\D_i'[1/t]).
    \]
    Recall that $\D_i'[1/t]=\D_i[1/t]$ by construction, so the intersection above makes sense.
    Moreover notice that $\mathrm{Fil}^j(y^*\D_i)$ is stable under $\Phi$ and $\Gamma$,
    since both $y^*\D_i$ and $\mathrm{Fil}^j(y'^*\D_i'[1/t])$ are.
    Finally, by Lemma \ref{Y_mu gives filtrations} we have that $\mathrm{Fil}^j(y^*\D_i)$ is locally free.
    Therefore $\mathrm{Fil}^\bullet(y^*\D_i)$ provides a triangulation of $\D_i$
    and hence it gives parameters $(\widetilde\delta_1,\dots,\widetilde\delta_i)\in\T^i(X_i)$ for the family $\D_i$.
    The last thing left to show is that
    $(\widetilde\delta_1,\dots,\widetilde\delta_i)=(x^{\mu_1}\widetilde\delta_1',\dots,x^{\mu_i}\widetilde\delta_i')$.
    We claim that
    \[
      \mathrm{gr}^j(p_2^*\mathbb{L}_i)=\robba_{R_i}^r(\widetilde\delta_j)\otimes_{\robba_{R_i}^r,\tau}R_i\llbracket t\rrbracket
    \]
    for all $1\leq j\leq i\leq n$.
    Let $q:X_i\rightarrow X_i'$ be the projection map; since we fixed a basis
    of the standard lattice $\mathbb{L}_i'$ over $R_i'$ respecting
    the triangulation of $\D'_i$, we have that
    \begin{align*}
      \mathrm{Fil}^j(R_i((t))^i)& =q^*\mathrm{Fil}^j(R'_i((t))^i)=
      q^*(\D'^r_j[1/t]\otimes_{\robba_{R_i'}^r[1/t],\tau}R_i'((t)))\\
      &=q^*\D_j'^r[1/t]\otimes_{\robba_{R_i}^r[1/t],\tau}R_i((t)).
    \end{align*}
    Thus
    \begin{align*}
      \mathrm{gr}^1(p_2^*\mathbb{L}_i)&=\mathrm{Fil}^1(p_2^*\mathbb{L}_i)=
      p_2^*\mathbb{L}_i\cap \mathrm{Fil}^1(R_i((t))^i)\\
      &=\left (\D_i^r\otimes_{\robba_{R_i}^r,\tau}R_i\llbracket t\rrbracket\right )\cap
      \left (\robba_{R_i}^r(\widetilde\delta_1')[1/t]\otimes_{\robba_{R_i}^r[1/t],\tau}R_i((t))\right )\\
      &=\robba_{R_i}^r(\widetilde\delta_1)\otimes_{\robba_{R_i}^r,\tau}R_i\llbracket t\rrbracket
    \end{align*}
    and
    \begin{align*}
      \mathrm{gr}^j(p_2^*\mathbb{L}_i)&= \Fil^j(p_2^*\mathbb{L}_i)/\Fil^{j-1}(p_2^*\mathbb{L}_i)\\
      &=(p_2^*\mathbb{L}_i\cap\Fil^j(R_i((t))^i))/(p_2^*\mathbb{L}_i\cap\Fil^{j-1}(R_i((t))^i))
    \end{align*}
    for all $2\leq j\leq i\leq n$.
    Notice that
    \begin{align*}
      p_2^*\mathbb{L}_i\cap\Fil^j(R_i((t))^i)&=p_2^*\mathbb{L}_i\cap(q^*\D_j'^r[1/t]\otimes_{\robba_{R_i}^r[1/t],\tau}R_i((t)))\\
      &=\D_j^r\otimes_{\robba_{R_i}^r,\tau}R_i\llbracket t\rrbracket
    \end{align*}
    for all $2\leq j\leq i\leq n$.
    Therefore
    \begin{align*}
      \mathrm{gr}^j(p_2^*\mathbb{L}_i)&=(p_2^*\mathbb{L}_i\cap\Fil^j(R_i((t))^i))/(p_2^*\mathbb{L}_i\cap\Fil^{j-1}(R_i((t))^i))\\
      &=(\D_j^r\otimes_{\robba_{R_i}^r,\tau}R_i\llbracket t\rrbracket)/(\D_{j-1}^r\otimes_{\robba_{R_i}^r,\tau}R_i\llbracket t\rrbracket)
      =\robba_{R_i}(\widetilde\delta_j)\otimes_{\robba_{R_i}^r,\tau}R_i\llbracket t\rrbracket
    \end{align*}
    for all $2\leq j\leq i\leq n$.
    Restricting ourselves to lattices in $Y_{(\mu_1,\dots,\mu_i)}$ as in the definition of the space $X_i$, we have
    \begin{align*}
      \robba_{R_i}^r(\widetilde\delta_j)\otimes_{\robba_{R_i}^r,\tau}R_i\llbracket t\rrbracket&=
      t^{\mu_j}\cdot \robba_{R_i}^r(\widetilde\delta_j')\otimes_{\robba_{R_i}^r,\tau}R_i\llbracket t\rrbracket\\
      &=\robba_{R_i}^r(x^{\mu_j}\widetilde\delta_j')\otimes_{\robba_{R_i}^r,\tau}R_i\llbracket t\rrbracket
    \end{align*}
    for all $1\leq j\leq i\leq n$, which concludes the proof.
  \end{proof}

  Exactly like in the case of dimension two, we have constructed a rigid space $X_n$ over $L$
  and a family of $\phigam$-modules $\D_n$ specializing to $\drig(\rho)$ at some point $x_n\in X_n$.

  \begin{rem}
    Notice that the above construcion works in any dimension $n$ and that for any
    $3\leq i\leq n$, we have $\D_i\in H^1_{\phi,\Gamma}(\D_{i-1}(\widetilde\delta_i^{-1}))$.
  \end{rem}

  We still have to show that the family $\D_n$ is regular in a dense subspace of a
  neighbourhood of $x_n\in X_n(L)$. As in the two-dimensional case, the space $X_n$ is not
  a vector bundle over an open of $\T^n$, hence we are again going to use a more elaborate
  argument to be able to conclude.

\subsection{Density of the regular locus}
\label{density of the regular locus}

  In order to simplify the notation, we denote by $G$ the map
  \[
    (\widetilde\delta_1,\dots,\widetilde\delta_n):X_n\rightarrow\T^n
  \]
  sending a point $y_n$ of $X_n$ to
  the parameters associated to the filtration of $y_n^*\D_n$.
  Moreover, let
  \[
    H_n:\T^n\rightarrow\W^n,
    (\widetilde\eta_1,\dots,\widetilde\eta_n)\mapsto(\widetilde\eta_{1|\mathbb{Z}_p^\times},\dots,\widetilde\eta_{n|\mathbb{Z}_p^\times})
  \]
  be the map sending continuous characters of $\Q_p^\times$ to their restriction to $\mathbb{Z}_p^\times$.
  Finally, let
  \begin{equation}
    \label{map F}
    F:X_n\xrightarrow G\T^n\xrightarrow{H_n}\W^n
  \end{equation}
  be the composition of the maps $G$ and $H_n$.
  To be able to prove the density of the regular locus, we need again to use some argument involving
  the dimension of the preimage through $F$ of certain subspaces of $\W^n$.

  First of all we need the following results.

  \begin{lem}
    \label{D[1/t] splits}
    Let $K$ be an extension of $\Qp$ and let $D$ be a $\phigam$-module over $\robba_K$ such that there exists a
    short exact sequence
    \[
      0\rightarrow E\rightarrow D\rightarrow \robba_K(\delta)\rightarrow0
    \]
    for some $\phigam$-module $E$ over $\robba_K$ and for some continuous character $\delta:\Qp^\times\rightarrow K^\times$
    such that $H^0_{\phi,\Gamma}(\robba_K(\delta))\neq0$.
    If the induced map $H^0_{\phi,\Gamma}(D)\rightarrow H^0_{\phi,\Gamma}(\robba_K(\delta))$ is surjective,
    then $D[1/t]=E[1/t]\oplus \robba_K(\delta)[1/t]$ as $\phigam$-module.
  \end{lem}

  \begin{proof}
    Let us denote by $\widetilde p$ the surjection of $\phigam$-modules
    \[
      \widetilde{p}:D\twoheadrightarrow\robba_K(\delta)
    \]
    and by $p$ the restriction of the map $\widetilde{p}$ to
    $\widetilde{p}^{-1}(\robba_K\cdot H^0_{\phi,\Gamma}(\robba_K(\delta)))$,
    so that
    \[
      p:\widetilde{p}^{-1}(\robba_K\cdot H^0_{\phi,\Gamma}(\robba_K(\delta)))
      \twoheadrightarrow \robba_K\cdot H^0_{\phi,\Gamma}(\robba_K(\delta)).
    \]
    We get the following map of short exact sequences:
    \[
      \begin{tikzcd}
        0\arrow[r]&\widetilde{p}^{-1}(\robba_K\cdot H^0_{\phi,\Gamma}(\robba_K(\delta)))\arrow[r]\arrow[d,twoheadrightarrow,"p"]
        &D\arrow[r]\arrow[d,twoheadrightarrow,"\widetilde{p}"]&T\arrow[r]\arrow[d,phantom,"=" labl]&0\\
        0\arrow[r]&\robba_K\cdot H^0_{\phi,\Gamma}(\robba_K(\delta))\arrow[r]
        &\robba_K(\delta)\arrow[r]&T\arrow[r]&0,
      \end{tikzcd}
    \]
    where $T$ is some torsion $\robba_K$-module.
    Let $d\in H^0_{\phi,\Gamma}(D)$ such that $\widetilde{p}(d)\neq0$, which exists by the hypothesis that
    $\widetilde{p}:H^0_{\phi,\Gamma}(D)\rightarrow H^0_{\phi,\Gamma}(\robba_K(\delta))$ is surjective.
    Since $E=\ker(\widetilde{p})$ and $H^0_{\phi,\Gamma}(\robba_K(\delta))$ is a 1-dimensional vector
    space over $K$, we have that
    \[
      \widetilde{p}^{-1}(\robba_K\cdot H^0_{\phi,\Gamma}(\robba_K(\delta)))=E\oplus\robba_K\cdot d
    \]
    as a $\phigam$-module over $\robba_K$.
    Observe in fact that $E$ is obviously stable under the action of $\Phi$ and $\Gamma$,
    since it is a $\phigam$-module itself; on the other hand, $\robba_K\cdot d$ is also
    stable under $\Phi$ and $\Gamma$, since $d\in H^0_{\phi,\Gamma}(D)$ and
    the 0-th $\phigam$-cohomology is the submodule of fixed elements by $\Phi$ and $\Gamma$ by definition.
    After inverting $t$, the torsion $\robba_K$-module $T[1/t]$ gets annihilated, hence (since localization is flat)
    we have the following commutative diagram
    \[
      \begin{tikzcd}
        E[1/t]\oplus(\robba_K\cdot d)[1/t]\arrow[r,"\sim"]\arrow[d,twoheadrightarrow,"p"]
        &D[1/t]\arrow[d]\arrow[d,twoheadrightarrow,"p"]\\
        \robba_K[1/t]\cdot H^0_{\phi,\Gamma}(\robba_K(\delta)\arrow[r,"\sim"]
        &\robba_K(\delta)[1/t],
      \end{tikzcd}
    \]
    where the horizontal arrows are isomorphisms of $\phigam$-modules.
  \end{proof}

  \begin{lem}
    \label{bound of H^0}
    Let $K$ be an extension of $\Qp$ and let $D_n,D_n'$ be two triangulable $\phigam$-modules over $\robba_K$ of the same rank.
    Let us denote by $D_i$ and $D_i'$ the $i$-th piece of the filtration of $D_n$ and $D_n'$ respectively
    and let $(\delta_1,\dots,\delta_n),(\delta_1',\dots,\delta_n')\in\T^n(K)$ be the parameters
    of $D_n$ and $D_n'$ associated to the fixed filtration respectively.
    Assume that
    \begin{enumerate}
      \item
        $\delta_i'=\delta_ix^{-\mu_i}$ for some $\mu_i\in\Z$ for all $1\leq i\leq n$;
      \item
        if $\delta_i/\delta_j\in\T^-(K)$, then $\delta_i'/\delta_j'\in\T^-(K)$ for all $i<j$;
      \item
        $D_i[1/t]=D_i'[1/t]$ for all $1\leq i\leq n$;
      \item
        $H^1_{\phi,\Gamma}(D_i'(\delta_j'^{-1})[1/t])\cong H^1_{\phi,\Gamma}(D_i'(\delta_j'^{-1}))$
        for all $i<j$.
    \end{enumerate}
    Then
    \[
      \dim_KH^0_{\phi,\Gamma}(D_{n-1}(\delta_n^{-1}))\leq\dim_KH^0_{\phi,\Gamma}(D_{n-1}'(\delta_n'^{-1})).
    \]
  \end{lem}

  \begin{proof}
    We proceed by induction on $n$.
    \begin{itemize}
      \item[$n=2:$]
      We distinguish two cases:
      \begin{itemize}
        \item
          if $\delta_1/\delta_2\in\T^-(K)$, then $\delta_1'/\delta_2'\in\T^-(K)$ and we have
          \[
            \dim_KH^0_{\phi,\Gamma}(\robba_K(\delta_1/\delta_2))=1=\dim_KH^0_{\phi,\Gamma}(\robba_K(\delta_1'/\delta_2')).
          \]
        \item
          If $\delta_1/\delta_2\notin\T^-(K)$, then
          \[
            \dim_KH^0_{\phi,\Gamma}(\robba_K(\delta_1/\delta_2))=0\leq\dim_KH^0_{\phi,\Gamma}(\robba_K(\delta_1'/\delta_2')).
          \]
      \end{itemize}
      \item[$n>2:$]
      We consider two different cases too:
      \begin{itemize}
        \item
          First assume that
          $\dim_KH^0_{\phi,\Gamma}(D_{n-1}(\delta_n^{-1}))=\dim_KH^0_{\phi,\Gamma}(D_{n-2}(\delta_n^{-1}))$.
          By inductive hypothesis we have that
          \[
            \dim_KH^0_{\phi,\Gamma}(D_{n-2}(\delta_n^{-1}))\leq\dim_KH^0_{\phi,\Gamma}(D'_{n-2}(\delta_n'^{-1}))
          \]
          so we have
          \[
            \dim_KH^0_{\phi,\Gamma}(D_{n-1}(\delta_n^{-1}))\leq\dim_KH^0_{\phi,\Gamma}(D'_{n-2}(\delta_n'^{-1})).
          \]
          On the other hand, from the short exact sequence
          \[
            0\rightarrow D_{n-2}'(\delta_n'^{-1})\rightarrow D_{n-1}'(\delta_n'^{-1})
            \rightarrow\robba_K(\delta_{n-1}'/\delta_n')\rightarrow0,
          \]
          we get the long exact sequence
          \[
            0\rightarrow H^0_{\phi,\Gamma}(D'_{n-2}(\delta_n'^{-1}))\rightarrow H^0_{\phi,\Gamma}(D'_{n-1}(\delta_n'^{-1}))
            \rightarrow H^0_{\phi,\Gamma}(\robba_K(\delta_{n-1}'/\delta_n'))\rightarrow\dots.
          \]
          Hence we have
          \[
            \dim_KH^0_{\phi,\Gamma}(D'_{n-2}(\delta_n'^{-1}))\leq\dim_KH^0_{\phi,\Gamma}(D'_{n-1}(\delta_n'^{-1})).
          \]
          Combining all of the inequalities above, we get
          \[
            \dim_KH^0_{\phi,\Gamma}(D_{n-1}(\delta_n^{-1}))\leq\dim_KH^0_{\phi,\Gamma}(D_{n-1}'(\delta_n'^{-1})).
          \]
        \item
          If instead $\dim_KH^0_{\phi,\Gamma}(D_{n-1}(\delta_n^{-1}))\neq\dim_KH^0_{\phi,\Gamma}(D_{n-2}(\delta_n^{-1}))$,
          then by the long exact sequence
          \[
            0\rightarrow H^0_{\phi,\Gamma}(D_{n-2}(\delta_n^{-1}))\rightarrow H^0_{\phi,\Gamma}(D_{n-1}(\delta_n^{-1}))
            \rightarrow H^0_{\phi,\Gamma}(\robba_K(\delta_{n-1}/\delta_n))\rightarrow\dots,
          \]
          we must have that $\dim_KH^0_{\phi,\Gamma}(\robba_K(\delta_{n-1}/\delta_n))\neq0$
          and the map
          \[
            H^0_{\phi,\Gamma}(D_{n-1}(\delta_n^{-1}))\twoheadrightarrow H^0_{\phi,\Gamma}(\robba_K(\delta_{n-1}/\delta_n))
          \]
          is surjective, in fact $\dim_KH^0_{\phi,\Gamma}(\robba_K(\delta_{n-1}/\delta_n))=1$.
          By Lemma \ref{D[1/t] splits}, $D_{n-1}(\delta_n^{-1})[1/t]$ is the split extension in
          $H^1_{\phi,\Gamma}(D_{n-2}(\delta_{n-1}^{-1})[1/t])$.
          We moreover have that by hypothesis
          \begin{align*}
            H^1_{\phi,\Gamma}(D_{n-2}(\delta_{n-1}^{-1})[1/t])\cong H^1_{\phi,\Gamma}(D_{n-2}'(\delta_{n-1}'^{-1})[1/t])
            \cong H^1_{\phi,\Gamma}(D_{n-2}'(\delta_{n-1}'^{-1})).
          \end{align*}
          We have that $D_{n-1}'(\delta_n'^{-1})$ is the image of $D_{n-1}(\delta_n^{-1})[1/t]$
          through the above isomorphism, hence this in particular implies that
          \[
            D_{n-1}'(\delta_n'^{-1})=D_{n-2}'(\delta_n'^{-1})\oplus\robba_K(\delta_{n-1}'/\delta_n')
          \]
          is also split.
          Observe that
          \begin{align*}
            H^0_{\phi,\Gamma}(D_{n-1}'(\delta_n'^{-1}))&=(D_{n-1}'(\delta_n'^{-1}))^{(\Phi,\Gamma)}\\
            &=(D_{n-2}'(\delta_n'^{-1}))^{(\Phi,\Gamma)}\oplus(\robba_K(\delta_{n-1}'/\delta_n'))^{(\Phi,\Gamma)}\\
            &=H^0_{\phi,\Gamma}(D_{n-2}'(\delta_n'^{-1}))\oplus H^0_{\phi,\Gamma}(\robba_K(\delta_{n-1}'/\delta_n')).
          \end{align*}
          By the second hypothesis, we have that
          \[
            \dim_KH^0_{\phi,\Gamma}(\robba_K(\delta_i/\delta_j))\leq\dim_KH^0_{\phi,\Gamma}(\robba_K(\delta_i'/\delta_j'))
          \]
          for all $i<j$.
          Finally, by inductive hypothesis we have
          \[
            \dim_KH^0_{\phi,\Gamma}(D_{n-2}(\delta_n^{-1}))\leq\dim_KH^0_{\phi,\Gamma}(D'_{n-2}(\delta_n'^{-1})),
          \]
          thus
          \begin{align*}
            \dim_KH^0_{\phi,\Gamma}(D_{n-1}(\delta_n^{-1}))&
            =\dim_KH^0_{\phi,\Gamma}(D_{n-2}(\delta_n^{-1}))+\dim_KH^0_{\phi,\Gamma}(\robba_K(\delta_{n-1}/\delta_n))\\
            &\leq\dim_KH^0_{\phi,\Gamma}(D'_{n-2}(\delta_n'^{-1}))+\dim_KH^0_{\phi,\Gamma}(\robba_K(\delta_{n-1}'/\delta_n'))\\
            &=\dim_KH^0_{\phi,\Gamma}(D_{n-1}'(\delta_n'^{-1})).
          \end{align*}
      \end{itemize}
    \end{itemize}
  \end{proof}

  \begin{lem}
    \label{satisfy hypothesis of bound H^0}
    Let $y=(y',\Lambda_y)\in X_n(K)$ for some extension $K$ of $L$.
    We have that the $\phigam$-modules $\D_n\widehat\otimes k(y)$
    and $\D_n'\widehat\otimes k(y')$ satisfy the hypothesis 1,2 and 3 of Lemma \ref{bound of H^0}.
  \end{lem}

  \begin{proof}
    \begin{enumerate}
      \item
        Let us denote by $(\widetilde\eta_1,\dots,\widetilde\eta_n),(\widetilde\eta_1',\dots,\widetilde\eta_n')\in\T^n(K)$
        the parameters of $\D_n\widehat\otimes k(y)$ and $\D_n'\widehat\otimes k(y')$ respectively.
        By Theorem \ref{construction of X_n} we have that
        \[
          \widetilde\eta_i=\widetilde\delta_i\widehat\otimes k(y)=x^{\mu_i}\widetilde\delta_i'\widehat\otimes k(y')=x^{\mu_i}\widetilde\eta_i'
        \]
        for all $1\leq i\leq n$.
      \item
        Let us assume that $\widetilde\eta_i/\widetilde\eta_j=x^{-a}$ for some $a\in\N$ and for some $i<j$.
        Then $\widetilde\eta_i'/\widetilde\eta_j'=\widetilde\eta_i/\widetilde\eta_j\cdot x^{-\mu_i+\mu_j}=x^{-a-\mu_i+\mu_j}$.
        Recall that by Lemma \ref{good neighbourhood} we chose the neighbourhood $\widetilde{\mathcal{U}}_n'$
        of $(\delta_1x^{-\mu_1},\dots,\delta_nx^{-\mu_n})$ in such a way so that
        $\widetilde\delta_i'/\widetilde\delta_j'\widehat\otimes k(u)=x^{a_{ij}}$ for some $a_{ij}\in\Z$ only if
        $x^{-\mu_i}\delta_i/x^{-\mu_j}\delta_j=x^{a_{ij}}$.
        Thus we have that $x^{-\mu_i}\delta_i/x^{-\mu_j}\delta_j=x^{-a-\mu_i+\mu_j}$, which implies that
        $\delta_i/\delta_j=x^{-a}\in\T^-(L)$.
        By Lemma \ref{iso of H^1 in high dim}, we have if $\delta_i/\delta_j\in\T^-(L)$
        then $x^{-\mu_i}\delta_i/x^{-\mu_j}\delta_j=x^{0}\in\T^-(L)$, hence $-a-\mu_i+\mu_j=0$.
        This concludes that $\widetilde\eta_i'/\widetilde\eta_j'=x^{-a-\mu_i+\mu_j}\in\T^-(K)$.
      \item
        Let $E_i\coloneqq\Fil^i(\D_n\widehat\otimes k(y))$ and $E_i'\coloneqq\Fil^i(\D_n'\widehat\otimes k(y'))$.
        We have by definition
        \begin{align*}
          E_i&=(\D_n\widehat\otimes k(y))\cap\Fil^i((\D_n'\widehat\otimes k(y'))[1/t])\\
          &=(\D_n\widehat\otimes k(y))\cap E_i'[1/t].
        \end{align*}
        Hence
        \[
          E_i[1/t]=(\D_n\widehat\otimes k(y))[1/t]\cap E_i'[1/t]=E_i'[1/t]
        \]
        because $(\D_n\widehat\otimes k(y))[1/t]=(\D_n'\widehat\otimes k(y'))[1/t]$.
    \end{enumerate}
  \end{proof}

  Recall the definition of the open neighbourhoods $\mcal U_i'$ of $\delta_ix^{-\mu_i}$ as
  in Proposition \ref{construction of Xn'}.
  Notice that the map $G:X_n\rightarrow \T^n$ factors through
  \[
    G:X_n\xrightarrow{G_n}X_{n-1}\times\mathcal{U}_{n}\xrightarrow{G_{n-1}}
    \dots\xrightarrow{G_3}X_2\times\mathcal{U}_3\times\dots\times\mathcal{U}_n
    \xrightarrow{G_2}\widetilde{\mathcal{U}}_n\coloneqq\mathcal{U}_{1}\times\dots\times\mathcal{U}_n
  \]
  where the spaces $\mathcal{U}_i\coloneqq x^{\mu_i}\mathcal{U}_i'\subset\T$
  are the open subspaces defined by the characters $\widetilde\delta_i$, as in Theorem \ref{construction of X_n}.
  From now on, for all $3\leq i\leq n$ we denote by $d_i$
  the dimension of the vector bundle $X_i'$ over $X_{i-1}'\times\mathcal{U}_i'$
  and by $d_2$ the dimension of the vector bundle $X_2'$ over $\mathcal{U}'_{1}\times\mathcal{U}'_2$.
  As explained in Theorem \ref{extensions}, we have that $d_i\coloneqq\dim_LH^1_{\phi,\Gamma}(D_{i-1}'(\delta_i^{-1}x^{\mu_i}))$
  and $d_2\coloneqq\dim_LH^1_{\phi,\Gamma}(\robba_L(\delta_1x^{-\mu_1}/\delta_2x^{-\mu_2}))$.

  \begin{dfn}
    We can partition the space $\T^n$
    into subspaces consisting of $(\widetilde\eta_1,\dots,\widetilde\eta_n)\in\widetilde{\mathcal{U}}_n$
    satisfying compatible conditions of the form
    \[
      \begin{array}{lcr}
        \widetilde\eta_i/\widetilde\eta_j=x^{-b_{ij}},
        &\widetilde\eta_i/\widetilde\eta_j=\chi x^{b_{ij}},
        &\widetilde\eta_i/\widetilde\eta_j\in\T^\reg
      \end{array}
    \]
    for some $b_{ij}\in\N$.
    More precisely, for any $2\leq j\leq n$, let us fix $\mathcal{L}_j,\mathcal{M}_j\subset\{1,\dots,j-1\}$.
    Moreover for every $i\in\mathcal{L}_j\cup\mathcal{M}_j$ we fix $b_{ij}\in\N$.
    We define
    \begin{equation}
      \label{def of Z}
      \mathcal{Z}_{(\mathcal{L}_j,\mathcal{M}_j)_j}^{(b_{ij})}\coloneqq\left\{
        \underline{\widetilde\eta}\in\T^n\colon
        \begin{array}{c}
          \widetilde\eta_i/\widetilde\eta_j=x^{-b_{ij}}\Leftrightarrow i\in\mathcal{L}_j,
          \widetilde\eta_i/\widetilde\eta_j=\chi x^{b_{ij}}\Leftrightarrow i\in\mathcal{M}_j\\
          \text{ and }\widetilde\eta_i/\widetilde\eta_j\in\T^\reg\Leftrightarrow i\notin\mathcal{L}_j\cup\mathcal{M}_j
        \end{array}
        \right\}.
    \end{equation}
  \end{dfn}

  We will now give a bound for the dimension of the preimage through $G$ of the
  spaces of the form $\mathcal{Z}_{(\mathcal{L}_j,\mathcal{M}_j)_j}^{(b_{ij})}$.
  We need first some preliminary results.

  \begin{lem}
    \label{only good Z}
    Let $K$ be an extension of $L$.
    We have that $\widetilde{\mathcal{U}}_n(K)$
    is covered by spaces of the form $\mathcal{Z}_{(\mathcal{L}_j,\mathcal{M}_j)_j}^{(a_{ij})}$
    where $\mathcal{L}_j\subseteq\{i<j\colon \delta_i/\delta_j\in\T^-\}$,
    $\mathcal{M}_j\subseteq\{i<j\colon \delta_i/\delta_j\in\T^+\}$
    and $a_{ij}\in\N$ such that $\delta_i/\delta_j=x^{-a_{ij}}$ or $\delta_i/\delta_j=\chi x^{a_{ij}}$.
  \end{lem}

  \begin{proof}
    We only need to show that for $\underline{\widetilde{\eta}}\in\widetilde{\mathcal{U}}_n(K)$
    we have
    \[
      \widetilde\eta_i/\widetilde\eta_j=x^{-b_{ij}}\text{ for some }b_{ij}\in\N\Rightarrow
      \delta_i/\delta_j=x^{-b_{ij}}
    \]
    and
    \[
      \widetilde\eta_i/\widetilde\eta_j=\chi x^{b_{ij}}\text{ for some }b_{ij}\in\N\Rightarrow
      \delta_i/\delta_j=\chi x^{b_{ij}}.
    \]
    Notice that
    \[
      \widetilde\eta_i/\widetilde\eta_j=x^{-b_{ij}}\Leftrightarrow
      x^{-\mu_i}\widetilde\eta_i/x^{-\mu_j}\widetilde\eta_j=x^{-b_{ij}-\mu_i+\mu_j},
    \]
    which implies by Remark \ref{good U_n'} that
    \[
      x^{-\mu_i}\delta_i/x^{-\mu_j}\delta_j=x^{-b_{ij}-\mu_i+\mu_j}\Leftrightarrow \delta_i/\delta_j=x^{-b_{ij}}.
    \]
    On the other hand,
    \[
      \widetilde\eta_i/\widetilde\eta_j=\chi x^{b_{ij}}\Leftrightarrow
      x^{-\mu_i}\widetilde\eta_i/x^{-\mu_j}\widetilde\eta_j=\chi x^{b_{ij}-\mu_i+\mu_j}.
    \]
    Observe that in the proof of Proposition \ref{n dim T-}, we have chosen $\widetilde{\mathcal{U}}_n$
    in such a way so that for every $i<j$ and every $u\in\widetilde{\mathcal{U}}_n$ we have
    $(\widetilde\delta_i'/\widetilde\delta_j')\widehat{\otimes}k(u)\notin\T^+$.
    Hence $b_{ij}-\mu_i+\mu_j<0$, which implies always by Remark \ref{good U_n'} that
    \[
      x^{-\mu_i}\delta_i/x^{-\mu_j}\delta_j=\chi x^{b_{ij}-\mu_i+\mu_j}\Leftrightarrow \delta_i/\delta_j=\chi x^{b_{ij}}.
    \]
  \end{proof}

  \begin{lem}
    \label{how many x' specializing to E}
    Let $y_i'=(y_{i-1}',\widetilde\eta_i',e)\in X_i'(K)=(X_{i-1}'\times\mathcal{U}_i'\times\mathbb{A}^{d_i})(K)$
    for some extension $K$ of $L$ (when $i=2$, we let $y_{i-1}'\coloneqq x^{-\mu_1}\widetilde{\eta}_1$ and $X_{i-1}'\coloneqq\mathcal{U}_1'$).
    Let us denote by $E_i'$ and $E_{i-1}'$ the $\phigam$-modules $\D_i'\widehat\otimes k(y_i')$
    and $\D_{i-1}'\widehat\otimes k(y_{i-1}')$ respectively.
    If $i=2$, then we denote by $\D_{i-1}'$ the $\phigam$-module $\robba_{\mathcal{U}'_1}(x^{-\mu_1}\widetilde{\delta}_1)$.
    We have that
    \begin{align*}
      \dim_K(\{z_i'\in X_i'(K)\colon \D_i'\widehat\otimes k(z_i')=E_i'\})
      =d_i-\dim_KH^1_{\phi,\Gamma}(E_{i-1}'(\widetilde\eta_i'^{-1})).
    \end{align*}
  \end{lem}

  \begin{proof}
    Notice that
    \begin{align*}
      \{z_i'\in X_i'(K)\colon \D_i'\widehat\otimes k(z_i')=E_i'\}
      &=\{(y_{i-1}',\widetilde\eta_i',e')\in X_i'(K)\colon \Psi_K(y_{i-1}',\widetilde\eta_i',e')=E_i'\}\\
      &=(y_{i-1}',\widetilde\eta_i')\times\{e'\in\mathbb{A}^{d_i}_K\colon\Psi_K(y_{i-1}',\widetilde\eta_i',e')=E_i'\},
    \end{align*}
    where $\Psi_K$ is the surjective map for $X_i'$ defined in Theorem \ref{extensions}.
    Thus
    \begin{align*}
      \dim_K(\{z_i'\in X_i'(K)\colon \D_i'\widehat\otimes k(z_i')=E_i'\})
      &=\dim_K\ker(\mathbb{A}_K^{d_i}\xrightarrowdbl{\Psi_K}H^1_{\phi,\Gamma}(E_{i-1}'(\widetilde\eta_i'^{-1})))\\
      &=d_i-\dim_K H^1_{\phi,\Gamma}(E_{i-1}'(\widetilde\eta_i'^{-1})).
    \end{align*}
  \end{proof}

  \begin{prop}
    \label{fiber dimension of G_i}
    For $2\leq i\leq n$, let $(y_{i-1},\widetilde\eta_i,\dots,\widetilde\eta_n)\in\im(G_i)(K)$ for some extension $K$ of $L$
    (if $i=2$, then $y_{i-1}=\widetilde{\eta}_1\in\T(K)$).
    We have
    \begin{align*}
      \dim_K G_i^{-1}(y_{i-1},\widetilde\eta_i,\dots,\widetilde\eta_n)&
      =\dim_KH_{\phi,\Gamma}^1((\D_{i-1}\widehat\otimes k(y_{i-1}))(\widetilde\eta_i^{-1}))\\
      &+d_i-\dim_KH_{\phi,\Gamma}^1((\D_{i-1}'\widehat\otimes k(y'_{i-1}))(x^{\mu_i}\widetilde\eta_i^{-1})).
    \end{align*}
    When $i=2$, we put $\D_{i-1}=\robba_{\mathcal{U}_1}(\widetilde{\delta}_1)$
    and $\D_{i-1}'=\robba_{\mathcal{U}_1'}(x^{-\mu_1}\widetilde{\delta}_1)$.
  \end{prop}

  \begin{proof}
    For simplicity, let us denote by $E_{i-1}\coloneqq\D_{i-1}\widehat\otimes k(y_{i-1})$
    and $E_{i-1}'\coloneqq\D_{i-1}'\widehat\otimes k(y'_{i-1})$
    (in case $i=2$, then $E_{i-1}=\robba_K(\widetilde{\eta}_1)$
    and $E_{i-1}'=\robba_K(x^{-\mu_1}\widetilde{\eta}_1)$).
    We then have that the fiber $G_i^{-1}(y_{i-1},\widetilde\eta_i,\dots,\widetilde\eta_n)$ is
    \begin{align*}
      \left\{
        (y_i',\Lambda_{y_i})\in X_i'\times Y_{(\mu_1,\dots,\mu_i)}\colon
        \begin{array}{c}
          \D_i'\widehat\otimes k(y_i')\in H_{\phi,\Gamma}^1(E_{i-1}'(x^{\mu_i}\widetilde\eta_i^{-1}))\text{ and}\\
          \Lambda_{y_i}\text{ is a }\Gamma_m\text{-stable lattice of }\\
          (\D_i'^r\widehat\otimes k(y_i'))[1/t]\otimes_{\robba_K^r[1/t],\tau}K((t))
        \end{array}
      \right\}.
    \end{align*}
    We define the set
    \begin{align*}
      S\coloneqq\left\{
        (y_i',E_i)\in X_i'\times H^1_{\phi,\Gamma}(E_{i-1}(\widetilde\eta_i^{-1}))\colon
        \begin{array}{c}
          \D_i'\widehat\otimes k(y_i')\in H_{\phi,\Gamma}^1(E_{i-1}'(x^{\mu_i}\widetilde\eta_i^{-1}))\\
          \text{such that }E_i[1/t]=(\D_i'\widehat\otimes k(y_i'))[1/t]
        \end{array}
      \right\}.
    \end{align*}
    Observe that there is a bijection
    \begin{align*}
      G_i^{-1}(y_{i-1},\widetilde\eta_i,\dots,\widetilde\eta_n)&\rightarrow S\\
      (y_i',E_i^r\otimes_{\robba_K^r,\tau}K\llbracket t\rrbracket)&\mapsfrom(y_i',E_i),
    \end{align*}
    where $E_i^r$ is a $\phigam$-module over $\robba_K^r$ such that $E_i=E_i^r\otimes_{\robba_K^r,\tau}\robba_K$.
    In fact, the map is well-defined because obviously $E_i^r\otimes_{\robba_K^r,\tau}K\llbracket t\rrbracket$
    is a $\Gamma_m$-stable lattice of $(\D_i'^r\widehat\otimes k(y_i'))[1/t]\otimes_{\robba_K^r[1/t],\tau}K((t))$
    and it is in $Y_{(\mu_1,\dots,\mu_n)}$:
    in fact
    \[
      \gr^j(E_i^r\otimes_{\robba_K^r,\tau}K\llbracket t\rrbracket)=
      t^{\mu_j}\gr^j(E_i'^r\otimes_{\robba_K^r,\tau}K\llbracket t\rrbracket)
    \]
    for all $1\leq j\leq i$ by similar computations to the ones in the proof of Lemma \ref{relative positions of lattices}.
    Then the bijection descends by the bijection of Proposition \ref{bijection of lattices and modules}.
    Finally by Lemma \ref{how many x' specializing to E} we have that
    \begin{align*}
      \dim_KG_i^{-1}(y_{i-1},\widetilde\eta_i,\dots,\widetilde\eta_n)
      =\dim_KS=&\dim_KH_{\phi,\Gamma}^1((\D_{i-1}\widehat\otimes k(y_{i-1}))(\widetilde\eta_i^{-1}))\\
      &+d_i-\dim_KH_{\phi,\Gamma}^1((\D_{i-1}'\widehat\otimes k(y'_{i-1}))(x^{\mu_i}\widetilde\eta_i^{-1})).
    \end{align*}
  \end{proof}

  \begin{lem}
    \label{bound of H^1}
    Let $D_{n-1}$ be a triangulable $\phigam$-module over $\robba_K$ of parameters
    $(\delta_1,\dots,\delta_{n-1})\in\T^{n-1}(K)$.
    Let $\delta_n\in\T(K)$ and let us define the following sets
    $\mathcal{L}\coloneqq\{i<n\colon\delta_i/\delta_n\in\T^-(K)\}$ and
    $\mathcal{M}\coloneqq\{i<n\colon\delta_i/\delta_n\in\T^+(K)\}$.
    We have
    \[
      \dim_KH^1_{\phi,\Gamma}(D_{n-1}(\delta_n^{-1}))\leq n-1+|\mathcal{L}|+|\mathcal{M}|.
    \]
  \end{lem}

  \begin{proof}
    We prove this result by induction.
    \begin{itemize}
      \item[$n=2:$]
        We distinguish two cases:
        \begin{itemize}
          \item
            if $\delta_1/\delta_2\in\T^\reg(K)$, then $\mathcal{L}=\mathcal{M}=\emptyset$
            and in fact
            \[
              \dim_K H^1_{\phi,\Gamma}(\robba_K(\delta_1/\delta_2))=1.
            \]
          \item
            If instead $\delta_1/\delta_2\notin\T^\reg(K)$, then $|\mathcal{L}|+|\mathcal{M}|=1$
            and
            \[
              2=\dim_KH^1_{\phi,\Gamma}(\robba_K(\delta_1/\delta_2))=1+|\mathcal{L}|+|\mathcal{M}|.
            \]
        \end{itemize}
      \item[$n>2:$]
        Let $D_{n-2}$ be the $n-2$-th piece of the filtration of $D_{n-1}$, so that
        $D_{n-1}\in H^1_{\phi,\Gamma}(D_{n-2}(\delta_{n-1}^{-1}))$.
        We have the following long exact sequence
        \[
          \dots\rightarrow H^1_{\phi,\Gamma}(D_{n-2}(\delta_n^{-1}))\rightarrow H^1_{\phi,\Gamma}(D_{n-1}(\delta_n^{-1}))
          \rightarrow H^1_{\phi,\Gamma}(\robba_K(\delta_{n-1}/\delta_n))\rightarrow\dots.
        \]
        Thus
        \[
          \dim_KH^1_{\phi,\Gamma}(D_{n-1}(\delta_n^{-1}))\leq\dim_KH^1_{\phi,\Gamma}(D_{n-2}(\delta_n^{-1}))
          +\dim_KH^1_{\phi,\Gamma}(\robba_K(\delta_{n-1}/\delta_n)).
        \]
        Now we distinguish three cases:
        \begin{itemize}
          \item
            if $\delta_{n-1}/\delta_n\in\T^\reg(K)$, then $\dim_KH^1_{\phi,\Gamma}(\robba_K(\delta_{n-1}/\delta_n))=1$;
            moreover $\mathcal{L}=\{i<n-1\colon\delta_i/\delta_n\in\T^-(K)\}$ and
            $\mathcal{M}=\{i<n-1\colon\delta_i/\delta_n\in\T^+(K)\}$.
            By inductive hypothesis then
            \[
              \dim_KH^1_{\phi,\Gamma}(D_{n-2}(\delta_n^{-1}))\leq n-2+|\mathcal{L}|+|\mathcal{M}|
            \]
            and hence
            \[
              \dim_KH^1_{\phi,\Gamma}(D_{n-1}(\delta_n^{-1}))\leq n-2+|\mathcal{L}|+|\mathcal{M}|+1
              =n-1+|\mathcal{L}|+|\mathcal{M}|.
            \]
          \item
            If $\delta_{n-1}/\delta_n\in\T^-(K)$,  then $\dim_KH^1_{\phi,\Gamma}(\robba_K(\delta_{n-1}/\delta_n))=2$.
            On the other hand, we have $\{i<n-1\colon\delta_i/\delta_n\in\T^-(K)\}=\mathcal{L}\setminus\{n-1\}$
            and $\mathcal{M}=\{i<n-1\colon\delta_i/\delta_n\in\T^+(K)\}$.
            By inductive hypothesis we have
            \[
              \dim_KH^1_{\phi,\Gamma}(D_{n-2}(\delta_n^{-1}))\leq n-2+|\mathcal{L}|-1+|\mathcal{M}|,
            \]
            thus
            \[
              \dim_KH^1_{\phi,\Gamma}(D_{n-1}(\delta_n^{-1}))\leq n-2+|\mathcal{L}|-1+|\mathcal{M}|+2
              =n-1+|\mathcal{L}|+|\mathcal{M}|.
            \]
          \item
            Finally if $\delta_{n-1}/\delta_n\in\T^+(K)$, we have again that
            \[
              \dim_KH^1_{\phi,\Gamma}(\robba_K(\delta_{n-1}/\delta_n))=2,
            \]
            $\mathcal{L}=\{i<n-1\colon\delta_i/\delta_n\in\T^-(K)\}$ and
            $\{i<n-1\colon\delta_i/\delta_n\in\T^+(K)\}=\mathcal{M}\setminus\{n-1\}$.
            Therefore by inductive hypothesis
            \[
              \dim_KH^1_{\phi,\Gamma}(D_{n-2}(\delta_n^{-1}))\leq n-2+|\mathcal{L}|+|\mathcal{M}|-1.
            \]
            Consequently
            \[
              \dim_KH^1_{\phi,\Gamma}(D_{n-1}(\delta_n^{-1}))\leq n-2+|\mathcal{L}|+|\mathcal{M}|-1+2
              =n-1+|\mathcal{L}|+|\mathcal{M}|.
            \]
        \end{itemize}
    \end{itemize}
  \end{proof}

  \begin{prop}
    \label{fiber dimension of G_n}
    Let $\mathcal{Z}_{(\mathcal{L}_j,\mathcal{M}_j)_j}^{(a_{ij})}\subset\T^n$ be a subspace
    defined as in (\ref{def of Z}).
    We have that for $\underline{\widetilde{\eta}}\in\mathcal{Z}_{(\mathcal{L}_j,\mathcal{M}_j)_j}^{(a_{ij})}\cap\widetilde{\mathcal{U}}_n(K)$
    \[
      \dim_K G^{-1}(\underline{\widetilde{\eta}})\leq
      \sum_{i=2}^n d_i +\sum_{j=2}^n(|\mathcal{L}_j|+|\mathcal{M}_j|).
    \]
  \end{prop}

  \begin{proof}
    Let $\underline{\widetilde{\eta}}\in\mathcal{Z}_{(\mathcal{L}_j,\mathcal{M}_j)_j}^{(a_{ij})}\cap\widetilde{\mathcal{U}}_n(K)$.
    For $2\leq i\leq n$, let $(y_{i-1},\widetilde\eta_i,\dots,\widetilde\eta_n)\in\im(G_i)(K)$
    (if $i=2$ we let $y_{i-1}=\widetilde{\eta}_1$,
    $\D_{i-1}=\mathcal{R}_{\mathcal{U}_1}(\widetilde{\delta}_1)$
    and $\D_{i-1}'=\mathcal{R}_{\mathcal{U}_1'}(x^{-\mu_1}\widetilde{\delta}_1)$).
    By Proposition \ref{fiber dimension of G_i} and Lemma \ref{bound of H^1} we have
    \begin{align*}
      \dim_K G_i^{-1}(y_{i-1},\widetilde\eta_i,\dots,\widetilde\eta_n)&
      =\dim_KH_{\phi,\Gamma}^1((\D_{i-1}\widehat\otimes k(y_{i-1}))(\widetilde\eta_i^{-1}))\\
      &+d_i-\dim_KH_{\phi,\Gamma}^1((\D_{i-1}'\widehat\otimes k(y'_{i-1}))(x^{\mu_i}\widetilde\eta_i^{-1}))\\
      &\leq i-1+|\mathcal{L}_{i}|+|\mathcal{M}_i|\\
      &+d_i-\dim_KH_{\phi,\Gamma}^1((\D_{i-1}'\widehat\otimes k(y'_{i-1}))(x^{\mu_i}\widetilde\eta_i^{-1})).
    \end{align*}
    On the other hand, by Theorem \ref{poincare duality}, we have
    \[
      \dim_KH_{\phi,\Gamma}^1((\D_{i-1}'\widehat\otimes k(y'_{i-1}))(x^{\mu_i}\widetilde\eta_i^{-1}))\geq i-1,
    \]
    thus
    \begin{align*}
      \dim_K G_i^{-1}(y_{i-1},\widetilde\eta_i,\dots,\widetilde\eta_n)&\leq i-1+|\mathcal{L}_{i}|+|\mathcal{M}_i|\\
      &+d_i-\dim_KH_{\phi,\Gamma}^1((\D_{i-1}'\widehat\otimes k(y'_{i-1}))(x^{\mu_i}\widetilde\eta_i^{-1}))\\
      &\leq d_i+|\mathcal{L}_{i}|+|\mathcal{M}_i|.
    \end{align*}
    This proves the wanted inequality.
  \end{proof}

  It is possible to compute the exact dimension of the preimage of the regular locus through $F$.

  \begin{prop}
    \label{dim of Wreg}
    We have
    \[
      \dim_KF^{-1}(\W_n^\reg)=2n+\sum_{i=2}^nd_i.
    \]
  \end{prop}

  \begin{proof}
    It is easy to see that
    \[
      \dim_K(\W_n^\reg\cap\im(H_n))=n=\dim_K(H_n^{-1}(\W_n^\reg)).
    \]
    Notice that $H_n^{-1}(\W_n^\reg)\subset\T_n^\reg$ by definition of $\W_n^\reg$.
    Fix $\underline{\widetilde\eta}\in H_n^{-1}(\W_n^\reg)\cap\widetilde{\mathcal{U}}_n(K)$.
    For $2\leq i\leq n$, let $(y_{i-1},\widetilde\eta_i,\dots,\widetilde\eta_n)\in\im(G_i)(K)\cap
    G_{i-1}^{-1}\circ\dots\circ G_2^{-1}(\underline{\widetilde\eta})$
    (if $i=2$, we have $y_{i-1}=\widetilde{\eta}_1\in\T(K)$).
    By Proposition \ref{fiber dimension of G_i} we have
    \begin{align*}
      \dim G_i^{-1}(y_{i-1},\widetilde\eta_i,\dots,\widetilde\eta_n)&
      =\dim_KH_{\phi,\Gamma}^1((\D_{i-1}\widehat\otimes k(y_{i-1}))(\widetilde\eta_i^{-1}))\\
      &+d_i-\dim_KH_{\phi,\Gamma}^1((\D_{i-1}'\widehat\otimes k(y'_{i-1}))(x^{\mu_i}\widetilde\eta_i^{-1})).
    \end{align*}
    Since $(y_{i-1},\widetilde\eta_i,\dots,\widetilde\eta_n)\in G_{i-1}^{-1}\circ\dots\circ G_2^{-1}(\underline{\widetilde\eta})$,
    we have that the parameters of the $\phigam$-module $\D_{i-1}\widehat\otimes k(y_{i-1})$
    are exactly $(\widetilde\eta_1,\dots,\widetilde\eta_{i-1})$.
    Since $\underline{\widetilde\eta}\in\T_n^\reg$, we have that
    \[
      \mathcal{L}_i\coloneqq\{j<i\colon \widetilde\eta_j/\widetilde\eta_i\in\T^-(K)\}=\emptyset
    \]
    and
    \[
      \mathcal{M}_i\coloneqq\{j<i\colon \widetilde\eta_j/\widetilde\eta_i\in\T^+(K)\}=\emptyset
    \]
    for all $2\leq i\leq n$.
    Thus by Lemma \ref{bound of H^1}, we have
    \[
      \dim_KH_{\phi,\Gamma}^1((\D_{i-1}\widehat\otimes k(y_{i-1}))(\widetilde\eta_i^{-1}))\leq i-1.
    \]
    On the other hand, by Theorem \ref{poincare duality}, we know that
    \[
      \dim_KH_{\phi,\Gamma}^1((\D_{i-1}\widehat\otimes k(y_{i-1}))(\widetilde\eta_i^{-1}))\geq i-1,
    \]
    allowing us to conclude that the above inequalities are actually equalities.
    Furthermore we have that $\D_{i-1}'\widehat\otimes k(y'_{i-1})$
    has parameters $(\widetilde\eta_1x^{-\mu_1},\dots,\widetilde\eta_{i-1}x^{-\mu_{i-1}})$.
    Since $\widetilde\eta_i/\widetilde\eta_j\notin\{x^a\colon a\in\Z\}\cup\{\chi x^a\colon a\in\Z\}$
    for all $i<j$
    (because $\underline{\widetilde\eta}\in H_n^{-1}(\W_n^\reg)$),
    we have that $x^{-\mu_i}\widetilde\eta_i/x^{-\mu_j}\widetilde\eta_j\in\T^\reg(K)$ for all $i<j$.
    Therefore the same exact argument as above applies to
    $\dim_KH_{\phi,\Gamma}^1((\D_{i-1}'\widehat\otimes k(y'_{i-1}))(x^{\mu_i}\widetilde\eta_i^{-1}))$
    and we can conclude that
    \[
      \dim_KH_{\phi,\Gamma}^1((\D_{i-1}'\widehat\otimes k(y'_{i-1}))(x^{\mu_i}\widetilde\eta_i^{-1}))=i-1.
    \]
    Finally, we have that
    \begin{align*}
      \dim G_i^{-1}(y_{i-1},\widetilde\eta_i,\dots,\widetilde\eta_n)&
      =\dim_KH_{\phi,\Gamma}^1((\D_{i-1}\widehat\otimes k(y_{i-1}))(\widetilde\eta_i^{-1}))\\
      &+d_i-\dim_KH_{\phi,\Gamma}^1((\D_{i-1}'\widehat\otimes k(y'_{i-1}))(x^{\mu_i}\widetilde\eta_i^{-1}))\\
      &=d_i.
    \end{align*}
    This ends the proof.
  \end{proof}

  \begin{rem}
    \label{Wreg and X' same dim}
    In particular observe that
    \begin{align*}
      \dim F^{-1}(\W_n^\reg)=2n+\sum_{i=2}^nd_i
      =\dim(\widetilde{\mathcal{U}}_n'\times\mathbb{A}_L^{d_2,\an}\times\dots\times\mathbb{A}_L^{d_n,\an})
      =\dim X'_n.
    \end{align*}
  \end{rem}

  Now our goal is to show that the dimension of the preimage through $F$ of certain
  subspaces of $\W^n$ is strictly less than the dimension of the preimage of the regular
  locus $\W_n^\reg$.

  Like in the case of $\T^n$, we define subspaces of $\W^n$ consisting of characters $(\eta_1,\dots,\eta_n)$
  satisfying conditions of the form:
  \[
    \begin{array}{lcr}
      \eta_i/\eta_j=x^{b_{ij}},
      &\eta_i/\eta_j=x^{-b_{ij}}
    \end{array}
  \]
  for some $b_{ij}\in\N$.

  \begin{dfn}
    For any $2\leq j\leq n$, fix  $\mathcal{L}_j,\mathcal{M}_j\subset\{1,\dots,j-1\}$;
    for all $2\leq j\leq n$ and for all $i\in\mathcal{L}_j\cup\mathcal{M}_j$
    let us fix $b_{ij}\in\N$.
    We define
    \begin{equation}
      \label{def of W_(L,M)}
      \mathcal{W}_{(\mathcal{L}_j,\mathcal{M}_j)_j}^{(b_{ij})}\coloneqq\left\{\underline{\eta}\in\W^n\colon
      \begin{array}{c}
      \eta_i/\eta_j=x^{-b_{ij}}\Leftrightarrow i\in\mathcal{L}_j\text{ and }\\
      \eta_i/\eta_j=x^{b_{ij}+1}\Leftrightarrow i\in\mathcal{M}_j
      \end{array}
      \right\}.
    \end{equation}
  \end{dfn}

  We then have that the union of such spaces covers the whole $\W^n$.
  Since we are again interested in the dimension of $F^{-1}(\mathcal{W}_{(\mathcal{L}_j,\mathcal{M}_j)_j}^{(b_{ij})})$,
  we want to consider only those spaces such that
  $H_n^{-1}(\mathcal{W}_{(\mathcal{L}_j,\mathcal{M}_j)_j}^{(b_{ij})})\cap \widetilde{\mathcal{U}}_n\neq\emptyset$.

  \begin{lem}
    \label{good Z for H^-1(W)}
    Let $\mathcal{W}_{(\mathcal{L}_j,\mathcal{M}_j)_j}^{(b_{ij})}$ be the set defined in (\ref{def of W_(L,M)}).
    We have that the space $H_n^{-1}(\mathcal{W}_{(\mathcal{L}_j,\mathcal{M}_j)_j}^{(b_{ij})})$ is covered
    by those $\mathcal{Z}_{(\mathcal{L}_j',\mathcal{M}_j')_j}^{(b_{ij})}$
    such that $\mathcal{L}_j'\subseteq\mathcal{L}_j$ and $\mathcal{M}_j'\subseteq\mathcal{M}_j$
    for all $2\leq j\leq n$.
  \end{lem}

  \begin{proof}
    Let $\underline{\widetilde{\eta}}\in H_n^{-1}(\mathcal{W}_{(\mathcal{L}_j,\mathcal{M}_j)_j}^{(b_{ij})})$.
    Using the isomorphism $\T\cong\W\times \mathbb{A}^1$, we write
    $\underline{\widetilde{\eta}}=(\eta_i,a_i)_{i=1,\dots,n}$ where $\eta_i\coloneqq\widetilde{\eta_i}_{|\Z_p^\times}$
    and $a_i\coloneqq\widetilde{\eta_i}(p)$.
    Then for any $2\leq j\leq n$ and any $i\in\mathcal{L}_j$, we let $i\in\mathcal{L}_j'$
    if and only if $a_i/a_j=p^{-b_{ij}}$.
    In the same way, for any $2\leq j\leq n$ and any $i\in\mathcal{M}_j$, we let $i\in\mathcal{M}_j'$
    if and only if $a_i/a_j=p^{b_{ij}}$.
    It is easy to see that $\underline{\widetilde{\eta}}\in \mathcal{Z}_{(\mathcal{L}_j',\mathcal{M}_j')_j}^{(b_{ij})}$.
  \end{proof}

  \begin{lem}
    \label{good W}
    Let $K$ be an extension of $L$ and let $\mathcal{W}_{(\mathcal{L}_j,\mathcal{M}_j)_j}^{(b_{ij})}$
    be as in (\ref{def of W_(L,M)}).
    Then
    $H_n^{-1}(\mathcal{W}_{(\mathcal{L}_j,\mathcal{M}_j)_j}^{(b_{ij})})\cap\widetilde{\mcal U}_n(K)$ is covered
    by spaces of the form $\mathcal{Z}_{(\mathcal{L}_j',\mathcal{M}_j')_j}^{(a_{ij})}$,
    where
    \[
      \mcal L'_j\subseteq\mcal L_j\cap\{i<j\colon b_{ij}=a_{ij}\}\cap\{i<j\colon \delta_i/\delta_j\in\T^-\}
    \]
    and
    \[
      \mcal M'_j\subseteq\mcal M_j\cap\{i<j\colon b_{ij}=a_{ij}\}\cap\{i<j\colon \delta_i/\delta_j\in\T^+\}.
    \]
  \end{lem}

  \begin{proof}
    Let $\widetilde{\underline{\eta}}\in H_n^{-1}(\mathcal{W}_{(\mathcal{L}_j,\mathcal{M}_j)_j}^{(b_{ij})})\cap\widetilde{\mcal U}_n(K)$.
    As proved in Lemma \ref{good Z for H^-1(W)}, we have
    $\widetilde{\underline{\eta}}\in \mathcal{Z}_{(\mathcal{L}_j',\mathcal{M}_j')_j}^{(b_{ij})}$ for some
    $\mathcal{L}_j'\subseteq\mathcal{L}_j$ and $\mathcal{M}_j'\subseteq\mathcal{M}_j$.
    By Lemma \ref{only good Z}, we have that
    $\widetilde{\underline{\eta}}\in \mathcal{Z}_{(\mathcal{L}_j',\mathcal{M}_j')_j}^{(b_{ij})}
    \cap \mathcal{Z}_{(\mathcal{L}_j'',\mathcal{M}_j'')_j}^{(a_{ij})}$ for some
    $\mcal L''_j\subseteq\{i<j\colon \delta_i/\delta_j\in\T^-\}$ and
    $\mcal M''_j\subseteq\{i<j\colon \delta_i/\delta_j\in\T^+\}$.
    As the sets of the form $\mathcal{Z}_{(\mathcal{L}_j,\mathcal{M}_j)_j}^{(b_{ij})}$ are disjoint,
    we must have that $\mathcal{Z}_{(\mathcal{L}_j',\mathcal{M}_j')_j}^{(b_{ij})}
    =\mathcal{Z}_{(\mathcal{L}_j'',\mathcal{M}_j'')_j}^{(a_{ij})}$; in particular
    we have $\widetilde{\underline{\eta}}\in\mathcal{Z}_{(\mathcal{L}_j',\mathcal{M}_j')_j}^{(a_{ij})}$
    with $\mcal L'_j\subseteq\mcal L_j\cap\{i<j\colon b_{ij}=a_{ij}\}\cap\{i<j\colon \delta_i/\delta_j\in\T^-\}$ and
    $\mcal M'_j\subseteq\mcal M_j\cap\{i<j\colon b_{ij}=a_{ij}\}\cap\{i<j\colon \delta_i/\delta_j\in\T^+\}$
    as wanted.
  \end{proof}

  \begin{lem}
    \label{dim of H^-1(eta)cap Z}
    Let us fix $\underline{\eta}\in\mathcal{W}_{(\mathcal{L}_j,\mathcal{M}_j)_j}^{(b_{ij})}(K)$.
    We have
    \[
      \dim_K\mathcal{W}_{(\mathcal{L}_j,\mathcal{M}_j)_j}^{(b_{ij})}
      =\dim_K(H_n^{-1}(\underline{\eta})\cap \mathcal{Z}_{(\mathcal{L}_j,\mathcal{M}_j)_j}^{(b_{ij})}).
    \]
  \end{lem}

  \begin{proof}
    Recall that
    \[
      \mathcal{W}_{(\mathcal{L}_j,\mathcal{M}_j)_j}^{(b_{ij})}\coloneqq\left\{\underline{\eta}\in\W^n(K)\colon
    \begin{array}{c}
    \eta_i/\eta_j=x^{-b_{ij}}\Leftrightarrow i\in\mathcal{L}_j\text{ and }\\
    \eta_i/\eta_j=x^{b_{ij}+1}\Leftrightarrow i\in\mathcal{M}_j
    \end{array}
    \right\}.
    \]
    Using the isomorphism $\T\cong\W\times\mathbb{A}^1$, we get
    \begin{align*}
      H_n^{-1}(\underline{\eta})\cap \mathcal{Z}_{(\mathcal{L}_j,\mathcal{M}_j)_j}^{(b_{ij})}
      \cong\left\{\underline{a}\in\mathbb{A}^n(K)\colon
      \begin{array}{c}
        a_i/a_j=p^{-b_{ij}}\Leftrightarrow i\in\mathcal{L}_j\text{ and}\\
        a_i/a_j=p^{b_{ij}}\Leftrightarrow i\in\mathcal{M}_j
      \end{array}
      \right\}.
    \end{align*}
    It is then obvious that these two spaces are isomorphic and hence have the same dimension.
  \end{proof}

  \begin{thm}
    \label{dim F-1(W)< dim regular locus for n=3}
    If $\sum_{j=2}^n|\mathcal{L}_j|+|\mathcal{M}_j|\leq 5$, then
    \[
      \dim_KF^{-1}(\mathcal{W}_{(\mathcal{L}_j,\mathcal{M}_j)_j}^{(b_{ij})})<\dim_KF^{-1}(\W_n^\reg).
    \]
  \end{thm}

  \begin{proof}
    Let us fix $\mathcal{W}_{(\mathcal{L}_j,\mathcal{M}_j)_j}^{(b_{ij})}$.
    We are going to show that for any $\mathcal{Z}_{(\mathcal{L}_j',\mathcal{M}_j')_j}^{(a_{ij})}$
    with $\mathcal{L}_j'\subseteq\mathcal{L}_j\cap \{i<j\colon b_{ij}=a_{ij}\}\cap\{i<j\colon \delta_i/\delta_j\in\T^-\}$
    and $\mathcal{M}_j'\subseteq\mathcal{M}_j\cap \{i<j\colon b_{ij}=a_{ij}\}\cap\{i<j\colon \delta_i/\delta_j\in\T^+\}$
    for all $2\leq j\leq n$ such that $\widetilde{\mathcal{U}}_n\cap\mathcal{Z}_{(\mathcal{L}_j',\mathcal{M}_j')_j}^{(a_{ij})}\neq\emptyset$,
    for any $\underline{\eta}\in\mathcal{W}_{(\mathcal{L}_j,\mathcal{M}_j)_j}^{(b_{ij})}(K)$
    and for any
    $\underline{\widetilde{\eta}}\in (H_n^{-1}(\underline{\eta})\cap\mathcal{Z}_{(\mathcal{L}_j',\mathcal{M}_j')_j}^{(a_{ij})})(K)
    \cap\widetilde{\mathcal{U}}_n$,
    we have
    \begin{equation}
      \dim_K\mathcal{W}_{(\mathcal{L}_j,\mathcal{M}_j)_j}^{(b_{ij})}+
      \dim_K(H_n^{-1}(\underline{\eta})\cap\mathcal{Z}_{(\mathcal{L}_j',\mathcal{M}_j')_j}^{(a_{ij})})
      +\dim_KG^{-1}(\underline{\widetilde{\eta}})<\dim_KF^{-1}(\W_n^\reg).
    \end{equation}
    By Lemma \ref{good W} this will prove the claim of the Theorem.
    By Proposition \ref{fiber dimension of G_n}, we have
    \[
      \dim G^{-1}(\underline{\widetilde{\eta}})\leq
      \sum_{i=2}^n d_i +\sum_{j=2}^n(|\mathcal{L}_j'|+|\mathcal{M}_j'|)
    \]
    and by Lemma \ref{dim of H^-1(eta)cap Z}, we have
    \[
      \dim_K\mathcal{W}_{(\mathcal{L}_j',\mathcal{M}_j')_j}^{(a_{ij})}
      =\dim_K(H_n^{-1}(\underline{\eta})\cap \mathcal{Z}_{(\mathcal{L}_j',\mathcal{M}_j)_j'}^{(a_{ij})}).
    \]
    Hence it is enough to show
    \begin{equation}
      \label{wanted inequality}
      \dim_K\mathcal{W}_{(\mathcal{L}_j,\mathcal{M}_j)_j}^{(b_{ij})}+
      \dim_K\mathcal{W}_{(\mathcal{L}_j',\mathcal{M}_j')_j}^{(a_{ij})}
      +\sum_{i=2}^n d_i +\sum_{j=2}^n(|\mathcal{L}_j'|+|\mathcal{M}_j'|)<\dim_KF^{-1}(\W_n^\reg).
    \end{equation}
    Observe that $\dim_K\mathcal{W}_{(\mathcal{L}_j,\mathcal{M}_j)_j}^{(b_{ij})}=\dim_K\mathcal{W}_{(\mathcal{L}_j,\mathcal{M}_j)_j}^{(c_{ij})}$
    for any $b_{ij},c_{ij}\in\Z$, in fact the dimension of these spaces only depends on
    the sets $\mcal L_j,\mcal M_j$.
    We claim now that it is enough to show the above inequality only for $\mathcal{L}_j'=\mathcal{L}_j$
    and $\mathcal{M}_j'=\mathcal{M}_j$:
    assume in fact that for any $\mathcal{W}_{(\mathcal{L}_j',\mathcal{M}_j')_j}^{(b_{ij})}$ we have proved
    \[
      2\dim_K\mathcal{W}_{(\mathcal{L}_j',\mathcal{M}_j')_j}^{(b_{ij})}
      +\sum_{i=2}^n d_i +\sum_{j=2}^n(|\mathcal{L}_j'|+|\mathcal{M}_j'|)<\dim_KF^{-1}(\W_n^\reg).
    \]
    If $\mathcal{L}_j'\subseteq\mathcal{L}_j$ and $\mathcal{M}_j'\subseteq\mathcal{M}_j$
    for all $2\leq j\leq n$, then
    \[
      \dim_K\mathcal{W}_{(\mathcal{L}_j,\mathcal{M}_j)_j}^{(b_{ij})}
    \leq \dim_K\mathcal{W}_{(\mathcal{L}_j',\mathcal{M}_j')_j}^{(c_{ij})}
    \]
    for any $b_{ij},c_{ij}\in\Z$.
    We thus get the wanted inequality (\ref{wanted inequality}) for any $\mathcal{W}_{(\mathcal{L}_j',\mathcal{M}_j')_j}^{(a_{ij})}$
    such that $\mathcal{L}_j'\subsetneq\mathcal{L}_j$ and $\mathcal{M}_j'\subsetneq\mathcal{M}_j$.
    It is then left to prove the inequality (\ref{wanted inequality}) when
    $\mathcal{L}_j'=\mathcal{L}_j$ and $\mathcal{M}_j'=\mathcal{M}_j$.
    In other words, we want to show
    \begin{equation}
      \label{better inequality}
      2\dim_K\mathcal{W}_{(\mathcal{L}_j,\mathcal{M}_j)_j}^{(b_{ij})}+\sum_{i=2}^n d_i +\sum_{j=2}^n(|\mathcal{L}_j|+|\mathcal{M}_j|)
      <\dim_KF^{-1}(\W_n^\reg)=2n+\sum_{i=2}^nd_i.
    \end{equation}
    Now we proceed by cases:
      \begin{itemize}
        \item
          if $\sum_{j=1}^n|\mathcal{L}_j|+|\mathcal{M}_j|=1$, then we have
          $\mathcal{W}_{(\mathcal{L}_j,\mathcal{M}_j)_j}^{(b_{ij})}$ is cut out by only one equation
          inside $\W^n$, hence $\dim_K\mathcal{W}_{(\mathcal{L}_j,\mathcal{M}_j)_j}^{(b_{ij})}=n-1$.
          Therefore we have
          \begin{align*}
            2\dim_K\mathcal{W}_{(\mathcal{L}_j,\mathcal{M}_j)_j}^{(b_{ij})}+\sum_{j=2}^n(|\mathcal{L}_j|+|\mathcal{M}_j|)
            =2n-1<2n.
          \end{align*}
        \item
          If $\sum_{j=1}^n|\mathcal{L}_j|+|\mathcal{M}_j|=2$, then
          $\dim_K\mathcal{W}_{(\mathcal{L}_j,\mathcal{M}_j)_j}^{(b_{ij})}= n-2$
          because we are cutting out two independent equations.
          As a consequence
          \begin{align*}
            2\dim_K\mathcal{W}_{(\mathcal{L}_j,\mathcal{M}_j)_j}^{(b_{ij})}+\sum_{j=2}^n(|\mathcal{L}_j|+|\mathcal{M}_j|)
            \leq 2n-2<2n.
          \end{align*}
        \item
          If $\sum_{j=1}^n|\mathcal{L}_j|+|\mathcal{M}_j|=3$, then
          $\dim_K\mathcal{W}_{(\mathcal{L}_j,\mathcal{M}_j)_j}^{(b_{ij})}\leq n-2$,
          because it could be that there are only two independent equations and that one of the three
          is a consequence of the other two.
          More precisely, we could have for example $\eta_{i}/\eta_j,\eta_k/\eta_j\in x^\Z$
          for some $i<k<j$: in this case, we will also have the equation $\eta_i/\eta_k\in x^\Z$,
          but it depends on the previous two.
          Thus we have
          \begin{align*}
            2\dim_K\mathcal{W}_{(\mathcal{L}_j,\mathcal{M}_j)_j}^{(b_{ij})}+\sum_{j=2}^n(|\mathcal{L}_j|+|\mathcal{M}_j|)
            \leq 2n-1<2n.
          \end{align*}
        \item
          If $\sum_{j=1}^n|\mathcal{L}_j|+|\mathcal{M}_j|=4$, then
          $\dim_K\mathcal{W}_{(\mathcal{L}_j,\mathcal{M}_j)_j}^{(b_{ij})}\leq n-3$,
          because we have at least three independent equations.
          So we get
          \begin{align*}
            2\dim_K\mathcal{W}_{(\mathcal{L}_j,\mathcal{M}_j)_j}^{(b_{ij})}+\sum_{j=2}^n(|\mathcal{L}_j|+|\mathcal{M}_j|)
            \leq 2n-2<2n.
          \end{align*}
        \item
          Finally, if $\sum_{j=1}^n|\mathcal{L}_j|+|\mathcal{M}_j|=5$, then
          $\dim_K\mathcal{W}_{(\mathcal{L}_j,\mathcal{M}_j)_j}^{(b_{ij})}\leq n-3$,
          because two of the five equations could depend on the other three.
          So we get
          \begin{align*}
            2\dim_K\mathcal{W}_{(\mathcal{L}_j,\mathcal{M}_j)_j}^{(b_{ij})}+\sum_{j=2}^n(|\mathcal{L}_j|+|\mathcal{M}_j|)
            \leq 2n-1<2n.
          \end{align*}
      \end{itemize}
      For all of the above cases we have then showed that the inequality (\ref{better inequality}) holds.
  \end{proof}

  \begin{cor}
    \label{dim F-1(W) for regular enough}
    If $|\{\delta_i/\delta_j\in\T\setminus\T^\reg\colon i<j\}|\leq5$,
    then for any $\mathcal{W}_{(\mathcal{L}_j,\mathcal{M}_j)_j}^{(b_{ij})}$ we have
    \[
      \dim_KF^{-1}(\mathcal{W}_{(\mathcal{L}_j,\mathcal{M}_j)_j}^{(b_{ij})})<\dim_KF^{-1}(\W_n^\reg).
    \]
  \end{cor}

  \begin{proof}
    The claim is true if $\sum_{j=2}^n|\mathcal{L}_j|+|\mathcal{M}_j|\leq 5$ by Theorem \ref{dim F-1(W)< dim regular locus for n=3}.
    Suppose now that $\sum_{j=2}^n|\mathcal{L}_j|+|\mathcal{M}_j|>5$.
    Exactly as in the proof of Theorem \ref{dim F-1(W)< dim regular locus for n=3},
    we want to show that for any $\mathcal{Z}_{(\mathcal{L}_j',\mathcal{M}_j')_j}^{(a_{ij})}$
    with $\mathcal{L}_j'\subseteq\mathcal{L}_j\cap \{i<j\colon b_{ij}=a_{ij}\}\cap\{i<j\colon \delta_i/\delta_j\in\T^-\}$
    and $\mathcal{M}_j'\subseteq\mathcal{M}_j\cap \{i<j\colon b_{ij}=a_{ij}\}\cap\{i<j\colon \delta_i/\delta_j\in\T^+\}$
    for all $2\leq j\leq n$ such that $\widetilde{\mathcal{U}}_n\cap\mathcal{Z}_{(\mathcal{L}_j',\mathcal{M}_j')_j}^{(a_{ij})}\neq\emptyset$,
    for any $\underline{\eta}\in\mathcal{W}_{(\mathcal{L}_j,\mathcal{M}_j)_j}^{(b_{ij})}(K)$
    and for any
    $\underline{\widetilde{\eta}}\in (H_n^{-1}(\underline{\eta})\cap\mathcal{Z}_{(\mathcal{L}_j',\mathcal{M}_j')_j}^{(a_{ij})})(K)
    \cap\widetilde{\mathcal{U}}_n$,
    we have
    \begin{equation}
      \label{equation for regular enough}
      \dim_K\mathcal{W}_{(\mathcal{L}_j,\mathcal{M}_j)_j}^{(b_{ij})}+
      \dim_K(H_n^{-1}(\underline{\eta})\cap\mathcal{Z}_{(\mathcal{L}_j',\mathcal{M}_j')_j}^{(a_{ij})})
      +\dim_KG^{-1}(\underline{\widetilde{\eta}})<\dim_KF^{-1}(\W_n^\reg).
    \end{equation}
    By Lemma \ref{good W} this will prove the claim of the Corollary.
    Since $\mcal L_j'\subseteq\mcal L_j$ and $\mcal M_j'\subseteq\mcal M_j$ for all $2\leq j\leq n$,
    we have $\dim_K\mathcal{W}_{(\mathcal{L}_j,\mathcal{M}_j)_j}^{(b_{ij})}\leq\dim_K\mathcal{W}_{(\mathcal{L}_j',\mathcal{M}_j')_j}^{(b_{ij})}$.
    Now notice that by construction $\sum_{j=2}^n|\mcal L_j'|+|\mcal M_j'|\leq 5$ and
    in the proof of Theorem \ref{dim F-1(W)< dim regular locus for n=3} we have showed that
    \[
      \dim_K\mathcal{W}_{(\mathcal{L}_j',\mathcal{M}_j')_j}^{(b_{ij})}+
      \dim_K(H_n^{-1}(\underline{\eta})\cap\mathcal{Z}_{(\mathcal{L}_j',\mathcal{M}_j')_j}^{(a_{ij})})
      +\dim_KG^{-1}(\underline{\widetilde{\eta}})<\dim_KF^{-1}(\W_n^\reg)
    \]
    holds.
    This implies that (\ref{equation for regular enough}) holds too.
  \end{proof}

  \begin{lem}
    \label{bound for dim H^2}
    Let $D_{n-1}$ be a triangulable $\phigam$-module over $\robba_K$
    of parameters $(\delta_1,\dots,\delta_{n-1})\in\T^{n-1}(K)$ for some extension $K$ of $L$.
    Let $\delta_n\in\T(K)$ and we define
    $\mathcal{M}\coloneqq\{i<n\colon \delta_i/\delta_n\in\T^+(K)\}$.
    We have that
    \[
      \dim_KH^2_{\phi,\Gamma}(D_{n-1}(\delta_n^{-1}))\leq|\mathcal{M}|.
    \]
  \end{lem}

  \begin{proof}
    We proceed by induction:
    \begin{itemize}
      \item[$n=2:$]
        If $\delta_1/\delta_2\notin\T^+(K)$,
        then $\dim_KH^2_{\phi,\Gamma}(\robba_K(\delta_1/\delta_2))=0$, so the claim is obviously true.
        If $\delta_1/\delta_2\in\T^+(K)$, then $\mathcal{M}=\{1\}$ and
        $\dim_KH^2_{\phi,\Gamma}(\robba_K(\delta_1/\delta_2))=1=|\mathcal{M}|$.
      \item[$n>2$:]
        Let $D_{n-2}$ be the $n-2$-th piece of the filtration of $D_{n-1}$ giving rise to the parameters
        $(\delta_1,\dots,\delta_{n-1})$; we have the long exact sequence
        \[
          \dots \rightarrow H^2_{\phi,\Gamma}(D_{n-2}(\delta_n^{-1}))\rightarrow H^2_{\phi,\Gamma}(D_{n-1}(\delta_n^{-1}))
          \rightarrow H^2_{\phi,\Gamma}(\robba_K(\delta_{n-1}/\delta_n))\rightarrow 0.
        \]
        As a consequence
        \[
          \dim_KH^2_{\phi,\Gamma}(D_{n-1}(\delta_n^{-1}))\leq
          \dim_KH^2_{\phi,\Gamma}(D_{n-2}(\delta_n^{-1}))+\dim_KH^2_{\phi,\Gamma}(\robba_K(\delta_{n-1}/\delta_n)).
        \]
        By inductive hypothesis we have
        \[
          \dim_KH^2_{\phi,\Gamma}(D_{n-2}(\delta_n^{-1}))\leq |\mathcal{M}'|,
        \]
        where $\mathcal{M}'\coloneqq\{i<n-1\colon \delta_i/\delta_n\in\T^+(K)\}$.
        \begin{itemize}
          \item
            If $\delta_{n-1}/\delta_n\notin\T^+(K)$, then $\mathcal{M}=\mathcal{M}'$ and
            \[
              \dim_KH^2_{\phi,\Gamma}(D_{n-1}(\delta_n^{-1}))\leq
              \dim_KH^2_{\phi,\Gamma}(D_{n-2}(\delta_n^{-1}))\leq |\mathcal{M}'|=|\mathcal{M}|.
            \]
          \item
            If $\delta_{n-1}/\delta_n\in\T^+(K)$, then $\mathcal{M}=\mathcal{M}'\cup\{n-1\}$ and
            \[
              \dim_KH^2_{\phi,\Gamma}(D_{n-1}(\delta_n^{-1}))\leq
              \dim_KH^2_{\phi,\Gamma}(D_{n-2}(\delta_n^{-1}))+1\leq |\mathcal{M}'|+1=|\mathcal{M}|.
            \]
        \end{itemize}
    \end{itemize}
  \end{proof}

  \begin{thm}
    \label{dim F-1(W)< dim regular locus for n=4}
    Let $n=4$.
    Let $\mathcal{L}_j\coloneqq \{i<j\colon (\delta_i/\delta_j)_{|\Z_p^\times}\in x^{-\N}\}$ and
    $\mathcal{M}_j\coloneqq\{i<j\colon (\delta_i/\delta_j)_{|\Z_p^\times}\in x^{\N_{>0}}\}$ for any
    $2\leq j\leq 4$.
    If
    \begin{equation}
      \label{condition}
      2\dim_K\mathcal{W}_{(\mathcal{L}_j,\mathcal{M}_j)_j}^{(b_{ij})}+\sum_{j=2}^4|\mathcal{M}_j|<8,
    \end{equation}
    then
    \[
      \dim_KF^{-1}(\mathcal{W}_{(\mathcal{L}_j,\mathcal{M}_j)_j}^{(b_{ij})})<\dim_KF^{-1}(\W_4^\reg).
    \]
  \end{thm}

  \begin{proof}
    Again, we want to show that for any $\mathcal{Z}_{(\mathcal{L}_j',\mathcal{M}_j')_j}^{(a_{ij})}$
    with $\mathcal{L}_j'\subseteq\mathcal{L}_j\cap \{i<j\colon b_{ij}=a_{ij}\}\cap\{i<j\colon \delta_i/\delta_j\in\T^-\}$
    and $\mathcal{M}_j'\subseteq\mathcal{M}_j\cap \{i<j\colon b_{ij}=a_{ij}\}\cap\{i<j\colon \delta_i/\delta_j\in\T^+\}$
    for all $2\leq j\leq 4$ such that $\widetilde{\mathcal{U}}_4\cap\mathcal{Z}_{(\mathcal{L}_j',\mathcal{M}_j')_j}^{(a_{ij})}\neq\emptyset$,
    for any $\underline{\eta}\in\mathcal{W}_{(\mathcal{L}_j,\mathcal{M}_j)_j}^{(b_{ij})}$
    and for any
    $\underline{\widetilde{\eta}}\in (H_4^{-1}(\underline{\eta})\cap\mathcal{Z}_{(\mathcal{L}_j',\mathcal{M}_j')_j}^{(a_{ij})}
    \cap\widetilde{\mathcal{U}}_4)(K)$,
    we have
    \begin{equation*}
      \dim_K\mathcal{W}_{(\mathcal{L}_j,\mathcal{M}_j)_j}^{(b_{ij})}+
      \dim_K(H_4^{-1}(\underline{\eta})\cap\mathcal{Z}_{(\mathcal{L}_j',\mathcal{M}_j')_j}^{(a_{ij})})
      +\dim_KG^{-1}(\underline{\widetilde{\eta}})<\dim_KF^{-1}(\W_4^\reg).
    \end{equation*}
    By Lemma \ref{dim of H^-1(eta)cap Z}, we have
    \[
      \dim_K\mathcal{W}_{(\mathcal{L}_j',\mathcal{M}_j')_j}^{(a_{ij})}
      =\dim_K(H_4^{-1}(\underline{\eta})\cap \mathcal{Z}_{(\mathcal{L}_j',\mathcal{M}_j)_j'}^{(a_{ij})}),
    \]
    thus it is enough to show
    \begin{equation}
      \label{wanted inequality for n=4}
      \dim_K\mathcal{W}_{(\mathcal{L}_j,\mathcal{M}_j)_j}^{(b_{ij})}+
      \dim_K\mathcal{W}_{(\mathcal{L}_j',\mathcal{M}_j')_j}^{(a_{ij})}
      +\dim_KG^{-1}(\underline{\widetilde{\eta}})<\dim_KF^{-1}(\W_4^\reg).
    \end{equation}
    Now notice that by Theorem \ref{dim F-1(W)< dim regular locus for n=3}, we have
    \[
      2\dim_K\mathcal{W}_{(\mathcal{L}_j',\mathcal{M}_j')_j}^{(b_{ij})}
      +\dim_KG^{-1}(\underline{\widetilde{\eta}})<\dim_KF^{-1}(\W_4^\reg)
    \]
    whenever $\sum_{j=2}^4|\mathcal{L}_j'|+|\mathcal{M}_j'|\leq 5$.
    Since $\dim_K \mathcal{W}_{(\mathcal{L}_j,\mathcal{M}_j)_j}^{(b_{ij})}\leq \dim_K\mathcal{W}_{(\mathcal{L}_j',\mathcal{M}_j')_j}^{(a_{ij})}$,
    we have that (\ref{wanted inequality for n=4}) holds when $\sum_{j=2}^4|\mathcal{L}_j'|+|\mathcal{M}_j'|\leq 5$.
    Notice that $|\{i<j\colon 2\leq j\leq 4\}|=6$, thus it is only left to show that the inequality
    (\ref{wanted inequality for n=4}) holds when $\sum_{j=2}^4(|\mathcal{L}_j|+|\mathcal{M}_j|)=6$ and
    $\mathcal{L}_j'=\mathcal{L}_j,\mathcal{M}_j'=\mathcal{M}_j$ for all $2\leq j\leq 4$.
    We will show that for any $y_2\in G_2^{-1}(\underline{\widetilde{\eta}})$,
    for any $y_3\in G_3^{-1}(y_2,\widetilde{\eta}_3,\widetilde{\eta}_4)$ and for any
    $y_4\in G_4^{-1}(y_{3},\widetilde{\eta}_4)$ we have
    \begin{equation}
      \label{new equation}
      \begin{split}
        2\dim_K\mathcal{W}_{(\mathcal{L}_j,\mathcal{M}_j)_j}^{(b_{ij})}&+\dim_KG_2^{-1}(\underline{\widetilde{\eta}})
        +\dim_K G_3^{-1}(y_{2},\widetilde{\eta}_3,\widetilde{\eta}_4)+\dim_K G_4^{-1}(y_3,\widetilde{\eta}_4)\\
        &<\dim_KF^{-1}(\W_4^\reg).
      \end{split}
    \end{equation}
    By Proposition \ref{fiber dimension of G_i}, we have
    \begin{align*}
      \dim G_i^{-1}(y_{i-1},\widetilde\eta_i,\dots,\widetilde\eta_n)&
      =\dim_KH_{\phi,\Gamma}^1((\D_{i-1}\widehat\otimes k(y_{i-1}))(\widetilde\eta_i^{-1}))\\
      &+d_i-\dim_KH_{\phi,\Gamma}^1((\D_{i-1}'\widehat\otimes k(y'_{i-1}))(x^{\mu_i}\widetilde\eta_i^{-1})).
    \end{align*}
    for all $2\leq i\leq 4$.
    For simplicity, let $E_i\coloneqq\D_{i}\widehat\otimes k(y_{i})$ and $E_i'\coloneqq\D_{i}'\widehat\otimes k(y'_{i})$.
    Hence the inequality (\ref{new equation}) becomes
    \begin{equation}
      \begin{split}
        2\dim_K\mathcal{W}_{(\mathcal{L}_j,\mathcal{M}_j)_j}^{(b_{ij})}&+
        \sum_{i=2}^4(\dim_KH_{\phi,\Gamma}^1(E_{i-1}(\widetilde\eta_i^{-1}))-\dim_KH_{\phi,\Gamma}^1(E_{i-1}'(x^{\mu_i}\widetilde\eta_i^{-1})))
        <2n.
      \end{split}
    \end{equation}
    By Theorem \ref{poincare duality}, we have
    \begin{align*}
      &\dim_KH_{\phi,\Gamma}^1(E_{i-1}(\widetilde\eta_i^{-1}))-\dim_KH_{\phi,\Gamma}^1(E_{i-1}'(x^{\mu_i}\widetilde\eta_i^{-1}))\\
      &=\dim_KH^0_{\phi,\Gamma}(E_{i-1}(\widetilde\eta_i^{-1}))+\dim_KH^2_{\phi,\Gamma}(E_{i-1}(\widetilde\eta_i^{-1}))\\
      &-\dim_KH^0_{\phi,\Gamma}(E_{i-1}'(x^{\mu_i}\widetilde\eta_i^{-1}))-\dim_KH^2_{\phi,\Gamma}(E_{i-1}'(x^{\mu_i}\widetilde\eta_i^{-1})).
    \end{align*}
    First of all notice that by the way we have chosen the neighbourhood $\widetilde{\mathcal{U}}_4'$, we have
    $x^{-\mu_i}\widetilde{\eta}_i/x^{-\mu_j}\widetilde{\eta}_j\notin\T^+(K)$
    for all $1\leq i< j\leq 4$, hence $\dim_KH^2_{\phi,\Gamma}(E_{i-1}'(x^{\mu_i}\widetilde\eta_i^{-1}))=0$
    by Lemma \ref{bound for dim H^2}.
    Moreover observe that since $\underline{\widetilde{\eta}}\in \mathcal{Z}_{(\mathcal{L}_j,\mathcal{M}_j)_j}^{(a_{ij})}$,
    we have that
    $\widetilde{\eta}_i/\widetilde{\eta}_j=\delta_i/\delta_j$ for all $i\in\mathcal{L}_j\cup\mathcal{M}_j$
    and $2\leq j\leq 4$.
    Since we are treating the case where $\sum_{j=2}^4|\mathcal{L}_j|+|\mathcal{M}_j|=6$, we have that
    $\widetilde{\eta}_i/\widetilde{\eta}_j=\delta_i/\delta_j$ for all $1\leq i<j\leq 4$.
    In particular this tells us that
    $\mathrm{wt}(x^{-\mu_i}\widetilde{\eta}_i/x^{-\mu_j}\widetilde{\eta}_j)\notin\N_{>0}$.
    By Proposition \ref{prop:high dim twist}, we have that
    \[
      H^1_{\phi,\Gamma}(E_{i-1}'(x^{\mu_i}\widetilde\eta_i^{-1}))\cong H^1_{\phi,\Gamma}(E_{i-1}'(x^{\mu_i}\widetilde\eta_i^{-1})[1/t])
    \]
    for all $2\leq i\leq n$.
    Moreover by Lemma \ref{satisfy hypothesis of bound H^0} we have that the $\phigam$-modules
    $E_{i-1},E_{i-1}'$ satisfy
    the hypothesis of Lemma \ref{bound of H^0} and hence
    \[
      \dim_KH^0_{\phi,\Gamma}(E_{i-1}(\widetilde\eta_i^{-1}))-\dim_KH^0_{\phi,\Gamma}(E_{i-1}'(x^{\mu_i}\widetilde\eta_i^{-1}))\leq 0.
    \]
    It is thus sufficient to prove that
    \begin{equation}
      \label{better inequality for n=4}
      \begin{split}
        2\dim_K\mathcal{W}_{(\mathcal{L}_j,\mathcal{M}_j)_j}^{(b_{ij})}&+
        \sum_{i=2}^4\dim_KH_{\phi,\Gamma}^2(E_{i-1}(\widetilde\eta_i^{-1}))<2n.
      \end{split}
    \end{equation}
    By Lemma \ref{bound for dim H^2} we have that
    \[
      \dim_KH_{\phi,\Gamma}^2(E_{i-1}(\widetilde\eta_i^{-1}))<|\mathcal{M}_i|
    \]
    for all $2\leq i\leq n$.
    It is then enough to show that
    \[
      2\dim_K\mathcal{W}_{(\mathcal{L}_j,\mathcal{M}_j)_j}^{(b_{ij})}+
        \sum_{i=2}^4|\mathcal{M}_i|<2n
    \]
    holds, but this true by hypothesis.
  \end{proof}

  \begin{rem}
    \label{when is condition necessary}
    Let
    \[
      \begin{array}{l l l}
        \mathcal{L}_j\coloneqq \{i<j\colon (\delta_i/\delta_j)_{|\Z_p^\times}\in x^{-\N}\}
        &\text{ and }
        &\mathcal{M}_j\coloneqq\{i<j\colon (\delta_i/\delta_j)_{|\Z_p^\times}\in x^{\N_{>0}}\}
      \end{array}
    \]
    for any $2\leq j\leq 4$.
    Notice in the proof of Theorem \ref{dim F-1(W)< dim regular locus for n=4} that the condition
    (\ref{condition}) is necessary only when $\sum_{j=2}^4|\mathcal{L}_j|+|\mathcal{M}_j|=6$
    and $\delta_i/\delta_j\notin\T^\reg(L)$ for all $1\leq i<j\leq 4$.
    Otherwise the claim holds by Theorem \ref{dim F-1(W)< dim regular locus for n=3} without the
    need of assuming (\ref{condition}).
  \end{rem}

  \begin{cor}
    \label{dim F-1(W) for n=3,4}
    For $n=3,4$ and for any $\mathcal{W}_{(\mathcal{L}_j,\mathcal{M}_j)_j}^{(b_{ij})}$ we have
    \[
      \dim_KF^{-1}(\mathcal{W}_{(\mathcal{L}_j,\mathcal{M}_j)_j}^{(b_{ij})})<\dim_KF^{-1}(\W_n^\reg).
    \]
  \end{cor}

  \begin{proof}
    We treat the two cases $n=3$ and $n=4$ separately.
    \begin{itemize}
      \item[$n=3$:]
        In this case we have that $\sum_{j=1}^3|\mathcal{L}_j|+|\mathcal{M}_j|\leq 3$
        for any $\mathcal{L}_j,\mathcal{M}_j$.
        Hence the claim descends directly from Theorem \ref{dim F-1(W)< dim regular locus for n=3}.
      \item[$n=4:$]
        If $\sum_{j=1}^4|\mathcal{L}_j|+|\mathcal{M}_j|\leq 5$, then the wanted inequality
        holds by Theorem \ref{dim F-1(W)< dim regular locus for n=3}.
        If $\sum_{j=1}^n|\mathcal{L}_j|+|\mathcal{M}_j|= 6$, then we distinguish two cases:
        \begin{itemize}
          \item
            if $\delta_i/\delta_j\in\T^\reg(L)$ for some $1\leq i<j\leq 4$,
            then by Remark \ref{when is condition necessary} we have that
            \[
              \dim_KF^{-1}(\mathcal{W}_{(\mathcal{L}_j,\mathcal{M}_j)_j}^{(b_{ij})})<\dim_KF^{-1}(\W_n^\reg)
            \]
            as wanted.
          \item
            If instead $\delta_i/\delta_j\notin\T^\reg(L)$ for all $1\leq i<j\leq 4$,
            then it means that there is
            $\underline{\widetilde{\eta}}\in\widetilde{\mathcal{U}}_n\cap\mathcal{Z}_{(\mathcal{L}_j,\mathcal{M}_j)_j}^{(a_{ij})}$
            such that
            $\widetilde{\eta}_i/\widetilde{\eta_j}=x^{-a_{ij}}\Leftrightarrow i\in\mathcal{L}_j$ and
            $\widetilde{\eta}_i/\widetilde{\eta_j}=\chi x^{a_{ij}}\Leftrightarrow i\in\mathcal{M}_j$
            for all $2\leq j\leq 4$.
            As a consequence we have
            $x^{-\mu_i}\widetilde{\eta}_i/x^{-\mu_j}\widetilde{\eta_j}=x^{-a_{ij}-\mu_i+\mu_j}\Leftrightarrow i\in\mathcal{L}_j$ and
            $x^{-\mu_i}\widetilde{\eta}_i/x^{-\mu_j}\widetilde{\eta_j}=\chi x^{a_{ij}-\mu_i+\mu_j}\Leftrightarrow i\in\mathcal{M}_j$
            for all $2\leq j\leq 4$.
            Since $(x^{-\mu_1}\widetilde{\eta}_1,\dots,x^{-\mu_4}\widetilde{\eta_4})\in\widetilde{\mathcal{U}}_4'(K)$,
            by Remark \ref{good U_n'} we have
            \[
              x^{-\mu_i}\delta_i/x^{-\mu_j}\delta_j=x^{-a_{ij}-\mu_i+\mu_j} \text{ for all } i\in\mathcal{L}_j
            \]
            and
            \[
              x^{-\mu_i}\delta_i/x^{-\mu_j}\delta_j=\chi x^{a_{ij}-\mu_i+\mu_j}\text{ for all } i\in\mathcal{M}_j
            \]
            and for all $2\leq j\leq 4$.
            Obviously we get $\delta_i/\delta_j=x^{-a_{ij}}$ for all $ i\in\mathcal{L}_j$ and
            $\delta_i/\delta_j=\chi x^{a_{ij}}$ for all $ i\in\mathcal{M}_j$ and for all $2\leq j\leq 4$.
            In particular we have $\mathcal{L}_j= \{i<j\colon \delta_i/\delta_j= x^{-a_{ij}}\}$ and
            $\mathcal{M}_j=\{i<j\colon \delta_i/\delta_j= \chi x^{a_{ij}}\}$ for all
            $2\leq j\leq 4$.
            It is easy to see that there are only three possibilities
            so that $\sum_{j=2}^4|\mathcal{L}_j|+|\mathcal{M}_j|=6$ and $\sum_{j=2}^4|\mathcal{M}_j|\geq 1$:
            \begin{enumerate}
              \item
                \[
                  \begin{array}{lcr}
                    \delta_1/\delta_2=\chi x^{a_{12}}, &\delta_1/\delta_3=\chi x^{a_{13}}, &\delta_1/\delta_4=\chi x^{a_{14}},\\
                    \delta_2/\delta_3=x^{-a_{12}+a_{13}}=x^{-a_{23}}, &\delta_2/\delta_4= x^{-a_{12}+a_{14}}=x^{-a_{24}}, &\\
                    \delta_3/\delta_4= x^{-a_{13}+a_{14}}=x^{-a_{14}}. & &
                  \end{array}
                \]
              \item
                \[
                  \begin{array}{lcr}
                    \delta_1/\delta_2= x^{-a_{12}}, &\delta_1/\delta_3=\chi x^{a_{13}}, &\delta_1/\delta_4=\chi x^{a_{14}},\\
                    \delta_2/\delta_3=\chi x^{a_{12}+a_{13}}=\chi x^{a_{23}}, &\delta_2/\delta_4=\chi x^{a_{12}+a_{14}}=\chi x^{a_{24}}, &\\
                    \delta_3/\delta_4= x^{-a_{13}+a_{14}}=x^{-a_{14}}. & &
                  \end{array}
                \]
              \item
                \[
                  \begin{array}{lcr}
                    \delta_1/\delta_2= x^{-a_{12}}, &\delta_1/\delta_3= x^{-a_{13}}, &\delta_1/\delta_4=\chi x^{a_{14}},\\
                    \delta_2/\delta_3= x^{a_{12}-a_{13}}=x^{-a_{23}}, &\delta_2/\delta_4=\chi x^{a_{12}+a_{14}}=\chi x^{a_{24}}, &\\
                    \delta_3/\delta_4=\chi x^{a_{13}+a_{14}}=\chi x^{a_{14}}. & &
                  \end{array}
                \]
            \end{enumerate}
            For the three above cases we have
            \[
              2\dim_K\mathcal{W}_{(\mathcal{L}_j,\mathcal{M}_j)_j}^{(b_{ij})}+\sum_{j=2}^4|\mathcal{M}_j|=2+3=5<8.
            \]
            This shows that the condition (\ref{condition}) is always satisfied,
            hence we can conclude by Theorem \ref{dim F-1(W)< dim regular locus for n=4}.
        \end{itemize}
    \end{itemize}
  \end{proof}

  \begin{lem}
    \label{description of pi-1(z')}
    Recall that by construction (Theorem \ref{extensions}) we have
    $X_n'\cong \widetilde{\mathcal{U}}_n'\times\times\mathbb{A}_L^{d_2,\an}\times\dots\times\mathbb{A}_L^{d_n,\an}$,
    where $d_i\coloneqq\dim_LH^1_{\phigam}(D_{i-1}'(\delta_i^{-1}x^{\mu_i}))$.
    Let $z'\coloneqq (\underline{\widetilde{\eta}}',\underline{a})
    \in(\widetilde{\mathcal{U}}_n'\times\times\mathbb{A}_L^{d_2,\an}\times\dots\times\mathbb{A}_L^{d_n,\an})(K)=X'_n(K)$
    for some extension $K$ of $L$
    and let
    \[
      p:X_n\rightarrow X_n'
    \]
    be the projection map.
    We have that the fiber $p^{-1}(z')$ is in bijection with the set
    \[
      \left\{
        \begin{array}{c}
          \phigam\text{-modules }E_n\text{ over }\robba_{K}\text{ of parameters }
          (\widetilde{\eta}_1'x^{\mu_1},\dots,\widetilde{\eta}'_nx^{\mu_n})\\
          \text{ such that }E_n[1/t]=z'^*\D_n'[1/t]
        \end{array}
      \right\}.
    \]
  \end{lem}

  \begin{proof}
    By definition of $X_n$ we have that the fiber $p^{-1}(z')$ is the set
    \begin{align*}
      \{\Lambda_z\in Y_{(\mu_1,\dots,\mu_n)}\colon
      \Lambda_z\text{ is a }\Gamma_m\text{-stable lattice of }z'^*\D_n'^r[1/t]\otimes_{\robba_K^r,\tau}K((t))\}.
    \end{align*}
    By Proposition \ref{bijection of lattices and modules}, we know that there is a bijection between the two sets
    \begin{align*}
      \{\Lambda_z\in \Gr_n\colon
      \Lambda_z\text{ is a }\Gamma_m\text{-stable lattice of }z'^*\D_n'^r[1/t]\otimes_{\robba_K^r,\tau}K((t))\}
    \end{align*}
    and
    \begin{align*}
      \{\phigam\text{-modules }E_n\text{ over }\robba_K\text{ such that }E_n[1/t]=z'^*\D_n'[1/t]\}.
    \end{align*}
    It is then enough to show that if $E_n$ has parameters $(\widetilde{\eta}_1'x^{\mu_1},\dots,\widetilde{\eta}'_nx^{\mu_n})$,
    then the corresponding $\Lambda_z$ through the bijection in Proposition \ref{bijection of lattices and modules}
    is in $Y_{(\mu_1,\dots,\mu_n)}$.
    Let us denote by $\Lambda_z'$ the $\Gamma_m$-stable lattice of $z'^*\D_n'^r[1/t]\otimes_{\robba_K^r,\tau}K((t))$
    associated to $z'^*\D_n'^r$.
    Observe that
    \begin{align*}
      \gr^1(\Lambda_z)&=\Lambda_z\cap\Fil^1(K((t))^n)
      =(E_n^r\otimes_{\robba_K^r,\tau}K\llbracket t\rrbracket)\cap(\robba_K^r(\widetilde{\eta}_1')[1/t]\otimes_{\robba_K^r[1/t],\tau}K((t)) )\\
      &=\robba_K^r(x^{\mu_1}\widetilde{\eta}_1')\otimes_{\robba_K^r,\tau}K\llbracket t\rrbracket
      =t^{\mu_1}\gr^1(\Lambda_z').
    \end{align*}
    Moreover for every $2\leq i\leq n$ let $E_i$ be the $i$-th piece of the filtration of
    $E_n$ giving parameters $(\widetilde{\eta}_1'x^{\mu_1},\dots,\widetilde{\eta}'_nx^{\mu_n})$.
    We have
    \begin{align*}
      \Fil^i(\Lambda_z)&=\Lambda_z\cap\Fil^i(K((t))^n)
      =(E_n^r\otimes_{\robba_K^r,\tau}K\llbracket t\rrbracket)\cap(z'^*\D_i'^r[1/t]\otimes_{\robba_K^r[1/t],\tau}K((t)))\\
      &=E_i^r\otimes_{\robba_K^r,\tau}K\llbracket t\rrbracket.
    \end{align*}
    Therefore
    \begin{align*}
      \gr^i(\Lambda_z)&=\Fil^i(\Lambda_z)/\Fil^{i-1}(\Lambda_z)
      =(E_i^r\otimes_{\robba_K^r,\tau}K\llbracket t\rrbracket)/(E_{i-1}^r\otimes_{\robba_K^r,\tau}K\llbracket t\rrbracket)\\
      &\cong\robba_K^r(x^{\mu_i}\widetilde{\eta}_i')\otimes_{\robba_K^r,\tau}K\llbracket t\rrbracket
      = t^{\mu_i}\gr^i(\Lambda_z').
    \end{align*}
    The above computations show that $\Lambda_z\in Y_{(\mu_1,\dots,\mu_n)}$ and ends the proof.
  \end{proof}

  \begin{lem}
    \label{homeo of regular locus}
    Let $W$ be the intersection of $\W_n^\reg$ and $\im(F)$ in $\W^n$, let
      \[
        W'\coloneqq\{(\eta_1x^{-\mu_1},\dots,\eta_n x^{-\mu_n})\colon(\eta_1,\dots,\eta_n)\in W\}\subset H_n(\widetilde{\mathcal{U}}_n')
      \]
       and let $T'$ be the intersection of $\widetilde{\mathcal{U}}_n'$ and $H_n^{-1}(W')$ in $\T^n$.
      The map
      \begin{equation}
        \label{homeomorphism}
        p:F^{-1}(W)\rightarrow T'\times\mathbb{A}_L^{d_2,\an}\times\dots\times\mathbb{A}_L^{d_n,\an}
        \subset\widetilde{\mathcal{U}}_n'\times\mathbb{A}_L^{d_2,\an}\times\dots\times\mathbb{A}_L^{d_n,\an}=X'
      \end{equation}
      given by the restriction of the map $X_n\rightarrow X'_n$ to $F^{-1}(W)$ is a homeomorphism.
  \end{lem}

  \begin{proof}
    First of all we show that if $z=(z',\Lambda_z)\in F^{-1}(W)$,
    then $z'\in T'\times\mathbb{A}_L^{d_2,\an}\times\dots\times\mathbb{A}_L^{d_n,\an}$;
    in other words we show that the parameters $\underline{\widetilde{\eta}}'$
    of $z'^*\D'_n$ are in $T'$:
    in fact by Theorem \ref{construction of X_n}, we have that $z^*\D_n$ has parameters
    $(\widetilde{\eta}_1'x^{\mu_1},\dots,\widetilde{\eta}_n'x^{\mu_n})$;
    moreover we know $z\in F^{-1}(W)$, which implies
    $H_n(\widetilde{\eta}_1'x^{\mu_1},\dots,\widetilde{\eta}_n'x^{\mu_n})\in W$;
    as a consequence we have $H_n(\underline{\widetilde{\eta}}')\in W'$, hence
    $\underline{\widetilde{\eta}}\in T'$.
    We have thus showed that the map $p$ is well-defined.
    Now let $T\coloneqq H_n^{-1}(W)\cap \widetilde{\mathcal{U}}_n$.
    Observe that
    \[
      \underline{\widetilde{\eta}}\in T\Leftrightarrow
      (\widetilde{\eta}_1x^{-\mu_1},\dots,\widetilde{\eta}_nx^{-\mu_n})\in T':
    \]
    in fact by definition we have that
    $\underline{\widetilde{\eta}}\in T=H_n^{-1}(W)\cap \widetilde{\mathcal{U}}_n$
    if and only if $H_n(\widetilde{\eta}_1x^{-\mu_1},\dots,\widetilde{\eta}_nx^{-\mu_n})\in W'$ and
    $(\widetilde{\eta}_1x^{-\mu_1},\dots,\widetilde{\eta}_nx^{-\mu_n})\in \widetilde{\mathcal{U}}_n'$,
    which means that
    $(\widetilde{\eta}_1x^{-\mu_1},\dots,\widetilde{\eta}_nx^{-\mu_n})\in H_n^{-1}(W')\cap \widetilde{\mathcal{U}}_n'= T'$.
    We will now prove that the map
    \begin{align*}
      p:F^{-1}(W)=\bigcup_{\underline{\widetilde{\eta}}\in T}G^{-1}(\underline{\widetilde{\eta}})
        \rightarrow &\bigcup_{\underline{\widetilde{\eta}}'\in T'}
        \underline{\widetilde{\eta}}'\times\mathbb{A}_L^{d_2,\an}\times\dots\times\mathbb{A}_L^{d_n,\an}\\
        &=T'\times\mathbb{A}_L^{d_2,\an}\times\dots\times\mathbb{A}_L^{d_n,\an}
    \end{align*}
    is a bijection.
    We will show this by proving that for every $\underline{\widetilde{\eta}}\in T(K)$ for some
    extension $K$ of $L$, we have
    that the map
    \[
      p:G^{-1}(\underline{\widetilde{\eta}})\rightarrow
      (\widetilde{\eta}_1x^{-\mu_1},\dots,\widetilde{\eta}_nx^{-\mu_n})
      \times\mathbb{A}_L^{d_2,\an}\times\dots\times\mathbb{A}_L^{d_n,\an}
    \]
    given by the restriction of the map $p$ to $G^{-1}(\underline{\widetilde{\eta}})$ is a bijection.
    Let $z'\coloneqq (\widetilde{\eta}_1x^{-\mu_1},\dots,\widetilde{\eta}_nx^{-\mu_n})\times\underline{a}\in
    (\widetilde{\eta}_1x^{-\mu_1},\dots,\widetilde{\eta}_nx^{-\mu_n})
    \times\mathbb{A}_L^{d_2,\an}\times\dots\times\mathbb{A}_L^{d_n,\an}$.
    By Lemma \ref{description of pi-1(z')}, we have that $p^{-1}(z')$ is in bijection with the set
    \[
      \{\phigam\text{-modules }E_n\text{ over }\robba_K\text{ of parameters }\underline{\widetilde{\eta}}
      \text{ such that }E_n[1/t]=z'^*\D_n'[1/t]\}
    \]
    and we will now show that such set has cardinality 1, which implies the desired bijection.
    Since $H_n(\underline{\widetilde{\eta}})\in\W_n^\reg(K)$, we have that $\mathrm{wt}(\widetilde{\eta}_i/\widetilde{\eta}_j)\notin\N_{>0}$
    for all $1\leq i<j\leq n$.
    This in particular implies that $\mathrm{wt}(x^{-\mu_i}\widetilde{\eta}_i/x^{-\mu_j}\widetilde{\eta}_j)\notin\N_{>0}$
    for all $1\leq i<j\leq n$ by Remark \ref{good U_n'}.
    Thus by Proposition \ref{prop:high dim twist}, we have the following isomorphism:
    \[
      H^1_{\phi,\Gamma}(\robba_K(\widetilde{\eta}_1/\widetilde{\eta}_2))
      \cong H^1_{\phi,\Gamma}(\robba_K(\widetilde{\eta}_1/\widetilde{\eta}_2)[1/t])
      \cong H^1_{\phi,\Gamma}(\robba_K(x^{-\mu_1}\widetilde{\eta}_1/x^{-\mu_2}\widetilde{\eta}_2)).
    \]
    Let us denote by $E_2$ the extension of $\robba_K(\widetilde{\eta}_2)$ by $\robba_K(\widetilde{\eta}_1)$
    corresponding to the extension $z'^*\D_2'\in H^1_{\phi,\Gamma}(\robba_K(x^{-\mu_1}\widetilde{\eta}_1/x^{-\mu_2}\widetilde{\eta}_2))$
    through the isomorphism above.
    In an iterative manner, for every $3\leq i\leq n$ we define the $\phigam$-module $E_i$
    to be the extension of $\robba_K(\widetilde{\eta}_i)$ by $E_{i-1}$ corresponding to the extension
    $z'^*\D_i'\in H^1_{\phi, \Gamma}(z'^*\D_{i-1}'(x^{\mu_i}\widetilde{\eta}_i^{-1}))$ through the
    isomorphism
    \[
      H^1_{\phi, \Gamma}(E_{i-1}(\widetilde{\eta}_i^{-1}))\cong H^1_{\phi, \Gamma}(E_{i-1}(\widetilde{\eta}_i^{-1})[1/t])
      \cong H^1_{\phi, \Gamma}(z'^*\D_{i-1}'(x^{\mu_i}\widetilde{\eta}_i^{-1}))
    \]
    given by Proposition \ref{prop:high dim twist}.
    We have then found a unique $E_n$ of parameters $\underline{\widetilde{\eta}}$ such that
    $E_n[1/t]=z'^*\D_n'[1/t]$ and this proves that the above map is a bijection.
    Finally we have that
    \begin{align*}
      F^{-1}(W)&=G^{-1}(H_n^{-1}(W)\cap\widetilde{\mathcal{U}}_n)=G^{-1}(T)\\
      &=X_n\times_{\widetilde{\mathcal{U}}_n}T
      =X_n'\times\Gr_n\times_{\Gr_n\times\Gr_n}Y_{(\mu_1,\dots,\mu_n)}\times_{\widetilde{\mathcal{U}}_n}T\\
      & = \widetilde{\mathcal{U}}_n'\times\mathbb{A}_L^{d_2,\an}\times\dots\times\mathbb{A}_L^{d_n,\an}
      \times\Gr_n\times_{\Gr_n\times\Gr_n}Y_{(\mu_1,\dots,\mu_n)}\times_{\widetilde{\mathcal{U}}_n}T\\
      & = \widetilde{\mathcal{U}}_n'\times_{\widetilde{\mathcal{U}}_n}T\times\mathbb{A}^{d_2,\an}_L\times\dots\times\mathbb{A}_L^{d_n,\an}
      \times\Gr_n\times_{\Gr_n\times\Gr_n}Y_{(\mu_1,\dots,\mu_n)}\\
      & = T'\times\mathbb{A}^{d_2,\an}_L\times\dots\times\mathbb{A}_L^{d_n,\an}\times\Gr_n\times_{\Gr_n\times\Gr_n}Y_{(\mu_1,\dots,\mu_n)}.
    \end{align*}
    Recall that for locally ringed spaces $X$, $Y$ over a field $k$, the natural projection $X\times_kY\rightarrow X$
    is open.
    In our case, this translates to the map
    \[
      p:T'\times\mathbb{A}_L^{d_2,\an}\times\dots\times\mathbb{A}_L^{d_n,\an}\times\Gr_n\times_{\Gr_n\times\Gr_n}Y_{(\mu_1,\dots,\mu_n)}
      \rightarrow T'\times\mathbb{A}_L^{d_2,\an}\times\dots\times\mathbb{A}_L^{d_n,\an}
    \]
    being open.
    We have then showed that $p$ is continuous, open and a bijection, hence a homeomorphism.
  \end{proof}

  \begin{thm}
    \label{X_n is irreducible}
    If $n=3,4$ or if $|\{\delta_i/\delta_j\in\T\setminus\T^\reg\colon i<j\}|\leq5$,
    the irredeucible component $Z$ of $X_n$ containing the point $x_n$ is equal to the closure of $F^{-1}(\W_n^\reg)$.
  \end{thm}

  \begin{proof}
    Let $W\coloneqq\W_n^\reg\cap\im(F)\subset\W^n$
    and let $Z'$ be the closure of $F^{-1}(W)$ in $X_n$.
    By Lemma \ref{homeo of regular locus} we have that $F^{-1}(W)$ is irreducible,
    thus $Z'$ is irreducible, as closure of irreducible.
    Now we show that $Z'$ is an irreducible component of $X_n$: assume $\widetilde Z\subset Z'$
    is an irreducible component of $X_n$ containing $Z'$.
    Let $W_0$ be an open in $W$, then $F^{-1}(W_0)$ is open in $Z'$ and $\widetilde Z$, hence
    \[
      \dim\widetilde Z=\dim F^{-1}(W_0)=\dim Z'.
    \]
    This shows that $Z'=\widetilde Z$, since they are both irreducible of the same dimension.
    Now we prove that $Z'=Z$:
    let
    \[
      V\coloneqq Z\setminus\bigcup_{Z''\neq Z}Z''\subset Z
    \]
    where the union ranges over all the irreducible components $Z''$ of $X_n$ that are not $Z$.
    Notice that $V$ is open in $Z$, hence by Theorem \ref{dimZ}
    and Remark \ref{Wreg and X' same dim} we have
    \[
      \dim V=\dim Z\geq \dim X_n'=\dim F^{-1}(\W_n^\reg).
    \]
    Finally, we have $V\subseteq F^{-1}(F(V))$, thus
    \[
      \dim F^{-1}(F(V))\geq \dim V\geq \dim F^{-1}(\W_n^\reg).
    \]
    On the other hand, by Corollary \ref{dim F-1(W) for n=3,4} and Corollary \ref{dim F-1(W) for regular enough}, for every
    $\mathcal{L}_j\subseteq\{i<j\colon (\delta_i/\delta_j)_{|\Z_p^\times}\in x^{-\N}\}$ and
    $\mathcal{M}_j\subseteq\{i<j\colon (\delta_i/\delta_j)_{|\Z_p^\times}\in x^{\N_{>0}}\}$
    we have that $\dim_KF^{-1}(\mathcal{W}_{(\mathcal{L}_j,\mathcal{M}_j)_j})<\dim_KF^{-1}(\W_n^\reg)$.
    Therefore we must have $F^{-1}(F(V))\cap F^{-1}(\W_n^\reg)\neq\emptyset$, which
    is a contradiction, as $F(V)\cap\W_n^\reg=\emptyset$ by definition of $V$.
  \end{proof}

  \begin{rem}
    Theorem \ref{X_n is irreducible} concludes the proof of Theorem \ref{result2}.
    Notice that the only part in which we need the assumption $n\leq 4$ or the assumption
    $|\{\delta_i/\delta_j\in\T\setminus\T^\reg\colon i<j\}|$ is for
    giving a bound to the dimension of the sets of the form $F^{-1}(\mathcal{W}^{(b_{ij})}_{(\mathcal{L}_j,\mathcal{M}_j)_j})$.
  \end{rem}

\bibliography{triangrep}
\bibliographystyle{amsalpha}

\end{document}

%% file: packages.tex
\PassOptionsToPackage{usenames,dvipsnames}{xcolor}
\usepackage{amsmath}
\usepackage{amsfonts}
\usepackage{amssymb}
\usepackage{mathtools}
\usepackage{stmaryrd}
\usepackage[utf8]{inputenc}
\usepackage[english]{babel}
\usepackage[T1]{fontenc}
\usepackage{amsthm}
\usepackage{comment}
\usepackage{tikz-cd}
\usepackage{enumerate}
\usepackage[numbers,square]{natbib}
\usepackage{thmtools, thm-restate}
\usepackage[hyperindex,breaklinks]{hyperref}
\usepackage{cleveref}
\usepackage{graphicx}
\usepackage{amsxtra}
\usepackage{fancyhdr}
\usepackage{etoolbox}
\usepackage{pdfpages}
\usepackage{multicol}
\usepackage{afterpage}

\input{shortcuts}

%% file: shortcuts.tex
\usepackage{xparse}
\usepackage{ifthen}

\ExplSyntaxOn
\NewDocumentCommand{\makeabbrev}{mmm}
 {
  \yoruk_makeabbrev:nnn { #1 } { #2 } { #3 }
 }

\cs_new_protected:Npn \yoruk_makeabbrev:nnn #1 #2 #3
 {
  \clist_map_inline:nn { #3 }
   {
    \cs_new_protected:cpn { #2 } { #1 { ##1 } }
   }
 }
\ExplSyntaxOff

\makeabbrev{\textsf}  {sf#1} {a,b,c,d,e,f,g,h,i,j,k,l,m,n,o,p,q,r,s,t,u,v,w,x,y,z,A,B,C,D,E,F,G,H,I,J,K,L,M,N,O,P,Q,R,S,T,U,V,W,X,Y,Z}
\makeabbrev{\mathfrak}{fk#1} {a,b,c,d,e,f,g,h,i,j,k,l,m,n,o,p,q,r,s,t,u,v,w,x,y,z,A,B,C,D,E,F,G,H,I,J,K,L,M,N,O,P,Q,R,S,T,U,V,W,X,Y,Z}
\makeabbrev{\mathbf}  {bf#1} {a,b,c,d,e,f,g,h,i,j,k,l,m,n,o,p,q,r,s,t,u,v,w,x,y,z,A,B,C,D,E,F,G,H,I,J,K,L,M,N,O,P,Q,R,S,T,U,V,W,X,Y,Z}
\makeabbrev{\mathrm}  {rr#1} {a,b,c,d,e,f,g,h,i,j,k,l,m,n,o,p,q,r,s,t,u,v,w,x,y,z,A,B,C,D,E,F,G,H,I,J,K,L,M,N,O,P,Q,R,S,T,U,V,W,X,Y,Z}
\makeabbrev{\mathcal} {cal#1}{A,B,C,D,E,F,G,H,I,J,K,L,M,N,O,P,Q,R,S,T,U,V,W,X,Y,Z}
\makeabbrev{\mathbb}  {bb#1} {A,B,C,D,E,F,G,H,I,J,K,L,M,N,O,P,Q,R,S,T,U,V,W,X,Y,Z}
\makeabbrev{\mathscr} {scr#1}{A,B,C,D,E,F,G,H,I,J,K,L,M,N,O,P,Q,R,S,T,U,V,W,X,Y,Z}

\newcommand{\mcal} [1]{\mathcal{#1}}
\newcommand{\rr} [1]{\mathrm{#1}}

\usepackage{scalerel,stackengine}
\stackMath
\newcommand\rlywdhat[1]{
  \savestack{\tmpbox}{\stretchto{
    \scaleto{\scalerel*[\widthof{\ensuremath{#1}} * \real{0.8}]
      {\kern.1pt\mathchar"0362\kern.1pt}{\rule{0ex}{\textheight}}}
    {\textheight}
    }{1.8ex}
  }
\stackon[-7pt]{#1}{\tmpbox}
}

\renewcommand{\geq}{\geqslant}
\renewcommand{\leq}{\leqslant}

%% file: definitions.tex
\theoremstyle{definition}
\newtheorem{thm}{Theorem}[section]
\newtheorem{prop}[thm]{Proposition}
\newtheorem{lem}[thm]{Lemma}
\newtheorem{cor}[thm]{Corollary}
\newtheorem{dfn}[thm]{Definition}
\newtheorem*{conj}{Conjecture}
\newtheorem*{notation}{Notation}

\declaretheoremstyle{definition}

\theoremstyle{remark}
\newtheorem{rem}[thm]{Remark}

\tikzset{
    labl/.style={anchor=center, rotate=90, inner sep=.5mm}
}

\newcommand{\xrightarrowdbl}[2][]{%
  \xrightarrow[#1]{#2}\mathrel{\mkern-14mu}\rightarrow
}

\newcommand{\Treg}{\mathcal{T}^{\mathrm{reg}}}

\newcommand{\absGal}{G_{\mathbb{Q}_p}}

\newcommand{\trivar}{X_{\mathrm{tri}}^\square(\overline{\rho})}

\newcommand{\drig}{\mathbb{D}_{\mathrm{rig}}^{\dag}}

\newcommand{\phigam}{(\phi,\Gamma)}

\newcommand{\Category}{(\phi,\Gamma)-\mathrm{Mod}(\mathcal{R}^r_A[1/t])\times_{\mathrm{Mod}(\prod A((t)))^\Gamma}\mathrm{Mod}(\prod_{K_m\hookrightarrow L}A\llbracket t\rrbracket)^\Gamma}

\newcommand{\Categoryr}{(\phi,\Gamma)-\mathrm{Mod}(\mathcal{R}^{r_0}_A[1/t])\times_{\mathrm{Mod}(\prod A((t)))^\Gamma}\mathrm{Mod}(\prod_{K_m\hookrightarrow L}A\llbracket t\rrbracket)^\Gamma}

\newcommand{\robba}{\mathcal{R}}

\newcommand{\defspace}{\mathfrak{X}_{\bar\rho}^\square}

\newcommand{\reg}{\mathrm{reg}}

\newcommand{\tri}{\mathrm{tri}}

\newcommand{\regtrivar}{U_{\mathrm{tri}}^\square(\bar\rho)}

\newcommand{\Qp}{\mathbb{Q}_p}

\newcommand{\T}{\mathcal{T}}

\newcommand{\W}{\mathcal{W}}

\newcommand{\D}{\mathbb{D}}

\newcommand{\Q}{\mathbb{Q}}

\newcommand{\N}{\mathbb{N}}

\newcommand{\Z}{\mathbb{Z}}

\newcommand{\rig}{\mathrm{rig}}

\newcommand{\Gr}{\mathrm{Gr}^{\rig}}

\newcommand{\imF}{\mathrm{im}(F)}

\newcommand{\an}{\mathrm{an}}

\newcommand{\Gm}{\mathbb{G}_m}

\newcommand*{\isoarrow}[1]{\arrow[#1,"\rotatebox{90}{\(\sim\)}"
]}

\newcommand{\im}{\mathrm{im}}

\newcommand{\Spec}{\mathrm{Spec}}

\newcommand{\Spf}{\mathrm{Spf}}

\newcommand{\BdR}{B_{\mathrm{dR}}}

\newcommand{\BdRA}{B_{\mathrm{dR},A}}

\newcommand{\pdR}{\mathrm{pdR}}

\newcommand{\Rep}{\mathrm{Rep}}

\newcommand{\Mod}{\mathrm{Mod}}

\newcommand{\WdR}{W_{\mathrm{dR}}}

\newcommand{\DpdR}{D_{\mathrm{pdR}}}

\newcommand{\Sp}{\mathrm{Sp}}

\newcommand{\Fil}{\mathrm{Fil}}

\newcommand{\st}{\mathrm{st}}

\newcommand{\dR}{\mathrm{dR}}

\newcommand{\gr}{\mathrm{gr}}

\newcommand{\Hom}{\rr{Hom}}